\documentclass{amsart}
\usepackage{amsmath,amssymb,amsthm,amsfonts,epsfig,verbatim,enumerate,comment}
\usepackage{graphicx,xcolor,color,wrapfig}
\usepackage{pxfonts}
\usepackage{hyperref}
\input{xy}
\xyoption{all}
\CompileMatrices

\newcommand{\ed}{{\rm d}}
\newcommand{\w}{{\mathchoice{\,{\scriptstyle\wedge}\,}{{\scriptstyle\wedge}}
      {{\scriptscriptstyle\wedge}}{{\scriptscriptstyle\wedge}}}}
\newcommand{\lhk}{\mathbin{\hbox{\vrule height1.4pt width4pt depth-1pt
             \vrule height4pt width0.4pt depth-1pt}}}

\newcommand{\del}{{\partial}}
\newcommand{\delb}{\bar{\partial}}
\newcommand{\delx}{\partial_{\xi}}
\newcommand{\delxb}{\partial_{\xib}}

\newcommand{\liealgebra}[1]{{\mathfrak {#1}}}
\newcommand{\g}{\liealgebra{g}}

\newcommand{\sla}{\liealgebra{sl}}
\newcommand{\so}{\liealgebra{so}}
\newcommand{\su}{\liealgebra{su}}

\newcommand{\liegroup}[1]{{\operatorname{#1}}}
\newcommand{\G}{\liegroup{G}}
\newcommand{\K}{\liegroup{K}}

\newcommand{\SL}{\liegroup{SL}}
\newcommand{\SO}{\liegroup{SO}}
\newcommand{\SU}{\liegroup{SU}}

\newcommand{\R}{\mathbb R}
\newcommand{\C}{\mathbb C}

\newcommand{\Z}{\mathbb Z}

\newcommand{\E}{\mathbb E}

\newcommand{\PP}{\mathbb P}

\renewcommand{\Re}{\operatorname{Re}}
\renewcommand{\Im}{\operatorname{Im}}

\DeclareMathOperator{\Ad}{Ad}

\newcommand{\tr}{\rm tr}
\newcommand{\dd}[2]{\frac{\partial {#1}}{\partial {#2}}}

\newcommand{\mcb}{\mathcal B}
\newcommand{\mcc}{\mathcal C}

\newcommand{\mce}{\mathcal E}
\newcommand{\mcf}{\mathcal F}

\newcommand{\mch}{\mathcal H}
\newcommand{\mci}{\mathcal I}
\newcommand{\mcj}{\mathcal J}

\newcommand{\mcl}{\mathcal L}
\newcommand{\mcm}{\mathcal M}
\newcommand{\mcn}{\mathcal N}
\newcommand{\mco}{\mathcal O}
\newcommand{\mcp}{\mathcal P}
\newcommand{\mcq}{\mathcal Q}
\newcommand{\mcr}{\mathcal R}
\newcommand{\mcs}{\mathcal S}

\newcommand{\mcu}{\mathcal U}
\newcommand{\mcv}{\mathcal V}
\newcommand{\mcw}{\mathcal W}

\newcommand{\mfj}{\mathfrak J}

\newcommand{\nA}{\textnormal{A}}
\newcommand{\nB}{\textnormal{B}}
\newcommand{\nC}{\textnormal{C}}
\newcommand{\nAb}{\ol{\textnormal{A}}}
\newcommand{\nBb}{\ol{\textnormal{B}}}

\newcommand{\Zh}{\hat{Z}}

\newcommand{\ES}{\textnormal{S}}
\newcommand{\x}{\textnormal{x}}
\newcommand{\xh}{\hat{\textnormal{x}}}
\newcommand{\im}{\textnormal{i}}
\newcommand{\JAI}{\textnormal{J}}

\newcommand{\btheta}{\boldsymbol{\theta}}

\newcommand{\ol}{\overline}

\newcommand{\zb}{\ol{z}}
\newcommand{\hb}{\bar h}

\newcommand{\bEb}{\ol{\bf E}}
\newcommand{\cb}{\ol{c}}
\newcommand{\ab}{\ol{a}}

\newcommand{\Ub}{\ol{U}}

\newcommand{\betab}{\ol{\beta}}
\newcommand{\gammab}{\ol{\gamma}}
\newcommand{\etab}{\ol{\eta}}

\newcommand{\thetab}{\ol{\theta}}
\newcommand{\zetab}{\ol{\zeta}}
\newcommand{\omb}{\ol{\omega}}
\newcommand{\bb}{\ol{b}}

\newcommand{\xib}{\ol{\xi}}

\newcommand{\bE}{{\bf E}}
\newcommand{\bn}{{\bf n}}

\newcommand{\bs}{{\bf s}}
\newcommand{\bt}{{\bf t}}

\newcommand{\s}[1]{{\mathbb S}^{#1}}
\newcommand{\F}[1]{\mcf^{(#1)}}
\newcommand{\Fh}[1]{\hat{\mcf}^{(#1)}}

\newcommand{\X}[1]{X^{(#1)}}
\newcommand{\Xh}[1]{\hat{X}^{(#1)}}

\newcommand{\I}[1]{{\rm I}^{(#1)}}
\newcommand{\Ih}[1]{{\hat{\rm I}}^{(#1)}}

\newcommand{\cp}[1]{{\C \mathbb{P}^{#1}}}
\newcommand{\pr}{{\mathbb{P}}}
\newcommand{\xinf}{X^{(\infty)}}
\newcommand{\xinfh}{\hat{X}^{(\infty)}}
\newcommand{\iinf}{{\rm I}^{(\infty)}}
\newcommand{\iinfh}{\hat{\rm I}^{(\infty)}}

\newcommand{\finf}{{\mcf}^{(\infty)}}
\newcommand{\Cv}[1]{\mathcal{H}^{#1}}
\newcommand{\Hb}{{\bar H}^1}
\newcommand{\Vp}[1]{\mcv^{(#1)}}
\newcommand{\mcuh}{\hat{\mcu}}
\newcommand{\mcjh}{\hat{\mcj}}

\DeclareMathOperator{\Gr}{\rm Gr}

\newcommand{\be}{\begin{equation}}
\newcommand{\ee}{\end{equation}}

\newcommand{\bp}{\begin{pmatrix}}
\newcommand{\ep}{\end{pmatrix}}

\newcommand{\lra}{\longrightarrow}

\newcommand{\hra}{\hookrightarrow}
\newcommand{\p}{\varphi}
\newcommand{\pt}{\tilde{\varphi}}
\newcommand{\ff}{{\rm I \negthinspace I}}
\newcommand{\Sigmah}{\hat{\Sigma}}

\newcommand{\n}{\notag}
\newcommand{\noi}{\noindent}
\newcommand{\np}{\newpage}
\newcommand{\tn}{\textnormal}
\newcommand{\tb}{\textbf}

\newcommand{\hook}{\hookrightarrow}
\newcommand{\marg}{\marginpar}
\newcommand{\one}{\vspace{1mm}}
\newcommand{\two}{\vspace{2mm}}

\newcommand{\qinf}{Q_{\infty}}

\newcommand{\da}{\dot{a}}
\newcommand{\db}{\dot{b}}
\newcommand{\dc}{\dot{c}}

\usepackage[mathscr]{eucal}
\usepackage{bbm}
\usepackage{oldgerm}

\theoremstyle{plain}
\newtheorem{thm}{Theorem}[section]
\newtheorem{lem}[thm]{Lemma}
\newtheorem{conj}[thm]{Conjecture}
\newtheorem{cor}[thm]{Corollary}

\newtheorem{prop}[thm]{Proposition}

\newtheorem{rem}[thm]{Remark}
\newtheorem{quest}[thm]{Question}
\theoremstyle{definition}
\newtheorem{defn}{Definition}[section]

\newtheorem{exam}[thm]{Example}

\begin{document}
\title[Conservation laws for CMC surfaces]{Conservation laws for surfaces of constant mean curvature in $3$-dimensional space forms}
\author{Daniel Fox and Joe S. Wang}
\address{Philadelphia, PA and Seoul, South Korea}
\email{foxdanie@gmail.com, jswang12@gmail.com}
\date{\today}
\dedicatory{Dedicated to our teacher Robert Bryant on the occasion of his $60 $th birthday.}
\thanks{We would like to thank Dominic Joyce, Fran Burstall, Joseph Krasil'shchik, George Bluman, Young-Heon Kim, Chanyoung Sung, Ben Mckay, Tom Ivey, and Abraham Smith  for useful discussions on this topic. 
Wang would like to thank KIAS, UCC, KonKuk University,  Korea University, 
and Texas  A\&M  for the support.}
\begin{abstract}
The exterior differential system for constant mean curvature (CMC) surfaces in a 3-dimensional space form is an elliptic Monge-Ampere system defined on the unit tangent bundle.
We determine the infinite sequence of higher-order symmetries and conservation laws via an enhanced prolongation modelled on a loop algebra valued formal Killing field.
To prove their existence we derive an explicit differential algebraic recursion,
which shows that the sequence of conservation laws are represented by 
1-forms generally singular at the umbilics.
As a consequence we establish Noether's theorem for the CMC system
and there is a canonical isomorphism between the symmetries and conservation laws.

A geometric interpretation of the $\s{1}$-family of associate surfaces leads to 
an integrable extension for a non-local symmetry called spectral symmetry. 
The sequence of higher-order conservation laws generated by recursion 
are homogeneous under this generalized symmetry.
We show that the corresponding spectral conservation law exists 
as a secondary characteristic cohomology class.

For a compact linear finite type CMC surface of arbitrary genus,
we observe that the monodromies of the associated flat $\sla(2,\C)$-connection commute with each other.
It follows that
a single spectral curve is defined as the completion of the  set of eigenvalues of 
the entire monodromies.
This agrees with the recent result  that
a compact  high genus linear finite type CMC surface 
necessarily factors through a branched covering of a torus.

We introduce a sequence of Abel-Jacobi maps defined by the periods of conservation laws.
For the class of deformations of CMC surfaces which scale the Hopf differential by a real parameter, we compute the first order truncated Picard-Fuchs equation.  
The resulting formulae for the first few terms exhibit a similarity with the Griffiths transversality theorem for variation of Hodge structures.
\end{abstract}
\maketitle

\np
\tableofcontents

\np
\section{Introduction}\label{sec:intro}  
Constant mean curvature (CMC) surfaces in a 3-dimensional Riemannian space form 
are by definition  the extremals for the area functional  with/without volume constraint. 
The theory of such CMC surfaces lies at the crossroads of 
classical differential geometry, integrable systems, and elliptic partial differential equation.
There exist a vast amount of literature on the subject. 
For the integrable systems aspects of the theory, especially for the genus 1 case,
let us cite the articles \cite{Pinkall1989}\cite{Hitchin1990}\cite{Bobenko1991}.
For the analytic aspects of the theory, 
let us cite the surveys  \cite{Osserman1986}\cite{Lawson1980}\cite{White2013}\cite{Brendle2013} 
and refer the reader to the included bibliographies for  further references. 
We mention in particular the recent proof of Lawson conjecture 
on the embedded minimal tori in $\s{3}$ by \cite{Brendle2012}, see also  \cite{Brendle2013a}.
A similar argument was applied to the proof of
Pinkall \& Sterling conjecture on the embedded CMC tori in $\s{3}$ in \cite{Andrews2012}.
 
\two
The second order elliptic partial differential equation\footnotemark\footnotetext{It is one of the elliptic sinh-Gordon, Liouville, or cosh-Gordon equation depending on the sign of a structural constant.} which locally describes CMC surfaces in a 3-dimensional space form is one of the prototypes of integrable equation. 
As is often true with integrable PDE's, it possesses an infinite sequence of higher-order symmetries and conservation laws.
The purpose of this work is twofold:

\begin{enumerate}[\qquad a)]
\item
to compute and determine the symmetries and conservation laws
for the exterior differential system (EDS) for CMC surfaces in a 3-dimensional space form
as the characteristic cohomology of the defining elliptic Monge-Ampere system,

\item
to propose the methods of application of conservation laws and their periods 
to understanding the global geometry CMC surfaces. 
\end{enumerate}

\two 
We shall closely follow the general theory of characteristic cohomology for differential systems
developed by Bryant \& Griffiths \cite{Bryant1995}.
The original idea of applying commutative algebraic analysis to differential equation 
is due to Vinogradov, Tsujishita, etc, \cite{Vinogradov19841}\cite{Vinogradov19842}\cite{Tsujishita1982}.
The coordinate-free framework of \cite{Bryant1995} conforms 
with the purpose of application to the global theory of CMC surfaces.  
   
\two
For an evolution equation a conservation law has a transparent practical meaning: 
an integral of local density which remains constant under time evolution. 
On the other hand for the EDS for CMC surfaces, or for elliptic equations in general, 
a conservation law implies a moment condition for the associated over-determined boundary value problem. 

To this end consider for example a  closed curve $\Gamma\subset \E^3$ in the 3-dimensional Euclidean  space with the prescribed tangent 2-plane fields $\Pi$ along $\Gamma$.
The boundary value problem here is to find a CMC-1 surface spanning $\Gamma$ which is also tangent to $\Pi$ along $\Gamma;$\footnotemark\footnotetext{The initial data $(\Gamma, \Pi)$ admits a unique prolongation to the infinite jet space of solutions to $(\star)$.}   
\be(\star)\;\left\{
\begin{array}{rl}
\Sigma&: \; \tn{CMC-1 surface in $\E^3$}, \n\\
\partial\Sigma&=\Gamma,  \n\\
T\Sigma&=\Pi \quad\tn{along}\; \Gamma. \n
\end{array}\right.\one
\ee

From the known existence results  for Dirichlet problem it is intuitively clear that the problem $(\star)$ is generally not solvable unless the pair $(\Gamma, \Pi)$ satisfies certain compatibility conditions.
More concretely suppose the given data $(\Gamma, \Pi)$ is real analytic. Then under a certain mild non-degeneracy assumption one can uniquely thicken $(\Gamma, \Pi)$ to a CMC-1 strip by Cauchy-Kowalewski theorem. Conservation laws represent the obstruction to close up this CMC-1 strip to a compact CMC-1 surface  (with boundary) with reasonable regularity.\footnotemark\footnotetext{A linearized version of this problem is to impose both Dirichlet and Neumann data for Laplace equation on a domain, \cite{Fox2009}.}

Such moment conditions are likely  to exist at least conceptually for general second order elliptic equations for one function of two variables. But it is conceivable that they may exist in a form which cannot be computed. The relevant question here is whether these moment conditions can be expressed as integrals of local density that can be detected within the framework of characteristic cohomology of differential equations. 

We shall find that the CMC system admits an infinite sequence of 
higher-order symmetries and conservation laws, 
and give the explicit differential algebraic recursion formulae for them, 
Part \ref{part:cvlaw}, and \ref{part:prolongation}.
We also lay a foundation for the application of conservation laws and their periods 
to the global geometry of CMC surfaces,
and give a basic analysis of the variation of periods
for the particular deformation by scaling Hopf differential, Part \ref{part:FT}, and \ref{part:period}.

\two
Kusner introduced the momentum class for classical conservation laws induced by the Killing fields of the ambient space form \cite{Kusner1991}.  This idea was applied to the moduli theory for complete CMC surfaces
\cite{Korevaar1989}\cite{Korevaar1992}\cite{Kusner2003}.
A refined interpretation of momentum class in a generalized relative homological setting was proposed in \cite{Edelen2013}.
Balancing condition from the momentum class also appeared in Kapouleas' gluing construction of complete  CMC-1 surfaces \cite{Kapouleas1990}. Given the results of our analysis it is interesting that in this construction the possible moment conditions from the higher-order conservation laws could be perturbed away by implicit function theorem.
 
Pinkall \& Sterling gave a classification of CMC-1 tori in $\E^3$ via linearization on the Jacobian of the associated spectral curve \cite{Pinkall1989}. Their work is strictly for tori for which there exist no umbilic points.
The main difference in the high genus CMC surface case is that the sequence of higher-order symmetries and conservation laws are singular at the umbilics. One of our initial motivation was to generalize  \cite{Pinkall1989} and obtain an analogous algebraic construction of high genus CMC surfaces.

Fox \& Goertsches determined the space of conservation laws for the elliptic sinh-Gordon equation which locally describes CMC surfaces \cite{Fox2011}. 
A distinctive aspect of the current paper compared to  \cite{Fox2011} is that our analysis is global; 
we work in an equivariant geometric setting independent of a particular local coordinate representation of the underlying differential equation,
with a view to application for the theory of compact high genus CMC surfaces, or
for  the moduli theory of complete CMC surfaces.

\two
There are many existing literature on the conservation laws for elliptic, or hyperbolic sin(h)-Gordon equations.
Let us cite \cite{Dodd1977}\cite{Ziber1979}\cite{Dorfman1993}.

We remark that the recent work \cite{Berstein2013} employs the conservation laws for the curvature entropy functional to give a new characterization of the classical minimal surfaces such as Enneper's surface, catenoid, and helicoid.

\two
Our basic method of analysis is the differential analysis of EDS in terms of differential forms, \cite{Bryant1991}\cite{Ivey2003}.
Due to the size of some of the computations involved, we relied on \texttt{Maple} for the longer computations, 
particularly to get an insight on the inductive formulae by trying out the first few terms.

In addition to \cite{Bryant1995},
the geometric perspective on differential equation suggested in \cite{Bryant1995d} 
also served as a general guidance for developing the subsequent applications of conservation laws.

\two
In recent years there have been extensive studies of surfaces of constant mean curvature 
in 3-dimensional homogeneous Riemannian manifolds, \cite{Fernandez2010} and the references therein.
We have not analyzed the conservation laws for these general cases,
but the results and perspectives from the present work 
may provide a general indication to the relevant structures.

\two
With a subject of this scope omissions are inevitable and the results and references un-cited here 
are in no way less important nor they deserve less credit.

\two
We give in the sub-sections below the summary of paper.
\subsection{Exterior differential system}
Let $M$ be the 3-dimensional Riemannian space form of constant curvature $\epsilon$.
Let $\underline{\x}:\Sigma\hook M$ be an immersed surface. 
The mean curvature $H$ of $\underline{\x}$ is a second order scalar invariant 
defined by the trace of second fundamental form
\[ H=\frac{1}{2}\tn{tr}_{\underline{\tn{I}}}\underline{\ff}. 
\]
Here $\underline{\tn{I}},\underline{\ff}$ are 
the induced metric, and second fundamental form of $\underline{\x}$ respectively. 
A \tb{CMC}-$\delta$ surface is an immersed surface 
with constant mean curvature 
$$H\equiv\delta.$$
CMC surfaces arise as the extremals for the area functional 
with/without the constant enclosed volume constraint  \cite{White2013}. 
We shall be mainly concerned with the case 
$$\epsilon+\delta^2\ne 0.$$

The differential equation for CMC surfaces is a homogeneous elliptic Monge-Ampere system $\mci$ defined on the $\s{2}$-bundle of unit tangent vectors $X\to M$.
In order to take into consideration the relations from the higher-order derivatives all at the same time,
the background for our analysis will be the infinite prolongation $(\xinf,\I{\infty})$ of $(X,\mci)$. 
In addition,  to support certain extension of the base function field to accommodate the solutions to the linearized Jacobi equation (see Sect.~\ref{sec:Jacobifields}), the actual analysis will be carried out on a double cover  $(\xinfh,\Ih{\infty})$ of $(\xinf,\I{\infty})$. 
The double cover $\xinfh$ amounts to introducing the double cover of a CMC surface 
defined by the square root of Hopf differential,
a holomorphic quadratic differential canonically associated with a CMC surface.
We shall see that the infinite sequence of conservation laws are defined on $\xinfh$ rather than on $\xinf$.

\subsection{Symmetries and conservation laws}
In this general set-up, there are two fundamental structural invariants: 
symmetry, and conservation law.

\two\quad a)
Let $\hat{T}^h=(\Ih{\infty})^{\perp}\subset T\xinfh$ be the formally integrable rank two Cartan distribution. By definition CMC surfaces correspond to $\hat{T}^h$-horizontal integral surfaces of $\iinfh$. A section of the quotient bundle $T\xinfh/\hat{T}^h$, or simply a vector field $V\in H^0(T\xinfh)$,  is a \tb{symmetry} when its formal infinitesimal action (Lie derivative) preserves the ideal $\iinfh.$

We find that a symmetry is uniquely generated by a scalar function on $\xinfh$ called \tb{Jacobi field.}
Intuitively a Jacobi field can be considered as an element in the kernel of the linearization of $\iinfh$;
it restricts to a CMC surface to be a solution to the usual Jacobi equation.

\quad b)
By construction the infinitely prolonged ideal $\Ih{\infty}$ is Frobenius. It follows that the quotient space $(\Omega^*(\xinfh)/\iinfh, \underline{\ed}),$ where $\underline{\ed}=\ed\mod\Ih{\infty},$ becomes a complex. A \tb{conservation law} is an element in the $1$-st characteristic cohomology
\be\label{eq:introcvlaw}
\mcc^{(\infty)}=H^1( \Omega^*(\xinfh)/\iinfh, \underline{\ed}).
\ee
The differential system $(\xinfh,\iinfh)$ is formally self-adjoint, 
and the symbol of a conservation law (co-symmetry) is also a Jacobi field.

The conservation law pulls back to any integral surface $\xh:\Sigmah\to\xinfh$ 
to define de Rham cohomology classes,\footnotemark\footnotetext{In the present case the conservation laws are represented by 1-forms which are singular at the umbilics. 
We shall find that the residues at the umbilics are generally non-trivial 
and this cohomology map should be adapted/understood accordingly.}
$$\xh^*: \mcc^{(\infty)} \to H^1_{\tn{dR}}(\Sigmah,\C).
$$
These classes should measure how complicated the surface is, 
considered as an integral manifold of $\iinfh$.  
This is a differential geometric analogue of the characteristic classes for vector bundles in topology, 
which measure the complexity of universal maps into the classifying spaces up to homotopy.

\subsection{Dimension bounds}
From the long exact sequence associated with
$$0\to\iinfh\to \Omega^*(\xinfh) \to \Omega^*(\xinfh)/\iinfh\to 0,$$ 
we obtain the (local) isomorphism
$$\mcc^{(\infty)} \simeq H^2(\iinfh, \ed)=:\Cv{(\infty)}$$
with the space of differentiated conservation laws (2-forms) $\Cv{(\infty)}$.
The general recipe provided by \cite{Bryant1995} yields a subspace of reduced 2-forms 
$H^{(\infty)}\subset\Omega^2(\iinfh)$ in which every conservation law has a unique representative.

\two\noi
\tb{Corollary~\ref{cor:harmonic}.} [Hodge theorem]
\emph{
There exists an isomorphism
$$\Cv{(\infty)}\simeq \{ \, \tn{closed 2-forms in $H^{(\infty)}$} \, \}.
\two$$}
\indent
The differential equation for closed 2-forms in $H^{(\infty)}$ 
is a linear differential system.
By examining the symbol of this system for the principal coefficients,
we obtain the uniform bounds 
on the dimension of the space of conservation laws.
 
\two\noi
\tb{Theorem~\ref{thm:dimbound}.} 
\emph{
For $k\geq 1$,
\begin{align*}
 \tn{dim}\; \mch^{2k-1} &=0, \\
\tn{dim}\; \mch^{2k}    &\leq 2. \n
\end{align*}
It follows that for $k\geq 1$,
\begin{align*}
 \tn{dim}\; \mcc^{2k-1} &\leq 2, \\
\tn{dim}\; \mcc^{2k}    &=0. \n
\end{align*} }
\indent Here $\mch^n,  \mcc^n$ denote the space of $n+1$-th order 
differentiated, un-differentiated conservation laws respectively.
A similar gap phenomenon has been observed for
the generic nonlinear elliptic Poisson equations in \cite{Fox2011}.
 
\two
For a proof of the above theorem we rely on the following technical lemma.

\two\noi
\tb{Lemma~\ref{lem:lemma5.4}.} 
\emph{\,Let $f: U\subset \xinf\to\C$ be a scalar function on an open subset $U\subset\xinf$ such that
\be\label{eq:Lemma5.4intro}
\delxb f =0.
\ee
Then $ f$ is a constant function.\one }

\indent Here for a scalar function $f$ we denote $\ed f\equiv (\delx f) \xi+(\delxb f)\xib\mod\iinfh,$
where $\xi$ is roughly a 1-form that restricts to be the unitary $(1,0)$-form on a CMC surface.
This lemma and its generalization Lem.~\ref{lem:lemma5.4'} 
are the key structural properties of the CMC system  
which will be used in proving various rigidity results.

\subsection{Enhanced prolongation}
For an integrable PDE it is generally the case that 
there exists a set of preferred good coordinates/functions on the infinite prolongation space. 
We introduce an enhanced prolongation to capture this symmetry property of the EDS for CMC surfaces, Eqs.~\eqref{eq:newrecursion},  Sec.~\ref{sec:recursion}.  
The enhanced prolongation is  modelled on a twisted loop algebra 
$\Lambda^{\sigma,\tau}\so(4,\C)$-valued formal Killing field. 
It is tailored to incorporate the characteristics of the structure equation for CMC surfaces. 

For our purpose it allows one to simply read off the sequence of 
higher-order Jacobi fields and conservation laws from the enhanced structure equations, Lem.~\ref{lem:bncn},  Prop.~\ref{prop:nontrivialvarphin}. 
The original idea to employ formal Killing fields to produce Jacobi fields can be traced to 
Burstall's works on harmonic maps into symmetric spaces \cite[Section 5]{Burstall1992}. 
We adapted this to our setting to construct the enhanced prolongation.
Compare Terng \& Wang's work on curved flats and conservation laws \cite{Terng2005}.

\two
In order to generate the desired sequence of higher-order conservation laws, 
we derive an explicit differential algebraic recursion  
for the formal Killing field coefficients, Sec.~\ref{sec:inductiveformula}.
The original formula is due to Pinkall \& Sterling in a slightly different form \cite{Pinkall1989}.

By construction the higher-order conservation laws
become singular at the umbilics when restricted to a CMC surface.
It would be natural to consider them as a twisted conservation law,
a smooth section of $\Omega^1(\Sigmah)\otimes(\hat{K})^q$ for some $q\in\Z_+.$
Here $\Sigmah$ is the double cover of a CMC surface defined by the square root of Hopf differential, 
and $\hat{K}\to\Sigmah$ is the canonical line bundle.

\subsection{Noether's theorem}
We have determined the uniform dimension bounds,
and derived the differential algebraic recursion 
for the desired sequence of higher-order Jacobi fields and conservation laws.
This is summarized in the higher-order extension of Noether's theorem.
It states that there exists a canonical isomorphism
between the spaces of symmetries and conservation laws for the CMC system.
  
\two\noi
\tb{Theorem~\ref{thm:higherNoether}.} [Noether's theorem]
\emph{
The Noether's theorem for classical symmetries and conservation laws Cor.~\ref{cor:Noether} admits the following higher-order extension.
\begin{enumerate}[\qquad a)]
\item
The space of conservation laws  is the direct sum of the classical conservation laws and the higher-order conservation laws:
\begin{align}
\Cv{(\infty)}&=\Cv{0}\oplus\cup_{k=1}^{\infty}\Cv{2k},\n\\
\mcc^{(\infty)}&=\mcc^0\oplus\cup_{k=1}^{\infty}\mcc^{2k-1}.\n
\end{align}
Here $\tn{dim}\,\Cv{0}=\tn{dim}\, \mcc^0=6$, and for $n\geq 1,$ 
\begin{align}
\Cv{n}&=\left\{
\begin{array}{rll}
&0  &\tn{when $n$ is odd,} \n \\
&\langle \Phi_{a^{2k+1}}, \Phi_{\ol{a}^{2k+1}} \rangle &\tn{when $n=2k$.}  \n 
\end{array}\right. \n \\
\mcc^{n}&=\left\{
\begin{array}{rll}
&\langle [\varphi_{k-1}], [\ol{\varphi_{k-1}}] \rangle &\tn{when $n=2k-1$}, \n\\
&0&\tn{when $n$ is even}.\n
\end{array}\right. \n
\end{align}
\item
Recall the exact sequence Eq.~\eqref{eq:exactsequence} from Sec.~\ref{sec:highercvlaws1},
\be 
0\to E^{0,1}_1\hook E^{1,1}_1\to E^{2,1}_1.\n
\ee
The injective map $E^{0,1}_1\hook E^{1,1}_1$ is also surjective and we have the isomorphisms
\[ \mathfrak{J}^{(\infty)} \simeq \mathfrak{S}_v \simeq \mcc^{(\infty)} \simeq \Cv{(\infty)}.\] 
\end{enumerate}}

\two
Here a few remarks are in order. 

In $\,a)$, the classical conservation laws are those induced by the Killing fields of the ambient space form,
$\Phi_{a^{2k+1}}\in\Cv{(\infty)}$ is the differentiated conservation law generated by Jacobi field $a^{2k+1}$,
and $[\varphi_{k-1}]\in \mcc^{(\infty)}$ is the corresponding un-differentiated conservation law. 
We have from recursion 
an explicit inductive formula  for the sequence of 1-forms $\varphi_{n}, n=0, 1, 2, \,... \, $.
 
In $\,b)$, the pieces $E^{p,q}_r$'s are the elements of a spectral sequence canonically associated with $(\Omega^*(\xinfh),\iinfh)$.  $\mathfrak{J}^{(\infty)}$  is the space of Jacobi fields, and $\mathfrak{S}_v$ is the space of (vertical) symmetries.
 
\subsection{Integrable extension}
The associate surfaces of a CMC surface are an $\s{1}$-family of CMC surfaces obtained 
by scaling the Hopf differential by unit complex numbers. 
The corresponding deformation vector field called  \tb{spectral symmetry } 
turns out to be a non-local object;
it is defined on an integrable extension\footnotemark\footnotetext{This is roughly an extension of a PDE by a compatible system of ODE's. In our case $Z\to \xinf$ is an affine bundle over $\g$ 
(the algebra of Killing fields on $M$).},   Sec.~\ref{sec:spectralsym}.
\be\label{eq:introextension}
 \Pi: (Z,\mcj ) \to (\xinf, \iinf), \quad \tn{such that}\;\,\Pi^*\iinf \subset\mcj \subset\Omega^*(Z).
\ee
Spectral symmetry is a  $\Pi$-horizontal vector field $\mcs$ on $Z$ such that
$$\mcl_{\mcs}(\Pi^*\iinf)\subset\mcj.\quad$$
The notion of weighted homogeneity under spectral symmetry 
will be important for our analysis throughout the paper.

\two
We show that there exists the corresponding \tb{spectral conservation law} 
which is defined as a  secondary characteristic cohomology class.  

\two\noi
\tb{Theorem~\ref{thm:spectralcvlaw}.}  [Secondary conservation law] 
\emph{ Let $\Sigma\hook M$ be a CMC surface. The spectral conservation law
$[\varphi_{v_0}] \in    H^1(\Omega^*(Z)/\mcj, \underline{\ed})$,   Eq.~\eqref{eq:spectralcvlaw},
restricts to represent an element in the quotient space
$$ [ \varphi_{v_0} ]\in  H^1(\Sigma,\R)/\langle \mcc^0 \rangle.\two$$}
\indent Here $\langle \mcc^0 \rangle \subset H^1(\Sigma,\R)$ is the subspace spanned by the restriction of classical conservation laws.  

The definition of spectral symmetry as a non-local object and some of the related materials 
are drawn from the lectures by Krasil'shchik \cite{Krasilshchik2012} on the non-local theory of differential equation. Compare also \cite{Vinogradov1989}\cite{Bluman2010}.

\subsection{Finite type surfaces}
The class of linear finite type CMC surfaces are defined by the condition
that there exists a constant coefficient linear relation among the canonical Jacobi fields.
For the case of torus without umbilics,
the Jacobi fields are smooth solutions to the elliptic Jacobi equation.
From the elliptic theory the kernel is finite dimensional,
and a CMC torus is necessarily linear finite type.

We consider the case of compact linear finite type surfaces of arbitrary genus with possible umbilics.
A polynomial Killing field still exists, Thm.~\ref{thm:polyKilling},
which forces the monodromies of the associated flat $\sla(2,\C)$-connection to commute with each other,
Cor.~\ref{cor:monocommute}.
This implies that a single spectral curve can be defined as the completion of the set of eigenvalues
of the entire monodromies of the flat $\sla(2,\C)$-connection, Cor.~\ref{cor:commonspectral}.
This agrees with the recent result of Gerding that
a compact high genus linear finite type CMC surface factors through a branched covering of a torus, Thm.~\ref{thm:Gerding}.
We do not pursue the rest of details, and refer the reader to the original work \cite{Gerding2011}.
 
One of the initial ideas was to employ the index theory to show that
a compact CMC surface is necessarily linear finite type, 
and thereby reducing the construction of examples 
to an analysis of spectral data.
Although this is not true, there is a possibility that
a compatible nonlinear finite type criterion 
may exist to produce high genus CMC surfaces.
We give an example of a class of nonlinear finite type surfaces
defined by a fourth order equation, Sec.~\ref{sec:nft}.

\subsection{Abel-Jacobi maps}
A direct computation shows that
at an umbilic point of degree $p\in\Z_+$ on a CMC surface $\Sigma$,
the 1-form  $\varphi_n$ representing the higher-order conservation law 
has the singularity of the following form.

\two\noi
\tb{Corollary~\ref{cor:varphininw}.} 
\emph{ With respect to the local formulation of CMC surfaces in Eqs.~\eqref{eq:localum},\eqref{eq:Hopfum}
adapted to an umbilic point, 
the higher-order conservation laws have the following normal form.
Here $z=w^2$.
\be\label{eq:varphininw}
\varphi_n=\frac{1}{w^{(2n-1)p+(4n-1)}}\mco(w^{4n}) \ed w 
+\frac{1}{w^{(2n-1)p+(4n-4)}}\mco(w^{4n-4})  \ol{w}\ed \ol{w},\quad n\geq 0.
\ee\indent}

Here the notation $\mco(w^{\ell})$ means a polynomial in $w$ of degree $\leq \ell$
with coefficients in the ring of functions generated by a scalar function (conformal factor)  $u$ 
and its successive $z$-derivatives.
Examining the first few conservation laws in the sequence $\varphi_n, n\geq 1$, 
we suspect that the residues are generally non-trivial
and detect certain higher-order quantitative invariants of a CMC surface at the umbilics.

Let $\Sigmah\to\Sigma$ be the double cover defined by the square root of Hopf differential $\ff$.
The normal form above shows that $\varphi_n \otimes (\sqrt{\ff})^{4n}$ is a smooth section of 
$\Omega^1(\Sigmah)\otimes (\hat{K})^{4n}.$
It also shows that a CMC surface with umbilics cannot be of linear finite type.
  
\two 
We introduce  an associated sequence of \tb{Abel-Jacobi maps}
defined as the dual of period map for the conservation laws
on the punctured CMC surface, Defn.~\ref{defn:AJmap}.
Set the truncated subspace of conservation laws
$$ \ol{H}^{(1,0)}_{n}=\langle\varphi_1, \varphi_2, \, ... \, \varphi_n \rangle.$$
Then the relevant filtration is
\be\label{eq:introfiltration}
 \ldots \, \subset \ol{H}^{(1,0)}_{n-1}\subset  \ol{H}^{(1,0)}_{n}\subset  \ol{H}^{(1,0)}_{n+1}\subset 
\ldots \; \subset\ol{H}^{(1,0)}_{\infty}\subset\mcc^{(\infty)},
\ee
and certain Hodge-like structure naturally emerges, Sec.~\ref{sec:abel}.
Although the analogy with variation of Hodge structures (VHS) in complex projective geometry is intriguing,
at the present the analysis is in the preparatory stage.

\subsection{Variation of periods}
There is a natural action of symmetries on conservation laws,
which turns out to be trivial except for 
the classical symmetries acting on the classical conservation laws, Cor.~\ref{cor:exactaction}.
This suggests to consider a different kind of deformation 
to obtain nontrivial variation of periods for conservation laws.

We examine the particular variation of CMC surfaces which scales the Hopf differential by a real parameter. 
The resulting first order truncated Picard-Fuchs equation, for the first few conservation laws in the sequence, 
exhibits a similarity with the Griffiths transversality theorem.

\two\noi
\tb{Conjecture.~\ref{conj:Gtransversality}.} [Griffiths transversality] 
\emph{ For the first order deformation given by the inhomogeneous Jacobi field $u=\im v_{\rho}+\im\delta v_0$, let $\psi_j, j=0,1,\, ... \,,$ be the sequence of 1-forms obtained from the variation of the higher-order conservation law  $\varphi_j$  by Eq.~\eqref{eq:dotvarphij}. Then as a characteristic cohomology class on $\Zh$ we have
\be\label{eq:introGtransversality}
[\psi_j] \in \langle [\varphi_0], [\varphi_1],\, ... \, [\varphi_{j-2}] \rangle \subset H^1(\Omega^*(\Zh)/\mcjh,\underline{\ed}).\two
\ee} 
\indent Here we denote the deformation of $\varphi_j$ by (upper-dot $( \,\dot{}\,)$ is the first order variation)
$$\dot{\varphi}_j=(2j-1)\varphi_j+\psi_j.$$
The proposed equation \eqref{eq:introGtransversality} is reminiscent of the Griffiths transversality  for VHS in the sense that in terms of the filtration \eqref{eq:introfiltration} this implies
$$\dot{ \left(  \ol{H}^{(1,0)}_{n} \right) }\subset  \ol{H}^{(1,0)}_{n-2} \mod \;\tn{scaling terms.}
$$

\subsection{Further research}
The application of these results to solving the differential equation for CMC surfaces, 
i.e., either to find a class of examples with as much control, 
or to understand the structure of the moduli space of compact high genus CMC surfaces, 
is yet to be explored.

\two
The methods of analysis developed for CMC surfaces can be applied to the following related problems.

\begin{itemize} 
\item  special Legendrian surfaces in $\ES^5$ 
\item  complex curves in $\ES^6$ 
\item  primitive harmonic maps from a surface to a general $k$-symmetric space 
\item  special Legendrian 3-folds in $\ES^7$. 
\end{itemize}

\np
\part{Classical, and higher-order conservation laws}\label{part:cvlaw}
\section{CMC surfaces in $3$-dimensional space forms}\label{sec:setup}
Let $M$ be the simply connected 3-dimensional Riemannian space form of constant curvature $\epsilon$.
Let \\
\centerline{\xymatrix{\SO(3) \ar[r] & \mathcal{F} \ar[d]^{\pi} \\ & M  }}\\

\noi be the bundle of oriented orthonormal coframes.  An element $\mathfrak{u}\in\mcf$ is an isometry 
\[\mathfrak{u}:T_{\pi(\mathfrak{u})}M\to\R^3,\]
where $\R^3$ is the fixed standard 3-dimensional Euclidean space. The structure group $\SO(3)$ acts on $\mcf$ on the right by
\[\mathfrak{u} \to g^{-1}\circ\mathfrak{u},\quad \mbox{for}\;g\in\SO(3).
\]

Let $(\omega^1,\,\omega^2,\,\omega^3)^t$ be the $\R^3$-valued tautological 1-form on $\,\mcf$ defined as follows. For a tangent vector $\tn{v}\in T_{\mathfrak{u}}\mcf$,  set
\[\mathfrak{u}(\pi_*(\tn{v}))=(\omega^1(\tn{v}),\omega^2(\tn{v}),\omega^3(\tn{v}))^t.
\]
The structure group $\SO(3)$ acts on $(\omega^1,\,\omega^2,\,\omega^3)^t$ on the right accordingly by
\[  \bp \omega^1 \\ \omega^2 \\ \omega^3\ep  \to g^{-1}  \bp \omega^1 \\ \omega^2 \\ \omega^3\ep,
\quad \mbox{for}\;g\in\SO(3).
\]
There exists a Lie algebra $\so(3)$-valued unique connection 1-form $(\omega^A_B)$ on $\mcf$ such that they satisfy the structure equation
\begin{align}\label{1strt1}
\ed \omega^A   &= -\omega^A_B \w \omega^B,\quad \omega^A_B=-\omega^B_A,\\
\ed \omega^A_B&= -\sum_{C=1}^3\omega^A_C\w \omega^C_B+ \epsilon\,\omega^A\w\omega^B, \quad A,B=1,2,3.\n
\end{align}

\one
Let $X=\mcf/\SO(2)$ be the unit tangent bundle of $M$.

\one
\centerline{\xymatrix{\s{2} \ar[r] & X \ar[d]^{\pi_0}\\ & M}}

\one
\noindent
For a constant $\delta$, define the differential ideal on $X$ generated by
\begin{equation}\label{1ideal1}
\mci=\langle\, \omega^3, \, - \ed \omega^3,\, \Psi \,\rangle
\end{equation}
where $\,- \ed \omega^3=\omega^3_1\w\omega^1+\omega^3_2\w\omega^2$ from Eq.~\eqref{1strt1}, and
$\,\Psi$ is the 2-form
\be\label{Psi}
\Psi=\omega^3_1\w\omega^2+\omega^1\w\omega^3_2-2\delta\omega^1\w\omega^2.
\ee
The structure equation \eqref{1strt1} shows that this set of differential forms, originally defined on $\mcf$, are invariant under the induced action by the structure group $\SO{(2)}$ of $X$. 
The ideal $\mci$ is well defined on $X$.

\two
Let $\underline{\x}:\Sigma\hook M$ be an immersed oriented surface. Let $\x:\Sigma\hook X$ be the lift of $\,\underline{\x}$ by the unit normal field compatible with the given orientation of $M$.  The standard moving frame computation shows that the induced metric $\underline{\rm{I}}$ and the second fundamental form $\underline{\ff}$ (corresponding to the unit normal chosen for the lift $\x$) of the surface $\underline{\x}$ are
\begin{align}
\underline{\rm{I}}&=(\omega^1)^2+(\omega^2)^2,\n\\
\underline{\ff}&=\omega^3_1\,\omega^1+\omega^3_2\,\omega^2.\label{eq:sff}
\end{align}
It follows that $\x$ is an integral manifold of $\,\mci$, i.e., $\x^*\mci=0$, when the surface $\underline{\x}$ has \tb{constant mean curvature}
$$\frac{1}{2}\tr_{\underline{\rm{I}}}(\underline{\ff})=\delta.$$
Conversely, the immersed integral surfaces of $\,\mci$ project to the possibly branched CMC surfaces in $M$. The EDS defined by the differential ideal ~\eqref{1ideal1} on the unit tangent bundle $X$ is therefore the fundamental object in the study of CMC  surfaces in the space form $M$.

\begin{rem}\label{1rem1}
When $\epsilon>0$, the space form $M\simeq \s{3}$ is parallelizable.
$X$ is diffeomorphic to $\s{3}\times\s{2}$, and the de Rham cohomology  is given by
\begin{equation}\label{1deRham1}
H^q_{dR}(X,\R)=\left\{
\begin{array}{rl}
\R  &\quad\mbox{$q=0, 2, 3, 5$} \\
0 \,&\quad\mbox{$q=1, 4.$} \\
\end{array} \right.
\end{equation}
The dual homology $2$-cycle is generated by any  fiber of $\pi_0: X\to M$. The dual homology  $3$-cycle is  generated by any  section of $\pi_0$ which is a unit vector field on $M$.
It can for example be given by an integral manifold of the Frobenius system
$\,\langle \,\omega^3_1+\omega^2,\, \omega^2_3+\omega^1 \, \rangle$  on $X$.

When $\epsilon \leq 0$, $\,M$ is contractible. $X$ is  diffeomorphic to $\,\R^3\times \s{2}$.
The de Rham cohomology is given by
\begin{equation}\label{1deRham0}
H^q_{dR}(X,\R)=\left\{
\begin{array}{rl}
\R    &\quad\mbox{$q=0, 2$} \\
0  \, &\quad\mbox{$q=1, 3, 4, 5.$} \\
\end{array} \right.
\end{equation}
\end{rem}

\two
The EDS for CMC surfaces is an elliptic Monge-Ampere system.  The ellipticity implies that the characteristic directions are complex and they induce a complex structure on the integral surfaces.  To this end, we introduce the complexified differential forms on $\mcf$ as follows. Here $\,\im=\sqrt{-1}$.
\begin{align}\label{1complex1}
\xi&=\omega^1+\im\,\omega^2,\\
\theta_0&=\omega^3, \n \\
\eta_1&=(\omega^3_1 -\im\,\omega^3_2) +\delta\,\xib, \n \\
\rho&=\omega^1_2.\n
\end{align}
Define the dual frame on $\mcf$ by
\be\label{dualframing}
\bp \bE\\ \bn \\ E_{\eta} \\E_{\rho} \ep
\iff
\bp \xi \\ \theta_0 \\ \eta_1 \\ \rho \ep.
\ee

In terms of the complex 1-forms,  the structure equation \eqref{1strt1} translates to
\begin{align}\label{eq:strt2}
\ed \xi &= \im \rho \w \xi -\theta_0\w (\etab_1+\delta \xi),\\
\ed \theta_0&=-\mbox{Re} (\eta_1\w\xi), \n\\
\ed \eta_1 &=-\im\rho\w\eta_1+\theta_0\w(\delta\,\eta_1+\gamma^2\,\xib),\n \\
\ed \rho&=-\frac{\im}{2}\eta_1\w\etab_1+\gamma^2\frac{\im}{2}\xi\w\xib+\delta\,\Psi. \n
\end{align}
Here 
$$\gamma^2=\epsilon+\delta^2$$ is the structure constant.\footnotemark\footnotetext{We adopt the following convention for $\gamma$:
\begin{equation}\label{1deRham1}
\gamma=\left\{
\begin{array}{cl}
+\sqrt{\epsilon+\delta^2} & \mbox{\qquad$\,\epsilon+\delta^2>0$} \\
0 & \mbox{if \quad$\,\epsilon+\delta^2=0$} \\
+\im \sqrt{-(\epsilon+\delta^2)} & \mbox{\qquad $\,\epsilon+\delta^2<0$.}
\end{array} \right.\n
\end{equation}} The 2-form $\Psi$ in the ideal $\mci$ is expressed in terms of the complex 1-forms by
\[\Psi=\Im (\eta_1\w\xi).
\]
The ideal \eqref{1ideal1} is now written as
\be\label{1ideal2}
\mci=\langle \,\theta_0,  \, \eta_1\w\xi, \, \etab_1\w\xib  \,\rangle.
\ee
We record the useful formula
\be\label{1dPsi}
\ed \Psi= \im\,\theta_0 \w (\eta_1\w\etab_1 + \gamma^2\,\xi\w\xib).
\ee

\two
Let us introduce a few notations and conventions which will be frequently used. 
Define the $\C$-linear operator $\JAI$ acting on 1-forms by
\be\label{1JAI}
\JAI\xi=-\im\xi,\; \JAI\theta_0=\theta_0,\; \JAI\eta_1=-\im \eta_1.
\ee
For a scalar function $f\in C^{\infty}(X)$, the covariant derivative (when pulled back to $\mcf$)
is written in the upper-index notation\footnotemark\footnotetext{The lower index notations $f_{\xi},\,f_{\xib}$ are reserved for the derivatives modulo the \emph{infinitely prolonged ideal} $\I{\infty}$. See Sec.~\ref{sec:prolongation1} for the details.}
$$\ed f=f^{\xi}\xi+f^{\xib}\xib+f^0\theta_0+f^1\eta_1+f^{\bar{1}}\etab_1.$$
Define the first order differential operators $\del,\,\delb$  by
$$\del f=f^{\xi}\xi+f^1\eta_1, \quad  \delb f=f^{\xib}\xib+f^{\bar{1}}\etab_1.$$

The second order covariant derivatives of a scalar function $f$ satisfy the following commutation
relations, which are drawn from the identity $\,\ed^2 f\equiv0\mod \theta_0$.
\be\label{1commut}\begin{array}{rlrrrlrl}
f^{\xi,\xib}&=f^{\xib,\xi}, &\;\;& f^{\xi, \bar{1}}&=f^{\bar{1}, \xi},&\;\;&
f^{\xi ,1}        &=f^{1, \xi}+\frac{1}{2}f^0,\\
f^{1 ,\bar{1}}&=f^{\bar{1}, 1}, &\;\;&  f^{\xib ,1}&=f^{1, \xib}, &\;\;
&  f^{\xib,\bar{1}}&=f^{\bar{1},\xib}+\frac{1}{2}f^0.  \end{array}
\ee
When $f^0=0$, we have the extra relations
\be\label{1extracommut}\begin{array}{rlrrl}
f^{1,0}& =f^{\xib}-\delta f^1,     &\;\;& f^{\xi,0}     &=\delta f^{\xi}-\gamma^2 f^{\bar{1} },\\
f^{\bar{1},0}&=f^{\xi} -\delta f^{\bar{1}}, &\;\;&  f^{\xib,0}&=\delta f^{\xib}-\gamma^2 f^{1 }.  \end{array}
\ee

\begin{exam}[Polar surfaces]
We give a brief analysis of the integral surfaces of $\mci$ in the case  $\epsilon=1$ and the ambient space form $M$ is the 3-sphere $\s{3}$.

When $M=\s{3}$,  the oriented orthonormal coframe bundle $\mcf$ is naturally identified with the special orthogonal group $\SO(4)$. For definiteness, we consider an element $\tn{e}\in\SO(4)$ as an ordered set of orthonormal column vectors
$$\tn{e}=(e_0,\,e_1,\,e_2,\,e_3).$$
The associated projection maps are\\
\centerline{\xymatrix@R=5mm{\quad\quad\quad \SO(4) \quad\, \ni \tn{e} \ar[d] \\ 
\qquad \qquad  \;\; X  \quad \ni (e_0, e_3) \ar[d]^{\pi_0}  \\
\quad\quad \; \;\quad\;\; M \qquad   \ni e_0.}}\\

\noindent
The components of $\tn{e}$ satisfy the structure equation
$$\ed e_A =e_B \omega^A_B,$$
where we set $\omega^A_0=\omega^A$.

Let $E_1=\frac{1}{2}(e_1-\im e_2),\, E_{\bar{1}}=\ol{E}_1$, and denote $e_0=\pi_0,\, e_3=\bn$.
In terms of the complexified 1-forms \eqref{1complex1}, the structure equation for the frame $\tn{e}$ translates to
\begin{align}\label{eq:3spherecase}
\ed  \pi_0 &=  \left( E_1 \xi + E_{\bar{1}} \xib \right)+ \bn \theta_0,\\
\ed  E_1 &= -\im E_1  \rho + \frac{1}{2} \bn \eta_1+\frac{1}{2}(\delta  \bn- \pi_0) \xib_1,\n\\
\ed \bn &= - \pi_0 \theta_0 - \left( E_1 (\delta \xi+\etab_1) + E_{\bar{1}} (\delta \xib +\eta_1) \right).\n
\end{align}

Let $\x: \Sigma \hook X$ be an immersed integral surface of $\mci$, and consider the pulled back
$\SO(2)$-bundle $\x^*\SO(4)\to\Sigma$. Assume first the independence condition $\xi\w\xib\ne0$.
The equation $\eta_1\w\xi=0$ implies that 
$$\eta_1=h_2\xi$$ 
for a coefficient $h_2$.
In terms of the induced complex structure on the surface determined by the $(1,0)$-form $\xi$,
see Sec.~\ref{sec:prolongation1}, a computation using Eq.~\eqref{eq:strt2} with $\theta_0=0$ leads to
\begin{align}\label{eq:3spherepair}
\ed *\ed \pi_0 &= 2(\delta \bn - \pi_0) \frac{\im}{2} \xi \w \xib, \\
\ed *\ed \bn &= 2 (\delta \pi_0-(\delta^2+|h_2|^2) \bn) \frac{\im}{2} \xi \w \xib.  \n
\end{align}

On the other hand, assuming the independence condition $\eta_1\w\etab_1\ne 0$,
the substitution $\xi = p_2 \eta_1$ for a coefficient $p_2=\frac{1}{h_2}$ gives
\be\label{3spheredual}
\ed *\ed \bn = 2 \left( \delta |p_2|^2 \pi_0-(\delta^2|p_2|^2+1) \bn \right) \frac{\im}{2} \eta_1 \w \etab_1.
\ee
It follows that when the mean curvature $\delta=0$, 
one obtains a dual pair of (possibly branched) minimal surfaces in $\s{3}$ from an integral surface of $\mci$ in $X$. The minimal surface described by $\bn$ is called the \tb{polar surface} of $\pi_0$ \cite{Lawson1970}. Note that the adapted oriented frame for the minimal surface $\bn$ is $(\bn, \, E_{\bar{1}},\,E_1, \pi_0)$.
\end{exam}
We now turn our attention to Hopf differential, a holomorphic quadratic differential canonically attached to  a CMC surface. Let $\x: \Sigma \hook X$ be an immersed integral surface of $\mci$. The structure equation \eqref{eq:strt2} restricted to the pulled back $\SO(2)$-bundle $\x^*\mcf\to\Sigma$ shows that $\,\Sigma$ inherits a unique complex structure defined by a section of either $\xi$, or $\eta_1$. Let $K=T^{*(1,0)}\to\Sigma$ be the canonical line bundle of $(1,0)$-forms of the underlying Riemann surface.
\begin{defn}\label{Hopf}
Let $\x: \Sigma \hook X$ be an immersed integral surface of the EDS for CMC surfaces.  Consider Eq.~\eqref{eq:strt2} mod $\mci$ as the induced structure equation on $\x^*\mcf$. The \tb{Hopf differential} of $\x$  is the quadratic differential
\begin{equation}\label{eq:Hopf}
\ff=\eta_1 \circ \xi  \in H^0(\Sigma, K^2).
\end{equation}
The structure equation ~\eqref{eq:strt2} shows that $\ff$ is a well defined holomorphic section.  

The \tb{umbilic divisor} $$\mcu_{\x} =(\ff)_0$$ is the zero divisor of $\ff$. By Riemann-Roch, we have $\tn{deg}(\mcu_{\x})=4\, \tn{genus}(\Sigma)-4.$
\end{defn}

We will find in the course of analysis that the (possibly double-valued) holomorphic 1-form
$\omega=\sqrt{\ff}$ is a very useful object.
Let us introduce the double cover of $\Sigma$ on which $\omega$ is well-defined.
\begin{defn}\label{defn:doublecover}
Let $\x:\Sigma\hook X$ be an immersed integral surface of the EDS for CMC surfaces.
Let $K\to\Sigma$ be the canonical line bundle under the induced complex structure. Let $\ff\in H^0(\Sigma, K^2)$ be the Hopf differential \eqref{eq:Hopf}.
The \tb{double cover} $\Sigmah$ of $\Sigma$ associated with $\ff$ is the Riemann surface of the complex curve
\[
\Sigma'= \left\{  \kappa  \in K\; \vert \; \kappa^2=\ff \right\} \subset K.
\]
Here we identify the section $\ff$ with its image in the line bundle $K^2$.
The \tb{square root $\omega$} of the Hopf differential $\ff$ is the holomorphic 1-form on $\Sigmah$ obtained by the pull-back of the restriction of the tautological 1-form on $K$ to $\Sigma'$.
\end{defn}
By construction, the natural projection $\nu:\Sigmah \to \Sigma$ is a double cover branched over the odd degree umbilics  in $\,\mcu_{\x}$.\footnotemark
\footnotetext{Let $\mcu_{\x}=\sum_{p\in\Sigma} m(p) p \in Div(\Sigma)$ be the umbilic divisor. The integer $m(p)$ is the umbilic degree of the point $p$. We have $\tn{deg}(\mcu_{\x})=\sum_{p\in\Sigma} m(p)=4\, \tn{genus}(\Sigma)-4.$
}
It is clear from the defining property of the tautological 1-form on $K$ that
$$\omega^2=\nu^*\ff.$$
\begin{defn}\label{exact}
An immersed integral surface  $\x:\Sigma\hook X$ of the EDS for CMC surfaces is \tb{split type} when there exists a holomorphic 1-form $\underline{\omega}$ on $\Sigma$ which is a square root of the Hopf differential $\ff$;
$$\underline{\omega}^2=\ff.$$
In this case, the double cover $\Sigmah$ consists of the two copies of $\Sigma$
corresponding to the sections $\pm\underline{\omega}$ of $K$.
\end{defn}
Note by definition of the square root $\omega$ that for a split type integral surface we have $\omega=\nu^*\underline{\omega}$.

\one
The square root $\omega$ on the double cover $\Sigmah$ is a holomorphic 1-form, and is necessarily closed.  We will see that the conservation laws for CMC surfaces lead to the $1$-forms which are closed on the integral surfaces of $\mci$, and $\omega$ provides the first example of a conservation law that we encounter which is not induced by the ambient symmetry of $X$.  It is semi-classical:  on the one hand we have not needed to prolong the natural EDS for CMC surfaces; on the other hand it only appears in normal form at the second prolongation.

A generic CMC surface is not split type, but the square root $\omega$ is defined globally on the double cover.  We will interpret the double cover associated with the Hopf differential as a base for a field extension. In a sense to be made precise below, this defines a splitting field for the linearized Jacobi equation for CMC surfaces.

\section{Classical conservation laws and Jacobi fields}\label{sec:classicallaws}
Given the differential ideal $\mci$ on the manifold $X$, it is natural to consider the cohomology of the quotient complex
\[(\underline{\Omega}^*,\, \underline{\ed}),\]
where $\underline{\Omega}^*=\Omega^*(X)/\mci$ and  $\underline{\ed}=\ed\mod \mci$. A \tb{classical conservation law} for CMC surfaces, as opposed to the higher-order generalization on the infinite prolongation to be discussed in Sec.~\ref{sec:highercvlaws1}, is by definition an element of the 1-st cohomology of the quotient complex
\[ \mcc^{(0)}= \mcc^{0}:=H^1(\underline{\Omega}^*,\, \underline{\ed}).\]
In this section we give an analytic description of the classical conservation laws  and show that they are equivalent to Kusner's momentum class defined in terms of the Killing fields of the ambient space form $M$  \cite{Kusner1991}. This serves as a motivation and a model for the study of higher-order conservation laws. Moreover, in hindsight  a recursion relation to be explored for the higher-order symmetries and conservation laws is already implicit in the structure equation for the classical conservation laws, Sec.~\ref{sec:resolution}.

\two
For computational purposes, among other practical reasons, we shall work directly with the differentiated classical conservation laws defined as follows. Consider the $\ed$-invariant filtration by the subspaces
\be\label{eq:filtration}
 F^p\Omega^q(X)=\tn{Image} \{ \, \underbrace{\mci\w\mci\w\,...\,\mci}_p\w\Omega^{*}(X) \to \Omega^{q}(X) \}.
\ee
For each fixed $p\geq0$, we have a well defined complex
\be\label{eq:Fpcomplex}
(F^p\Omega^*(X), \,\ed).
\ee
Let $H^{p,q}(F^p\Omega^*(X),\,\ed)$ denote the associated cohomology at $F^p\Omega^q(X)$.

Let $\Gr^{p,q}(X)=F^p\Omega^q(X)/F^{p+1}\Omega^q(X)$ be the associated graded, and consider the complex
\[ (\Gr^{p,*}(X), \,\ed^p),\]
where $\ed^p=\ed\mod F^{p+1}\Omega^*(X)$.
Let $H^{p,q}(X, \ed^p)$ denote the associated cohomology at $\Gr^{p,q}(X)$.

By definition, we have the isomorphism
\be\label{eq:Grcohomology}
 H^1(\underline{\Omega}^*,\, \underline{\ed})=H^{0,1}(X,  \ed^0).
\ee
The exterior derivative defines a map
\be\label{eq:edmap}
\ed: H^{0,1}(X, \ed^0) \to  H^{1,2}(F^1\Omega^*(X),\,\ed).
\ee
According to the fundamental results established in \cite{Bryant1995}, this is an isomorphism when one restricts to a small contractible open subset of $X$. We shall postpone the global aspects for the moment and work directly with the cohomology $ H^{1,2}(F^1\Omega^*(X),\,\ed)$, which is known as the space of differentiated conservation laws.
\begin{defn}
Let $(X,\mci)$ be the EDS for CMC surfaces. Let \eqref{eq:Fpcomplex} be the associated complex under the filtration defined by \eqref{eq:filtration}.  A \tb{classical differentiated conservation law} is an element of the $2$-nd cohomology
\[   \mch^{(0)}=\mch^0:= H^{1,2}(F^1\Omega^*(X),\,\ed)=H^{2}(\mci,\,\ed).\]
\end{defn}
The filtered object $F^1\Omega^2(X)\subset\Omega^2(X)$ is a subspace of 2-forms, rather than a quotient space in the definition of un-differentiated conservation law. As we will see below, there exists a proper subspace of $F^1\Omega^2(X)$ in which every classical conservation law has a unique representative.

\two
We proceed to the computation of classical differentiated conservation laws.\footnotemark
\footnotetext{The differential analysis is carried out on the $\SO(2)$-bundle $\mcf\to X$ via pull-back.}
Let $\Phi\in F^1\Omega^2(X)$ be a 2-form in the ideal.  Up to addition by $\ed(F^1\Omega^1(X))$, an exact 2-form in the ideal of the form $\ed(f\theta_0), \,f\in C^{\infty}(X)$, one may write
\be\label{eq:H0}
\Phi=A \Psi + \theta_0\w\sigma, \quad A\in C^{\infty}(X),\, \sigma\in\Omega^1(X).
\ee
\begin{defn}
Let $(X,\mci)$ be the EDS for CMC surfaces. Let \eqref{eq:filtration} define the associated filtration. The subspace of \tb{reduced 2-forms} $H^{(0)}\subset F^1\Omega^2(X)$ is the subspace of  2-forms specified in Eq.~\eqref{eq:H0}.
\end{defn}
Note by construction that $H^{(0)}$ is transversal to the subspace $\ed(F^1\Omega^1(X))$ in $F^1\Omega^2(X)$. It follows that every classical conservation law has a unique representative in $H^{(0)}$, and we have the isomorphism
$$ \mch^{(0)} \simeq \{ \,\tn{closed 2-forms in }\, H^{(0)} \, \}.
$$

\two
The condition for a reduced 2-form $\Phi$ to be a conservation law is that it is closed, $\ed \Phi=0$.
From Eqs.~\eqref{eq:strt2}, \eqref{1dPsi}, the equation $\ed\Phi \equiv0\mod\theta_0$ gives
$$-\im \ed A\w (\eta_1\w\xi-\etab_1\w\xib)\equiv\sigma\w (\eta_1\w\xi+\etab_1\w\xib) \mod
\;\;\theta_0.$$
One may thus put
\be\label{1sigma}
\sigma=-\JAI \ed A.
\ee
Here $\JAI$ is the operator defined in \eqref{1JAI}.

We continue to examine the equation $\ed \Phi=0$ with the given $\sigma.$
The equation $\,\ed\Phi+\im\,\theta_0\w \ed(\ed(A))\equiv 0\mod\etab_1,\xib$ gives $A^0=0$.
Utilizing the relations \eqref{1commut}, \eqref{1extracommut}, one gets  ('$\cdot$' denote  0)
\be\label{eq:classicalJacobi}
\ed \begin{pmatrix} A\\A^1\\A^{\bar{1}}\\A^{\xi}\\A^{\xib} \end{pmatrix} =
\im \begin{pmatrix} \cdot \\A^1\\-A^{\bar{1}}\\-A^{\xi}\\A^{\xib}\end{pmatrix}\rho
+\begin{pmatrix}
A^{\xi}& A^{\xib}&\cdot&  A^1& A^{\bar{1}}\\
A^{1,\xi}&\cdot&  ( A^{\xib}-\delta A^{{1}}  )&
A^{{1,1}}&-\frac{1}{2}A\\
\cdot& A^{{\bar{1},\xib}}&  ( A^{{\xi}}-\delta A^{{\bar{1}}}  ) &
-\frac{1}{2}A & A^{\bar{1},\bar{1}}\\
A^{\xi,\xi}&-\frac{\gamma^2}{2}A&
 ( \delta A^{\xi}-\gamma^2 A^{{\bar{1}}}) &
 A^{1,\xi}&\cdot\\
-\frac{\gamma^2}{2}A &
 A^{\xib,\xib}&
  ( \delta A^{\xib}-\gamma^2 A^{1})& \cdot&
 A^{\bar{1},\xib}  \end{pmatrix}
\begin{pmatrix} \xi \\ \xib \\ \theta_0 \\ \eta_1 \\ \etab_1 \end{pmatrix}.
\ee

Denote $\,A^{1,\xi}-A^{\bar{1},\xib}=B\,$ for a new variable $B$. From the equations $\ed(\ed(A^1))\equiv 0,\,\ed(\ed(A^{\xi}))\equiv 0 \mod \eta_1, \xi$, and  similarly $\ed(\ed(A^{\bar{1}}))\equiv 0,\,\ed(\ed(A^{\xib}))\equiv 0 \mod \etab_1, \xib$, one gets
$$A^{\xib,\xib}=\gamma^2 A^{1,1},\; A^{\xi,\xi}=\gamma^2 A^{\bar{1},\bar{1}},\;
A^{1,\xi}+A^{\bar{1},\xib}=-\delta A.$$
The remaining un-determined second derivatives of $A$ at this stage are $\,B, \,A^{1,1},A^{\bar{1},\bar{1}}.$

Differentiating again, the equations $\ed(\ed(A^1))\equiv 0\mod \eta_1,\, \ed(\ed(A^{\bar{1}}))\equiv 0 \mod \etab_1$ imply
$$\ed B=(-\delta A^{\xi}+\gamma^2 A^{\bar{1}})\xi+(\delta A^{\xib}-\gamma^2 A^{1})\xib
+(A^{\xib}-\delta A^1)\eta_1+(-A^{\xi}+\delta A^{\bar{1}})\etab_1.$$
The analysis is divided here into two cases.

\two
$\bullet$ Case $\gamma^2=\epsilon+\delta^2\ne0.$
The identities from $\ed(\ed B)=0$ imply $A^{1,1}=A^{\bar{1},\bar{1}}=0.$ The linear differential system for the six variables $\{A,A^1,A^{\bar{1}},A^{\xi},A^{\xib}, B\}$ becomes closed and compatible, i.e., their derivatives are expressed as functions of themselves alone and $\ed^2=0$ is a formal identity. By the existence and uniqueness theorem of ODE,  the space of  (local and differentiated) classical conservation laws is six dimensional. We shall give an argument below that these conservation laws, both differentiated and un-differentiated, are globally defined on $X$.

\one
$\bullet$ Case $\gamma^2=\epsilon+\delta^2=0.$
The equations $\ed(\ed B)=0$, and $\ed(\ed A^{\xi})=0,\,\ed(\ed A^{\xib})=0$ are identities. The remaining compatibility equations for $\ed A^1,\,\ed A^{\bar{1}}$ show that the linear differential system for  the six variables $\{A,A^1,A^{\bar{1}},A^{\xi},A^{\xib}, B\}$ is involutive.\footnotemark
\footnotetext{A PDE is involutive when the associated generic nested sequence of infinitesimal Cauchy problems
is well posed as a whole. It is one of the central concepts in Cartan's theory of exterior differential systems.
The Cartan-K\"ahler theorem, a geometric generalization of the Cauchy-Kowalewski theorem,
states that for an involutive PDE the generic nested sequence of actual Cauchy problems is solvable
and it admits many local solutions in the real analytic category. See \cite{Bryant1991}\cite{Ivey2003}.}
In the sense of Cartan-K\"ahler theory, locally it admits solutions depending on two arbitrary functions of 1 variable. The space of (local and differentiated) classical  conservation laws is infinite dimensional.

The CMC surfaces in this case are, away from the umbilics, locally described by 
the elliptic Liouville equation $\Delta u+ e^{u}=0.$
It is well known that this equation admits a Weierstra\ss\,-type representation formula for solutions in terms of one arbitrary holomorphic function of 1 complex variable. Up to the contribution from the ambient symmetry, the classical conservation laws are equivalent to the Cauchy integral formula. This case includes the minimal surfaces in the Euclidean space,  and the CMC-1 surfaces in hyperbolic space (Bryant's surfaces, \cite{Bryant1987}).

\two
For the rest of the paper, we will restrict ourselves to the following non-degenerate case:
\begin{center}
\fbox{\quad$\gamma^2\ne 0$\quad}
\end{center}
\one
The analysis for the case $\gamma^2=0$ is, although related, different in many aspects and
deserves a separate treatment.

\two
Motivated by the structure equation \eqref{eq:classicalJacobi} satisfied by the coefficients of the classical differentiated conservation laws, we make the following definition of second order differential operators for scalar functions.
\begin{defn}\label{e0operators}
For a scalar function $f:X\to\C$, define the pair of second order linear differential operators by
\begin{align}
\mce^0_H(f)&:=f^{\xi, \xib} +\frac{\gamma^2}{2} f, \label{E0H}  \\
\mce^0_V(f)&:=f^{1,\bar{1}} +\frac{1}{2} f.  \label{E0V}
\end{align}
\end{defn}
The subscripts $H$ and $V$ indicate horizontal and vertical. Note that
\be
\ed \JAI (\ed f) \equiv   \left(4\mce^0_H(f)- 2\gamma^2 f \right) \frac{\im}{2}\xi \w \xib
+ \left(4\mce^0_V(f) -2 f \right) \frac{\im}{2} \eta_1 \w \etab_1\mod\theta_0,\,\ed \theta_0.
\ee
\begin{defn}\label{defn:classicalJacobifield}
A scalar function $A:X\to\C$ is a \tb{classical Jacobi field} if it satisfies
\begin{align}\label{eq:classicalJacobifield}
A^0=A^{1,\xib}=A^{\bar{1},\xi}&=0,\\
\mce^0_H(A)&=0, \n\\
\mce^0_V(A)&=0. \n
\end{align}
The vector space of classical Jacobi fields is denoted by $\mathfrak{J}^{(0)}=\mathfrak{J}^0$.
\end{defn}

The analysis for classical conservation laws in this section is summarized as follows.
\begin{prop}\label{prop:classicalJacobifield}
A classical Jacobi field $A$ satisfies the closed structure equation \eqref{eq:classicalJacobi}. There exists a canonical (local) isomorphism
\[ \mathfrak{J}^{(0)}\simeq \mch^{(0)}.\]
\end{prop}
\begin{proof}
Direct computation. We omit the details. The analysis in the below shows that both classical Jacobi fields and classical conservation laws are globally defined, and this isomorphism holds globally on $X$.
\end{proof}
Eq.~\eqref{1sigma} shows that a classical Jacobi field uniquely determines a classical conservation law;
the coefficients of the 2-form $\,\Phi$ are determined by the Jacobi field $\,A$  and its derivatives. A version of this phenomenon will continue to hold for the higher-order conservation laws,  Sec.~\ref{sec:highercvlaws1}.
\begin{rem}\label{rem:h2introduce}
The set of equations \eqref{eq:classicalJacobifield} can be more compactly written as follows.
Let $\eta_1=h_2\xi,\,\etab_1=\hb_2\xib$ for coefficients $h_2,\,\hb_2$ on an integral surface ($h_2,\,\hb_2$ are the prolongation variables. See Sec.~\ref{sec:prolongation1}). Consider the differential equation for a scalar function $A\in C^{\infty}(X)$,
\be\label{eq:classicalJacobiequation}
 \mce(A):=A_{\xi,\xib}+\frac{1}{2}(\gamma^2+h_2\hb_2)A=0.
\ee
Here the lower index notation $A_{\xi,\xib}$ represents the total derivative with respect to $\xi,\xib$ on the integral surfaces.

One computes (since $\delxb h_2=0$, Sec.~\ref{sec:prolongation1})
\begin{align}
A_{\xi}&=A^{\xi}+h_2 A^1, \n \\
\mce(A)&
=A^{\xi,\xib}+\hb_2A^{\xi,\bar{1}}+h_2A^{1,\xib}+h_2\hb_2A^{1,\bar{1}}+\frac{1}{2}(\gamma^2+h_2\hb_2)A. \label{eq:expandedJacobi}
\end{align}
If a function $A$ on $X$ satisfies Eq.~\eqref{eq:classicalJacobiequation} for any integral surface of $\mci$, then it must satisfy Eq.~\eqref{eq:expandedJacobi}. Since the EDS for CMC surfaces is involutive, this implies that each coefficient of $h_2,\,\hb_2,\,h_2\hb_2$ in \eqref{eq:expandedJacobi} must vanish separately.  This  shows that all the conditions in Eq.~\eqref{eq:classicalJacobifield} except $A^0=0$ are imposed. Requiring that $\ed^2=0$ then forces $A^0=0$. Hence a classical Jacobi field can be equivalently defined by the single equation  \eqref{eq:classicalJacobiequation} involving the total derivative $A_{\xi,\xib}$ and the prolongation variables $h_2, \hb_2$.
\end{rem}

\one
A geometric concept closely related to, and somewhat dual to, conservation law is symmetry.
Let us give a definition of classical symmetry, which turns out to be equivalent to classical Jacobi field. 
\begin{defn}
Let $(X,\,\mci)$ be the EDS for CMC surfaces. A vector field $V \in  H^0(TX)$ is a \tb{classical symmetry} if it preserves the differential ideal \eqref{1ideal2} under the Lie derivative,   
$$\mcl_V \mci \subset \mci.$$ 
The algebra of classical symmetries is denoted by $\mathfrak{S}^0.$
\end{defn}
\begin{prop}\label{prop:classicalsymmetrytocv}
There exists a canonical isomorphism
$$ \mathfrak{S}^{0}\simeq \Cv{(0)}.$$
\end{prop}
\begin{proof}
We outline the necessary calculations.  The equation $\mcl_V \theta_0\equiv 0\mod \theta_0$ implies that, in terms of the dual  frame \eqref{dualframing},
\be\label{eq:Jacobisymmetry}
V=V_0 \bn -2V^1 \bE -2V^{\ol{1}} \bEb + 2V^{\xi}E_{\eta} + 2V^{\xib}E_{\etab}
\ee
for a scalar function $V_0$ on $X$. Next we look for the conditions that
$\mcl_V \Psi \equiv a \Psi+b \ed \theta_0 \mod \theta_0$ for some functions $a, b$.  This imposes a set of equations on the second derivatives of $V_0$. Combined with the compatibility equations \eqref{1commut}, it follows that the function $V_0$ satisfies the same set of equations \eqref{eq:classicalJacobi} as the generating function for a classical conservation law.  We mention that $V_0^0=0$, although one has to work with the third derivatives of $V_0$ in order to uncover this equation.  
\end{proof}
\begin{rem}
We remark that  a classical symmetry $V$ satisfies the stronger relations
\be\label{eq:symmetryinvariance}
 \mcl_V\theta_0=0,\;\mcl_V\Psi=0.
\ee
\end{rem}

\one
Propositions  \ref{prop:classicalJacobifield} and \ref{prop:classicalsymmetrytocv}  give the canonical isomorphisms between classical Jacobi fields, classical symmetries, and classical conservation laws which is generally known as Noether's theorem. There is a general formalism (for Euler-Lagrange equations) by which this isomorphism between symmetries and conservation laws are more directly established, \cite{Bryant2003}.

Given a classical Jacobi field $A$, let $V_A$ denote the corresponding classical symmetry.
\begin{defn}\label{defn:0PCform}
The  $0$-th order \tb{Poincar\'e-Cartan form} is the differential 3-form $\Upsilon_0$ on $X$
defined by
\be\label{eq:classicalCartanform}
\Upsilon_0=\theta_0\w\Psi\in F^2\Omega^3(X).
\ee
The 3-form $\Upsilon_0$ is closed, $\ed\Upsilon_0=0$, and it defines a class
\[ [ \Upsilon_0 ] \in H^{2,3}(F^2\Omega^*(X), \ed ).\]
\end{defn}
Since $\theta_0$ is up to scale by functions the unique 1-form in the ideal $\mci$, the cohomology class $[\Upsilon_0]$ is nontrivial (it is easily checked that $\dim_{\C}H^{2,3}(F^2\Omega^*(X), \ed)=1$).

Both $\theta_0$ and $\Psi$ are invariant forms under the classical symmetry as in \eqref{eq:symmetryinvariance}, and  we have for any $V\in\mathfrak{S}^0$
\[ \mcl_V\Upsilon_0=0.\]
From this and Cartan's formula for Lie derivative $\mcl_V=\ed \circ V\lrcorner+V\lrcorner \circ\ed,$  it follows that the following map directly establishes the desired isomorphism between the space of classical Jacobi fields and the space of classical conservation laws.
\be\label{eq:Jacobicvlaw}
A\mapsto V_A\lrcorner \Upsilon_0.
\ee

It is clear that the classical symmetries are induced by the prolonged action of the group of isometry of the ambient space form $M$ to the unit tangent bundle $X$. In particular, a classical symmetry is globally well defined on $X$.
\begin{cor}[\tb{Noether's theorem} for classical conservation law]\label{cor:Noether}
The space of classical Jacobi fields, the space of classical symmetries, and the space of classical differentiated/un-differentiated conservation laws are all canonically isomorphic to each other. In particular, they are globally defined on $X$.
\end{cor}
\begin{proof}
From the explicit local isomorphisms \eqref{eq:Jacobisymmetry}, \eqref{eq:Jacobicvlaw} established earlier, it follows that both classical Jacobi field and classical differentiated conservation laws are also globally well defined on $X$. It thus suffices to show that every classical un-differentiated conservation law has a globally defined representative. In the below, in relation to Kusner's momentum class, we show explicitly that a classical differentiated conservation law admits a corresponding globally defined un-differentiated representative.
\end{proof}

The classical conservation laws were used by Kusner \cite{Kusner1991} to define certain global invariants for CMC surfaces that became important for the structure theorems of properly embedded CMC surfaces in 3-dimensional space forms.  Let us give a discussion of some of the main results in \cite{Korevaar1989}, \cite{Korevaar1992} adapted to our framework.

Let $\G$ be the six dimensional group of isometries of $M$. Let  $\g$ be the Lie algebra of $\G$. We identify an element $Y \in \g$ with the corresponding Killing vector field on $M$. It is a standard fact that the divergence of a Killing field of any Riemannian manifold vanishes,
\be\label{eq:Ydivergence}
 \nabla \cdot Y=0.
\ee
Since $H^2(M, \R)=0$, for a given Killing field $Y$ there exists a vector field  $Z^Y\in H^0(TM)$ such that
\be\label{eq:Ycurl} 
Y=\nabla\times Z^Y \quad\tn{(curl)}.
\ee
Such a vector field $Z^Y$ is not unique, but by choosing a basis of $\g$ let us assume we have an injective homomorphism from $\g$ to $H^0(TM)$ that maps $Y\to Z^Y$.

Let $\x:\Sigma\hook M$ be a CMC-$\delta$ surface. Let $e_3\in H^0(\x^*TM)$ be the unit normal vector field to $\Sigma$. For a Killing field $Y$ of the ambient space $M$, let $Y=Y^T+Y^N$ be the orthogonal decomposition into the tangential and normal components along $\Sigma$. A computation shows that
\[ \nabla^{\Sigma}\cdot Y^T=2\delta \,\langle e_3,Y \rangle.  
\]
Here $\nabla^{\Sigma}\cdot Y^T$ is the divergence of $Y^T$ as a vector field on $\Sigma$.

Let $\Gamma\hook\Sigma$ be an oriented closed curve representing a homology class $[\Gamma] \in H_1(\Sigma,\Z)$.  
Let $e_1$ be the tangent vector field of $\Gamma$, and $e_2$ be the unit conormal field along $\Gamma$ such that the frame $\{e_1, e_2, e_3\}$ is positively oriented.
Let $D \hook M$ be a surface (2-simplex) such that $\del D = \Gamma$. 
Let $\mathfrak{n}$ the oriented unit normal field to $D$ which is in the direction of $e_2$ along $\Gamma$.  

Kusner defined the \tb{momentum class}  
\[
\mu \in H^1(\Sigma,\R) \otimes \g^*
\]
by the integral formula
\be\label{eq:momentumclass}
\langle \mu([\Gamma]), Y \rangle= \int_{\Gamma} \langle e_2,Y \rangle + 2\delta \int_{D} \langle  \mathfrak{n},Y \rangle.
\ee
Under the action of $g\in\G$ on $\x$, the momentum class $\mu$ transforms as
\[
g^*(\mu)=\Ad^*(g) \cdot \mu.
\]

In order to see momentum class is well defined, let  $(\Gamma', D')$ be another set of data representing the given homology class $[\Gamma]$. There exists a 3-simplex $U\subset M$ and a subset $S\subset\Sigma$ such that
\[\qquad\quad \del U =  D-D'+S,\quad\; \del S=\Gamma-\Gamma'\quad\tn{(outward normal)}.
\]
Applying the divergence theorem to $Y$ and $Y^T$ in turn, we find that
\begin{align}\label{eq:divergence}
2\delta \int_{D} \langle  \mathfrak{n},Y \rangle -2\delta \int_{D'} \langle  \mathfrak{n'},Y \rangle
&=-\int_S 2\delta \langle e_3,Y \rangle,\\
&=- \int_{\Gamma} \langle e_2,Y \rangle  +  \int_{\Gamma'} \langle e_2',Y \rangle, \n
\end{align}
and the claim follows.

\two
We wish to show that the momentum class is equivalent to the classical conservation laws. Note from the identity \eqref{eq:Ycurl} and Stokes' theorem that the momentum class can be expressed entirely in terms of line integrals
\[
\langle \mu([\Gamma]), Y \rangle= \int_{\Gamma} \langle e_2,Y \rangle 
+ 2\delta \int_{\Gamma} \langle  e_1, Z^Y \rangle.
\]
Analytically this can be interpreted as follows. From the structure equation \eqref{eq:classicalJacobi} we have the identity
\be\label{eq:classicalcvlawidentity}
 \ed \Big(\im(A^{\bar{1}}\xi-A^1\xib)\Big)=\Phi_A+\delta\im \Big(A\xi\w\xib-2\theta_0\w(A^{\bar{1}}\xi-A^1\xib)\Big).
\ee
Here $\Phi_A$ denotes the differentiated conservation law \eqref{eq:H0} generated by a Jacobi field $A$.
One may check that $\Xi_A=\im \Big(A\xi\w\xib-2\theta_0\w(A^{\bar{1}}\xi-A^1\xib)\Big)$ is a closed 2-form defined on the ambient space form $M$ (this is equivalent to \eqref{eq:Ydivergence}).  Since $H^2(M,\R)=0$, there exists a 1-form $\chi_A$ on $M$ such that
$\ed\chi_A=\Xi_A$. Consequently we have
\[ \ed \Big(\im(A^{\bar{1}}\xi-A^1\xib)-\delta\chi_A\Big)=\Phi_A.
\]
By the defining properties of the contact 1-form $\theta_0$, Eq.~\eqref{eq:divergence} is recovered by first integrating $\Xi_A$ on the 2-cycle $\del U$, and then integrating \eqref{eq:classicalcvlawidentity} on the 2-simplex $S$.

As a result of the momentum class analysis, we conclude that for a classical differentiated conservation law  there exists an un-differentiated representative which is globally defined on $X$.
In particular, Kusner's momentum class accounts for all the classical conservation laws.
\begin{rem}\label{rem:Edelen}
Edelen and Solomon gave a generalization of momentum class to a relative homological setting
for hypersurfaces with constant weighted mean curvature \cite{Edelen2013}.
This extension of momentum class is suited for the analysis of the CMC (hyper)surfaces invariant under Killing fields.
For an application they gave a geometric derivation of the first integral for CMC twizzlers\footnotemark\footnotetext{A CMC twizzler is a helicoidal surface of 
nonzero constant mean curvature in the Euclidean 3-space generated by 
the trace of a planar curve under an orthogonal screw motion, \cite{Perdomo2012}.}
in terms of the period of a classical conservation law.
\end{rem}
It is known that there exists an infinite sequence of  non-classical higher-order symmetries and  conservation laws for the elliptic $\sinh$-Gordon equation which locally describes the CMC surfaces for the case  $\gamma^2>0$, \cite{Fox2011}.  We shall show that Noether's theorem admits an extension 
to the higher-order symmetries and conservation laws.  
See Sec.~\ref{sec:highercvlaws1}, \ref{sec:recursion}.
 
In Sec.~~\ref{sec:families} we will return to the adjoint action of the isometry group G on $\mu$ and its higher-order counterparts, and formulate these in terms of the action of symmetry vector fields on the characteristic cohomology.  We put it off so that we may deal with the classical and higher-order conservation laws at the same time.
  
\section{Prolongation}\label{sec:prolongation1}
Let $\x:\Sigma\hook X$ be an immersed integral surface of the ideal $\mci.$  For the analysis in this section, assume the independence condition $\xi\w\xib\ne0$ so that $\x$ projects to an immersed CMC surface in $M$.

Let $\x^*\mcf\to \Sigma$ be the induced $\SO(2)$-bundle.  The structure equation \eqref{eq:strt2} restricted to $\x^*\mcf$ becomes
\begin{align}\label{eq:strtx}
\ed\xi&=\im\,\rho\w\xi,\\
0&=\eta_1\w\xi, \n \\
\ed\eta_1&=-\im\,\rho\w\eta_1, \n \\
\ed\rho&=-\frac{\im}{2}\eta_1\w\etab_1+\gamma^2\frac{\im}{2}\xi\w\xib. \n
\end{align}
The first equation shows that $\Sigma$ has an induced complex structure defined by
a section of $\xi$ as a $(1,0)$-form.  Let $K=T^{*(1,0)}\to \Sigma$ be the canonical line bundle of $(1,0)$-forms.

From the second equation, it follows that there exists a coefficient $h_2$ defined on $\x^*\mcf$
such that
\be\label{eq:h2}
\eta_1= h_2 \xi.
\ee
Differentiating this using the third equation of \eqref{eq:strtx}, one gets
$$\ed h_2\equiv-2 \im\,h_2\rho,\mod \;\xi.$$
The Hopf differential of $\x$, Definition ~\ref{Hopf}, is now written as
$$\ff=h_2\xi^2.$$
Put $h_2=a-\im\,b$ for real coefficients $a,\,b$. Then
\be\label{eq:SecondFundamentalForm}
\bp \omega^3_1\\ \omega^3_2 \ep =
\bp a+\delta &b \\ b & -a+\delta \ep \bp \omega^1 \\ \omega^2 \ep,\n
\ee
and Hopf differential is a complexification of the traceless part of the second fundamental form of the surface.

From the Gau\ss\, equation, the fourth equation of \eqref{eq:strtx}, the curvature $R$ of the metric  $\xi\circ\xib$ satisfies the equation
\be\label{eq:Gausscurvature}
R=\gamma^2-h_2\hb_2.
\ee

\one
Summarizing the analysis above we state the classical Bonnet-type theorem for CMC surfaces. It shows that Eq.~\eqref{eq:Gausscurvature} is essentially the only compatibility equation to be satisfied to prescribe a CMC surface. The proof shall be omitted.
\begin{thm}
Let $\Sigma$ be a Riemann surface. Let $(\underline{\rm{I}}, \ff)$ be a pair of conformal metric and a holomorphic quadratic differential which satisfies the compatibility equation
\[ R_{\underline{\rm{I}}}=\gamma^2-\vert \ff \vert^2_{\underline{\rm{I}}}.
\]
Here $R_{\underline{\rm{I}}}$ is the curvature of the conformal metric $\underline{\rm{I}}$, and $\vert \ff \vert^2_{\underline{\rm{I}}}$ is the norm square of the holomorphic differential with respect to $\underline{\rm{I}}$.

Let $\pi: \widetilde{\Sigma}\to\Sigma$ be the universal cover.  Then there exists a CMC-$\delta$ immersion $\tilde{\x}: \widetilde{\Sigma}\hook M$ which realizes $(\pi^*\underline{\rm{I}}, \pi^*\ff)$ as the induced Riemannian metric and the Hopf differential. Such $\tilde{\x}$ is unique up to motion by isometry of the 3-dimensional space form $M$.
\end{thm}
We give a definition for the compatible Bonnet data.
\begin{defn}\label{defn:Bonnettriple}
Let $M$ be a 3-dimensional Riemannian space form of constant curvature $\epsilon.$  An \tb{admissible triple} for  CMC-$\delta$ surface in $M$ consists of  a Riemann surface, a conformal metric, and a holomorphic quadratic  differential which satisfy the compatibility equation \eqref{eq:Gausscurvature} with the structure constant $\gamma^2=\epsilon+\delta^2.$
\end{defn}
For the purpose of our analysis it would be sometimes convenient to identify a CMC surface with an admissible triple, and set aside the global issue of monodromy of the associated CMC immersion for separate treatment.
 
\two
We now turn our attention to the infinite prolongation of the EDS for CMC surfaces. 
Prolongation is a general framework that allows one to access 
the higher-order structure of a differential equation, \cite[Chapter \rm{VI}]{Bryant1991}.
Through the process of successive differentiation, it provides a proper setting to 
take into consideration the derivatives of all orders at the same time. 
We will find that the EDS for CMC surfaces possesses rich higher-order structures.

To begin we consider the infinite prolongation of the induced structure equation on a CMC surface.
Note the higher-order derivatives of the structure coefficient $\, h_2$ on a CMC surface can be introduced inductively by the sequence of equations
\be\label{eq:dhj}
\ed h_j +  \im   \,j    h_j  \, \rho = h_{j+1} \, \xi + T_{j} \, \xib,
\; \; j = 2, \, 3, \, ... \, ,
\ee
where
\begin{align}\label{eq:Tj}
T_{2}&=0, \,     \\
T_{j+1}&=   \sum_{s=0}^{j-2}
a_{js} \,  h_{j-s} \, \partial^{s}_{\xi} R,    \; \; \; \mbox{for} \; \; j\geq 2,  \n   \\
& \quad  a_{js} =\frac{(j+s+2) }{2} \frac{ (j-1) !}{(j-s-2)!(s+2)!}
             =\frac{(j+s+2) }{2\, j} {j \choose s+2}, \n\\
& \quad  \partial^{s}_{\xi}  R=\delta_{0s}\gamma^2 -h_{2+s} \hb_2.\n
\end{align}
Each coefficient $h_j$ represents a section of $K^j$, which is generally not holomorphic for $j\geq 3$.
\begin{rem}
We will use $f_{\xi}$ or $\,\del_{\xi}f$ to denote the $\xi$-coefficient of $\,\ed f$ (and similarly for $f_{\xib}$ or  $\del_{\xib}f$) on an integral surface.  We can iterate this to obtain the quantities $\del^s_{\xi}f$.  One then computes  that $\, \partial^{s}_{\xi}  R=\delta_{0s}\gamma^2 -h_{2+s} \hb_2\,$  for $\, s\geq 0$.
\end{rem}
The formula for $T_{j+1}$ is uniquely determined by the condition that
$\ed(\ed h_j)=0$ is an identity for $j=2,\,3,\,...\,$,  \cite{Wang2012}. This implies the recursive relation
\be \label{eq:Trecurs}
T_{j+1} = \partial_{\xi} T_j + \frac{j}{2}R\, h_j.
\ee
One may check that the sequence \eqref{eq:Tj} is the unique solution.
\begin{rem}
The recursion relation \eqref{eq:Trecurs} can be interpreted as a commutator relation for
the natural pair of first order differential operators $(\delx,\,\delxb)$ on the sections of the line bundles $\Gamma(K^j)$ as follows.
\be\label{3operators}
\xymatrix{
 \ar@<1ex>[r]   &\ar@<1ex>[l]  \Gamma(K^{j-1})   \ar@<1ex>[r]^{\del_{\xi}} &  \Gamma(K^j)\ar@<1ex>[l]^{\del_{\xib}}
\ar@<1ex>[r]^{\del_{\xi}}   &  \ar@<1ex>[l]^{\del_{\xib}}  \Gamma(K^{j+1})   \ar@<1ex>[r]
&\ar@<1ex>[l]
 }
\ee
A direct computation shows that when acting on the sections of $K^j$,
\be\label{3anticommut}
\del_{\xib}\del_{\xi}-\del_{\xi}\del_{\xib}=\frac{j}{2}R.
\ee
This identity, and its generalizations to be recorded at the end of this section, will be frequently used.

When applied to $\,T_j$ for  example (note that $T_j$ has weight $(j-1)$ under the $\SO{(2)}$ action),
this yields
\begin{align}
\delxb(\delx T_j)-\delx(\delxb T_j)&=R\frac{(j-1)}{2}T_j,\n\\
\delxb(T_{j+1}-R\frac{j}{2}h_j)-\delx(\delxb T_j)&=R\frac{(j-1)}{2}T_j,\n\\
\delxb T_{j+1}&=\delx(\delxb T_j)+\delxb (R\frac{j}{2}h_j)+R\frac{(j-1)}{2}T_j.\n
\end{align}
\end{rem}

\one
With this preparation we proceed to define the infinite prolongation of the EDS for CMC surfaces. 
The infinitely prolonged structure equation established in this section 
will serve as a basis for our analysis throughout the paper.  

Set $(X^{(0)},\,\I{0})=(X,\,\mci)$ be the original differential system \eqref{1ideal1}. Inductively define  $(\X{k+1},\,\I{k+1})$ as the prolongation of $(\X{k},\,\I{k})$ so that 
$$\pi_{k+1, k}: \X{k+1}\to \X{k}$$
is the bundle of oriented integral 2-planes of $(\X{k},\,\I{k})$, and  that $\I{k+1}$ is generated by $\pi_{k+1,k}^*\I{k}$ and the canonical contact ideal on $\X{k+1}\subset\Gr^+(2,T\X{k})$.

For each $k\geq 0$, the space $\X{k+1}$ is a smooth manifold. The projection $\pi_{k+1,k}$ is a smooth submersion with two dimensional fibers isomorphic to $\C\mathbb{P}^1$, see Sec.~\ref{sec:doublecover} for details. Note this implies that the de Rham cohomology $H^1(\X{k}, \C)=0$ for all $k\geq 0$.
\begin{defn}
The \tb{infinite prolongation} $(\xinf,\I{\infty})$  is defined as the limit  
\begin{align}\label{eq:infiniteprolongation}
 \xinf &=\lim_{\longleftarrow}  \X{k}, \\
\I{\infty}&=\cup_{k\geq 0} \I{k}. \n
\end{align}
Here we identity $\I{k}$ with its image $\pi_{\infty,k}^*\I{k}\subset\Omega^*(\xinf).$
\end{defn}
Note by construction that the sequence of Pfaffian systems satisfy the inductive closure condition
$$\ed\I{k}\equiv 0\mod \I{k+1}, \; k\geq 1.$$
 
\two
Recall the $\SO(2)$-bundle  $\Pi:\mcf\to X=\mcf/\SO(2)$. Set $\F{k}=\Pi^*\X{k}$ for $k\geq 1$, and define
\[\F{\infty}= \lim_{\longleftarrow} \,\F{k}.\]
The bundle $\Pi: \F{k}\to\X{k}, \,0\leq k\leq \infty,$ is a principal right $\SO(2)$-bundle. Our analysis will be practically carried out on $(\F{k},\,\Pi^*\I{k})$ in a way that is invariant under the $\SO(2)$ action so that it has a well defined meaning on $\X{k}$. For simplicity, let us continue to use $\,\I{k}$ to denote the differential ideal $\Pi^*\I{k}$  on  $\F{k}$.

\two
The preceding analysis suggests to define the following sequence of generators for the ideal
$\I{\infty}$ on a particular open subset of $\F{\infty}$.  Let $\X{1}_0\subset\X{1}$ be the open subset defined by the independence condition
\[
\X{1}_0=\{ \;\, \tn{E}\in \X{1} \;\vert\; \;  \xi\w\xib_{\vert_{\tn{E}}}\ne 0\,\}.
\]
Inductively define the corresponding sequence of open subsets $\X{k+1}_0=\pi_{k+1,k}^{-1}(\X{k}_0)$ and $\F{k+1}_0=\Pi^{-1}(\X{k+1}_0)$ for $k\geq 1$. Let $\,\xinf_0=\lim \X{k}_0$, and $\,\F{\infty}_0=\lim \F{k}_0$.
 
Let us define the following set of differential forms, vector fields, and functions on $\F{\infty}_0$ following Eqs.~\eqref{eq:dhj}, \eqref{eq:Tj}. The covariant objects, functions and differential forms, are essentially defined on some finite $\F{k}_0$,
whereas the contravariant objects, points and vector fields, are defined on $\F{\infty}_0$.

\one
\begin{itemize}
\item  Differential forms:
\be\label{eq:thetaj}
\begin{array}{rlllll}
\eta_j&=\ed h_j&+\im j h_j \rho&&-T_j \xib,\,\quad&\tn{for}\;j\geq 2,\\
\theta_j&=\eta_j &&-h_{j+1}\xi, &&\tn{for}\; j\geq 1. \\
\end{array}
\ee
On the open subset $\,\F{k}_0$, the  set of 1-forms
$$\{ \rho; \, \xi, \xib, \theta_0, \theta_1 , \thetab_{1}, \ldots, \,\theta_k, \thetab_{k};\,\eta_{k+1}, \etab_{k+1}\}$$
form a coframe. On the open subset $\,\F{\infty}_0$, the  set of 1-forms
\be\label{Finfframe}
\{ \rho; \, \xi, \xib, \theta_0, \theta_1 , \thetab_{1},  \,\theta_2, \thetab_{2}, \ldots \}
\ee
form a coframe.  
 
\one
\item Vector fields:
\begin{align}
\delx&:=   \tn{total derivative with respect to}\;\xi\mod\iinf,\n\\
\delxb&:= \tn{total derivative with respect to}\;\xib\mod\iinf.\n
\end{align}
On the open subset $\,\F{\infty}_0$, define the  set of vector fields
$$\{ E_{\rho}; \, \delx, \delxb, E_0, E_1 , E_{\bar{1}}, E_2,E_{\bar{2}}, \ldots \}$$
as the dual frame to \eqref{Finfframe}.
\end{itemize}
 
\one
In terms of the 1-forms introduced above, we denote the covariant derivative of  $T_{j},\,j\geq 3$, by
\be\label{3dTi}
\ed T_{j} +\im \, (j-1) T_{j} \rho=(\del_{\xi} T_{j}) \xi + (\del_{\xib} T_{j}) \xib +T_{j}^{\bar{2}}\thetab_{2}
+\sum_{s=2}^{j-1}T_{j}^s\theta_s,
\ee
where $T_j^s=E_s(T_j)$.

For a later use, extend the $\C$-linear operator $\JAI$ acting on 1-forms by
$$\JAI\xi=-\im\xi,\; \JAI\theta_0=\theta_0, \;\JAI\theta_j=-\im\theta_j \;\tn{for}\; j\geq 1.$$
 
\two
Recall from Eqs.~\eqref{eq:strt2},
\begin{align}\label{3strt2}
\ed \xi - \im \rho \w \xi &=\hspace{12.5mm}-\theta_0\w\thetab_1
\hspace{15mm}-\theta_0\w (\delta \xi +\hb_2\xib),\\
\ed \theta_0&=\hspace{39.6mm}-\frac{1}{2} (\theta_1\w\xi+\thetab_1\w\xib),\n\\
\ed \theta_1 +\im\rho\w\theta_1&=-\theta_2\w\xi
+\theta_0\w(\delta\theta_1+h_2\thetab_1)
+\theta_0\w(2\delta h_2\xi+(\gamma^2+h_2\hb_2)\xib),\n\\
\ed \eta_1 +\im\rho\w\eta_1&=\hspace{13.5mm}\delta\theta_0\w\theta_1
\hspace{17.5mm}+\theta_0\w(\delta h_2\xi+ \gamma^2 \xib),\n\\
\ed \rho&=R\frac{\im}{2}\xi\w\xib-\frac{\im}{2}\theta_1\w\thetab_1
+\frac{\im}{2}(-\delta\theta_1+h_2\thetab_1)\w\xi
+\frac{\im}{2}(-\hb_2\theta_1+\delta\thetab_1)\w\xib.\n
\end{align}

\two
With this preparation, we record the structure equation satisfied by the 1-forms $\theta_j, \eta_j$  on $\,\mcf^{(\infty)}_0$.  
\begin{lem}\label{lem:prolongedstrt}
For  $\,j\geq 2$,
\begin{align}
\ed \theta_j + \im\, j\,\rho\w\theta_j&=-\theta_{j+1}\w\xi+ \theta_0\w(T_j\theta_1+h_{j+1}\thetab_1)
+\frac{jh_j}{2}\theta_1\w\thetab_1\n \\
&\quad+\tau_j'\w\xi+\tau_j''\w\xib,  \quad\tn{where}\n\\
\tau_j'&= (\delta h_{j+1}+h_2T_j)\theta_0+\frac{j h_j}{2}(\delta \theta_1-h_2\thetab_1),\n\\
\tau_j''&=(\hb_2 h_{j+1}+\delta T_j)\theta_0+\frac{j h_j}{2}(\hb_2 \theta_1-\delta\thetab_1)\n \\
&\quad-(\, T_{j}^{\bar{2}}\thetab_{2}+\sum_{s=2}^{j-1}\,T_{j}^s\theta_s), \n\\
\ed \eta_j + \im\, j\,\rho\w\eta_j&=-T_{j+1}\xi\w\xib+ \theta_0\w(T_j\theta_1 )
+\frac{jh_j}{2}\theta_1\w\thetab_1\n \\
&\quad+\varepsilon_j'\w\xi+\varepsilon_j''\w\xib,  \quad\tn{where}\n\\
\varepsilon_j'&= (  h_2T_j)\theta_0+\frac{j h_j}{2}(\delta \theta_1-h_2\thetab_1),\n\\
\varepsilon_j''&=(  \delta T_j)\theta_0+\frac{j h_j}{2}(\hb_2 \theta_1-\delta\thetab_1)\n \\
&\quad-(\, T_{j}^{\bar{2}}\thetab_{2}+\sum_{s=2}^{j-1}\,T_{j}^s\theta_s).\n
\end{align}
\end{lem}
\begin{proof}
Let us denote $\,\tn{D}\theta_j= \ed \theta_j + \im\,j\,\rho\w\theta_j$,
and similarly for $\eta_j, \, h_j, \,T_j$, i.e., $\tn{D}h_j=\ed h_j+\im\,j\,h_j\rho$, etc.
Taking the exterior derivative of $\theta_j$ in \eqref{eq:thetaj},
\begin{align}
\tn{D}\theta_j&=-\tn{D}h_{j+1}\w\xi-\tn{D}T_j\w\xib+\im\,j\,h_j (\ed\rho)-h_{j+1}(\tn{D} \xi)-T_j(\tn{D} \xib)\n \\
&=-\theta_{j+1}\w\xi-(\tn{D}T_j-T_{j+1}\xi)\w\xib+\im\,j\,h_j (\ed\rho)-h_{j+1}(\tn{D}\xi)-T_j(\tn{D}\xib).\n
\end{align}
From Eq.~\eqref{eq:Trecurs},
$$\tn{D}T_j-T_{j+1}\xi\equiv-\frac{j}{2}R h_j\xi \mod\;\;\xib,\; \I{\infty},$$
which gives
\begin{align}\label{3Dthetai}
\tn{D}\theta_j
&=-\theta_{j+1}\w\xi+\frac{j}{2}R h_j \xi \w\xib-(T_{j}^{\bar{2}}\thetab_{2}+\sum_{s=2}^{j-1}T_{j}^s\theta_s)\w\xib\n\\
&\quad+\im\,j\,h_j (\ed\rho)-h_{j+1}(\tn{D} \xi)-T_j(\tn{D} \xib).\n
\end{align}
Lemma follows from this and \eqref{3strt2}. The analysis for $\ed\eta_i$ is similar, and shall be omitted.
\end{proof}
The structure equation implies the following commutation relations for the dual frame of vector fields. For simplicity we shall exclude the relations involving $E_0,\,E_1,\,E_{\bar{1}}$. This suffices for our computations in the later sections.
\begin{cor}\label{lem:commutation}
The dual frame $E_j$'s satisfy the following commutation relations.
\begin{align}
[E_{\xi},\,E_{\xib}]&=-R\frac{\im}{2}E_{\rho},\n \\
[E_2,\,E_{\xi} ]&=E_1+\sum_{j\geq2} \ol{T_j}^{2}E_{\bar{j}},\n\\
[E_{\bar{2}},\,E_{\xib} ]&=E_{\bar{1}}+\sum_{j\geq2} T_j^{\bar{2}}E_j,\n\\
[E_s,\,E_{\xi} ]&=E_{s-1},\n\\
[E_{\bar{s}},\,E_{\xib} ]&=E_{\ol{s-1}},\quad \tn{for}\;s\geq 3,\n\\
[E_{\ell},\,E_{\xib} ]&=\sum_{j\geq {\ell}+1}T^{\ell}_j E_j,\n\\
[E_{\bar{{\ell}}},\,E_{\xi} ]&=\sum_{j\geq {\ell}+1}\ol{T_j}^{\bar{{\ell}}} E_{\bar{j}},\quad \tn{for}\;{\ell}\geq 2,\n\\
[E_j,\,E_{k} ]&=[E_j,\,E_{\bar{k}} ]=[E_{\bar{j}},\,E_{\bar{k}} ]=0,\quad \tn{for}\;j,\,k\geq 2.\n
\end{align}
\end{cor}
\begin{proof}
Let $\theta$ be a differential 1-form, and let $E_a,\,E_b$ be vector fields.
The corollary follows from the Cartan's formula
\[ \ed\theta(E_a,\,E_b)=E_a(\theta(E_b))-E_b(\theta(E_a))-\theta([E_a,\,E_b]).
\]
\end{proof}
We record the following identities satisfied by the sequence $\{\,T_j\,\}$, which generalize Eq.~\eqref{eq:Trecurs}.
\begin{cor}\label{cor:Tidentity}
Recall 
$$T_2=0,  \, T_3=h_2(\gamma^2-h_2\hb_2) ,\,T_4=\frac{5}{2}\gamma^2h_3 -\frac{7}{2}h_2\hb_2h_3.$$
By definition, $T_j$ is at most linear in $\hb_2$, and at most quadratic in $h_2$. It satisfies the following identities.
\be\label{eq:Tidentity1}
\;\;\;\left\{
\begin{array}{rl}
T_{j+1} &= \partial_{\xi} T_j + \frac{j}{2}R\, h_j,  \\
\delx T_j&=\sum_{s=2}^{j-1}T^s_j h_{s+1},  \\
\delxb T_j&=T^{\bar{2}}_j\hb_3+\sum_{s=2}^{j-1}T^s_j T_s, \\
(j-1)T_j&=-2T^{\bar{2}}_j\hb_2+\sum_{s=2}^{j-1}T^s_j s h_s,  \\
T_j  &=  -T_j^{\bar{2}}\hb_2 +\sum_2^{j-1} T_j^s h_s ,\\
\end{array} \right.
\ee
\be\label{eq:Tidentity2}
\hspace{7mm}\left\{
\begin{array}{rl}
\delx (T^2_j)&=T^2_{j+1}+\frac{jh_j}{2}\hb_2-\delta_{2j}R, \\
\delx (T^{\bar{2}}_j)&=T^{\bar{2}}_{j+1}+\frac{jh_j}{2}h_2, \\
\delx (T^s_j)&=T^s_{j+1}-T^{s-1}_j,  \qquad \quad \tn{for}\;3 \leq s\leq j-1, \,j\geq 4, \\
 T^j_{j+1}-T^{j-1}_j &=\frac{j}{2}R,  \quad \tn{for}\; j\geq 3, \\
\end{array} \right.
\ee
\be\label{eq:Tidentity3}
\hspace{2cm}\left\{
\begin{array}{rl}
\delxb (T^{\bar{2}}_j)&= (\delxb T_j)^{\bar{2}}-\sum^{j-1}_{s=2} T^s_j T^{\bar{2}}_s,  \\
                     &= \sum^{j-1}_{s=2} T^{\bar{2},s}_jT_s,  \\
\delxb (T^{\ell}_j)&=(\delxb T_j)^{\ell}-\sum^{j-1}_{s=\ell+1} T^{\ell}_sT^s_j,  \\
                     &=T^{\ell,\bar{2}}_j\hb_3+\sum^{j-1}_{s=2} T^{\ell,s}_jT_s, \qquad   \tn{for}\;\ell\geq 2. \\
\end{array} \right.
\ee
\end{cor}
\begin{proof}
These formulae are obtained by applying the commutation relations of Lemma \ref{lem:commutation} to $T_j$.
Let us present the computation for two cases.

Note by definition $E_{\rho}( h_j )=-\im j h_j$, and since $T_j$ is homogeneous of weight $j-1$ under the action by the structure group $\SO(2)$,
\begin{align}
E_{\rho}(T_j)&=-\im (j-1) T_j,\n\\
                  &=T^{\bar{2}}_j E_{\rho}(\hb_2)+\sum_{s=2}^{j-1}T^s_j E_{\rho}(h_s).\n
\end{align}
This gives the last equation of \eqref{eq:Tidentity1}.

For the last equation of \eqref{eq:Tidentity2}, we apply the relation $[E_s, E_{\xi}]=E_{s-1}, \, s\geq 3,$ to $T_j$ and get
\begin{align}
 T^{s-1}_j&=(\delx T_j)^s-\delx(T^s_j)\n\\
      &=T^s_{j+1}-(\frac{j}{2}R h_j)^s-\delx(T^s_j).\n
\end{align}
Evaluating this for $s=j\geq 3$ gives the desired formula.
\end{proof}

In the course of analysis, the following adapted functions well defined on $\xinf$ away from the umbilic divisor  $\{\,h_2 =0\,\}$ will be convenient.\footnotemark\footnotetext{Due to the possible odd powers of $h_2^{-\frac{1}{2}}$ appearing in the formulae, to be precise they are defined on a double cover of $\xinf$.
This will be defined in Sec.~\ref{sec:doublecover}.}
\begin{itemize}
\item $ z_j:=h_2^{-\frac{j}{2}}h_j,  \quad j\geq 3.$
\end{itemize}

\one\noi
These functions will turn out to be weighted homogeneous under a generalized symmetry of the CMC system.
 
\one
We introduce some notations to keep track of orders.
\begin{itemize}
\item  $\mco(k)$: scalar functions on $\xinf$ that do not depend on $h_j$ for $j>k+1$ when $k\geq 2$, or on $\hb_j$ for $j>k+1$ when $k\leq -2$.
\item $\mco(k,m)$: $\mco(k)\cap\mco(m)$ for $k\geq 2, m\leq -2$.
\end{itemize}

\section{Higher-order conservation laws}\label{sec:highercvlaws1}
Fox \& Goertsches's results on the elliptic $\sinh$-Gordon equation in \cite{Fox2011} shows that there exist infinitely many  higher-order conservation laws for the differential equation which locally describes the CMC surfaces for the case $\gamma^2>0$. This suggests to consider the analogous problem for CMC surfaces. In this section, we give an analysis of the higher-order conservation laws defined on the infinite prolongation $\,(\X{\infty},\,\I{\infty})$ utilizing the differential algebraic structure of the Frobenius  ideal $\,\I{\infty}$. 

Let us give a brief summary of the analysis. Following the general recipe for computing the conservation laws of an involutive EDS established in \cite{Bryant1995}, we first obtain a rough normal form of a differentiated conservation law as a closed 2-form in the ideal by a direct computation. We shall then extract a geometric conclusion from this that 
 \emph{one may get at most  2  new higher-order un-differentiated conservation laws on every odd prolongation
$\, \X{2k+1}\to\X{2k-1}$}.
See Thm. ~\ref{thm:dimbound} in the below for a precise statement.

\two
Let $\underline{\Omega}^*=\Omega^*(\xinf)/\iinf$, and $\underline{\ed}=\ed\mod\iinf$.
Recall that the prolongation sequence of Pfaffian systems satisfy the inductive closure condition
$$\ed\I{k}\equiv 0\mod \I{k+1}, \; k\geq 1.$$
This implies that $\iinf$ is Frobenius, and the quotient space $(\underline{\Omega}^*, \underline{\ed})$ becomes a complex.
\begin{defn}
Let $(\xinf,\iinf)$ be the infinite prolongation of the EDS for CMC surfaces. 
A \tb{conservation law} is an element in the $1$-st characteristic cohomology
$$\mcc^{(\infty)}:=H^1(\underline{\Omega}^*, \underline{\ed}).$$
A conservation law of \tb{order at most $k+1, k\geq 0,$} is an element in the $1$-st  characteristic cohomology
$$\mcc^{(k)}=H^1(\Omega^*(\X{k})/\I{k}, \underline{\ed})$$
so that $\mcc^{(\infty)}=\cup_{k=0}^{\infty}\mcc^{(k)}.$
A conservation law of \tb{order $k+1, k\geq 0,$} is an element in the quotient space
$$\mcc^{k}= \mcc^{(k)}/\mcc^{(k-1)}.$$
Here we set $\mcc^{(-1)}=0$, and $\mcc^0=\mcc^{(0)}$ is the space of classical conservation laws.
\end{defn}
\begin{rem}\label{directsum}
It is conceivable that certain elements of $\mcc^{(k)}$ become  trivial in $\mcc^{(k')}$ for $ k'>k$.
This does not occur, and we will have
$$\mcc^{(\infty)}  \simeq\oplus_{k=0}^{\infty} \, \mcc^{k}.$$
\end{rem}

\one
Let us start the analysis with a formulation of conservation laws in terms of an associated spectral sequence. The original idea of applying commutative algebraic analysis to differential equation is due to Vinogradov, Tsujishita, and Bryant \& Griffiths, etc. We shall closely follow \cite{Bryant1995}.

Consider the filtration by subspaces
$$ F^p\Omega^q =\tn{Image}\{\underbrace{\iinf\w\iinf\, ... }_{p}\w:\Omega^*(\xinf)\to\Omega^{q}(\xinf)\}.$$
From the associated graded
$F^p\Omega^*/F^{p+1}\Omega^*,$
a standard construction yields the spectral sequence
$$(E^{p,q}_r, \ed_r), \quad \ed_r \;\tn{has bidegree}\;(r, 1-r), \quad r\geq 0.$$
 
From the fundamental theorem \cite[p562, Theorem 2 and Eq.(4)]{Bryant1995}, the following sub-complex is exact,
\be\label{eq:exactsequence}
0\to E^{0,1}_1\hook E^{1,1}_1\to E^{2,1}_1.
\ee
Here by definition the first piece is given by
\begin{align}\label{eq:E011}
E^{0,1}_1 &=\{ \varphi\in\Omega^1(\xinf) \vert \ed\varphi\equiv 0\mod\iinf\}/
\{ \ed\Omega^0(\xinf) +\Omega^1(\iinf)\} \\
&=H^1(\Omega^*(\xinf)/\iinf, \underline{\ed})\n\\
&=\mcc^{(\infty)},\n
\end{align}
and it is the space of conservation laws. Unwinding the definition, the second piece $E^{1,1}_1$ is the space of cosymmetries (symbols, or generating functions of conservation laws). In our case the EDS is formally self-adjoint and it turns out that this is the space of Jacobi fields (Eq.~\eqref{eq:Agenerator}, Sec.~\ref{sec:Jacobifields}),
$$E^{1,1}_1=\{\tn{\,Jacobi fields\,}\}.$$
The derivative $\ed_1: E^{0,1}_1\hook E^{1,1}_1$ maps a conservation law to its generating Jacobi field.

We mention in passing that in this context the higher-order Noether's theorem claims that the injective symbol map
$$\ed_1: E^{0,1}_1\hook E^{1,1}_1 $$ is also surjective.\footnotemark\footnotetext{It will be shown that the higher-order  Noether's theorem holds for the EDS for CMC surfaces.}

\two
Let us comment on an important technical aspect before we proceed to the analysis of conservation laws. Due to the quotient factor $\ed\Omega^0(\xinf)$ in \eqref{eq:E011}, the space $E^{0,1}_1$ is not readily accessible by means of local analysis. For the computational purpose, it is advantageous to adopt the following differentiated version of conservation laws, similarly as we did for the analysis of classical conservation laws.

From the exact sequence
$$0\to  \iinf  \to \Omega^*(\xinf)\to  \Omega^*(\xinf)/\iinf\to 0,$$
at least locally we have
$$\mcc^{(\infty)}\simeq H^2(\iinf, \ed)\simeq\tn{ker}\{\ed_1: E^{1,1}_1\to E^{2,1}_1\}.$$
\begin{defn}
Let $(\xinf,\iinf)$ be the infinite prolongation of the EDS for CMC surfaces. 
A \tb{differentiated conservation law} is an element in the $2$-nd  cohomology
$$\mch^{(\infty)}:=H^2(\iinf, \ed).$$
A differentiated conservation law of \tb{order at most $k+1, k\geq 0,$} is an element in the $2$-nd  cohomology
$$\mch^{(k)}:=H^2(\I{k}, \ed)$$
so that $\mch^{(\infty)}=\cup_{k=0}^{\infty}\mch^{(k)}.$
A differentiated conservation law of \tb{order $k+1, k\geq 0,$} is an element in the quotient space
$$\mch^{k}:= \mch^{(k)}/\mch^{(k-1)}.$$
Here we set $\mch^{(-1)}=0$, and $\mch^0=\mch^{(0)}$ is the space of classical differentiated conservation laws.
\end{defn}
We find that the space of differentiated conservation laws $\mch^{(\infty)}$ is susceptible to local analysis partly because of the intrinsic differential algebraic structure of the infinitely prolonged ideal $\iinf$.

\two
The main result of this section is the following  uniform bounds on the dimension of the space of conservation laws.
\begin{thm}\label{thm:dimbound}
For $k\geq 1$,
\begin{align}\label{dimbound}
 \tn{dim}\; \mch^{2k-1} &=0, \\
\tn{dim}\; \mch^{2k}    &\leq 2. \n
\end{align}
It follows that for $k\geq 1$,
\begin{align}\label{dimboundC}
 \tn{dim}\; \mcc^{2k-1} &\leq 2, \\
\tn{dim}\; \mcc^{2k}    &=0. \n
\end{align}
\end{thm}
We will show in the later section that $\tn{dim}\,\mch^{2k}=\tn{dim}\,\mcc^{2k-1} = 2$ for every $k\geq 1$.
Combined with Remark \ref{directsum}, this shows that there exists an infinite sequence of nontrivial higher-order conservation laws for the EDS for CMC surfaces.

\subsection{Rough normal form}\label{sec:roughnormal}
With a view to establishing the dimension bounds \eqref{dimbound} simultaneously,
we shall give an analysis for the even case $\mch^{(2k)} .$
The vanishing of  $\mch^{2k-1}$ follows from this by setting certain coefficients to zero.

Let us make one technical assumption on $k$, restricted only to this sub-section:
$$ k\geq3.$$
This is in order to give a uniform treatment. The cases $k=1, \,2$ can be treated by a separate but similar computation.

We give an outline of the analysis in this sub-section.  Given a conservation law $\Phi\in \mch^{(\infty)}$, we utilize the equation $\ed\Phi=0$ in an inductive manner to reduce and successively normalize the coefficients of $\Phi$. This leads to a subspace of higher-order reduced 2-forms $H^{(\infty)}\subset F^1\Omega^2$ in which every conservation law has a unique representative. The analysis shows that there exists a pair of distinguished principal coefficients for a conservation law $\Phi\in \mch^{(\infty)}$, which determines the order of the conservation law. The proposed dimension bounds follow from the analysis of the canonical differential equation satisfied by the principal coefficients.

\two
Let $\Phi\in\mch^{(2k)}$. From the results established in \cite{Fox2011} one may write
\be\label{eq:Phiform0}
\begin{split}
\Phi&=A \Psi +\theta_0 \w \sigma   \\
&\quad+\sum_{1\leq i < j \leq 2k}    \left( B^{i,j}\theta_i \w \theta_j    + C^{i,j}\thetab_i \w \thetab_j\right)\\
&\quad+\sum_{1\leq i \leq j \leq 2k} \left( F^{i,j}\theta_i \w \thetab_j  + G^{i,j}\thetab_i \w \theta_j \right),
\end{split}
\ee
where  $F^{i,i}=-G^{i,i}$, and
$$\sigma=S^{\xi}\xi+S^{\xib}\xib+S^0\theta_0+
\sum_{j=1}^{2k}(S^j\theta_j+S^{\bar{j}}\thetab_j)+S^{2k+1}\eta_{2k+1}+S^{\ol{2k+1}}\etab_{2k+1}.$$

In order to facilitate the differential analysis computation, let us introduce
the following weights for the 1-forms on $\mcf^{(\infty)}$.
\one
\begin{center}
\begin{tabular}{ccccccc}
&\vline&\tn{weight} &\vline\vline&  &\vline&\tn{weight}\\
\hline
$\xi$&\vline& $-1$ &\vline\vline& $\theta_j$&\vline& $j$\\
$\xib$&\vline& $1$ &\vline\vline& $\thetab_j$&\vline& $-j$\\
$\theta_0$&\vline& $0$ &\vline\vline& $\eta_j$&\vline&$j$\\
& &  &\vline\vline& $\etab_j$&\vline& $-j$
\end{tabular}
\end{center}
\one
The prolonged structure equation in Lemma ~\ref{lem:prolongedstrt} shows that these correspond to the weights under the induced action by the structure group $\SO(2)$ of the principal bundle $\mcf^{(\infty)}\to\xinf$. The weight can be extended to the monomial 2-forms in a natural way,
e.g., $\tn{weight} (\theta_i\w\eta_j) =i+j$, etc.
Note for example that the components of the 2-form $\Psi$ have weight $0.$

The conservation law $\Phi$ in ~\eqref{eq:Phiform0} now admits the corresponding weight decomposition as follows.
\begin{align}\label{eq:Phiform1}
\Phi&=A \Psi +\theta_0 \w \sigma   \\
&\quad+\sum_{\substack{1\leq i < j \leq 2k \\ i+j\leq 2k+1}}
\left( B^{i,j}\theta_i \w \theta_j    + C^{i,j}\thetab_i \w \thetab_j\right) + \Phi' \n\\
&\quad+\sum_{1 \leq j \leq 2k} \left( F^{1,j}\theta_1 \w \thetab_j  + G^{1,j}\thetab_1 \w \theta_j \right)
+\Phi^{''}, \n
\end{align}
where
\begin{align}
\Phi'&=\sum_{\substack{1\leq i < j \leq 2k \\ i+j\geq 2k+2}}
\left( B^{i,j}\theta_i \w \theta_j    + C^{i,j}\thetab_i \w \thetab_j\right), \n \\
\Phi^{''}&=\sum_{2\leq i \leq j \leq 2k} \left( F^{i,j}\theta_i \w \thetab_j  + G^{i,j}\thetab_i \w \theta_j \right).\n
\end{align}

\two
By definition of conservation law, $\,\Phi$ is a closed 2-form, $\,\ed\Phi=0.$
This is a priori a system of first order linear differential equations for the coefficients of $\,\Phi$.
On the other hand, considering the weight decomposition of $\,\Phi$,
this equation also implies a set of linear equations for the coefficients of $\,\Phi$.
For example, one may check by direct computation that one of the consequences of the equation
$\,\ed\Phi\equiv 0\mod \theta_0, \theta_1,\thetab_1$ is that $\,B^{2k-1,2k}=C^{2k-1,2k}=0$,
the highest weight coefficients must vanish.\footnotemark
\footnotetext{See \textbf{Step 1} of the proof of Prop.~\ref{prop:roughnormal}.}

The strategy of our analysis is to apply the argument of this kind repeatedly to reduce
the algebraic normal form of $\,\Phi$.
In the course of analysis, the subset of coefficients of weight $\,2k+1$ will play a distinguished role.
\begin{defn}
Let $\Phi\in\mch^{(2k)}$ be a differentiated conservation law written in decomposed form ~\eqref{eq:Phiform1}.
The \tb{principal coefficients} are
$$\{\,B^{i,j},\, C^{i,j}\,\;|\;\; i+j=2k+1  \; \}\cup\{\, S^{2k+1},\,S^{\ol{2k+1}}\,\}.$$
\end{defn}
\begin{defn}
 Let $\I{2k}$ be the $2k$-th prolongation of the differential ideal for CMC surfaces.
 The subspace of \tb{reduced 2-forms} $H^{(2k)}\subset \Omega^2(\I{2k})$ is
 the subspace of 2-forms specified by the following normalization.
\begin{align}\label{eq:H2k}
\Phi&=A \Psi +\theta_0 \w \sigma   \\
&\quad+\sum_{\substack{1\leq i < j \leq 2k \\ i+j\leq 2k+1}}
\left( B^{i,j}\theta_i \w \theta_j    + C^{i,j}\thetab_i \w \thetab_j\right)  \n\\
&\quad+\sum_{1 \leq j \leq 2k-2} \left( F^{1,j}\theta_1 \w \thetab_j  + G^{1,j}\thetab_1 \w \theta_j \right), \n
\end{align}
where $F^{1,1}=-G^{1,1}$, and the principal coefficients are determined by
\be\label{eq:H2krelations}\begin{array}{rlrll}
B^{j, 2k+1-j}&=(-1)^{j+1}\,B^{1,2k}, & & & \\
C^{j, 2k+1-j}&=(-1)^{j+1}\,C^{1,2k}, &   \tn{for} \;\; j&= 1,  2, \,...   \, k, &       \\
A^{2k+1}&=\im\,B^{1,2k},& S^{2k+1}&=-B^{1,2k},&\\
A^{\ol{2k+1}}&=-\im\,C^{1,2k},&  S^{\ol{2k+1}}&=-C^{1,2k}.&
\end{array}\ee
Denote the entire space of reduced 2-forms by 
$$H^{(\infty)}=\cup_{k=0}^{\infty}H^{(2k)} \subset \Omega^2(\I{\infty}).$$
\end{defn}
\one

The main result of this sub-section is a rough normal form for a conservation law, and a rough structure equation for the principal coefficients. This would suffice for our purpose to give the proposed dimension bounds on the space of higher-order conservation laws.
\begin{prop}\label{prop:roughnormal}
Let $\Phi\in\mch^{(2k)}$ be a higher-order conservation law written in decomposed form ~\eqref{eq:Phiform1}.
Then
\begin{enumerate}[\qquad a)]
\item $\Phi\in H^{(2k)}.$

\item The principal coefficients $\,B^{1,2k}, \,C^{1,2k}$ satisfy the structure equations
\begin{align}
\ed B^{1,2k} +\im \, (2k+1) B^{1,2k}\rho &\equiv 0 \mod\;  \xi,   \theta_0,  \theta_1,  \thetab_1,
\theta_2, \theta_3, \, ... \, \theta_{k+1}, \n\\
\ed C^{1,2k} - \im \, (2k+1) C^{1,2k}\rho &\equiv 0 \mod\;  \xib, \theta_0,  \theta_1,  \thetab_1,
\thetab_2, \thetab_3, \, ... \, \thetab_{k+1}. \n
\end{align}

\item $(B^{ij})^{\bar{s}}=(C^{ij})^s=0 \;\; {\rm for} \; i, s \geq 2.$

\item $B^{1,2k-1}, \,S^{2k}$, and $C^{1,2k-1}, \,S^{\ol{2k}}$
lie in the linear span of the first order derivatives of $\, B^{1,2k}$, and  $\,C^{1,2k}$ respectively.
\end{enumerate}
\end{prop}

We present the proof divided in $9$ steps.
\begin{proof}

\textbf{Step 1}. We prove that $\Phi'=0$.  From Eq.~\eqref{eq:Phiform1}, consider the equation $\ed \Phi\equiv 0\mod \;\theta_0,\theta_1,\thetab_1$.
Since
$$\Psi, \,\ed\Psi, \,\ed\theta_0\equiv 0\mod\;\theta_0,\theta_1,\thetab_1,$$
the terms $A\Psi+\theta_0\w\sigma$ do not have any contribution.

We shall make use of the following form of the structure equation in Lemma ~\ref{lem:prolongedstrt}.
\be\label{eq:mod01m1}
\begin{split}
\ed\theta_j+\im\,j\,\rho\w\theta_j &\equiv -\theta_{j+1}\w\xi-(T^{\ol{2}}_j\thetab_2+\sum_{s=2}^{j-1}T^s_j\theta_s)\w\xib
\mod \;\theta_0,\theta_1,\thetab_1,\quad j\geq 2,\\
\ed\theta_1+\im\,\rho\w\theta_1 &\equiv -\theta_2\w\xi,\,\mod\;\theta_0.
\end{split}
\ee

Note first that a mixed term of $\Phi$, a $\theta\w\thetab$-term for short, has the weight in the closed interval
$[-(2k-1),\,2k-1]$, whereas $\Phi'$ consists of the terms of weight bounded by $\geq 2k+2$, or by $\leq -(2k+2)$.

Let $\Phi'=\sum_{w=2k+2}^{4k-1}(\Phi'_w+\Phi'_{-w})$ be the  weight decomposition of $\Phi'$
into terms $\Phi'_w$ of weight $w$.
We apply the induction argument on decreasing $w$ and show that $\Phi'_w=0$  for all $\,w\geq 2k+2$.
The complex conjugate of this argument then implies $\Phi'_{-w}=0$, and hence $\Phi'=0$.

Let us proceed to prove this claim.
Consider first the $\theta_{2k-1}\w\eta_{2k+1}\w\xi$-term of weight $4k-1$ in
$\ed\Phi\mod \theta_0,\theta_1,\thetab_1$.  It is easily checked that the only contribution comes from
the derivative of $\Phi'_{4k-1}$, which consists of the single element $B^{2k-1,2k}\theta_{2k-1}\w\theta_{2k}.$
Taking the exterior derivative,
$$\ed (B^{2k-1,2k}\theta_{2k-1}\w\theta_{2k})=
B^{2k-1,2k} \theta_{2k-1}\w\eta_{2k+1}\w\xi+\,... \,.$$
This implies that the highest weight term  $\Phi'_{4k-1}=0$.

Suppose $\Phi'_{s+1}=0$ for $\, s \geq 2k+2$.
For $1\leq i<\frac{s}{2}$, consider the term $\theta_i\w\theta_{s-i+1}\w\xi$ of weight $s$ in
$\ed\Phi\mod \theta_0,\theta_1,\thetab_1$.
Since the exterior derivative of a $\theta\w\theta$-term does not increase the weight,
by the induction hypothesis  and \eqref{eq:mod01m1} the only contribution comes from
\be\label{Bterms}
\ed (B^{i,s-i}\theta_i\w\theta_{s-i})=
B^{i,s-i}(\theta_{i+1}\w\theta_{s-i}\w\xi+\theta_i\w\theta_{s-i+1}\w\xi+\,... \,,
\ee
and possibly from the exterior derivatives of the mixed terms.

But \eqref{eq:mod01m1} shows that the maximum weight of the $\theta\w\theta\w\xi$-terms
that one may get from the exterior derivative of the mixed terms,
specifically those from $\theta\w\ed\thetab$, is at most $2k+2-1=2k+1$, which is strictly less than $s\geq2k+2$.
By induction on decreasing $s$ from $s=4k-1$ to $\, 2k+2$,  one gets $\Phi'_w=0$
for all $w\geq 2k+2$,  and our claim follows.

\textbf{Step 2}.  We prove that $\Phi''=0$.  Consider the equation
\[
\ed \Phi\equiv 0\mod \;\theta_0,\theta_1,\thetab_1,\theta_2,\thetab_2.
\]
Collecting the  $\theta\w\thetab\w\xi,\,\tn{and}\; \theta\w\thetab\w\xib$-terms, a similar argument as the one given above
shows that the only contributions come from differentiating $\Phi^{''}$.
In order to analyze the 3-forms of these kinds, let us define the \emph{absolute weight} by
$$\tn{absolute weight}\,(\theta_i\w\thetab_j\w\xi ) =i+j-1, \;\; \tn{absolute weight}\,( \theta_i\w\thetab_j\w\xib)  =i+j+1.$$

Expanding $\ed\Phi^{''}$, we get
$$\sum_{2\leq i \leq j\leq 2k}
   F^{i,j}(\theta_{i+1}\w\thetab_j\w\xi+\theta_i\w\thetab_{j+1}\w\xib)
+G^{i,j}(\thetab_{i+1}\w\theta_j\w\xib+\thetab_i\w\theta_{j+1}\w\xi) + \, ... \,,$$
where the  remainder are the lower absolute weight terms.
By induction on the decreasing absolute weight, it follows that the only possibly nonzero term in $\Phi^{''}$ is
$2F^{2,2} \theta_2\w\thetab_2$.

Consider now the equation $\ed \Phi\equiv 0\mod \;\theta_0,\theta_1,\thetab_1$,  and we look for the terms that
contain $\theta_2,\,\thetab_2$.
Collecting the  $\theta\w\thetab\w\xi,\,\tn{and}\; \theta\w\thetab\w\xib$-terms,
the contribution from $\Phi^{''}$ is
$2F^{2,2}(\theta_3\w\thetab_2\w\xi+\theta_2\w\thetab_3\w\xib)$ modulo terms of lower absolute weight.
On the other hand,  the rest of $\Phi$ can only contribute the terms such as
$\thetab_2\w\theta_j\w\xib$, or $\theta_2\w\thetab_j\w\xi$.
As a consequence, $F^{2,2}=0$, and we have $\Phi^{''}=0$.

\textbf{Step 3}.  We prove that
\begin{align*}
B^{j, 2k+1-j}&=(-1)^{j+1}\,B^{1,2k}\\
C^{j, 2k+1-j}&=(-1)^{j+1}\,C^{1,2k} \; \;\; \text{for}  \; j= 1,  2, \,...   \, k.
\end{align*}
Consider the equation $\ed \Phi\equiv 0\mod \;\theta_0,\theta_1,\thetab_1$ again,
and we look for $\theta\w\theta\w\xi$-terms of weight $2k+1$.
From the normalization in Step 2, there is no contribution from the exterior derivative of the mixed terms in $\Phi$.

As remarked earlier, the  exterior derivative does not increase the weight of $\theta\w\theta$-terms.
Hence the highest weight $\theta\w\theta\w\xi$-terms in $\ed \Phi\mod\;\theta_0,\theta_1,\thetab_1$ are
$$\sum_{1\leq i\leq k}B^{i,2k+1-i}(\theta_{i+1}\w\theta_{2k+1-i} +\theta_{i}\w\theta_{2k+2-i})\w\xi.$$
This implies
$$ B^{i, 2k+1-i}=(-1)^{i+1}\,B^{1,2k},\quad \tn{for}\; 2\leq i \leq k.$$
Taking the complex conjugate of this analysis, we get
$$ C^{i, 2k+1-i}=(-1)^{i+1}\,C^{1,2k},\quad \tn{for}\; 2\leq i \leq k.$$

\textbf{Step 4}.  We prove that
\begin{align*}
A^{2k+1}&=\im\,B^{1,2k},\quad S^{2k+1}=-B^{1,2k}.\\
A^{\ol{2k+1}}&=-\im\,C^{1,2k},\quad S^{\ol{2k+1}}=-C^{1,2k}.
\end{align*}
Consider the equation $\ed \Phi\equiv 0\mod \;\theta_0$ only,
and look for the $\theta_1\w\xi\w\eta_{2k+1}$-terms. Since $\ed \Psi\equiv 0\mod\;\theta_0$, the only contributions for weight $2k+1$ are
$$
A^{2k+1}\Psi\w\eta_{2k+1}+\ed\theta_0\w(S^{2k+1}\eta_{2k+1})+B^{1,2k}\theta_1\w\eta_{2k+1}\w\xi,
$$
where $A^{2k+1}$ is the $\eta_{2k+1}$-derivative of $A$.  Recall $\Psi=-\frac{\im}{2}(\theta_1\w\xi-\thetab_1\w\xib),\; \, \ed\theta_0=-\frac{1}{2}(\theta_1\w\xi+\thetab_1\w\xib)$.
We consequently find that $-\frac{\im}{2}A^{2k+1}-\frac{1}{2}S^{2k+1}-B^{1,2k}=0$.  Turning to the coefficient of the $\thetab_1\w\xib\w\eta_{2k+1}$-term we uncover only that $\frac{1}{2}A^{2k+1}-\frac{1}{2}S^{2k+1}=0$.  Together these vanishing conditions imply that
\be\label{eq:A2k+1}
A^{2k+1}=\im\,B^{1,2k},\quad S^{2k+1}=-B^{1,2k}.
\ee
Taking the complex conjugate of this argument reveals that
\be\label{eq:A2k+1b}
A^{\ol{2k+1}}=-\im\,C^{1,2k},\quad S^{\ol{2k+1}}=-C^{1,2k}.
\ee

\textbf{Step 5}.  We prove that $ F^{1,2k}=G^{1,2k}=0$.
Consider first the equation $\ed \Phi\equiv 0\mod \;\theta_0,\theta_1,\xib$,
and we look for $\thetab_1\w\eta_{2k+1}\w\xi$-terms.
By  the  relations  $\Psi,\, \ed \Psi, \, \ed\theta_0 \equiv 0\mod \;\theta_0,\theta_1,\xib$ ,
there are  no contributions from differentiating $\,A\Psi+\theta_0\w\sigma$.

It is clear then the only contribution comes from differentiating the mixed terms.  We compute
$$\ed (G^{1,2k}\thetab_1\w\theta_{2k})=G^{1,2k}\thetab_1\w\eta_{2k+1}\w\xi+\, ... \,$$
so that $G^{1,2k}=0$. Taking the complex conjugate of this argument produces
$$ F^{1,2k}=0.$$

At this stage, note that when the two principal coefficients $\,B^{1,2k}, \,C^{1,2k}$ vanish,
$\Phi$ is semi-basic for the projection $\X{2k} \to \X{2k-1}$ and it lies in the algebraic span of the (pull-back of) exterior 2-forms on $\X{2k-1}$. 
 In fact since $\Phi$ is a closed 2-form, this implies that $\Phi$ is defined on $\X{2k-1}$.

\textbf{Step 6}. We prove that $ (B^{i,j})^{\ol{s}}=(C^{i,j})^{s}=0$ for $i>2$.
Consider the equation $\ed \Phi\equiv 0\mod \;\theta_0,\theta_1,\thetab_1,\xi,\xib$.
The structure equation \eqref{eq:mod01m1} shows that the only contribution comes from differentiating
the coefficients $B^{i,j}, \,C^{i,j}$.

Note first that all the $\eta_{2k+1},\,\etab_{2k+1}$-derivatives of $B^{i,j}, \,C^{i,j}$ vanish.
Collecting the $\thetab\w\theta\w\theta$-terms, and $\theta\w\thetab\w\thetab$-terms, we get
$$ (B^{i,j})^{\ol{s}}=(C^{i,j})^{s}=0,\quad \tn{for}\;\; i\geq 2, \;\;  s=2, 3,\, ...\, 2k+1.$$
Here $(B^{i,j})^{\ol{s}}$ denotes the $\thetab_s$(or $\etab_s$)-derivative of $B^{i,j}$,
and similarly for $(C^{i,j})^{s}$.

\textbf{Step 7}. We prove part b) of the Proposition.  Collecting the $\theta\w\theta\w\theta$-terms, the contributions are from
(modulo $\theta_0,\theta_1,\thetab_1,\xi,\xib$)
\begin{align}
&\ed B^{1,2k}\w(\theta_1\w\theta_{2k}-\theta_2\w\theta_{2k-1}+\,...\,+(-1)^{k+1}\theta_k\w\theta_{k+1})\\
&\quad+ \sum_{\substack{2\leq i<j\leq2k\\ i+j\leq2k}}\,\ed B^{i,j}\w\theta_i\w\theta_j.\n
\end{align}
Evaluating this $\mod \theta_2,\theta_3, \,...\,,\theta_{k-1}$, we get
$$ (B^{1,2k})^s=0,\quad \tn{for}\;\; s=k+2, k+3,\, ...\, (2k+1).$$

Consider now the equation $\ed \Phi\equiv 0\mod \;\theta_0,\theta_1,\thetab_1,\thetab_2, \xi$.
Collecting the $\theta_2\w\theta_{2k-1}\w\xib$-terms, the only contribution is from
$$\ed (B^{2,2k-1}\theta_2\w\theta_{2k-1})=(B^{2,2k-1})^{\xib}\theta_2\w\theta_{2k-1}\w\xib+\,...$$
because $B^{j,2k-1}=0$ for $j>2$.  This implies that $(B^{1,2k})^{\xib}= (-B^{2,2k-1})^{\xib}=0$.

Considering the complex conjugate of these arguments,
we have at this stage the following structure equation for the principal coefficients.
\be
\begin{split}
\ed B^{1,2k}+\im\,(2k+1) B^{1,2k}\rho&\equiv 0 \mod\;\; \xi, \theta_0, \theta_1,\thetab_1,
\theta_2,  \theta_3, \, ... \,  \theta_{k+1}, \n \\
\ed C^{1,2k}-\im\,(2k+1) C^{1,2k}\rho&\equiv 0 \mod\;\; \xib, \theta_0, \theta_1,\thetab_1,
\thetab_2,\thetab_3, \, ... \, \thetab_{k+1}. \n
\end{split}
\ee

\textbf{Step 8}.
Consider the equation $\ed \Phi\equiv 0\mod \;\theta_0,\theta_1,\thetab_1,\theta_2,\xi$,
and collect $\thetab_2\w\theta_{2k-1}\w\xib$-terms. Since $\,T_2=T_{\ol{2}}=0$, the only contribution comes from
$$\ed (G^{1,2k-1}\thetab_1\w\theta_{2k-1})=G^{1,2k-1}\thetab_2\w\theta_{2k-1}\w\xib+\, ... \, .$$
Taking the complex conjugate, we get
$$F^{1,2k-1}=G^{1,2k-1}=0.$$

\textbf{Step 9}. We finally prove part d) of the Proposition.  One finds that the coefficient of the term $\theta_i \w \theta_j \w \xi$ for $2 \leq i < j \leq 2k$ is
\[
(B^{ij})^{\xi}+B^{i-1,j}+B^{i,j-1}=0.
\]
Consider the subset of identities $(B^{j,2k+1-j})^{\xi}+B^{j-1,2k+1-j}+B^{j,2k-j}=0$ for $j=2 \ldots k$.  Using this along with the first identity in Eq.~\eqref{eq:H2krelations} we compute
\begin{align*}
(k-1)(B^{1,2k})^{\xi}&=\sum_{j=2}^{k}(-1)^j (B^{j,2k+1-j})^{\xi}\\
&=\sum_{j=2}^{k}(-1)^j (B^{j-1,2k+1-j}+B^{j,2k-j}  )=-B^{1,2k-1}
\end{align*}
where the last equality arises from the cancellation of terms in the penultimate sum.  Thus $B^{1,2k-1} \equiv 0 \; \mod \; (B^{1,2k})^{\xi}$. The complex conjugate of this argument gives the analogous result for $C^{1,2k-1}$.

Consider next the equation $\ed \Phi\equiv 0\mod \;\theta_0$,
and collect $\theta_1\w\theta_{2k}\w\xi, \; \thetab_1\w\theta_{2k}\w\xib$-terms.
Using the vanishing results already established we find the contributions
$$
A^{2k}\Psi\w\theta_{2k}+\ed\theta_0\w(S^{2k}\theta_{2k})+(B^{1,2k})^{\xi} \theta_1\w\theta_{2k}\w\xi+B^{1,2k-1} \theta_1\w\theta_{2k}\w\xi.
$$
This implies that
\[
\frac{\im}{2}A^{2k}+\frac{1}{2}S^{2k}+(B^{1,2k})^{\xi}+B^{1,2k-1}=0.
\]
Collecting $\; \thetab_1\w\theta_{2k}\w\xib$-terms only involves
$$
A^{2k}\Psi\w\theta_{2k}+\ed\theta_0\w(S^{2k}\theta_{2k})
$$
and produces the vanishing condition
\[
-\frac{\im}{2}A^{2k}+\frac{1}{2}S^{2k}=0
\]
which implies $S^{2k}=\im A^{2k}$.  These two vanishing conditions imply that
$$ A^{2k}=\im\,(B^{1,2k})^{\xi}+ \im \, B^{1,2k-1}.
$$
Thus $A^{2k}$ and $ S^{2k}$ vanish modulo $(B^{1,2k})^{\xi}$ because we have already established this for $B^{1,2k-1}$.  Taking the complex conjugate of this argument leads to the analogous conditions for $ A^{\ol{2k}}$ and $S^{\ol{2k}}.$
\end{proof}
\begin{cor}[Hodge theorem]\label{cor:harmonic}
There exists an isomorphism
$$\Cv{(\infty)}\simeq \{ \, \tn{closed 2-forms in $H^{(\infty)}$} \, \}.
$$
\end{cor}
Cor.~\ref{cor:harmonic}, which identifies a conservation law originally defined as a cohomology class with a closed $2$-form in the subspace $H^{(\infty)}$, resembles the Hodge theorem for the existence of harmonic representatives of de Rham classes. 
We shall freely identify a differentiated conservation law with a reduced closed 2-form from now on.
\begin{cor}
\begin{align*}
\mch^{(\infty)}  &\simeq\mch^0\oplus_{k=1}^{\infty} \, \mch^{2k}, \\
\mcc^{(\infty)}  &\simeq\mcc^0\oplus_{k=1}^{\infty} \, \mcc^{2k-1}.
\end{align*}  
\end{cor}
\begin{proof}
Note that since the set of monomial components of the 2-form $\Phi$ in \eqref{eq:H2k} is linearly independent, if a conservation law $\Phi$ belongs to $H^{(2k)}$ algebraically then the equation $\ed\Phi=0$ implies that it is independent of the variables $h_j, \hb_j$ for $j\geq 2k+2$. Hence it is defined on $\X{2k}$. The corollary follows from the transversality of the subspace of 2-forms \eqref{eq:Phiform0} (see the footnote), and the linear inclusion relation $H^{(2k)}\subset H^{(2k')}$ for $k<k'$.
\end{proof}
\subsection{Dimension bounds}\label{sec:dimensionbound}
Thm. ~\ref{thm:dimbound} now follows from Prop. ~\ref{prop:roughnormal}

\emph{Proof of Thm. ~\ref{thm:dimbound}}.
Suppose the principal coefficients $B^{1,2k},C^{1,2k}$ vanish identically. Prop.~\ref{prop:roughnormal} shows that $\Phi$ is defined on $\X{2k-2}$, and hence  $\Phi\in\mch^{(2k-2)}$.
This shows that $\tn{dim}\,\mch^{(2k-1)}/\mch^{(2k-2)}=0$.

From the general theory of ODE,  it now suffices for the proposed dimension bound $\,\tn{dim}\,\mch^{(2k)}/\mch^{(2k-2)}\leq 2$
to show that the set of equations $\{\, B^{1,2k}=C^{1,2k}=0\,\}$ imposes exactly two independent \emph{algebraic} equations on $\mch^{(2k)}.$ This follows from the following lemma.
\begin{lem}\label{lem:lemma5.4}
Let $f: U\subset \xinf\to\C$ be a scalar function on an open subset $U\subset\xinf$ such that
\be\label{eq:Lemma5.4}
\delxb f =0.
\ee
Then $ f$ is a constant function.
\end{lem}
We postpone the proof of this lemma to the next sub-section.

\two
Lemma \ref{lem:lemma5.4} and b) of Prop.~\ref{prop:roughnormal} imply that $B^{1,2k},\,C^{1,2k}$ are constant multiple of
$h_2^{-\frac{2k+1}{2}},\,\hb_2^{-\frac{2k+1}{2}}$ respectively.
Hence the set of equations  $\{\, B^{1,2k}=C^{1,2k}=0\,\}$
imposes two algebraic equations, and only two equations on $\mch^{(2k)}$.
\hfill$\square$

\two
The preceding lemma suggests and justifies the following decomposition of $\mch^{2k}$.
\begin{defn}
Let $\mch^{2k}=\mch^{(2k)}/\mch^{(2k-1)}$ be the space of differentiated conservation laws of order $2k+1$.
The  \tb{$\,\pm$-decomposition} of $\mch^{2k}=\mch^{2k}_+\oplus\mch^{2k}_-$ is defined by
\begin{align}
\mch^{2k}_+&=\{\; [\Phi]\in \mch^{2k}\;|\;\;  \tn{the principal coefficient  $\,C^{1,2k}$ of $\,\Phi$ vanishes}\,\}, \n \\
\mch^{2k}_- &=\{\; [\Phi]\in \mch^{2k}\;|\;\;  \tn{the principal coefficient  $\,B^{1,2k}$ of $\,\Phi$ vanishes}\,\}.\n
\end{align}
\end{defn}
\begin{cor}
For $\,k\geq 1$,
$$\tn{dim}\,\mch^{2k}_{\pm}\leq 1.$$
\end{cor}

In the Part 2 of the paper, we will show that $\tn{dim}\,\mch^{2k}_{\pm}=1$ for every $k\geq 1$
by producing a sequence of higher-order Jacobi fields and their associated conservation laws via an explicit recursion formula.

\subsection{Proof of Lem.~\ref{lem:lemma5.4}} 
This result is essentially due to Lie who classified the Darboux integrable (hyperbolic) Poisson equations, \cite[p182]{Goursat1896}:  the only Darboux integrable nonlinear Poisson equation $u_{xy}=F(u)$ is the Liouville equation $u_{xy} = Ae^{Bu}$ for constants $A, B.$\footnotemark\footnotetext{Darboux integrability is by definition the existence of a function $f$ on the (infinite) jet space for which $\ed f$ gives the characteristic. For the Poisson equation at hand one looks for such $f$ which is functionally independent from $z$ (or $\zb$).} 
The argument works for the elliptic case $u_{z\ol{z}}=F(u)$ without much change.

We propose the following generalization of  Lem.~\ref{lem:lemma5.4} which is
suitable for proof by induction.
\begin{lem}\label{lem:lemma5.4'}
Let $f: U\subset \xinf\to\C$ be a scalar function on an open subset $U\subset\xinf$ such that
$$\delxb f= c h_2^{-\frac{1}{2}}$$ 
for a constant $c$.  From the structure equation this implies that for some $k>0$
\be\label{eq:Lemma5.4'}
\ed f    \equiv c h_2^{-\frac{1}{2}}\xib \mod\;  \xi,   \theta_0,  \theta_1,  \thetab_1, \theta_2,
\theta_3,\,...\,\theta_{k}.
\ee
Then $c=0$ necessarily and $f$ is a constant function.
\end{lem}
The proof by induction is divided into the following three sub-sections.
Recall the notation 
$$g\in\mco(j)$$ 
for $j\geq 2$ means the function $g$ depends at the highest on the variable $h_j.$

For a uniform treatment, assume $k\geq 5$. 
The case $k\leq 4$ can be checked by direct computation.

\subsubsection{Initial analysis}
Let $\mathtt{I}$ be the Pfaffian system generated by
\[ \mathtt{I}=\langle \, \thetab_1,\,\theta_0,\,\xi,\,\theta_1,\,\theta_2,\, ... \,\theta_{k}\,\rangle.
\]
Our claim is that there is no nonzero closed 1-form $\alpha$ of the form
$$\ed f=\alpha= c h_2^{-\frac{1}{2}}\xib + \btheta, \quad \btheta\in\mathtt{I}.
$$
Note from the structure equation 
$$\ed (h_2^{-\frac{1}{2}}\xib)= - h_2^{-\frac{1}{2}}\theta_0\w\left(\theta_1+h_2\xi+\delta\xib\right)
-\frac{1}{2}h_2^{-\frac{3}{2}}\left(\theta_2\w\xib+h_3\xi\w\xib\right).
$$


\textbf{Step $\bar{1}$}.
Denote  a 1-form in $\texttt{I}$ by 
$\btheta=a_{\bar{1}}\thetab_1+a_0\theta_0+a_{\xi}\xi+a_1\theta_1+\sum_{j=2}^{k}a_j\theta_j.$
From the equation $\ed\alpha\equiv 0\mod\mathtt{I}$, collect $\thetab_2\w\xib$-terms and we get
\[ a_{\bar{1}}=-\sum_{j=3}^ka_j T^{\bar{2}}_j.
\]
Let $\mathtt{I}^{(0)}$ be the system generated by
\[ \mathtt{I}^{(0)}=\langle \, \btheta_0,\,\xi,\,\btheta_1,\,\btheta_2,\, ... \,\btheta_{k}\,\rangle,
\]
where we set
\begin{align}
\btheta_{0}&=\theta_0, \;\btheta_1=\theta_1, \;\btheta_2=\theta_2,\n\\
\btheta_j&=\theta_j-T^{\bar{2}}_j\thetab_1, \quad \tn{for}\;j\geq 3.\n
\end{align}

\textbf{Step $0$}.
Denote a 1-form in $\mathtt{I}^{(0)}$ by
$\btheta= a_0\btheta_0+a_{\xi}\xi+a_1\btheta_1+\sum_{j=2}^{k}a_j\btheta_j$.
Set 
$$\alpha= c h_2^{-\frac{1}{2}}\xib + \btheta$$ 
with this new 1-form $\btheta.$

From the equation $\ed\alpha\equiv 0\mod\mathtt{I}^{(0)}$, collect the $\thetab_1\w\xib$-terms and we get
\[ a_{0}=\sum_{j=2}^k a_j \left( -\delta j h_j +\bt_{0,j} \right),
\]
where
\[\bt_{0,j}=2\Big(\delxb(T^{\bar{2}}_j) -\sum_{s=2}^{j-1}T^s_j T^{\bar{2}}_s \Big).
\]
Let $\mathtt{I}^{(\xi)}$ be the  system generated by
\[ \mathtt{I}^{(\xi)}=\langle \,  \xi,\,\btheta^0_1,\,\btheta^0_2,\, ... \,\btheta^0_{k}\,\rangle,
\]
where we set $\btheta^0_j$'s by
\begin{align}
\btheta^0_{1}&=\theta_1,\n\\
\btheta^0_j&=\theta_j-T^{\bar{2}}_j\thetab_1
+ \left( -\delta j h_j+\bt_{0,j} \right)\theta_0,
\quad \tn{for}\;j\geq 2.\n
\end{align}

\textbf{Step $\xi$}.
Denote a 1-form in $\mathtt{I}^{(\xi)}$ by
$\btheta=  a_{\xi}\xi+a_1\btheta^0_1+\sum_{j=2}^{k}a_j\btheta^0_j$.
Set $\alpha= c h_2^{-\frac{1}{2}}\xib + \btheta$ with this new 1-form $\btheta.$

From the equation $\ed\alpha\equiv 0\mod\xib,\,\mathtt{I}^{(\xi)}$, collect the $\theta_0\w\thetab_1$-terms and we get
\[ a_{\xi}=a_1h_2+\sum_{j=2}^k a_j (h_{j+1}+\bt_{\xi,j}),
\]
where
\begin{align}
\bt_{\xi,j}&=\delta (j-1)T^{\bar{2}}_j
+\sum_{s=2}^{j-1}(-\delta sh_s+\bt_{0,s})(T^{\bar{2}}_j)^s
-\sum_{s=2}^{j-1}T^{\bar{2}}_s(\bt_{0,j})^s,\n\\
&=-2\delta T^{\bar{2}}_j
+\sum_{s=2}^{j-1}\left( \bt_{0,s}(T^{\bar{2}}_j)^s-T^{\bar{2}}_s(\bt_{0,j})^s\right).\n
\end{align}
Let $\mathtt{I}^{(1)}$ be the  system generated by
\[ \mathtt{I}^{(3)}=\langle \,  \btheta^1_1,\,\btheta^1_2,\, ... \,\btheta^1_{k}\,\rangle,
\]
where we set $\btheta^1_j$'s by
\begin{align}
\btheta^1_{1}&=\eta_1,\n\\
\btheta^1_j&=\eta_j-T^{\bar{2}}_j\thetab_1+ \left( -\delta j h_j+\bt_{0,j} \right)\theta_0
+\bt_{\xi,j}\xi, \quad \tn{for}\;j\geq 2.\n
\end{align}

\textbf{Step $1$}.
Denote a 1-form in $\mathtt{I}^{(1)}$ by
$\btheta= a_1\btheta^1_1+\sum_{j=2}^{k}a_j\btheta^1_j$.
Set $\alpha= c h_2^{-\frac{1}{2}}\xib + \btheta$ with this new 1-form $\btheta.$
Note that until this step there are no contribution from the derivative of $c h_2^{-\frac{1}{2}}\xib.$

From the equation $\ed\alpha\equiv 0\mod\xi,\,\mathtt{I}^{(1)}$, collect the $\theta_0\w\xib$-terms and we get
\[  \bt^1_1 a_{1}+\sum_{j=2}^k a_j  \bt^1_j=c 2\delta h_2^{-\frac{1}{2}},
\]
where
\begin{align}
\bt^1_1&= \gamma^2,\n \\
\bt^1_j &=2\delta(T_j-2T^{\bar{2}}_j\hb_2)-\bt_{\xi,j}\hb_2-\delxb(\bt_{0,j})
                 +\sum_{s=2}^{j-1}T^s_j\bt_{0,s},\quad \tn{for}\;j\geq 2.\n
\end{align}
Let $\mathtt{I}^{(2)}$ be the   system generated by
\[ \mathtt{I}^{(2)}=\langle \, \btheta^2_2,\, \btheta^2_3,\, ... \,\btheta^2_{k}\,\rangle,
\]
where we set $\btheta^2_j$'s by
\be
\btheta^2_j  = \btheta^1_{j}- \frac{\bt^1_j}{\gamma^2} \btheta^1_1, \quad \tn{for}\;j\geq 2.\n
\ee

\textbf{Step $2$}.
Denote a 1-form in $\mathtt{I}^{(2)}$ by
$\btheta= a_2\btheta^2_2+\sum_{j=3}^{k}a_j\btheta^2_j$.
Set 
$$\alpha= c\left( h_2^{-\frac{1}{2}}\xib +\frac{2\delta}{\gamma^2}h_2^{-\frac{1}{2}}\theta^1_1\right)+ \btheta$$ 
with this new 1-form $\btheta.$

From the equation $\ed\alpha\equiv 0\mod \thetab_2,\,\thetab_1,\,\theta_0,\, \xi,\,\mathtt{I}^{(2)}$, collect the $\theta_1\w\xib$-terms 
and we get
\[  \bt^2_2 a_{2}+\sum_{j=3}^k a_j  \bt^2_j=0,
\]
where
\begin{align}
\bt^2_2&= h_2\hb_2,\n \\
\bt^2_j &= \frac{jh_j}{2}\hb_2
+\frac{1}{\gamma^2}\Big( \delxb \bt^1_j - \sum_{s=2}^{j-1}T^s_j\bt^1_s \Big),\quad\tn{for}\;j\geq 3.\n
\end{align}
Let $\mathtt{I}^{(3)}$ be the system generated by
\[ \mathtt{I}^{(3)}=\langle \, \btheta^3_3,\, \btheta^3_4,\, ... \,\btheta^3_{k}\,\rangle,
\]
where we set $\btheta^3_j$'s by
\be
\btheta^3_j  = \btheta^2_{j}-\frac{\bt^2_j}{\bt^2_2} \btheta^2_2, \quad \tn{for}\;j\geq 3.\n
\ee

\textbf{Step $3$}.
Denote a 1-form in $\mathtt{I}^{(3)}$ by
$\btheta= a_3\btheta^3_3+\sum_{j=4}^{k}a_j\btheta^3_j$.
Set 
$$\alpha= c\left( h_2^{-\frac{1}{2}}\xib +\frac{2\delta}{\gamma^2}h_2^{-\frac{1}{2}}\theta^1_1\right)+ \btheta$$ 
with this new 1-form $\btheta.$

From the equation $\ed\alpha\equiv 0\mod \thetab_2,\,\thetab_1,\,\theta_0,\, \xi, \theta_1, \mathtt{I}^{(3)}$,  
collect the $\theta_2\w\xib$-terms and we get
\[  \bt^3_3 a_{3}+\sum_{j=4}^k a_j  \bt^3_j= c\frac{1}{2}h_2^{-\frac{3}{2}},
\]
where
\begin{align}
\bt^3_3&= -\gamma^2+2h_2\hb_2,\n \\
\bt^3_j &=x_{j-1},\quad\tn{for}\;j\geq 4.\n
\end{align}
Here 
$$x_{i}\in\mco(i)$$ 
is a generic notation for an element in $\mco(i).$

\subsubsection{Induction}
Suppose the claim of Lem.~\ref{lem:lemma5.4'} is true up to $k-1$.
Suppose also that the above analysis continues to work up to \tb{Step $i-1$} for $i\leq k$.

We then have the following schematic diagram for the 1-forms $\theta^i_j$.
\be
\begin{array}{rlllllllll}
\theta^3_3=&\theta^2_3 &+x_3\theta^2_2 &&&&&&& \\
\theta^4_4=&\theta^2_4 &+x_3\theta^3_3 &+x_4\theta^2_2&&&&&& \\
\theta^5_5=&\theta^2_5 &+x_4\theta^4_4 &+x_4\theta^3_3 &+x_5\theta^2_2&&&&& \\
\theta^6_6=&\theta^2_6 &+x_5\theta^5_5 &+x_5\theta^4_4 &+x_5\theta^3_3 &+x_6\theta^2_2&&&& \\
 &&&... &&&&&&\\
\theta^i_i=&\theta^2_i &+x_{i-1}\theta^{i-1}_{i-1}  &+x_{i-1}\theta^{i-2}_{i-2}
&+x_{i-1}\theta^{i-3}_{i-3}&...&+x_{i-1}\theta^3_3 &+x_i\theta^2_2& \\
\theta^i_{i+1}=&\theta^2_{i+1}&\qquad\cdot &+x_{i}\theta^{i-1}_{i-1}&+x_{i}\theta^{i-2}_{i-2}
&...& +x_{i}\theta^4_4 &+x_{i}\theta^3_3 &+x_{i+1}\theta^2_2& \\
\theta^i_{i+2}=&\theta^2_{i+2}&\qquad\cdot &\qquad\cdot&+x_{i+1}\theta^{i-1}_{i-1}
& ... &+x_{i+1}\theta^5_5  & +x_{i+1}\theta^4_4 &+x_{i+1}\theta^3_3 &+x_{i+2}\theta^2_2  \\
 &&& &&... &&&&
\end{array}
\ee
Here '$\cdot$' denotes $0$.

Denote a 1-form in $\mathtt{I}^{(i)}$ by
$\btheta= a_i\btheta^i_i+\sum_{j=i+1}^{k}a_j\btheta^i_j$. Set 
$$\alpha= c \left(h_2^{-\frac{1}{2}}\xib 
+\frac{2\delta}{\gamma^2}h_2^{-\frac{1}{2}}\theta_1+\sum_{j=3}^{i-1}g_j\theta^j_j\right)+ \btheta,$$ 
with this new 1-form $\btheta.$
 
From the equation $\ed\alpha\equiv 0\mod  \xi,\,\mathtt{I}^{(i)}$,  
collect the $\theta_{i-1}\w\xib$-terms.
Then we get
\be\label{eq:ithterm}
\bt^i_i a_{i}+\sum_{j=i+1}^k a_j  \bt^i_j=- c \delxb(g_{i-1}),
\ee
where we claim that
\begin{align}
\bt^i_i&=T^{i-1}_i+\delxb(x_{i-1})\in\mco(i-1),\label{eq:firstxi-1} \\
\bt^i_j&=x_{j-1},\quad\tn{for}\;j\geq i+1.\label{eq:secondj-1}
\end{align}
 
We verify this claim. The first equation \eqref{eq:firstxi-1} follows from the formula for $\theta^i_i$ in the diagram. For the second claim, again from the diagram when one differentiates $\theta^i_j, j\geq i+1,$ the possible contributions to $\theta_{i-1}\w\xib$-term are
\be\left\{\begin{array}{rl}
T^{i-1}_j\equiv x_{j-1}& \tn{from}\;\ed \theta^i_j, \n \\
\delxb(x_{j-1})& \tn{from}\;\ed (x_{j-1}\theta^{i-1}_{i-1}). 
\end{array}\right.\ee

In order to continue this induction to \tb{Step $i$}, we must show that $\bt^i_i\ne 0$.
By definition of $T_i$ in Eq.~\eqref{eq:Tj}, one finds
$$T_i^{i-1}=a_{i-1,0}R-a_{(i-1),(i-3)}h_2\hb_2.
$$
Suppose $\bt^i_i\equiv 0$ and there exists $x_{i-1}$ such that
\be\label{eq:observation}
\delxb(x_{i-1})=-T^{i-1}_i.
\ee
Since $\delxb (h_2^{-1}h_3)=R$ and $a_{(i-1),(i-3)}\ne 0$, there exists a constant $c'$ such that
$$\delxb(h_2^{-\frac{1}{2}}x_{i-1}-c'h_2^{-\frac{3}{2}}h_3)=\tn{(nonzero constant)}\cdot h_2^{-\frac{1}{2}}.
$$
This contradicts the induction hypothesis. Hence there does not exist such function $x_{i-1},$
and  $\bt^i_i$ does not vanish identically. We restrict to the open subset on which $\bt^i_i\ne 0$, the complement of which is dense in $\X{k}$.

\subsubsection{Completion of proof}
Continuing this analysis up to \tb{Step $k$}, one ends up with
$$\ed f=\alpha= c \left(h_2^{-\frac{1}{2}}\xib 
+\frac{2\delta}{\gamma^2}h_2^{-\frac{1}{2}}\theta_1+\sum_{j=3}^{k}g_j\theta^j_j\right),
$$ 
where the sequence $g_j$ is inductively defined by
\begin{align*}
g_3&=\frac{h_2^{-\frac{3}{2}}}{2(-\gamma^2+h_2\hb_2)},  \\
g_j&=-\frac{1}{\bt^j_j}\delxb (g_{j-1})\in\mco(j-1), \quad j\geq 4.
\end{align*}
In any case the structure equation for $f$ takes the form
$$\ed f=  c  \left(h_2^{-\frac{1}{2}}\xib +f^{\xi}\xi+\sum_{j=-1}^{k}f^j\theta_j\right),
$$ 
where $f^j\in\mco(k-1)$ for $j\geq 3.$
 
\begin{itemize}
\item  Suppose $c=0$. Then $\ed f =0$ and $f$ is a constant.

\item Suppose $c\ne 0$. We show that this leads to a contradiction. 
Applying  the commutation relation $ [E_k, \delxb]=\sum_{j\geq k+1}T_j^k E_j$ to the given function $f$,
$$ \delxb(f^k)=E_k(h_2^{-\frac{1}{2}})=0.
$$
By the induction hypothesis $f^k$ is a constant multiple of $h_2^{-\frac{k}{2}}$.
\begin{itemize}
\item  Suppose $f^k=0$. Then $f\in\mco(k-1)$ and again by the induction hypothesis $f\equiv \tn{constant},$
and hence $c=0$, a contradiction.

\item Suppose $f^k\ne 0$.
Applying  the commutation relation $ [E_{k-1}, \delxb]=\sum_{j\geq k}T_j^{k-1} E_j$ to the given function $f$,
$$ -\delxb(f^{k-1})=T^{k-1}_k f^k. 
$$
Since $k\geq 5$,  by the similar argument as before, following Eq.~\eqref{eq:observation}, this leads to a contradiction.
\end{itemize}
\end{itemize}
This completes the proof.

\section{Jacobi fields}\label{sec:Jacobifields}
Let $\Phi\in\Cv{(2k)}$ be a conservation law in reduced form \eqref{eq:H2k} under the normalization \eqref{eq:H2krelations}. Prop.~ \ref{prop:roughnormal} shows that the pair of principal coefficients $\{ B^{1,2k},\,C^{1,2k}\}$ indicate the order of $\Phi$; the order of $\Phi$ is less than $2k+1$ and $\Phi\in\Cv{(2k-2)}$ when the principal coefficients vanish. By a repeated application  of Prop.~\ref{prop:roughnormal}, $\Phi$ is a classical conservation law when all the potential principal coefficients $\{B^{1,2j},\,C^{1,2j}\}$ vanish for $j=k,\,k-1,\,...\, 1$. In this case, from the structure equation \eqref{eq:classicalJacobi} a classical conservation law $\Phi$ becomes trivial when the coefficient $A$ also vanishes.

The relations \eqref{eq:A2k+1}, \eqref{eq:A2k+1b} on the other hand, which state that  the principal coefficients $\{ B^{1,2k},\,C^{1,2k}\}$ are the derivatives of the coefficient $A$, suggests a different, somewhat opposite perspective. Suppose for a conservation law $\Phi\in\Cv{(2k)}$ the coefficient $A$ vanishes. Then from the argument above all the potential principal coefficients necessarily vanish and $\Phi$ is classical, and hence it is trivial. Since the equation $\ed\Phi=0$ is a system of linear differential equations for the coefficients of $\Phi$, this implies that \emph{every coefficient of a conservation law $\Phi$ lies in the linear span of the coefficient $A$ and its successive derivatives.} It follows that for a reduced 2-form $\Phi\in H^{(2k)}$ the conservation law equation $\ed\Phi=0$ is a system of linear partial differential equations for the single variable $A$.
 
The principal part of this differential equation for the generating function $A$ is a second order linear differential equation called Jacobi equation.\footnotemark\footnotetext{It turns out that the remaining equations for $A$ are differential consequences of Jacobi equation.} It can be regarded as the linearization of the EDS for CMC surfaces, although more precisely Jacobi equation governs the infinitesimal conformal CMC-$\delta$ deformation with the prescribed Hopf differential. In terms of the exact sequence \eqref{eq:exactsequence} Jacobi equation is the defining equation for the second piece $E^{1,1}_1.$

\two
We proceed to derive Jacobi equation. Given a differentiated conservation law $\Phi\in \Cv{(\infty)}$ as a reduced 2-form, consider first the equation $\ed\Phi\equiv 0\mod \,\theta_0$. One finds that
\be\label{eq:Agenerator}
\ed A \w \Psi +\ed \theta_0\w\sigma\equiv 0\mod \,\theta_0,\,F^2\Omega^3. 
\ee
Since $\Psi=\Im(\theta_1\w\xi),\,\ed\theta_0=-\Re(\theta_1\w\xi)$, this implies that
\be\label{eq:sigmaform}
\sigma \equiv \im (\del A-\delb A) = -\JAI \ed A \;\; \mod \I{\infty}
\ee
(in fact $\sigma=-\JAI \ed A$).

Recall that for a function $f$ on $\mcf^{(\infty)}$, we use the notation
\[ \ed f \equiv f_{\xi}\xi + f_{\xib}\xib \mod \rho, \, \I{\infty}  \]
for total derivatives with respect to $\xi,\,\xib$. With the given $\sigma$, the  $\theta_0 \w \xi \w \xib$-term of $\ed \Phi$ then shows that
\be
A_{\xi,\xib}+A_{\xib,\xi}+(\gamma^2+h_2\hb_2) A=0.\n
\ee
Since the total derivatives commute, $A_{\xi,\xib}=A_{\xib,\xi}$, this becomes
\be\label{eq:surfaceJacobi}
A_{\xi,\xib}+\frac{1}{2}(\gamma^2+h_2\hb_2) A=0.
\ee

Consider an immersed CMC surface $\Sigma$ with the induced metric $\underline{\rm{I}}=\xi\circ\xib.$ The operator $\JAI$ restricts to the usual Hodge star operator on $\Sigma$, and one has
\[ -\ed \JAI \ed A=-4 A_{\xi,\xib} \frac{\im}{2}\xi\w\xib=\Delta A \frac{\im}{2}\xi\w\xib,\]
where we are using the sign convention for the Laplacian $\Delta$ as
\[
-\ed \JAI \ed A = \Delta(A) \ed\tn{Area}.
\]
 Thus Eq.~\eqref{eq:surfaceJacobi} is recognized as the familiar elliptic Jacobi equation on a CMC surface;
 \be\label{eq:LaplaceJacobi}
\Delta A-2(\gamma^2+h_2\hb_2) A=0.
\ee

For a scalar function $A:\X{\infty}\to\C$ we set the \tb{Jacobi operator}
\[
\mce(A):=A_{\xi,\xib}+\frac{1}{2}(\gamma^2+h_2\hb_2) A.
\]
We will also denote by $\mce$ the induced operator on CMC surfaces. Motivated in part by Definition \ref{defn:classicalJacobifield}, we give the following definition.
\begin{defn}\label{defn:infJacobifield}
A scalar function $A$ on $\xinf$ is a \tb{Jacobi field} if it lies in the kernel of Jacobi operator,
\be\label{eq:infJacobifield}
\mce{(A)}:= A_{\xi,\xib}+\frac{1}{2}(\gamma^2+h_2\hb_2) A=0.
\ee
The vector space of Jacobi fields is denoted by  $\mfj^{(\infty)}$.

Let $\mfj^{(k)}\subset\mfj^{(\infty)}$ be the subspace of \tb{Jacobi fields of order $\leq k+1$}  defined on $\X{k}$. The space of \tb{Jacobi fields of order $k+1$} is defined as the quotient space
\begin{align}\label{eq:highJacobik}
\mfj^k&=\mfj^{(k)}/\mfj^{(k-1)},\quad k\geq 1,   \\
\mfj^0&=\mfj^{(0)}=\{\tn{classical Jacobi fields}\}.\n
\end{align}
\end{defn}
 
In Sec.~\ref{sec:classicallaws} we recorded the structure equation for classical Jacobi fields. As remarked there, note that the pair of equations $\mce^0_H=\mce^0_V=0$ in \eqref{eq:classicalJacobifield} imply Eq.~\eqref{eq:infJacobifield}. Thus a classical Jacobi field is indeed a Jacobi field.

\begin{exam}\label{ex:HopfJacobi}
We give a description of a (unique) differentiated conservation law of order 3 defined on $\X{2}$ and its associated Jacobi field,

Consider the square root of the Hopf differential $\omega=h_2^{\frac{1}{2}}\xi$. It is a (possibly double-valued) holomorphic 1-form on an integral surface. Hence $\ed\omega\equiv 0\mod \I{\infty}$, and $\omega$ represents a conservation law. Explicitly, one computes
\[ \ed\omega=\frac{1}{2}h_2^{-\frac{1}{2}} \theta_2\w\xi
-h_2^{\frac{1}{2}}\theta_0\w\thetab_{1}
-h_2^{\frac{1}{2}} \hb_2\theta_0\w\xib
-\delta h_2^{\frac{1}{2}}\theta_0\w\xi.
\]

In order to put this into the normal form of $H^{(2)}$, we augment $\omega$ by adding an appropriate 1-form in the ideal $\I{2}$.
\[ {\omega}' = \omega +\left( \frac{1}{2}h_2^{-\frac{1}{2}}\theta_1 +\frac{1}{4}\frac{h_3}{h_2^{\frac{3}{2}}}\theta_0\right).\]
Then we get
\begin{align}
\ed {\omega}' &=
\frac{1}{8}\frac{h_{3}}{h_{2}^{\frac{3}{2}}}(\theta_1\w\xi-\thetab_1\w\xib)
+\theta_0\w\left(
\frac{1}{8}{\frac { \left( 3\,h_3^{2}-2\,h_{{4}}h_{{2}}
 \right) }{h_{2}^{\frac{5}{2}}}}\xi \
+\frac{1}{4}\frac{(\gamma^2-h_2\hb_2)}{h_{2}^{\frac{1}{2}}}\xib
\right)\\
&\quad+\theta_0\w\left(
-\frac{1}{2}h_{2}^{\frac{1}{2}}\thetab_1
+\frac{1}{2}\delta h_{2}^{-\frac{1}{2}}\theta_1
+\frac{3}{8}\frac{h_3}{h_{2}^{\frac{5}{2}}}\theta_2
-\frac{1}{4} h_{2}^{-\frac{3}{2}}\theta_3
\right)\n\\
&\quad+\frac{1}{4}h_2^{-\frac{3}{2}}\theta_1\w\theta_2.\n
\end{align}
The Jacobi field of order 3 is the coefficient of the 2-form $(\theta_1\w\xi-\thetab_1\w\xib)$, which is  up to scaling by constant  
\[  z_3=h_{2}^{-\frac{3}{2}}h_{3}.\]
\end{exam}
 
\one
The recursion formula for Jacobi field to be developed in Sec.~\ref{sec:inductiveformula} 
will produce an infinite sequence of higher-order Jacobi fields 
with $z_3$ as the initial data.

\two 
Let us record a lemma on the normal form of a Jacobi field.
\begin{lem}\label{lem:Jacobinormal}
Let $A\in \mfj^{(k)}\subset\mco(k+1)$ be a Jacobi field. Suppose $A^{k+1}=E_{k+1}(A)\ne 0$.
Then $A$ is at most linear in the variable $z_{k+1}=h_{2}^{-\frac{k+1}{2}}h_{k+1}$, and up to constant scale
$$ A=z_{k+1}+\mco(k).
$$
\end{lem}
\begin{proof}
Since $A\in\mco(k+1)$, we have
\begin{align}
A_{\xi}&\equiv h_{k+2} A^{k+1} \mod \mco(k+1), \n\\
A_{\xi,\xib}&\equiv h_{k+2} \delxb(A^{k+1})\mod \mco(k+1),\n \\
&\equiv 0\mod\mco(k+1), \quad \tn{for $\mce(A)=0$.}\n
\end{align}
This forces $\delxb(A^{k+1})=0$. By Lemma \ref{lem:lemma5.4},  $A^{k+1}$ is a constant multiple of $h_2^{-\frac{k+1}{2}}$. It follows that $A$ is linear in the highest order variable $z_{k+1}.$
\end{proof}
Although it appears as a simple observation, 
this normal form will be refined in Sec.~\ref{sec:Noethertheorem} 
to the splitting of a Jacobi field  into classical, and higher order parts.

\subsection{Jacobi fields and symmetries}\label{sec:JacobiSymmetry}
It was shown in Sec.~\ref{sec:classicallaws} that there exists a canonical isomorphism between the classical Jacobi fields and the classical symmetries. The higher-order analogue of this isomorphism is true, and a Jacobi field uniquely determines a (vertical) symmetry vector field on $\xinf$.  In this section, we determine the  inductive differential algebraic formula for the coefficients of a symmetry generated by a Jacobi field.

\two
Recall the coframe of $\mcf^{(\infty)}$,
\[\{\,\rho,\,\xi,\,\xib,\,\theta_0,\,\theta_1,\,\thetab_1,\,\theta_2,\, \thetab_2,\,... \,\}.\]
The dual frame is denoted by
\[\{\,E_\rho,\,E_{\xi},\,E_{\xib},\,E_0,\,E_1,\,E_{\bar{1}},\,E_2,\, E_{\bar{2}},\,... \,\}.\]
By a vector field on $\xinf$ we mean a vector field (derivation) on $\mcf^{(\infty)}$ of the form
\be\label{eq:symmetryvector}
V=V_{\xi}E_{\xi} +V_{\xib}E_{\xib}+V_0 E_0+\sum_{j=1}^{\infty} ( V_j E_j + V_{\bar{j}}E_{\bar{j}})
\ee
which is invariant under the induced action of the structure group $\SO(2)$ of the bundle $\mcf^{(\infty)}\to\xinf$.

\begin{defn}
Let $V\in H^0(T\xinf)$ be a vector field. It is a \tb{symmetry} of the EDS $(\xinf,\I{\infty})$ when the Lie derivative $ \mcl_V$ preserves the ideal $\I{\infty}$,
\[  \mcl_V  \I{\infty} \subset \I{\infty}.\]
The algebra of symmetry is denoted by $\mathfrak{S}$.

A symmetry $V\in\mathfrak{S}$ is \tb{vertical} when $V_{\xi}=V_{\xib}=0$ and it has no $E_{\xi}, \,E_{\xib}$ components. The subspace of vertical symmetry is denoted by $\mathfrak{S}_v$.\footnotemark
\footnotetext{The subspace $\mathfrak{S}_v$ does not form a subalgebra.}

The \tb{degree} of a symmetry $V\in\mathfrak{S}$ is the least integer $m\geq 0$ such that
\[  \mcl_V  \I{k} \subset \I{k+m+1},\quad \forall \;k\geq 0.\]
\end{defn}
\noi
Note that the classical symmetries $\mathfrak{S}^{0}$ have degree $-1$ by definition.
\begin{exam}
Consider the vector fields of the form $\{ \, V_{\xi}E_{\xi}+V_{\xib}E_{\xib} \}$. They are horizontal with respect to the formally integrable distribution defined by $\iinf$, and hence they are symmetries by definition. These vector fields are called  trivial symmetries. The structure equation for prolongation in Lemma \ref{lem:prolongedstrt} shows that a nonzero trivial symmetry has degree 1. This may explain the appearance of the extra '$1$' in '$k+m+1$' in the definition of degree of a symmetry.

In the literature a symmetry is sometimes defined as a section of the quotient space $T\xinf/(\I{\infty})^{\perp}.$
\end{exam}

We wish to give an analytic characterization of a symmetry making use of the structure equation \eqref{3strt2}, and Lemma \ref{lem:prolongedstrt}. Without loss of generality, consider a vertical symmetry
\be\label{eq:symvf}
V=V_0 E_0+\sum_{j=1}^{\infty} ( V_j E_j + V_{\bar{j}}E_{\bar{j}}).
\ee
We first claim that $V_j, V_{\bar{j}}$'s are determined by $V_0$ and its derivatives.

\two
\textbf{Step 0}.
The condition that the Lie derivative $\mcl_V\theta_0\equiv 0\mod\I{\infty}$ gives
\begin{align}
\ed V_0+V\lrcorner \left(-\frac{1}{2}(\theta_1\w\xi+\thetab_1\w\xib)\right)&\equiv
\ed V_0  -\frac{1}{2}(V_1 \xi+V_{\bar{1}} \xib) \n\\
&\equiv 0 \mod \iinf.\n
\end{align}
Thus
\be\label{eq:symstep0}
\ed V_0\equiv \frac{1}{2}(V_1 \xi+V_{\bar{1}} \xib)\mod \iinf.
\ee

\textbf{Step 1}.
The condition that the Lie derivatives $\mcl_V\theta_1, \, \mcl_V\thetab_1\equiv 0\mod\I{\infty}$ gives
\be\label{eq:symstep1}\begin{array}{rcrll}
 \ed V_1 + \im V_1\rho&\equiv&
  (V_2-2\delta h_2 V_0 )\xi&
-(\gamma^2+h_2\hb_2)V_0\xib.& \\
\ed V_{\bar{1}} - \im V_{\bar{1}}\rho&\equiv&
(V_{\bar{2}}-2\delta \hb_2 V_0 )\xib&
-(\gamma^2+h_2\hb_2)V_0\xi\,&\mod\;\iinf.
\end{array}
\ee

\textbf{Step j}.
The condition that the Lie derivatives $\mcl_V\theta_j, \, \mcl_V\thetab_j\equiv 0\mod\I{\infty}$ for $j\geq 2$ gives after a short computation
\be\label{eq:symstepj}\begin{array}{rcll}
 \ed V_j +j \im V_j\rho&\equiv&
 \left(V_{j+1} - (\delta h_{j+1}+h_2T_j) V_0 -\frac{jh_j}{2}(\delta V_1-h_2V_{\bar{1}}) \right)\xi&\\
&&\Big(  - (\hb_2 h_{j+1}+\delta T_j) V_0 -\frac{jh_j}{2}(\hb_2  V_1- \delta V_{\bar{1}})
+(T^{\bar{2}}_jV_{\bar{2}}+ \sum_{s=2}^{j-1} T^s_jV_s)\Big)\xib& \\
&&\mod\;\iinf,&\\
 \ed V_{\bar{j}} -j \im V_{\bar{j}}\rho&\equiv&
 \left(V_{\ol{j+1}} - (\delta \hb_{j+1}+\hb_2\ol{T}_j) V_0 -\frac{j\hb_j}{2}(\delta V_{\bar{1}}-\hb_2V_{1}) \right)\xib&\\
&&\Big(  - (h_2 \hb_{j+1}+\delta \ol{T}_{j}) V_0 -\frac{j\hb_j}{2}(h_2  V_{\bar{1}}- \delta V_{1})
+(\ol{T}^{2}_{j}V_{2}+ \sum_{s=2}^{j-1} \ol{T}^{\bar{s}}_{j}V_{\bar{s}})\Big)\xi& \\
&&\mod\;\iinf.&\\
\end{array}
\ee

Note that all the coefficients $\{\,V_j, \,V_{\bar{j}}\,\}$ for $j\geq 1$ are determined by $V_0$ and its successive $\xi, \xib$-derivatives. The coefficient $V_0$ is called the generating function of the symmetry.
\begin{prop}\label{prop:symmetry}
The generating function of a symmetry is a Jacobi field. 
Conversely, a Jacobi field $A$ uniquely determines a vertical symmetry $V_A$ of the form \eqref{eq:symvf} with the generating function $V_0=A$.
We consequently have a canonical isomorphism
\[ \mfj^{(\infty)}\simeq \mathfrak{S}_v.\]
\end{prop}
\begin{proof}
From Eq.~\eqref{eq:symstep0}, \eqref{eq:symstep1}, one notes that
\be
(V_0)_{\xi,\xib}=\frac{1}{2}(V_1)_{\xib}=-\frac{1}{2}(\gamma^2+h_2\hb_2)V_0.\n
\ee

Conversely, suppose a Jacobi field $A=V_0$ is given. The $\xi$-components of the recursive equations ~\eqref{eq:symstep0}, \eqref{eq:symstep1}, and \eqref{eq:symstepj} determine $\{\,V_j\,\}$ for all $j\geq 1$. Given this sequence of coefficients, the compatibility condition to be checked is that for each $j\geq 0$ the following equation holds.
\[ (V_j)_{\xi,\xib} -(V_j)_{\xib,\xi} =\frac{j}{2}RV_j.\]

The cases $j=0, \,1,\,2$ are easily checked. For $j \geq 3$, one computes from Eq.~\eqref{eq:symstepj} as follows.
\begin{align}
 (V_j)_{\xi,\xib}&=
\delxb \Big(V_{j+1} - (\delta h_{j+1}+h_2T_j) V_0 -\frac{jh_j}{2}(\delta V_1-h_2V_{\bar{1}}) \Big)\n\\
&=\Big(- (\hb_2 h_{j+2}+\delta T_{j+1}) V_0 -\frac{(j+1)h_{j+1}}{2}(\hb_2  V_1- \delta V_{\bar{1}})
+(T^{\bar{2}}_{j+1} V_{\bar{2}}+ \sum_{s=2}^{j} T^s_{j+1}V_s)\Big) \n\\
&-(\delta T_{j+1}+h_2\delxb T_j)V_0
-\frac{j T_j}{2}(\delta V_1-h_2V_{\bar{1}}) \n \\
&- (\delta h_{j+1}+h_2T_j) \frac{1}{2}V_{\bar{1}}
-\frac{jh_j}{2}\left(-\delta (\gamma^2+h_2\hb_2)V_0-h_2(V_{\bar{2}}-2\delta\hb_2V_0 )  \right).\n
\end{align}
On the other hand,
\begin{align}
(V_j)_{\xib,\xi}&=\delx\Big(  - (\hb_2 h_{j+1}+\delta T_j) V_0 -\frac{jh_j}{2}(\hb_2  V_1- \delta V_{\bar{1}})
+(T^{\bar{2}}_jV_{\bar{2}}+ \sum_{s=2}^{j-1} T^s_jV_s)\Big)\n\\
&=\Big(  - (\hb_2 h_{j+2}+\delta \delx (T_j)) V_0 -\frac{jh_{j+1}}{2}(\hb_2  V_1- \delta V_{\bar{1}})
+(\delx (T^{\bar{2}}_j)V_{\bar{2}}+ \sum_{s=2}^{j-1} \delx(T^s_j)V_s)\Big)\n\\
&+\Big(  - (\hb_2 h_{j+1}+\delta T_j) \frac{1}{2}V_1
-\frac{jh_j}{2}(\hb_2  (V_2-2\delta h_2V_0)+ \delta (\gamma^2+h_2\hb_2)V_0)\Big)\n\\
&  + T^{\bar{2}}_j \Big(  - h_2 \hb_{3} V_0 -\hb_2(h_2  V_{\bar{1}}- \delta V_{1}) \Big)
+ \sum_{s=2}^{j-1} T^s_j\delx V_s. \n
\end{align}

Now we compute $(V_j)_{\xi,\xib}-(V_j)_{\xib,\xi}$ by collecting each $V_s$-terms.
\begin{itemize}
\item $V_0$-term:
After cancellations using Eq.~\eqref{eq:Trecurs},
\[\Big(\delta (-\delx T_j+\sum_{s=2}^{j-1}T^s_j h_{s+1})
+h_2(-\delxb T_j+T^{\bar{2}}_j\hb_3+\sum_{s=2}^{j-1}T^s_j T_s)\Big).
\]
This vanishes for $T_j$ is a function of $\{\,\hb_2,\,h_2,\, ... \, h_{j-1}\,\}$.

\item $V_1$-term:
\[ \frac{\delta}{2}\Big(-2T^{\bar{2}}_j\hb_2+\sum_{s=2}^{j-1}T^s_j s h_s-(j-1)T_j \Big).\]
This vanishes for $T_j$ is homogeneous of weight $(j-1)$ under the induced action by the structure group $\SO(2)$.

\item $V_{\bar{1}}$-term:
$ -\frac{h_2}{2}\Big(-2T^{\bar{2}}_j\hb_2+\sum_{s=2}^{j-1}T^s_j s h_s-(j-1)T_j \Big).$

\item $V_2$-term:
$-\delx T^2_j+T^2_{j+1}+\frac{jh_j}{2}\hb_2.$

\item $V_{\bar{2}}$-term:
$-\delx T^{\bar{2}}_j+T^{\bar{2}}_{j+1}+\frac{jh_j}{2}h_2.$

\item $V_s$-term:
$-\delx T^s_j+T^s_{j+1}-T^{s-1}_j.$

\item  $V_j$-term:
$T^j_{j+1}-T^{j-1}_j -\frac{j}{2}R.$
\end{itemize}
\two\noi
These identities for $T_j$ are  recorded in Cor.~\ref{cor:Tidentity}. The compatibility equation for $V_{\bar{j}}$ is verified similarly.
\end{proof}
 
\section{Relation to the PDE system}\label{sec:PDEsystem}
In this section we give a description of the relation of the EDS for CMC surfaces to the PDE for a scalar function of two variables 
\be\label{eq:uPDE}
u_{z \zb} + f(z,\zb,u)=0
\ee
for an appropriate function $f(z,\zb,u)$.  
This would allow one to apply the results for elliptic Poisson equation \cite{Fox2011} to CMC surfaces.
One of the other objectives is to provide a setting to apply the analytic methods and results of the elliptic PDE \eqref{eq:uPDE} to the study of the geometry of CMC surfaces at  umbilics.
\subsection{Integrable extension}\label{sec:localextension}
We refer to Sec.~\ref{sec:extension} for the relevant details on integrable extension.

Let $X'=\C \times \R \times \C$ with coordinates $z,u,u_1$. Let $f:\R \to \R$ be a scalar function in '$u$' variable.  On $X'$ define
\begin{align*}
\xi'&=\ed z,\\
\theta_0'&=\ed u - u_1 \ed z - u_{\bar{1}} \ed \zb, \\
\eta_1'&=\ed u_1 + f(u) \ed \zb,\\
\Psi'&= \Im(\eta_1' \w \xi'),
\end{align*}
and the ideal 
$$\mci'=\langle \theta_0', \ed\theta_0', \Psi' \rangle.$$  
The integral surfaces of $\mci'$ on which $\xi'\w\xib'\ne 0$ are locally the graphs of  solutions to the PDE \eqref{eq:uPDE}.  
\begin{rem}
The elliptic Monge-Ampere system $(X',\mci')$ possesses the function $z$ for which $\ed z$ is the characteristic. Hence it is not locally equivalent to the CMC system $(X,\mci).$
\end{rem}

Let $U \subset X$ be an open subset. Choose a section of the 1-form $\xi$ on $U$.  
Let 
$$Y'=U \times X',$$
and define the ideal generated by
\[  \check{\tn{I}}'=\mci_{\vert_U}\oplus \mci' \oplus \langle \Theta_1, \Theta_2, \Theta_3 \rangle,
\]
where
\begin{align}\label{eq:Thetaforms}
\Theta_1&=\xi-e^u\xi',\n\\
\Theta_2&=\eta_1 - e^{-u} \xi', \n\\
\Theta_3&=\im  \rho -(u_{\bar{1}} \xib' - u_1 \xi'), \n
\end{align}
and 
\be
f(u)=\frac{1}{4}(\gamma^2 e^{2u}-e^{-2u}).\label{eq:PDEf}
\ee
\begin{lem}\label{lem:PDEextension}
The EDS $(Y',\check{\tn{I}}')$ is an integrable extension of $(X',\mci')$.
\end{lem}
\begin{proof}
By a direct computation one finds that
\begin{align}
\mci_{|_U}&\equiv 0 \mod \theta_0,\Theta_1,\Theta_2,\Theta_3, \n\\
\ed\Theta_i&\equiv0\mod\theta_0,\Theta_1,\Theta_2,\Theta_3,\mci'.\n
\end{align}
\end{proof}
On an integral surface of $(Y',\check{\tn{I}}')$ satisfying the independence condition 
$\xi\w\xib\ne 0$, one calculates that $h_2 \equiv e^{-2u}$.  Using this in the structure equation for $\ed h_2$ uncovers the relation $h_3 \equiv-4 u_1 e^{-3u}.$ Thus on an integral surface we have 
$$h_2^{-\frac{3}{2}}h_3  \equiv -4 u_1.$$ 
This observation has the following implications. For the PDE system $\mci'$ with the given \eqref{eq:PDEf} it was found that the spectral symmetry could be extended to an $\s{1}$-action on the infinite prolongation of the EDS and that this symmetry has the generating function $\Im(z u_1)$, \cite{Fox2011}.  We would like to find the corresponding spectral symmetry for the CMC EDS studied in this paper.  Notice that
\[
\ed z =\xi' \equiv e^{-u} \xi \equiv \sqrt{h_2} \xi =:\omega \, \mod \, \check{\tn{I}}'.
\]
Thus in order to write down a local generating function on $(\Xh{\infty}_*,\Ih{\infty})$ for the spectral symmetry, it is natural to introduce a primitive of $\omega$, say $z$, and then let $q=\Im(z h_2^{-\frac{3}{2}}h_3)$.  We will develop the related ideas in Sec.~\ref{sec:spectralsymmetry}. 

\two
Implicit in the above discussion is the assumption that we are studying the integral surfaces on which $\eta_1 \equiv e^{-u} \ed z$ does not vanish. This is equivalent to assuming that there are no umbilic points. In order to study the spectral symmetry in the context of CMC surfaces with umbilics, one would need to make an integrable extension which allows for $\eta_1$ to vanish at a point. There is such a generalization that can allow for umbilics. 

The 1-form $\omega$ pulls back to any (smooth) integral surface to be a holomorphic form. Near a zero of $\omega$ a local holomorphic coordinate can be chosen so that the Hopf differential is written as $\eta_1 \circ \xi \equiv z^m (\ed z)^2$ for an integer $m\geq 1$. For simplicity we will make the assumption that $m=2n$ is even so that locally $\omega \equiv z^n \ed z$ (if $m$ were not even, one could work on the normalization of the double cover of the integral surface defined by the square root of the Hopf differential). One finds that such integral surfaces are locally equivalent to the solutions to a slightly different PDF system. We present the details of this analysis in the below using a new integrable extension. 

Let $X'=\C \times \R \times \C$ with coordinates $z,u,u_1$. 
Let $f:\C \times \R \to \R$ be a scalar function in '$z, \zb, u$' variables.  
On $X'$ define as before
\begin{align*}
\xi'&=\ed z\\
\theta_0'&=\ed u - u_1 \ed z -  u_{\bar{1}} \ed \zb \\
\eta_1'&=\ed u_1 + f(z, \zb, u) \ed \zb\\
\Psi'&= \Im(\eta_1' \w \xi')
\end{align*}
and the ideal $\mci'=\langle \theta_0', \ed\theta_0', \Psi' \rangle$.  

Let $U \subset X$ be an open set on which there is a section of 1-form $\xi$.  On $Y'=U \times X'$ define the ideal
\[\check{\tn{I}}'=  \mci_{\vert_U} \oplus \mci' \oplus \langle \Theta_1, \Theta_2, \Theta_3 \rangle,
\]
where
\begin{align}
\Theta_1&=\xi-e^u\xi' ,   \label{eq:strtum}   \\
\Theta_2&=\eta_1 - e^{-u} z^{2n}\xi' ,  \n  \\
\Theta_3&=\im  \rho -(u_{\bar{1}} \xib' - u_1 \xi'),  \n 
\end{align}
and
\be
f(z, \zb, u) =\frac{1}{4}(\gamma^2 e^{2u}-|z|^{4n}e^{-2u}). \label{eq:PDEfum}
\ee
The definition of $f(z, \zb, u)$ and $\Theta_2$ now reflect the existence of an even degree umbilic point at $z=0$; otherwise the system is the same as above.

The following  lemma is the analogue of Lemma \ref{lem:PDEextension} in the presence of umbilics. 
\begin{lem}\label{lem:PDEextension2}
The EDS $(Y',\check{\tn{I}}')$ is an integrable extension of $(X',\mci')$.
\end{lem}
 
\one
We give an indication of an idea how Lemma \ref{lem:PDEextension2} can be used to understanding the umbilics of CMC surfaces. Suppose that $n>0$. For smooth solutions of 
\[
u_{z \zb} +\frac{1}{4}(\gamma^2 e^{2u}-|z|^{4n}e^{-2u})=0 
\]
near $z=0$, $u(z,\zb)$ is close to a solution of the Liouville equation 
\[
u_{z \zb} +\frac{1}{4}\gamma^2 e^{2u}=0.
\] 
It is well known that solutions of the Liouville equation are locally equivalent to holomorphic curves in $\C^2$, \cite{Bryant2003}. Thus we expect that CMC surfaces  should behave like holomorphic curves  near the umbilic points. Considering the inverse coordinate on the open subset of $\X{1}$ on which $\eta_1\ne 0$, this approximately-holomorphic property is likely true also at the ends of complete CMC surfaces which correspond to the poles of  Hopf differential.

\subsection{Local equivalence with Abelian extension}\label{sec:localequivalence}
Let $X'^{(k-2)}=\C \times \R \times \C^{k-1}$ with coordinates $z,u=u_0,u_1, \, ... \, u_{k-1}$.
Define inductively for $j\geq 1$,
\begin{align*}
\eta_j'&=\ed u_j + S_j \ed \zb,\\
\theta_j'&= \eta_j' - u_{j+1} \ed z,
\end{align*}
where 
\begin{align*}
S_1&=f(u) \quad \tn{from} \;\eqref{eq:PDEf}, \\
S_j(u,u_1,u_2, \, ... \, u_{j-1})&=\del_z S_{j-1}.
\end{align*} 
Here $\del_z$ is the total derivative with respect to $z$ such as
\begin{align*}
S_2&=\del_z f(u) =f'(u) u_1, \\
S_3&=\del_z S_2 =f'(u) u_2+f''(u)u_1^2.
\end{align*}

Set the sequence of ideals inductively by
$$\mci'^{(k-2)}= \mci'^{(k-3)}\oplus \langle \theta_{k-2}', \eta_{k-1}'\w\ed z \rangle.$$  
Then $(X'^{(k-2)}, \mci'^{(k-2)})$ is the $(k-2)$-th prolongation of $(X',\mci').$

\two
Let
$$Y'^{(k-2)}=U \times X'^{(k-2)},$$
and define the ideal generated by
\[  \check{\tn{I}}'^{(k-2)}=\mci_{\vert_U}\oplus\mci'^{(k-2)} \oplus
 \langle \Theta_1, \Theta_2, \Theta_3  \rangle.
\]
Similarly as before we have,
\begin{lem}\label{lem:PDEextensionk-2}
The EDS $(Y'^{(k-2)},\check{\tn{I}}'^{(k-2)} )$ is an integrable extension of $(X'^{(k-2)}, \mci'^{(k-2)})$.
\end{lem}

We give an interpretation of $(Y'^{(k-2)},\check{\tn{I}}'^{(k-2)} )$ in term of the CMC system.
Recall $(  U \subset X, \mci\vert_U)$.
Let $(\hat{U}^{(k)}, \Ih{k}\vert_{\hat{U}^{(k)}})$ be the double cover of its $k$-th prolongation.\footnotemark\footnotetext{The double cover will be defined in Sec.~\ref{sec:doublecover}. For now it suffices to take a section of $h_2^{\frac{1}{2}}$ on $U^{(k)}$.}
Set 
$$\Zh^{(k)}=\hat{U}^{(k)}\times \C$$ 
and let $z_0$ be the $\C$-coordinate. Define the Abelian\footnotemark\footnotetext{An integrable extension is Abelian when it is defined by introducing a potential for a (possibly trivial) conservation law.}  integrable extension generated by
$$\check{\tn{I}}^{(k)}=\Ih{k}\vert_{\hat{U}^{(k)}}\oplus \langle \Theta_{z_0}  \rangle,
$$
where
$$ \Theta_{z_0}=\ed z_0 - h_2^{\frac{1}{2}}\xi.
$$ 
Let $Y^{(k)}\subset \Zh^{(k)}$ be the subset defined by 
$$ h_2-\hb_2=0.$$
Consider the induced system $(Y^{(k)}, \check{\tn{I}}^{(k)}).$
The following isomorphism follows from the construction.
\begin{lem}\label{lem:YY'iso}
There exists an isomorphism (for $2\leq k\leq\infty$)
$$(Y^{(k)}, \check{\tn{I}}^{(k)}) \simeq  (Y'^{(k-2)},\check{\tn{I}}'^{(k-2)} ).$$
\end{lem}
This has the following consequences.
From the isomorphism we have the correspondence
\begin{align}
z_0&\equiv z, \n \\
h_2=\hb_2&\equiv e^{-2u},\n \\
z_3&\equiv -4 u_1,  \label{eq:zurelation}  \\
z_4-\frac{3}{2}z_3^2&\equiv -4u_2, \n \\
p_{j+2}&\equiv-4 u_j, \quad j\geq 3,\n
\end{align}
where the sequence $p_{j+2}, j=1, 2, \, ... \,, $ is inductively determined by
$$ p_{j+2}=h_2^{-\frac{1}{2}}\delx p_{j+1}.$$
Assign the weights for the variables
$z_k=h_2^{-\frac{k}{2}} h_k, u_{k-2},$ for $k\geq 3,$
by
$$\tn{weight} (z_k) =k-2,\quad \tn{weight}(u_{k-2})=k-2.$$ 
Then from the formula
$$h_2^{-\frac{1}{2}}\delx z_k =z_{k+1}-\frac{k}{2}z_3z_k,$$
each $p_{j+2}=z_{j+2}+\,...\, $ is a weighted homogeneous polynomial of degree $j$ in the variable $z_k$'s.
\begin{cor}\label{cor:zurelation}
Under the isomorphism of Lem.~\ref{lem:YY'iso}, the functions of the variable
$\{ z_3, z_4, \, ... \, \}$ correspond to the functions of  $\{ u_1, u_2, \, ... \, \}$ preserving the given weights.
\end{cor}
As a result the higher-order Jacobi fields and conservation laws for the CMC system 
correspond to those of the PDE \eqref{eq:PDEf}.
We shall use this isomorphism to reduce the problem of classification of Jacobi fields for the CMC system to the known classification for the PDE.
 
\np
\part{Enhanced prolongation}\label{part:prolongation}
\section{CMC surfaces as primitive maps}\label{sec:primitivemap}
In this section we recall the notion of primitive maps and describe how CMC surfaces in a 3-dimensional space form arise as primitive maps in an associated 4-symmetric space.  
This description of CMC surfaces allows one to introduce naturally the spectral parameter  $\lambda$ which will play the important role of connecting the theory of CMC surfaces with integrable systems. 

\two
Let $\G^{\C}$ be a complex semi-simple Lie group. Let $\G$ be the compact real form of $\G^{\C}$ with respect to an anti-holomorphic involution $\sigma:\G^{\C} \to \G^{\C}$.  Let $\tau:\G^{\C} \to \G^{\C}$ be a commuting automorphism of order $k$,
\[\sigma\circ\tau=\tau\circ\sigma, \quad \tau^k=\tn{identity},\]
and let $\K^{\C}\subset\G^{\C}$ be its stabilizer subgroup. Let $\K=\K^{\C}\cap\G$ be the corresponding real form of $\K^{\C}$. Consider the homogeneous space
\[ X=\G/\K\]
equipped with a choice of $\G$-invariant Riemannian metric. By the commuting property, the automorphism $\tau$ descends on $X$ to an isometry of order $k$ with the identity coset as an isolated fixed point. The Riemannian homogeneous space constructed in this way from the triple $(\G^{\C},\sigma,\tau)$ is called a \tb{Riemannian $k$-symmetric space}. We often simply refer to a homogeneous Riemannian manifold $X=\G/\K$ as a $k$-symmetric space and it is understood that some fixed triple $(\G^{\C},\tau,\sigma)$ gives rise to $X$.

The automorphisms $\sigma,\, \tau$ of the Lie group $\G^{\C}$ induce the commuting automorphisms (called by the same name) of the Lie algebra $\g^{\C}$ of $\G^{\C}$.   Let $\varepsilon$ be a fixed primitive $k$-th root of unity and denote the eigenspace decomposition of $\tau$ by (the indices are defined modulo $k$)
\be
\g^{\C}=\g_{0}\oplus \g_{1}\oplus \ldots \oplus \g_{-1},\n
\ee
where each $\g_{j}$ is the eigenspace with eigenvalue $\varepsilon^j$. By definition we have
\[ [ \, \g_i,\,\g_j  ]\subset \g_{i+j}.
\]
This implies that there exists the corresponding decomposition of the complexified tangent bundle $T^{\C}X$ of $X$,
\be\label{eq:TCXdecomp}
T^{\C}X\simeq\g^{\C}/\g_0= \tn{E}_{1}\oplus  \tn{E}_2\oplus \ldots \oplus  \tn{E}_{-1}.
\ee
\begin{defn}
Let $X$ a $k$-symmetric space described as above with the associated decomposition \eqref{eq:TCXdecomp} of the complexified tangent bundle. Let $\Sigma$ be a Riemann surface. A smooth map $\x:\Sigma \to X$ is \tb{primitive} when the (1,0)-part  of the derivative map $\x_*$ takes values in $\tn{E}_{-1}$,
\[ \x_*:T^{(1,0)}\Sigma \to \tn{E}_{-1}\subset T^{\C}X.\]
\end{defn}
Note that the notion of primitive map is meaningful when $k\geq 3$. 

The first order differential equation for primitive maps can be considered as a generalization of the equation for pseudo-holomorphic curves in an almost complex manifold, though the $k$-symmetric space $X$ in general need not have an invariant almost complex structure.  It is known that primitive maps are harmonic maps with respect to the appropriate choice of invariant Riemannian metric on the $k$-symmetric space.

\two
We now describe the example of 4-symmetric space we will be concerned with. The analysis is for the case $\gamma^2>0$, and we give a remark for the necessary changes for the case $\gamma^2<0$. Take $\G^{\C}=\SO(4,\C)$ with the commuting automorphisms
\[ \sigma(g)=\ol g,\quad \tau_{\tn{R}}(g)=\tn{R}g\tn{R}^{-1},
\]
where $\tn{R}$ is the real matrix
\[\tn{R}=\bp -1&\cdot&\cdot&\cdot \\ \cdot&\cdot&-1&\cdot \\ \cdot&1&\cdot&\cdot \\ \cdot&\cdot&\cdot&1 \ep,\quad \tn{R}^4=\rm{I}_4.
\]
Notice that
\[\tn{R}=\tn{R}_1\tn{R}_2=\bp -1&\cdot&\cdot&\cdot \\ \cdot&1&\cdot&\cdot    \\ \cdot&\cdot&1&\cdot \\ \cdot&\cdot&\cdot&1 \ep
 \bp 1&\cdot&\cdot&\cdot \\ \cdot&\cdot&-1&\cdot   \\ \cdot&1&\cdot&\cdot \\ \cdot&\cdot&\cdot&1 \ep.
\]
The automorphism $\tau_{\tn{R}_1}$ gives rise to the 2-symmetric space $\s{3}=\SO(4)/\SO(3)$, and the matrix $\tn{R}_2$ represents a rotation by angle $\frac{\pi}{2}$ in a $2$-plane of $\R^4$.  The stabilizer subgroup of $\tau_{\tn{R}}$ is $\K\simeq\SO(2)$, and $X$ is the oriented unit tangent bundle of $\s{3}$.  The eigenspace decomposition of the Lie algebra $\so(4,\C)$ admits the following explicit description, Figure \ref{fig:taudecompo}.
Here $\{ \, a, \,b,\,c, \, e,\,f, \, g\,\}$ are complex coefficients.
\begin{figure}
\begin{center}
\begin{tabular}{lcc}
\qquad\qquad eigenspace    &\vline & eigenvalue\\
\hline
$\g_0=\bp \cdot&\cdot&\cdot& \cdot\\ \cdot& \cdot &-a&\cdot \\
\cdot &a&\cdot &\cdot \\ \cdot & \cdot & \cdot &\cdot  \ep$ &\vline  & $+1$\\
$\g_1=\bp \cdot&-c&-\im c& \cdot\\ c& \cdot &\cdot&b \\
\im c &\cdot&\cdot &- \im b \\ \cdot & -b & \im b &\cdot  \ep$ &\vline & $+\im$\\
$\g_2=\bp \cdot&\cdot&\cdot& \im e\\ \cdot& \cdot &\cdot&\cdot \\\cdot &\cdot&\cdot &\cdot \\
-\im e & \cdot & \cdot &\cdot  \ep$ &\vline & $-1$\\
$\g_{-1}=\bp \cdot&-f&\im f& \cdot\\ f& \cdot &\cdot&-g \\
 -\im f &\cdot&\cdot &- \im g \\ \cdot & g & \im g &\cdot  \ep$ &\vline  & $-\im$\\
\end{tabular}
\end{center}
\caption{Eigenspace decomposition of $\so(4,\C)$ under $\tau$}
\label{fig:taudecompo}
\end{figure}
\begin{rem}
For the case $\gamma^2<0$, choose the involutions
\[ \sigma(g)=\tn{R}_1\ol{g}\tn{R}_1, \quad \tau= \tau_{\tn{R}}.
\]
This gives rise to the hyperbolic space form $\mathbb{H}^3=\SO(1,3)/\SO(3)$ 
and its unit tangent bundle.
\end{rem}

The primitive maps into $X$ are equivalent to the integral surfaces of the EDS for CMC surfaces \eqref{1ideal2}. To see this consider an integral surface $\x:\Sigma\to X$. On the induced $\SO(2)$-bundle $\x^*\mcf\to \Sigma$ define a $\g$-valued 1-form  
\[\psi=\psi_{-1}+\psi_{0}+\psi_{1}
\]
decomposed into three parts according to the decomposition of $T^{\C}X$, where
\be
\psi_{-1} = \frac{1}{2}
\bp \cdot & -\gamma \xi & \im \gamma \xi &\cdot \\
\gamma \xi &\cdot& \cdot & -h_2 \xi \\
-\im \gamma \xi  &\cdot &\cdot & -\im h_2 \xi \\
\cdot & h_2 \xi &\im h_2 \xi &\cdot \ep,\n
\ee
\be \psi_{0} =
\bp \cdot &\cdot & \cdot &\cdot \\
\cdot &\cdot& \rho &\cdot \\
\cdot  & -\rho &\cdot &\cdot\\
\cdot &\cdot &\cdot &\cdot \ep,\n
\ee
\be \psi_{1} = \frac{1}{2}
\bp \cdot & -\gamma \xib &- \im \gamma \xib &\cdot \\
\gamma \xib &\cdot& \cdot & -\hb_2 \xib \\
\im \gamma \xib  &\cdot &\cdot & \im \hb_2 \xib \\
\cdot & \hb_2 \xib &-\im \hb_2 \xib &\cdot \ep.\n
\ee
Notice that $\psi_{1}=\sigma(\psi_{-1})$, and the 1-form $\psi$ satisfies the Maurer-Cartan equation $\ed\psi+\psi\w\psi=0.$ Since $\psi_{-1}\equiv 0\mod\xi$, by definition $\x$ is a primitive map. Reversing the argument it is easily checked that the converse is also true.

\two
The formulation of CMC surfaces as primitive maps has an important geometric consequence, namely the introduction of \emph{spectral parameter.} Following the general theory of integrable systems, let $\lambda\in\C^*$ be the spectral parameter and set the $\SO(4,\C)[\lambda^{-1},\lambda]$-valued \tb{extended Maurer-Cartan form}  
\be\label{eq:primitiveform}
\psi_{\lambda}=\lambda^{-1}\psi_{-1}+\psi_0+\lambda \psi_1.
\ee
The 1-form $\psi_{\lambda}$ takes values in the Lie algebra $\g$ when $\lambda$ is a unit complex number. A direct computation shows that the extended 1-form $\psi_{\lambda}$ still satisfies the Maurer-Cartan equation
\be\label{eq:primitiveMC}
\ed\psi_{\lambda}+\psi_{\lambda}\w\psi_{\lambda}=0
\ee
whenever the data $\{\, \gamma^2, \, \xi,\,h_2,\,\rho\,\}$ satisfy the structure equation for CMC surfaces,  Eq.~\eqref{eq:strtx} in Sec.~\ref{sec:prolongation1}.

In order to incorporate and make use of the present extra symmetry structure involving the spectral parameter $\lambda$, it seems desirable to extend the underlying symmetry group from the simple Lie group $\G$ to the associated loop group. This extension should be considered as a starting point of the theory of integrable systems in CMC surface theory.

\section{Formal Killing field}\label{sec:formalKillingfields}
We introduce the formal Killing fields associated with primitive maps adapted to our setting
following Burstall and Pedit, \cite{Burstall1995}.  
The formal Killing fields were defined in \cite{Burstall1993} as a way of packaging together
the  infinitesimal symmetries of harmonic maps.
We shall extract a recursion scheme for Jacobi fields from the structure equation for formal Killing fields.
Then by examining the first few terms of the sequence of coefficients,
we are motivated to use the square root of  Hopf differential. 
This will lead to a double cover of the infinite prolongation which supports the splitting field for the linearized Jacobi equation. 

\two
Recall $\g^{\C}=\so(4,\C)$. Consider the {\bf based loop algebra} associated to $\g^{\C}$:
\[
\mcl(\g^{\C})= \left\lbrace  X_{\lambda} \in\g^{\C}[[ \lambda^{-1},\lambda]] \;:\; {\rm if}\; X_{\lambda}=\sum_{n=-\infty}^{\infty}X^n \lambda^n \;{\rm then}\;X^0=0\right\rbrace,
\]
and the {\bf twisted based loop algebra} associated to the $k$-symmetric space $(\G^{\C},\sigma, \tau)$:
\begin{align}\label{defn:twistedloop}
\mcl^{\sigma,\tau}(\g^{\C})=\left\lbrace  X_{\lambda} \in \mcl(\g^{\C}):\; \sigma(X_{\ol{\lambda}^{-1}})=X_{\lambda},  \;\; \tau(X_{\varepsilon \lambda})=X_{\lambda}, \;\; \varepsilon=e^{\frac{2\pi \im}{k}} \right\rbrace.
\end{align}
Note that the extended Maurer-Cartan form $\psi_{\lambda}$ from the previous section is an $\mcl^{\sigma,\tau}(\g^{\C})$-valued 1-form.
\begin{defn}
Suppose we have a $\mcl^{\sigma,\tau}(\g^{\C})$-valued 1-forms $\psi_{\lambda}$ on a Riemann surface $\Sigma$ as given in \eqref{eq:primitiveform} which satisfies the structure equation \eqref{eq:primitiveMC} independent of the spectral parameter $\lambda$. A \tb{formal Killing field} for  $\psi_{\lambda}$ is a map  $X_{\lambda}:\Sigma \to \mcl^{\sigma, \tau}(\g^{\C})$ which satisfies
\be\label{eq:KillingEquation}
\ed X_{\lambda}+[\psi_{\lambda}, X_{\lambda}]=0.
\ee
\end{defn}

In order to make use of the formal Killing fields to the higher-order analysis for CMC surfaces,
we give an explicit invariant decomposition of $X_{\lambda}$. 
Let us write a formal Killing field as
\[
X_{\lambda}=\sum_{n=-\infty}^{\infty} \left( X^{kn}\lambda^{kn}+X^{kn+1}\lambda^{kn+1}+\ldots +X^{kn+k-1}\lambda^{kn+k-1}    \right)
\]
where $X^{kn+j} \in \g_{j}$. The reality condition $\sigma(X_{\ol{\lambda}^{-1}})=X_{\lambda}$ implies that
\be\label{eq:realitycondition}
X^{-m}=\sigma({X^m}). 
\ee
In terms of this expansion, the $\lambda^{kn+j}$-component of Eq.~\eqref{eq:KillingEquation} gives
\be\label{eq:formalKilling_nj}
\ed X^{kn+j}+[\psi_0,X^{kn+j}]+[\psi_{-1},X^{kn+j+1}]+[\psi_{1},X^{kn+j-1}]=0.
\ee
Note in particular that an element of the based loop algebra has $X^0=0$ so that the $\lambda^0$-component gives the initial relation
\[
[\psi_{-1},X^1]+[\psi_1, X^{-1}]=0.
\]

\one
For the example at hand ($k=4$), we adopt the index notation as follows.
\be\label{eq:weightindex}
\left( X^{4n}, \,X^{4n+1},\,X^{4n+2},\,X^{4n+3}\right)
=\left(\{a^n\}, \, \{b^n,\,c^n\},\,\{e^n\},\,\{f^n,g^n\}\right).
\ee
The reality condition \eqref{eq:realitycondition} translates to
\be\label{eq:newreality}\begin{array}{rl}
a^{-n}&=\ol{a^n}, \\
b^{-n}&=-\ol{g^{n-1}}\\
c^{-n}&= \tn{sign}(\gamma^2)\,\ol{f^{n-1}}\\
e^{-n}&=-\tn{sign}(\gamma^2)\,\ol{e^{n-1}}\\
f^{-n}&=\tn{sign}(\gamma^2)\,\ol{c^{n-1}}\\
g^{-n}&=-\ol{b^{n-1}},\qquad\qquad n\geq 1.
\end{array}\ee
\begin{rem}
We will only consider the terms of non-negative index $n\geq 0$.
The following arguments are valid independent of the sign of $\gamma^2$. 
\end{rem}

The Killing field equation is expanded component-wise in these variables as follows.

\one
\emph{$\bullet\;n$-th equation, $n\geq1$}:
\begin{align}\label{eq:formalKilling_n}
&\ed  a^n +(\im \gamma c^n+\im h_2 b^n)\xi+(-\im \gamma  f^{n-1}+\im \hb_{2} g^{n-1})\xib = 0,\\
&\ed  b^n  -\im  b^n \rho+\frac{\im \gamma }{2} e^n \xi +\frac{\im}{2}  \hb_{2}a^n \xib  =0,\n\\
&\ed c^n +\im c^n \rho+\frac{\im}{2}  h_2 e^n  \xi+\frac{ \im \gamma }{2} a^n \xib = 0,\n\\
&\ed e^n  +(\im h_2 f^n-\im \gamma  g^{n})\xi+(\im \hb_{2} c^n+\im \gamma b^n)\xib  = 0,\n\\
&\ed  f^n  -\im  f^n \rho-\frac{\im \gamma }{2} a^{n+1} \xi+\frac{\im}{2} \hb_{2} e^n \xib = 0,\n\\
&\ed  g^n +\im  g^n \rho+\frac{\im}{2} h_2 a^{n+1}  \xi-\frac{\im \gamma }{2}  e^n \xib = 0.\n
\end{align}

\emph{$\bullet\;0$-th equation}:
\begin{align}\label{eq:formalKilling_0}
&a^0=0,\\
&\gamma  c^0+h_2 b^0 =0, \n \\
&\ed  b^ 0  -\im  b^ 0 \rho+\frac{\im \gamma }{2} e^ 0 \xi  = 0, \n \\
&\ed c^ 0 +\im c^ 0 \rho+\frac{\im}{2}  h_2 e^ 0  \xi = 0,\n \\
&\ed e^ 0  +(\im h_2 f^ 0-\im \gamma  g^{ 0})\xi+(\im \hb_{2} c^ 0+\im \gamma b^ 0)\xib  = 0, \n \\
&\ed  f^ 0  -\im  f^ 0 \rho-\frac{\im \gamma }{2} a^{ 1} \xi+\frac{\im}{2} \hb_{2} e^ 0 \xib =0, \n \\
&\ed  g^ 0 +\im  g^ 0 \rho+\frac{\im}{2} h_2 a^{ 1}  \xi-\frac{\im \gamma }{2}  e^ 0 \xib =0.\n
\end{align}

A part of interests in formal Killing fields is from that they give rise to Jacobi fields.  We demonstrate this for the specific case at hand.
\begin{lem}\label{lem:formalJacobi}
If the coefficients
$\{ a^n, b^n, c^n \}$, or $\{a^n, f^{n-1}, g^{n-1} \}$ satisfy the formal Killing field equation then $a^n$ is a Jacobi field,
\[
\mce(a^n):=\Delta a^n-2\left( \gamma^2+|h_2|^2 \right)a^n=0.
\]
Similarly, if  the coefficients
$\{ e^n, b^{n},c^{n} \}$, or $\{ e^n, f^{n}, g^{n} \}$ satisfy the formal  Killing field equation then $e^n$ is a Jacobi field.
\end{lem}
\begin{proof}
By definition, $-\ed \JAI \ed a^n = -4(a^n_{\xi,\xib})\frac{\im}{2}\xi\w\xib= \Delta a^n \frac{\im}{2} \xi \w \xib$.
From Eqs.~\eqref{eq:formalKilling_n},
\begin{align}
\ed  a^n &= -(\im \gamma c^n+\im h_2 b^n)\xi-(-\im \gamma  f^{n-1}+\im \hb_{2} g^{n-1})\xib, \n\\
\JAI \ed a^n &= -(\gamma c^n+h_2 b^n)\xi-( \gamma  f^{n-1}- \hb_{2} g^{n-1})\xib, \n \\
\ed \JAI \ed a^n &=  -2(\gamma ^2+|h_2|^2)a^n \frac{\im}{2} \xi \w \xib. \n
\end{align}
The $e^n$ case is verified similarly.
\end{proof}
The formally integrable structure equation for formal Killing fields is equivalent to the recursive sequence of essentially determined first order equations \eqref{eq:formalKilling_nj}, or in components \eqref{eq:formalKilling_n}. The lemma above indicates that by solving this sequence of equations one may obtain a canonical (possibly degenerate) sequence of Jacobi fields. The explicit differential algebraic recursion formulae for the formal Killing coefficients will be given in Sec.~\ref{sec:inductiveformula}.

\two
Let us give here an initial analysis for the $0$-th equation.
To begin, note that the based loop algebra condition $X^0=0$ gives $a^0=0$.
Differentiating this equation twice, one gets $\gamma  c^0+h_2 b^0 =0$ and
\[ -\im \gamma h_2 e^0 + h_3 b^0=0.
\]
Assuming $h_2\ne 0$, one may solve this equation for $e^0$. Substituting this in the formula for $\ed b^0$, we get the closed structure equation for $b^0$,
\[ \ed  b^ 0  -\im  b^ 0 \rho+\frac{h_3}{2h_2} b^ 0 \xi  = 0.
\]
By inspection $b^0$ must be a constant multiple of $h_2^{-\frac{1}{2}}$.

Let us set accordingly
\begin{align}
b^0&=  h_2^{-\frac{1}{2}}, \n \\
c^0&=  -\frac{1}{\gamma} h_2^{\frac{1}{2}}, \n \\
\im e^0&= \frac{1}{\gamma} h_2^{-\frac{3}{2}}h_3. \n
\end{align}
By Lemma \ref{lem:formalJacobi},
$e^0$ is a Jacobi field. One may verify this by a direct computation using the structure equation for prolongation. See also Exam.~\ref{ex:HopfJacobi} for a geometric derivation of this higher-order Jacobi field.

The Jacobi field $e^0$ admits a simple local representation for a choice of complex local coordinate on a CMC surface.
Given a CMC surface, let $\omega$ be a local square root of the Hopf differential away from the umbilics.
Let $z$ be a local coordinate so that
\[ \omega = \ed z= \frac{1}{b^0}\xi.
\]
Then one computes that the connection 1-form $\rho$ on the CMC surface is
\[
\rho =  \im\left( \del \log(\ol{b^0})- \delb \log({b^0}) \right).
\]
The Jacobi field is given by
\[\im  \gamma  e^0= h_2^{-\frac{3}{2}}h_3= \frac{\partial \;}{\partial z}\log(|h_2|^2)
=-2 \frac{\partial \;}{\partial z}\log(|b^0|^2).
\]

\two
Continuing this computation one may generate a sequence of higher-order Jacobi fields. 
For now let us indicate how to proceed to determine the next Killing coefficients $f^0, \,g^0, \,a^1.$ 

Differentiating the given $\,\im\gamma e^0$, one gets
\[  \gamma(h_2f^0-  \gamma g^0)=\delx (h_2^{-\frac{3}{2}}h_3).
\]
Differentiating this equation again, $a^1$ is determined as a function of $f^0$,
\[\im\gamma a^1=-\frac{h_3}{h_2}f^0+\frac{1}{\gamma h_2}\delx^2(h_2^{-\frac{3}{2}}h_3).
\]
Substituting this to the structure equation for $f^0$, one again gets a closed structure equation.
By inspection again this equation can be solved\footnotemark\footnotetext{This relies on the fact that $\,e^0$ is a Jacobi field.} and we have
\begin{align}
f^0&=\frac{h_2^{-\frac{5}{2}}}{2\gamma}(h_4-\frac{5}{4}h_3^2h_2^{-1}),\n\\
a^1&=-\frac{\im h_2^{-\frac{5}{2}}}{\gamma^2}(h_5-5h_4h_3h_2^{-1}+\frac{35}{8}h_3^3h_2^{-2}).\n
\end{align}
\begin{rem}\label{rem:babyrecursion}
Note from Lemma \ref{lem:lemma5.4} that given the Jacobi field $\,e^0$ the first order linear differential equation
\[\delxb f^0=-\frac{\im}{2}\hb_2 e^0
\]
determines $\,f^0$ up to a constant multiple of $\,h_2^{-\frac{1}{2}}$. Assuming a solution exists, the formula
\[\delx f^0=\frac{\im\gamma}{2}a^1
\]
determines the new Jacobi field $\,a^1$ from $\,e^0$ up to a constant multiple of $\,z_3=h_2^{-\frac{3}{2}}h_3$, which is itself a Jacobi field. 
\end{rem}

The important role played by the coefficient $b^0$ of the formal Killing field, and the formulae for the first two higher-order Jacobi fields $e^0,\,a^1$ suggest that we should adopt a version of prolongation in which the square root $\sqrt{h_2}$ rather than $h_2$ is a fundamental quantity.  We turn to this now and then will resume the study of the formal Killing field equations.
\section{Double cover of the prolongation}\label{sec:doublecover}
The preceding formal Killing field analysis, and the double cover $\Sigmah\to\Sigma$ of a CMC surface introduced in Definition \ref{defn:doublecover} prompt the definition of a global double cover $\hat{X}^{(\infty)}$ of the infinite prolongation $\xinf$ such that we have the following commutative diagram: \two\\
\centerline{\xymatrix{\Sigmah \ar[d]  \;  \ar@{^{}.>}[r] &\;\;\hat{X}^{(\infty)} \ar[d] \\
                                \Sigma  \;  \ar[r]  &\;\;\xinf  }} \\ 

\one\noi
The double cover will provide a natural setting to address the global questions for CMC surfaces. 

\two
Recall the $\SO(2)$-principle bundle $\Pi:\mcf \to X$, and the structure equation \eqref{eq:strt2}.  Let
\[
K\to X
\]
be the complex line bundle generated a section of the $1$-form $\xi$, and let $K_{\eta_1}\to X$ be the complex line bundle generated by a section of the $1$-form $\eta_1$. The structure equation \eqref{eq:strt2} shows that there exists a well defined $\SO(2)$-equivariant pairing
\[
K \otimes K_{\eta_1} \to X\times\C,
\]
so that
\[ K_{\eta_1}=K^{-1}.
\]
By definition $\Pi^*(K) \cong \mcf \times \C$ is  trivial.

For a point $u\in\mcf$, denote by $(\theta_0^u,\xi^u,\eta_1^u)$ the value of the 1-forms $(\theta_0,\xi,\eta_1)$ at $u$. Recall that $\mcf^{(1)}=\Pi^*(\X{1})$.
\begin{lem}
$\mcf^{(1)}\cong \pr(\Pi^*K \oplus \Pi^*K^{-1}) \cong \mcf \times \cp{1}.$
\end{lem}
\begin{proof}
The second isomorphism follows from the first because $\Pi^*K^{\pm 1}$ is trivial. We now prove the first isomorphism.  An element of $\mcf^{(1)}$ is given by
\[
\left(u, \ker\{\theta_0^u, t_0(u)\xi^u-t_{\infty}(u)\eta_1^u \} \right) \in \mcf^{(1)},
\]
where $t_0(u), \,t_{\infty}(u)$ are the linear fiber coordinates on $\,\Pi^*K^{\pm 1}$ at $u$ respectively.
Define the map $\mcf^{(1)} \to \pr(\Pi^*K \oplus \Pi^*K^{-1})$ by
\[
\left(u,\ker\{\theta_0^u, t_0(u)\xi^u-t_{\infty}(u)\eta_1^u\} \right) \mapsto [t_0(u)\xi^u, t_{\infty}(u)\eta_1^u].
\]
It is clear that this is a well defined isomorphism.
\end{proof}

The lemma leads to an invariant description of the first prolongation $\X{1}\to X$.
\begin{cor}
Let $\X{1}\to X$ be the first prolongation of the EDS for CMC surfaces.  Then we have the following isomorphisms of the $\C P^1$-bundles over $X$.
\begin{enumerate}[\qquad a)]
\item $\X{1}\cong \pr(K^{p+1} \oplus K^{p-1})$,
\item $\F{1}\cong \pr(\Pi^*K^{p+1} \oplus \Pi^*K^{p-1})$,
\end{enumerate}
for any integer $p$.
\end{cor}
\begin{proof}
Simply note that the objects descend properly. The projectivization is invariant under tensoring by $K^p$ or $\Pi^*K^{p}$.
\end{proof}

In order to define the double cover, consider the isomorphism $\X{1} \cong \pr(K^{2} \oplus \C)$.
\begin{defn}
The \tb{double cover} $\Xh{1}$ of the first prolongation $\X{1}$ is defined by
\[ \Xh{1}= \pr(K \oplus \C).
\]
Define similarly
\[ \Fh{1}= \pr(\Pi^*K \oplus \C).
\]
\end{defn}
Following the notation used above, the branched double cover of the prolonged frame bundle is given by
\begin{align}
\bs:\Fh{1} &\to \F{1},\n\\
[t_0(u)\xi^u, t_{\infty}] &\mapsto [\left( t_0(u)\xi^u \right)^2, t_{\infty}^2].\n
\end{align}
This is a bundle map which is the standard double cover on the fibers
\begin{align}
\cp{1} &\to \cp{1},\n\\
[x,y] & \mapsto [x^2,y^2].\n
\end{align}
It is branched at the two points $0 = [0,1]$ and $\infty = [1,0]$.
\begin{cor}
There is a natural induced double covering map
\begin{align}
\bs:\Xh{1} &\to \X{1}.\n
\end{align}
\end{cor}

Let $\Ih{1}=\bs^*\I{1}$ be the pulled back ideal, and define a new EDS
\[
(\Xh{1},\,\Ih{1})
\]
which double covers $(\X{1},\, \I{1})$.  It is on this double cover that we will be able to define global conservation laws. This does not impose any restrictions on the study of CMC surfaces as every integral surface of $(X,\,\mci)$ has a possibly branched double cover which is an integral surface of  $(\Xh{1}, \Ih{1})$.

\two
We will now derive the structure equations on $\Fh{1}$.  
It will be useful to separate the two open sets which correspond to the standard affine charts of $\cp{1}$. Let $\Pi_1:\Fh{1}\to\Xh{1}$ denote the projection map.
\begin{defn}
Define the spaces
\begin{align*}
 \Fh{1}_0&:= \left\{\; [t_0(u)\xi^u, t_{\infty}] \; \vert \;\; t_{\infty} \neq 0 \right\},\\
 \Fh{1}_{\infty}&:= \left\{\; [t_0(u)\xi^u, t_{\infty}] \; \vert \; \;t_0(u)  \neq 0 \right\}.
\end{align*}
\begin{align*}
 \Xh{1}_0&:= \Pi_1(\Fh{1}_0),\\
 \Xh{1}_{\infty}&:=\Pi_1(\Fh{1}_{\infty}).
\end{align*}
\begin{align*}
 \X{1}_0&:= \bs(\Xh{1}_0),\\
 \X{1}_{\infty}&:=\bs(\Xh{1}_{\infty}).
\end{align*}
\end{defn}
The space $\X{1}_0$ corresponds to the standard prolongation of $(X,\mci)$ with the independence condition $\xi \w \xib \neq 0$ described in Sec.~\ref{sec:prolongation1}.  This space would seem to suffice for studying the global geometry of CMC surfaces in the space form $M$.  But in searching for global conservation laws we need to not only use a branched double cover of the prolongation, but also to work on the entire prolongation space $\Xh{1}$. We will see that we may consider the higher-order Jacobi fields either as singular functions 
which is smooth on $\Xh{1}_0$, or as smooth sections of line bundles over $\Xh{1}$. 

On $\Fh{1}_0$ we will use $h_1$ as the fiber coordinate such that
\[h_1^2=\bs^*(h_2).
\]
Set $p_2=h_2^{-1}$. On $\Fh{1}_{\infty}$ we will use $p_1$ as the fiber coordinate such that
\[p_1^2=\bs^*(p_2).
\]
In terms of these coordinates, we have the following descriptions of the ideals.
On $\Fh{1}_0$ we have $\Pi^*(\Ih{1}_0)= \langle \theta_0,\theta_1^0\rangle$ where
\[\theta_1^0=\eta_1-h_1^2 \xi.\]
On $\Fh{1}_{\infty}$ we have $\Pi^*(\Ih{1}_{\infty})= \langle \theta_0,\theta_1^{\infty}\rangle$ where
\[\theta_1^{\infty}=\xi-p_1^2 \eta_1.
\]

We could proceed in a number of ways.  We may construct the prolongation tower of $(\Xh{1},\Ih{1})$ and study the geometry there.  Or we could simply pull back the standard prolongation tower $(\X{k},\I{k})$ to $\Xh{1}$ using $\bs$.  It is this latter approach we will take as currently we  do not see any benefit of the extra structure obtained by the former.  It appears that introducing the square root of the Hopf differential is sufficient to obtain global conservation laws for CMC surfaces. Thus we make,
\begin{defn}
For $k \geq 2$ define the \tb{double cover} of the prolongations
\begin{align*}
\Fh{k}&:= \bs^*(\F{k}), \\
 \Xh{k}&:= \bs^*(\X{k}).
\end{align*}
We have the commutative diagram\one
\\
\centerline{
\xymatrix{
 \Xh{k} \ar[d]^{\hat{\pi}_{k,1}} \ar[r]^{\bs_k} & \X{k} \ar[d]^{\pi_{k,1}}  \\
 \Xh{1}  \ar[r]^{\bs} & \X{1}
}}\\

\one\noi and we set the ideals $\Ih{k}:=\bs_k^*(\I{k})$.

Let $\{ \Xh{k}_0, \Xh{k}_{\infty}, \X{k}_0, \X{k}_{\infty}\}$ be the inverse images under the projection maps
$\,\hat{\pi}_{k,1}, \, \pi_{k,1}$ of the open subsets
$\{\Xh{1}_0, \Xh{1}_{\infty}, \X{1}_0, \X{1}_{\infty}\}$ respectively.
Define similarly the corresponding open subsets  
$\{ \Fh{k}_0, \Fh{k}_{\infty}, \F{k}_0, \F{k}_{\infty}\}$ of $\Fh{k}, \,\F{k}$.

For simplicity, we denote the pull-back bundles for $1\leq k \leq \infty$ by 
\begin{align}\label{canonicalbundle}
\pi_{k,0}^*K&=K\to \X{k},\\
(\bs\circ\hat{\pi}_{k,1})^*K&=\hat{K} \to \Xh{k}.\n
\end{align}
\end{defn}

The following proposition shows that the double cover of prolongations constructed in this section fits naturally with the double cover of a CMC surface associated with the Hopf differential.  The proof follows  from the definitions.
\begin{prop}
Let $\x:\Sigma\hook X$ be an immersed integral surface of the EDS for CMC surfaces.
Let $\x^{(k)}:\Sigma\hook\X{k},\,1\leq k \leq\infty$, be the prolongation of $\x$.
Let $\nu:\Sigmah\to\Sigma$ be the double cover associated with the Hopf differential of $\x$, Definition \ref{defn:doublecover}.
There exists a lift $\xh^{(1)}:\Sigmah\hook\Xh{1}$ and the associated sequence of prolongations
\[ \xh^{(k)}:\Sigmah\hook\Xh{k},\quad 2\leq k \leq\infty,
\]
such that
\begin{enumerate}[\qquad a)]
\item  each $\xh^{(k)},\,k\geq 1,$ is integral to $\Ih{k}$,  
\item  $\x^{(k)}\circ \nu =\bs_k\circ\xh^{(k)}$.
\end{enumerate}
The lift $\,\xh^{(1)}$ and its prolongation sequence $\{\,\xh^{(k)}\,\}$ are uniquely determined by these  properties.
\end{prop}

It will often be useful to separate the following open subsets, $1\leq k\leq \infty$.
\begin{align*}
\Fh{k}_*&:=\Fh{k}_0 \cap \Fh{k}_{\infty}, \\
\F{k}_*&:=\F{k}_0 \cap \F{k}_{\infty}. \\
\Xh{k}_*&:=\Xh{k}_0 \cap \Xh{k}_{\infty}, \\
\X{k}_*&:=\X{k}_0 \cap \X{k}_{\infty}.
\end{align*}
Let us mention an important property of these subsets. On these subsets there exist a preferred set of coordinate functions which are balanced and adapted to the intrinsic structure of the prolongation. 
These functions will play a crucial role in determining the recursion relations for the formal Killing coefficients.
\begin{lem}\label{lem:zj}
There exists a sequence of functions 
$z_j : \Xh{\infty}_* \to \C,   j \geq 2,$
 defined as follows:  for $u \in \Fh{\infty}_*$, part of the information stored in $u$ is that of $[\eta_j^u-h_{j+1} \xi^u]$ for $j\geq 2$ and $[\eta_1^u-(h_1)^2 \xi^u]$ where $h_j, h_1 \neq 0$ and both depend on $u$.  We set
\[z_j=\frac{h_j}{h_1^j}, \quad  j \geq 2.
\]
Suppose that $\xh^{\infty}:\Sigmah \hook \Xh{\infty}_*$ is a smooth integral surface.  Locally  the induced function $(\xh^{\infty})^*(z_j)$ can be expressed as the quotient of a smooth function and a holomorphic function.
\end{lem}
\begin{proof}
Let $\omega$ be the square root of the Hopf differential for $\xh^{\infty}$ on $\Sigmah$. Then
\[ (\xh^{\infty})^*(z_j)= \frac{\bs^*(h_j\xi^j)}{\omega^j}.
\]
\end{proof}

There is a choice of coframe on $\Xh{\infty}_*$ which is better adapted with the functions $z_j,\zb_j.$
Define  
\begin{align*}
r&:=|h_2|=(h_2\hb_2)^{\frac{1}{2}},\\
\omega&:=h_1 \xi,\\
\zeta_j&:=h_1^{-j}\theta_j,\quad j\geq 0,\\
\hat{T}_j &:=h_1^{-(j-1)}T_j,\\
\sigma_j'&:=h_1^{-(j+1)} \tau_j',\\
\sigma_j'' &:=h_1^{-(j-1)} \tau_j'',\quad j\geq 2.
\end{align*}
Due to the transformation properties under the action of the structure group $\SO(2)$,  they are well defined on $\Xh{\infty}_*$. They satisfy the following structure equations:
\begin{align}\label{eq:zetastruct}
\ed \omega &=(\frac{1}{2}\zeta_2-\delta\theta_0)\w\omega-r\theta_0\w(\ol{\zeta_1}+\omb), \\
\ed \theta_0&=-\frac{1}{2}(\zeta_1\w\omega+\zetab_{1}\w\omb),\n\\
\ed \zeta_1&=-\left(\frac{1}{2} \zeta_{2}-\frac{1}{2}z_3 \zeta_1-\delta\theta_0\right) \w \omega
-\frac{1}{2}\zeta_2\w(\zeta_1+\omega)+r \theta_0\w(\ol{\zeta_1}+\omb)\n\\
&\quad+\theta_0\w(\delta(\zeta_1+\omega)+\gamma^2r^{-1}\omb),\n\\
\ed \zeta_j &=-\left( \zeta_{j+1}-\frac{j}{2} z_3 \zeta_j\right) \w \omega -\frac{j}{2}\zeta_2\w\zeta_j
+ \theta_0 \w \left( \hat{T}_j \zeta_1+r z_{j+1} \zetab_1 \right)\n\\
&\quad+\frac{j}{2}r z_j  \zeta_1 \w \zetab_1+ \sigma_j' \w \omega + r^{-1} \sigma_j'' \w \omb,
\quad{\rm for} \; j \geq 2,\n\\
\ed r&=\frac{r}{2}(z_3\omega+\zb_3\omb), \n\\
\ed z_j& =\zeta_j -\frac{j}{2} z_j\zeta_2+\left(z_{j+1}-\frac{j}{2} z_3 z_j \right) \omega + \hat{T}_{j} r^{-1}\omb, \quad{\rm for} \; j \geq 3,\n\\
 \hat{T}_{j+1}&= \hat{T}_{j,\omega}+\frac{j-1}{2}z_3  \hat{T}_j+\frac{j}{2}(\gamma^2-r^2)z_j.\n
\end{align}
We can now express the ideal $\iinfh$ on $\Xh{\infty}_*$ as 
$$\iinfh=\langle \theta_0, \zeta_j, \zetab_j \rangle,$$
and the generators are well defined forms on $\Xh{\infty}_*$.

We note an important property of $\hat{T}_j$. It is not difficult to see  that 
\begin{lem}\label{lem:Tjh}
\be\label{eq:Tjh}
\hat{T}_j \in \C[r^2,z_3,z_4,z_5,\ldots].
\ee  
\end{lem}
\begin{proof}
Define the operator $\del_{\omega}$ by $\ed f \equiv (\del_{\omega}f)\omega\mod\omb$ for a scalar function $f$.
Then
\begin{align}
\del_{\omega} z_j&=z_{j+1}-\frac{j}{2} z_3 z_j, \n\\
\del_{\omega} r^2&=r^2 z_3.\n
\end{align}
It follows from the inductive formula for $\hat{T}_j$, and the fact $\hat{T}_3=R.$
\end{proof}
\noi
This observation is relevant for the analysis of refined normal form of Jacobi fields.

\two
Let us make a technical remark on the notation. Recall the set of notations such as $\mathfrak{J}^{(\infty)}, \mathfrak{S}, \mcc^{(\infty)}, \Cv{(\infty)}$ for Jacobi fields, symmetries, and conservation laws. We shall omit the upper hat `$\;\hat{}\;$' and continue to use them for the corresponding objects on the double cover $\xinfh.$

\two
In Galois theory, one is interested in finding the smallest field extension in which a given polynomial splits.  This serves as an analogy for our case.  If, instead of finding the roots of a polynomial, we take the Jacobi equation $\mce$ as the equation to be solved, then it turns out that the field
\begin{equation}
\mathfrak{F}:=\C (z_3,z_4,z_5,\ldots)
\end{equation}
is the smallest field containing all of the polynomial solutions. 
It is remarkable that the splitting field $\mathfrak{F}$ is an algebraic extension of $\C (h_2,h_3,h_4,\ldots)$.  For other primitive maps one will not generally have a quadratic extension, but it should still be algebraic.  It seems highly likely that for PDE's with the Cartan character $s_2 \neq 0$ the analogous splitting fields will not be algebraic extensions.  This should be one of the non-local phenomena \cite{Vinogradov1989}.  While the analogy with Galois theory is enticing, it seems not quite justified without further results. 

\section{Formal Killing field as enhanced prolongation}\label{sec:kfpro}
We now return to the analysis of the formal Killing fields.
We define a tower of EDS based upon the formal Killing field equation.  
The full structure equations are remarkably simple compared to those of the EDS $(\X{k},\I{k})$. 
The main simplification comes in the absence of recursive terms such as $\tau'_j$ and $\tau''_j$ in Lemma \ref{lem:prolongedstrt}. It is common for an integrable system to have a set of `good coordinates'. The formal Killing field provides such good coordinates for the prolongation of the CMC system. 

Considering the analysis of the first few terms of the formal Killing field in Sec.~\ref{sec:formalKillingfields},
 it is not surprising that there is an equivalence between the integral surfaces of the Killing field EDS and those of the CMC EDS. It turns out that the CMC EDS can be considered as a natural closure of the Killing field EDS.

\two
For a motivation for the following construction, we refer the reader to Sec.~\ref{sec:recursion}. We start by laying the basis for our construction. On $\Fh{1}_*$, take the coframe $\{ \rho,\theta_0,\omega,\beta_2,\eta'_2\}$ where we recall
\begin{align}
h_1^2&=h_2,\n\\
\omega&=h_1 \xi, \n \\
\beta_2&=h_1^{-1}\eta_1-\omega=h_1^{-1}\theta_1,\n \\
\eta'_2&=\ed \log(h_1)+\im \rho.\n 
\end{align}
With this initial data, inductively define the tower of spaces 
\begin{align}
\Fh{2}_{KF}&=\Fh{1}_*,\n\\
\Fh{2n+1}_{KF}&=\Fh{2n}_{KF} \times \C, \;\;\;\;\;\;\;\;\;  n \geq 1,\n \\
\Fh{2n+2}_{KF}&=\Fh{2n+1}_{KF} \times \C^2,   \;\;\;\;\; n \geq 1,\n
\end{align}
where the fiber coordinate of 
$\Fh{2n+1}_{KF} \to \Fh{2n}_{KF}$ is $A^{2n+1},$ and the fiber coordinates of 
$\Fh{2n+2}_{KF} \to \Fh{2n+1}_{KF}$ are $B^{2n+2}$ and $C^{2n+2}$.

Let 
\begin{align*}
s&=|h_1|=\sqrt{r}, \\
B^2 &=\gamma, \\
C^2 &=-1.
\end{align*} 
Define 
$\gamma_2 = \eta'_2 - A^3 \omega$ on $\Fh{3}_{KF}$.
For $n \geq 2 $ define 
\begin{align*}
\alpha_{2n-1}&=\ed A^{2n-1}-\frac{1}{2} \left( B^{2n}+\gamma C^{2n} \right)\omega-\frac{1}{2} \left( \gamma s^{-2} B^{2n-2}+s^2 C^{2n-2}\right)\omb, \\
\beta_{2n}&=\ed B^{2n}-\left( A^3B^{2n} -\gamma A^{2n+1} \right) \omega+s^2A^{2n-1} \omb, \\
\gamma_{2n}&=\ed C^{2n}+\left( A^{2n+1}+A^3 C^{2n}  \right) \omega+\gamma s^{-2} A^{2n-1} \omb. 
\end{align*}
\begin{defn}
The \tb{ideal for formal Killing fields} $\Ih{\infty}_{KF}$ on $\Fh{\infty}_{KF}$ is the differential ideal generated by
$$\Ih{\infty}_{KF}=\langle \theta_0, \beta_{2j}, \gamma_{2j}, \alpha_{2j+1} \rangle_{j=1}^{\infty}.$$
\end{defn}

\one
We now record the structure equation on $\Fh{\infty}_{KF}$. 
Note first by definition that  for $n\geq 2$,
\begin{align*}
\ed A^{2n-1}&=\alpha_{2n-1}+\frac{1}{2} \left( B^{2n}+\gamma C^{2n} \right)\omega+\frac{1}{2} \left( \gamma s^{-2} B^{2n-2}+s^2 C^{2n-2}\right)\omb, \\
\ed B^{2n}&=\beta_{2n}+\left( A^3B^{2n} -\gamma A^{2n+1} \right) \omega-s^2A^{2n-1} \omb,\\
\ed C^{2n}&=\gamma_{2n}-\left( A^{2n+1}+A^3 C^{2n}  \right) \omega-\gamma s^{-2} A^{2n-1} \omb.
\end{align*}
The following structure equation is obtained by a direct computation.
\begin{align*}
\ed \theta_0&=-\frac{1}{2} ( \beta_2 \w \omega + \betab_2 \w \omb),\\
\ed \rho &= \frac{\im}{2} \left( \gamma^2 s^{-2} - s^2 \right) \omega \w \omb +\frac{\im}{2} (s^2 \betab_2 - \delta \beta_2 ) \w \omega + \frac{\im}{2} (\delta \betab_2 -s^2 \beta_2 )\w  \omb -\frac{\im}{2} s^2 \beta_2 \w \betab_2,\\
\ed \omega &=\gamma_2 \w \omega - \theta_0 \w \left( s^2[\betab_2+ \omb]+ \delta \omega \right),\\
\ed h_1 &=-\im h_1 \rho+h_1 \gamma_2 + h_1A^3 \omega, \\
\ed s &=\frac{s}{2} \left( A^3 \omega + \ol{A^3} \omb + \gamma_2 +\gammab_2 \right),\\
\ed \beta_2 &= -2 \gamma_2 \w \omega +\beta_2 \w (\gamma_2 - \delta \theta_0) -\betab_2 \w s^2 \theta_0+ \theta_0 \w (2 \delta \omega +(s^2+\gamma^2 s^{-2}) \omb),\\
\ed \gamma_2 &=  -\alpha_3 \w \omega =A^3 \gamma_2 \w \omega + \frac{1}{2} (\delta \beta_2 - s^2 \betab_2) \w \omega + \frac{1}{2} (s^2 \beta_2 - \delta \betab_2)  \\
& +A^3 \theta_0 (\delta \omega + s^2 \omb) + \frac{1}{2} s^2 \beta_2 \w \betab_2 + s^2 A^3 \theta_0 \w \betab_2,  \\
\ed \alpha_3 &= -\frac{1}{2} \left( \beta_4 + \gamma \gamma_4 +(B^4+\gamma C^4)\gamma_2 \right) \w \omega + \left( \frac{1}{2}[s^2 + \gamma^2 s^{-2}] \gamma_2 + s^2 \gammab_2  \right) \w \omb\\
& +\frac{1}{2} \theta_0 \w  \left[ \left(-s^4 + \gamma^2 + \delta(B^4 + \gamma C^4) \right) \omega+  \left( s^2[B^4+\gamma C^4] +\delta (\gamma^2 s^{-2} - s^2)\right)  \omb  \right] \nonumber \\
&+ \frac{1}{2} \theta_0 \w \left[ (\gamma^2 - s^4) \beta_2 + s^2 [B^4 + \gamma C^4 ] \betab_2 \right].
\end{align*}
And then for $n \geq 2$,
\begin{align*}
&\ed \beta_{2n}=\left( \gamma \alpha_{2n+1} -A^3 \beta_{2n} - B^{2n} \alpha_3 +(\gamma A^{2n+1} -A^3 B^{2n}) \gamma_2 \right) \w \omega  \\
&+s^2\left( \alpha_{2n-1} + A^{2n-1} \gamma_2 + 2 A^{2n-1} \gammab_2  \right) \w \omb \nonumber \\
&+\theta_0 \w \left( [\delta(A^3 B^{2n} - \gamma A^{2n+1}) -s^4 A^{2n-1}] \omega +s^2[-\gamma A^{2n+1}+A^3 B^{2n} - \delta A^{2n-1}  ] \omb \right)  \nonumber \\
& +\theta_0 \w \left( [-s^4 A^{2n-1}] \beta_2 +s^2[A^3 B^{2n} -\gamma A^{2n+1}] \betab_2 \right),  \nonumber \\
&\nonumber \\
&\ed \gamma_{2n}= \left( \alpha_{2n+1} +A^3 \gamma_{2n}+C^{2n} \alpha_3 +[A^{2n+1} + A^3 C^{2n}] \gamma_2 \right) \w \omega + \gamma s^{-2} \left( \alpha_{2n-1} -A^{2n-1} \gamma_2  \right) \w \omb\\
&  -\theta_0 \w \left[  \left(  \delta(A^{2n+1}+A^3C^{2n})  +\gamma A^{2n-1}  \right)  \omega  
  - \left(  s^2(A^{2n+1}+A^3 C^{2n}) +\delta  \gamma s^{-2} A^{2n-1} \right) \omb \right] \nonumber \\
&  -\theta_0 \w \left(  \gamma A^{2n-1} \beta_2 + s^2 (A^{2n+1} + A^3 C^{2n} ) \betab_2 \right),  \nonumber \\
\nonumber \\
&\ed \alpha_{2n+1}=-\frac{1}{2} \left( \beta_{2n+2}+\gamma \gamma_{2n+2} \right) \w \omega
 -\frac{1}{2} \left( B^{2n+2}+\gamma C^{2n+2} \right) \gamma_2 \w \omega \\
&-\frac{1}{2} \left( \gamma s^{-2} \beta_{2n}+s^2 \gamma_{2n} \right)  \w \omb 
 +\left( \frac{1}{2}[\gamma s^{-2} B^{2n} -s^2 C^{2n})\gamma_2 -s^2 C^{2n} \gammab_2  \right) \w \omb \nonumber  \\
 &+ \frac{1}{2} \left(  \delta B^{2n+2} +\delta \gamma C^{2n+2}+ \gamma  B^{2n} +s^4 C^{2n}  \right) \theta_0 \w \omega \nonumber  \\
 &+ \frac{1}{2} \left(  s^2 (B^{2n+2} + \gamma C^{2n+2})+ \delta(\gamma s^{-2}  B^{2n} +s^2 C^{2n})  \right) \theta_0 \w \omb \nonumber  \\
  &+ \frac{1}{2} \theta_0 \w \left( (\gamma B^{2n} +s^4 C^{2n})\beta_2  +s^2 (B^{2n+2}+\gamma C^{2n+2}) \betab_2\right).  \nonumber 
\end{align*}
One may check that they formally satisfy the compatibility equation $\ed^2=0$. 

\two
A question arises as to if this extension to formal Killing field EDS does keep the set of integral surfaces, or specifically given a CMC surface how unique the associated formal Killing field is (recall the existence of certain ambiguity, integration constants, when solving for the first few formal Killing coefficients in Sec.~\ref{sec:primitivemap}). We wish to show that the original differential system $(\Fh{\infty}, \Ih{\infty})$ is the natural closure of the formal Killing field system $(\Fh{\infty}_{KF}, \Ih{\infty}_{KF})$. This implies in particular that a (non-totally geodesic) CMC surface admits a unique associated formal Killing field.

In the below we draw from the results of Sec.~\ref{sec:recursion}. From the inductive formulae to be defined in Sec.~\ref{sec:inductiveformula}, there exists an embedding
\be 
\mu:\Fh{\infty}\hook\Fh{\infty}_{KF}. \n
\ee
Let $Q^{2j}, \,j=1, 2, \, ... \, $, be the sequence of functions which defines the relation between the formal Killing field coefficients $b^{2j}$ and $c^{2j}$, e.g., 
\begin{align}\label{eq:Qseriesformula}
Q^{2}&=\gamma+b^2 c^2, \\
Q^{4}&=h_2^{\frac{1}{2}}b^{4}-\gamma h_2^{-\frac{1}{2}} c^{4}-m_{4} a^3 a^{3},\n\\
&\;... \, \n\\
Q^{2j+2}&=h_2^{\frac{1}{2}}b^{2j+2}-\gamma h_2^{-\frac{1}{2}} c^{2j+2}-m_{2j+2} a^3 a^{2j+1}+
\, \tn{lower order terms},\n \\
&\;... \,  \n
\end{align}
(here $b^{2j}, c^{2j}$ are constant multiples of $h_1^{-1} B^{2j}, h_1 C^{2j}$ respectively).
Let $\mathcal{Q}\subset\Omega^0(\Fh{\infty}_{KF})$ be the ideal of functions generated by $\{Q^{2j}\}_{j=1}^{\infty}$. The image $\mu(\Fh{\infty})$ is cut out by  $\mathcal{Q}$, and moreover it is clear by definition of the embedding $\mu$ that $\mathcal{Q}$ is complete in the sense that any function on $\Fh{\infty}_{KF}$ which vanishes identically on $\mu(\Fh{\infty})$ lies in $\mathcal{Q}$.
\begin{lem}
Consider the extended ideal $(\Ih{\infty}_{KF})^+$ on $\Fh{\infty}_{KF}$ generated by
\be
(\Ih{\infty}_{KF})^+= \Ih{\infty}_{KF}\oplus\mathcal{Q}.\n
\ee
Then for each $j\geq 1$
\be\label{eq:Q2j}
\ed Q^{2j}\equiv 0\mod (\I{\infty}_{KF})^+,
\ee
and $(\I{\infty}_{KF})^+$ is formally a Frobenius system.
It follows that the embedding $\mu$ induces a natural isomorphism.
\be 
(\Fh{\infty}, \Ih{\infty}) \simeq (\Fh{\infty}_{KF}, (\Ih{\infty}_{KF})^+).\n
\ee 
\end{lem}
\begin{proof}
Suppose 
$$\ed Q^{2j}\equiv Q^{2j}_{\xi}\xi+Q^{2j}_{\xib}\xib,\mod  \Ih{\infty}_{KF}.$$
Since $Q^{2j}$ vanishes when restricted to $\mu(\Fh{\infty})$,  the derivatives $Q^{2j}_{\xi}, Q^{2j}_{\xib}$ must also vanish. The claim follows from the completeness of the ideal $\mathcal{Q}$. 
\end{proof}
We shall identify $\Fh{\infty}$ with its image $\mu(\Fh{\infty})\subset\Fh{\infty}_{KF}$ from now on.
\begin{rem} $\mu^*\Ih{\infty}_{KF}=\Ih{\infty}$.
\end{rem}
We claim that for a non-totally geodesic CMC surface the associated formal Killing field is uniquely determined.
\begin{prop}\label{prop:closure}
A (non-totally geodesic) integral surface of $(\Fh{k}_{KF},\Ih{k}_{KF})$ necessarily lies in $\Fh{k}$ as an integral surface of $\Ih{k}$.
The system $(\Fh{k},\Ih{k})$ is a natural closure of $(\Fh{k}_{KF},\Ih{k}_{KF})$ by the ideal of functions $\mathcal{Q}$. 
\end{prop}
\begin{proof}
First we show that each $Q^{2j}$ vanishes on every integral surface of $\Ih{\infty}_{FK}$ (which necessarily satisfies the independence condition $\xi\w\xib\ne 0$). 

Note that $c^2\ne 0$ on $\Fh{\infty}_{KF}$ by definition. We shall consider the following refinement of the equation \eqref{eq:Q2j}.
\be
\ed Q^{2j}\equiv Q^{2j}\chi_{2j} \mod  \Ih{\infty}_{KF}, \cup_{\ell=1}^{j-1}Q^{2\ell}.\n
\ee

\two
$\bullet$ case $Q^2$: We have $Q^2=\gamma-b^2c^2$. Differentiating this we get
\be
-\frac{\ed Q^2}{Q^2}\equiv\frac{\im}{2}c^2a^3\xi \mod \Ih{\infty}_{KF}. \n
\ee
If $Q^2$ were not identically zero on an integral surface, we must have
$$ \ed (c^2a^3\xi) \equiv  c^2 a^3_{\xib}\,\xib\w\xi \equiv 0\mod\Ih{\infty}_{KF}.$$
This implies from the structure equation that 
\be\label{eq:Q2relation}
\gamma b^2-(\bar{c^2})^2 c^2=0.
\ee
Differentiating this again by $\del_{\xib}$, we get 
$$a^3=0.$$
Thus the integral surface is necessarily flat and $|c^2|^2=|\gamma|=\tn{const}$. It then follows from \eqref{eq:Q2relation} that $Q^{2}=0$, a contradiction. Thus $Q^2$ vanishes on every integral surface of $\Ih{\infty}_{KF}$.

$\bullet$ case $Q^{2j}, j\geq 2$: From the formula \eqref{eq:Qseriesformula}, by a direct computation one finds that
\be
\frac{\ed Q^{2j+2}}{Q^{2j+2}}\equiv \big(\frac{1}{2}-\im\,m_{2j+2}\big) a^3\xi \mod  \I{\infty}_{KF}, \cup_{\ell=1}^{j}Q^{2\ell}.\n
\ee
Note here that the $a^{2j+3}$-terms cancel. It is easily checked that $m_{2j+2}$ is real and hence $\frac{1}{2}-\im\,m_{2j+2}\ne 0$. 
By the similar argument as above, if $Q^{2j+2}$ were not identically zero this leads to a contradiction.
\end{proof}

\one
We close this section by introducing a class of generalized finite-type CMC surfaces  which can be conveniently expressed in terms of  $(\Fh{k}_{KF},\Ih{k}_{KF})$.
\begin{defn}
Let 
\[Y^k=\{ h_1,A^3,B^4,C^4, \ldots, A^{2k-1}, B^{2n}, C^{2k},\hb_1, \ol{A}^3,\ol{B}^4,\ol{C}^4, \ldots, \ol{A}^{2k-1}, \ol{B}^{2n}, \ol{C}^{2k}\}.\] 
Thus $Y^k  \cong  \C^{2k}$. For a scalar function $F :Y^k \to \C$ let $\ed F \equiv F_{\omega}\omega+F_{\omb} \omb$ to define $F_{\omega}$ and $F_{\omb}$.  Using the structure equations above we find that
\begin{align}
F_{\omb}=\dd{F}{A^j}+ \ldots\,  \n
\end{align}
We will refer to an equation of the form
\begin{equation}\label{eq:finitetypeF}
F_{\omb}=\frac{1}{2}\left( \gamma s^{-2} B^{2k}+s^2 C^{2k}\right)
\end{equation}
as the {\bf finite-type equation} (recall that $s=|h_1|$).
\end{defn}
Solutions to the finite-type equation are locally bountiful as it is a single equation for a single function. 
\begin{lem}
Fix $k>2$.  Let $F:Y^k \to \C$ be a solution to Eq.~\eqref{eq:finitetypeF}. The structure equations
\begin{align*}
B^2&=\gamma, \;C^2=-1,\\
\ed s &=\frac{s}{2} \left( A^3 \omega + \ol{A^3} \omb\right),\\
\ed A^{2n-1}&=\frac{1}{2} \left( B^{2n}+\gamma C^{2n} \right)\omega+\frac{1}{2} \left( \gamma s^{-2} B^{2n-2}+s^2 C^{2n-2}\right)\omb, \\
\ed B^{2n}&=\left( A^3B^{2n} -\gamma A^{2n+1} \right) \omega-s^2A^{2n-1} \omb,\\
\ed C^{2n}&=-\left( A^3 C^{2n} +A^{2n+1} \right) \omega-\gamma s^{-2} A^{2n-1} \omb,
\end{align*}
for $n=2,3,\,...\, k$, along with the condition $A^{2k+1}=F(s,A^3, \ldots, C^{2n}, \ol{A}^3, \ldots, \ol{C}^{2n})$ formally satisfy $\ed^2=0$.
\end{lem}
\begin{proof}
We only need to check the last case of $n=k$. A direct calculation shows that $\ed^2 B^{2k} = \ed^2 C^{2k} =0$ if and only if $F_{\omb}=\frac{1}{2}\left( \gamma s^{-2} B^{2k}+s^2 C^{2k}\right)$.
\end{proof}

\begin{cor}
Each solution $F$ of the finite-type equation defines a reduction of the infinite structure equations to a finite-dimensional Lie-Cartan system.  Thus each solution $F$ of the finite-type equation defines a finite-dimensional space of local integral manifolds for the CMC system.
\end{cor}
The challenge now lies in finding a nonlinear $F$ for which we can gain a detailed geometric information about the integral manifolds.

\begin{rem}
It appears that one could specify $F$ so that $\dd{F}{\ol{A}^j}=\dd{F}{\ol{B}^j}=\dd{F}{\ol{C}^j}=0$.
\end{rem}


\begin{rem}
A candidate for a definition of an \emph{integrable} Monge-Ampere EDS on a five dimensional manifold is the following: determine the rough structure equations on the infinite prolongation. If the system admits infinitely many (nonlinear) finite-type reductions then it is integrable. A stronger notion would be to require that there exist  linear or polynomial such reductions. 
\end{rem}

Returning to the original CMC system, this general existence of reductions is only of theoretical interest.  For the practical purpose of solving the CMC equation explicitly in anyway we now must search for a scalar function $F$ that is preferably a polynomial function on $Y$. If $F$ is a polynomial, then the finite-type equation is likely to impose further relations that restrict to a submanifold of $Y$, as happens for the case when $F$ is linear. 

\begin{quest}
Is it possible that we might use the Cartan-K\"ahler approach to specify the initial data for $F$ on a hypersurface as a polynomial and then prove that the thickening is still polynomial?  

Is there anyway to relate this finite type equation to the hydrodynamic reduction, though it seems almost the reverse of hydrodynamic reduction.
\end{quest}
 
\section{Spectral symmetry}\label{sec:spectralsymmetry}
Recall from Lem.~\ref{lem:Tjh} that $\hat{T}_j$ is an element in the polynomial ring $\C[r^2,z_3, z_4, \,... \, ].$
From the formulae 
\begin{align*}
\del_{\omega} z_j&=z_{j+1}-\frac{j}{2}z_3z_j, \\
\del_{\omega} r^2&=r^2 z_3,
\end{align*}
suppose we assign the weights $\tn{wt}(z_j)=j-2, \,\tn{wt}(r^2)=0.$\footnotemark\footnotetext{From Cor.~\ref{cor:zurelation}, $\,z_j$ corresponds to  $u_{j-2}$ modulo lower order terms. }
Then the recursive defining equation \eqref{eq:zetastruct} for $\hat{T}_j$ shows that $\hat{T}_j$ is weighted homogeneous of weight $j-3$.

We introduce a new non-local symmetry called spectral symmetry 
to capture this auxiliary symmetry of the CMC system, 
which is uncovered only in terms of the balanced coordinates $z_j$'s of $\xinfh_*$.
For a nonlinear PDE this would be an exceptional phenomenon and 
in turn gives an indication of the underlying geometric structure.
\subsection{Integrable extension}\label{sec:extension}
We start with a general description of deformation of submanifolds in a homogeneous space via moving frame method.  
 
Recall that $\g$ is the six dimensional Lie algebra of the group of isometries $\G$ of the ambient 3-dimensional space form $M$. We shall identify $\g$ with the space of Killing vector fields on $M$. Let $\phi$ be a $\g$-valued 1-form on a manifold satisfying the Maurer-Cartan structure equation
$$\ed\phi+\phi\w\phi=0.$$
Consider a one parameter family of such 1-forms
$$\phi(t)\equiv\phi+t\delta\phi\mod \mco(t^2).$$
Then $\delta\phi$ satisfies the deformation equation
$$\ed(\delta\phi)+\delta\phi\w\phi+\phi\w\delta\phi=0.$$

Symbolically, let $e$ be a $\G$-frame ($\G$-valued function) for $\phi$ that satisfies the structure equation
$$\ed e=e \phi.$$
For a $\g$-valued function $v$, let 
$$e(t)\equiv e(\tn{I}+t v)\mod \mco(t^2)$$
be the deformation of frame corresponding to $\phi(t)$. Then $v$ satisfies the corresponding deformation equation
\be\label{eq:deformationequation}
\ed v=-\phi v+v\phi+\delta\phi.
\ee
We wish to introduce a generalized symmetry of CMC surfaces based on this description of deformation.

\two
A general criterion for the concept of symmetry of a differential equation would be a structure/rule that allows one to produce new solutions from the given ones. In this broader sense the following  spectral deformation  qualifies as a symmetry for the CMC system.

Consider the original structure equation \eqref{eq:strt2} on $X$ modulo $\mci$, i.e., we only consider \eqref{eq:strt2} as the induced structure equation on a CMC surface. For a unit complex number $e^{\im t}$, one notes that the structure equation modulo $\mci$ is invariant under the deformation
\be\label{eq:associatesymmetry}\left\{
\begin{array}{rlrl}
\xi&\lra\;\; \xi, &\quad\xib&\lra\;\; \xib,  \\
\rho&\lra\;\; \rho,&&\\
\eta_1&\lra\;\;   e^{2\im t} \eta_1, &\quad\etab_1&\lra\;\;  e^{-2\im t} \etab_1.
\end{array}\right.
\ee
The corresponding isometric $\s{1}$-family of (possibly multi-valued) CMC surfaces are called \tb{associate surfaces} to the original surface. 
\begin{rem} 
Compared to the symmetry defined in Sec.~\ref{sec:JacobiSymmetry}, one may heuristically consider \eqref{eq:associatesymmetry} as
$$\tn{symmetry}\mod\iinf.$$
It is evident from the structure equation that spectral deformation cannot be extended to a symmetry of $(\xinf,\iinf)$ in the usual sense.
\end{rem}
An analytic formulae for spectral symmetry can be derived based on the deformation equation \eqref{eq:deformationequation}.
Let $\{v_0,  v_{1,0}, v_{\bar{1},0},  v_{\rho}, v_{0,1}, v_{0,\bar{1}}\}$ be the components of the Lie algebra valued deformation function $v$ above corresponding to the coframe $\{\theta_0, \xi,\xib, \rho,\eta_1,\etab_1\}$. Then the deformation equation \eqref{eq:deformationequation} is written in these variables as follows (here $v_0$ is real, $v_{\rho}$ is imaginary,
and $v_{\bar{1},0}=\ol{v}_{1,0}, v_{0,\bar{1}}=\ol{v}_{0,1}$):
\begin{align}\label{eq:spectraldeform}
\pi_0:=\ed  v_0&-\Big[  (v_{1,0}+ h_2 v_{0,1})\xi +(v_{\bar{1},0}+ \hb_2 v_{0,\bar{1}})\xib+
v_{{0,1}}\theta_{{1}}+v_{{0,\bar{1}}}\thetab_1 \Big]\equiv 0,\\
\pi_{0,1}:=\ed  v_{0,1} - \im  v_{0,1} \rho&-\Big[-\frac{1}{2} (\delta v_0-v_{\rho}) \xi -\frac{1}{2}  \hb_{2}v_0 \xib
+\left( v_{{\bar{1},0}}-\delta v_{{0,1}} \right) \theta_{{0}}-\frac{1}{2} v_{{0}}\thetab_1 \Big]\equiv 0,\n\\
\pi_{0,\bar{1}}:=\ed  v_{0,\bar{1}}+\im  v_{0,\bar{1}} \rho&
-\Big[ -\frac{1}{2}  h_{2}v_0 \xi - \frac{1}{2} (\delta v_0+v_{\rho}) \xib
+\left( v_{{1,0}}-\delta v_{{0,\bar{1}}} \right) \theta_{{0}}-\frac{1}{2} v_{{0}}\theta_1 \Big]\equiv 0,\n\\
\pi_{\rho}:=\ed  v_{\rho}&-\Big[ \Big( (-\delta v_{1,0}+h_2 v_{\bar{1},0})+(\gamma^2 v_{0,\bar{1}}-\delta h_2 v_{0,1})\Big)\xi   \n\\
                   &\;\quad+\Big( (\delta v_{\bar{1},0}-\hb_2 v_{1,0})
+(-\gamma^2 v_{0,1}+\delta \hb_2 v_{0,\bar{1}})\Big)\xib\n\\
 &\;\quad+ \left( v_{{\bar{1},0}}-\delta v_{{0,1}} \right) \theta_{{1}}
+ \left( -v_{{1,0}}+\delta v_{{0,\bar{1}}} \right) \thetab_1 \Big]\equiv 0,\n\\
\pi_{1,0}:=\ed  v_{1,0} + \im  v_{1,0} \rho&
-\Big[  \Big(-\frac{1}{2}h_2 (\delta v_0-v_{\rho}) +\im h_2\Big)\xi -\frac{1}{2}\gamma^2 v_0 \xib \n\\
&\;\quad + \left( \delta v_{{1,0}}-\gamma^2 v_{{0,\bar{1}}}\right) \theta_{{0}}
-\frac{1}{2} \left( \delta v_{{0}}-v_{{\rho}} \right) \theta_{{1}} \Big]\equiv 0,\n\\
\pi_{\bar{1},0}:=\ed  v_{\bar{1},0}-\im  v_{\bar{1},0} \rho&-\Big[ -\frac{1}{2}\gamma^2 v_0 \xi 
+\Big(-\frac{1}{2}\hb_2 (\delta v_0+v_{\rho}) -\im \hb_2\Big)\xib \n\\
&\;\quad + \left( \delta v_{{\bar{1},0}}-\gamma^2 v_{{0,1}}\right) \theta_{{0}}
-\frac{1}{2} \left( \delta v_{{0}}+v_{{\rho}} \right) \thetab_{{1}} \Big]\equiv 0, 
\mod\;\iinf.\n
\end{align}
One may check by direct computation that this structure equation is compatible mod $\iinf$, i.e., when considered as a structure equation for the variables $\{v_0,  v_{1,0}, v_{\bar{1},0},  v_{\rho}, v_{0,1}, v_{0,\bar{1}}\}$, the identity $\ed^2=0$ holds mod $\iinf$.

Note the inhomogeneous terms $\im h_2, -\im \hb_2$ in $\delx v_{1,0}, \delxb v_{\bar{1},0}$. Without these terms Eq.~\eqref{eq:spectraldeform} is isomorphic to the structure equation for the classical conservation laws Eq.~\eqref{eq:classicalJacobi}.  This implies that the space of spectral symmetries is formally a six dimensional affine space modeled on the Lie algebra $\g$. 

Note also that $v_0$ is a Jacobi field, while $v_{\rho}$ is not and satisfies 
\be\label{eq:vrho}
\mce(v_{\rho})=-\im h_2\hb_2.
\ee
See Sec.~\ref{sec:PicardFuchs} for a geometrical interpretation of this equation.

\two
The analysis in the next section will reveal that the solutions to \eqref{eq:spectraldeform} do not exist on $\xinf$. In order to incorporate the spectral symmetry to our analysis, which is relevant  among other things for our proof for the nontriviality of higher-order conservation laws, let us introduce an integrable extension. The idea is to add the spectral deformation coefficients $\{v_0, v_{1,0}, v_{\bar{1},0},  v_{\rho}, v_{0,1}, v_{0,\bar{1}}\}$ as new variables subject to the compatible differential relation defined by \eqref{eq:spectraldeform}. One may  consider this as a PDE analogue of field extension for algebraic equation, \cite{Vinogradov1989, Krasilshchik2004}.

Recall the $\SO(2)$-bundle $\finf\to\xinf.$ Define the product bundle
\be\label{eq:ZFbundle}
Z_{\mcf}=\finf\times_{\SO(2)}\g \to \finf, 
\ee 
where $\SO(2)$ group action on $\g$ is specified by the connection 1-form $\rho$ in \eqref{eq:spectraldeform}.
Set the affine bundle modeled on $\g$ by 
\be\label{eq:Zbundle}
Z=Z_{\mcf}/\SO(2) \to\xinf.
\ee
We have the commutative diagram: \one\\
\centerline{\xymatrix{Z_{\mcf} \ar[d]^{\Pi} \ar[r]&Z \ar[d]^{\Pi}\\
\finf\ar[r] &\xinf}}\\
\begin{defn}\label{defn:integrableextension} 
Let $\Pi: Z \to\xinf$ be the affine bundle \eqref{eq:Zbundle} modeled on the Lie algebra $\g$ with fiber coordinates  
$\{v_0, v_{1,0}, v_{\bar{1},0},  v_{\rho}, v_{0,1}, v_{0,\bar{1}}\}$.
Let $Z_{\mcf}\to Z$ be the associated principal right $\SO(2)$ bundle defined by \eqref{eq:ZFbundle}. 

Let $\mcj_{\mcf}$ be the extended ideal on $Z_{\mcf}$ generated by
\begin{align}
\mcj^{\pi}_{\mcf}&=\langle \pi_0,  \pi_{1,0}, \pi_{\bar{1},0},  \pi_{\rho}, \pi_{0,1}, \pi_{0,\bar{1}} \rangle,\n\\
\mcj_{\mcf}&=\mcj^{\pi}\oplus\Pi^*\iinf _{\mcf}.\n
\end{align}
It is clear that these ideals are invariant under $\SO(2)$ action, and let $\mcj^{\pi}, \mcj$ be the corresponding ideals on $Z$. By construction $\mcj$ is Frobenius. The extended EDS $(Z,\mcj)$ is the \tb{integrable extension} of $(\xinf,\iinf)$ for the spectral symmetry. 

A vector field $V$ on $Z$ is  $\Pi$\tb{-horizontal} if it lies in the kernel of $\mcj^{\pi}$. 
A differential form $\Phi$ on $Z$ is  $\Pi$\tb{-horizontal} if it lies pointwise in the span of $\Pi^*\iinf$.
\end{defn}
\begin{rem}
Let $\xinfh\to\xinf$ be the double cover. The pull-back integrable extension is denoted by
$$\hat{Z} \to\xinfh.$$
We denote the corresponding ideals by the upper-hat notation, $\hat{\mcj}, \hat{\mcj^{\pi}}.$
\end{rem}
Let us mention two technical points. 
First, we shall identify $\iinf$ with its image $\Pi^*\iinf\subset\mcj$ (and similarly for $\Pi^*\iinf_{\mcf}\subset\mcj_{\mcf}$).
Second, the vector fields $\delx,\delxb, E_0, E_j, E_{\bar{j}}$ on $\finf$ admit unique lifts to $Z_{\mcf}$ as $\Pi$-horizontal vector fields. We denote these horizontal lifts by the same notation. Symbolically we have
$$ \{ \delx,\delxb, E_0, E_j, E_{\bar{j}}\} \subset (\mcj^{\pi}_{\mcf})^{\perp}.
$$

\one
On this integrable extension $Z$ we shall show that the spectral symmetry has a representation as a generalized symmetry vector field.\footnotemark\footnotetext{This kind of generalized symmetry is sometimes called a \emph{shadow}, \cite{Krasilshchik2004}\cite{Vinogradov1989}. See Definition \ref{defn:shadow}.} The non-local spectral symmetry will play an important role in our analysis of the (local) higher-order Jacobi fields and conservation laws. For example it allows one to decompose the relevant structures of $(\xinf,\iinf)$
into the weighted homogeneous pieces under spectral symmetry.
 
\two
We close this section with two lemmas on the properties of the extension $(Z,\mcj)$.
The following lemma  gives a property which separates the cases $\epsilon\ne0,$ or $\epsilon=0$.
\begin{lem}\label{lem:noZintegral}
Let $u:Z\to\C$ be a scalar function such that
$$\ed u \equiv 0 \mod \mcj.$$
Then,

\one
\qquad Case $\epsilon\ne 0$: $u$ is a constant.

\qquad Case $\epsilon=0$: the rank three sub-ideal generated by
$$\mcj^{\tn{FI}}=\langle 2\im\delta\theta_0+\pi_{\rho}, \theta_1+\im(\pi_{1,0}-\delta\pi_{0,\bar{1}}),\theta_{\bar{1}}-\im(\pi_{\bar{1},0}-\delta\pi_{0,1}) \rangle\subset\mcj$$
is Frobenius, and
$$\ed u \in \mcj^{\tn{FI}}.$$
\end{lem}
\begin{proof}
Let
$$\ed u = \pi^u + \theta^u,$$
where
\begin{align}
\pi^u &=
 b^0\pi_0+b^{1,0}\pi_{1,0}+b^{\bar{1},0}\pi_{\bar{1},0}
+b^{\rho}\pi_{\rho}+b^{0,1}\pi_{0,1}+b^{0,\bar{1}}\pi_{0,\bar{1}}, \n\\
\theta^u&=  u^0\theta_0+\sum_{j=1}^k (u^j\theta_j+u^{\bar{j}}\thetab_j).\n
\end{align}
We check the consequences of the identity $\ed(\ed u)=0$.

From the structure equation we have
$$\ed \pi^u\equiv 0 \mod \mcj^{\pi}, \theta_0,\theta_1,\thetab_1,\theta_2,\thetab_2.$$
Considering the equation
$$\ed(\ed u)\equiv 0\mod \mcj^{\pi}, \theta_0,\theta_1,\thetab_1,\theta_2,\thetab_2, \, ... \, \theta_j,\thetab_j,$$
an inductive argument shows that $u^j=u^{\bar{j}}=0$ for $j=k, k-1, \, ... \, 2$. We may thus assume $k=1$. The remaining analysis follows by direct computation.
\end{proof}

An integrable extension is called an \emph{Abelian extension} for a conservation law $\varphi$ when the extension is given by introducing the anti-derivative for $\varphi$,
$$\pi=\ed v -\varphi,$$
and hence trivializes the conservation law. 
The following lemma shows that $(\Zh,\mcjh)$ does not include the Abelian extension for $\omega=\sqrt{\ff}$.
\begin{lem}\label{lem:phi0onZ}
Let $\omega=h_2^{\frac{1}{2}}\xi=\sqrt{\ff}$, which represents a nontrivial higher-order conservation law on $\Xh{1}$.
Then for any nonzero complex number $\lambda\in\C^*$ the  characteristic cohomology class 
$$[\tn{Re}(\lambda \omega)]\in H^1(\Omega^*(\Zh)/\mcjh, \underline{\ed})$$ 
is  nontrivial.
\end{lem}
\begin{proof}
It suffices to show that there does not exist a 1-form $\Theta\in\Omega^1(\mcjh)$ such that
$$\ed(\tn{Re}(\lambda \omega)+\Theta)=0.$$
From the structure equation for $\omega=h_2^{\frac{1}{2}}\xi$ one may assume that
$$\Theta\equiv 0\mod \theta_0,\theta_1,\thetab_1,\mcjh^{\pi}.$$
The claim is verified by a direct computation.
\end{proof}
 
\subsection{Spectral symmetry}\label{sec:spectralsym}
Assign the weights
\be 
\tn{weight}(z_j)=j-2,\qquad \tn{weight}(\zb_j)=-(j-2), \qquad j\geq 3.\n
\ee
We wish to find a symmetry vector field (on $\Zh$) of the form
\be\label{spectralsymmetry}
\mcs  = V_{\xi}\delx+V_{\xib}\delxb+V_0 E_0+\sum_{j=1}^{\infty} \left(  V_j E_j + V_{\bar{j}}  E_{\bar{j}} \right) 
\ee
with the following properties:

\begin{enumerate}[\qquad a) ]
\item  it is purely imaginary,  $\overline{\mcs}=-\mcs$,  and hence
$$ V_{\xib}=-\ol{V_{\xi}},\qquad V_{\bar{j}}=-\overline{V_j }\;\;\tn{for}\;j\geq 0,$$

\item  when restricted to the variables $z_j, \zb_j$'s, it is the weighted Euler operator, 
$$ \mcs(z_j)=(j-2)z_j, \qquad \mcs(\zb_j)=-(j-2)\zb_j,$$

\item   it commutes with the Jacobi operator $\mce$ acting on the scalar functions on $\xinfh,$ 
\be\label{eq:commuterelation}
\mce\circ\mcs=\mcs \circ\mce.
\ee
\end{enumerate}
The condition c) means that $\mcs$ maps Jacobi fields to possibly non-local Jacobi fields.

\begin{rem}
For a classical symmetry $V$ (with horizontal part) which is generated by the isometries of $M$, we have
$$ \mcl_V \omega=0,\quad V(h_2\hb_2)=0,\quad V(z_j)=V(\zb_j)=0, \;\;\tn{for}\; j\geq 3.$$
It is easily checked that this property characterizes the subspace of classical symmetries among the symmetries. 

It is also clear that a classical symmetry is a symmetry of the Jacobi operator and satisfies  c). This shows that the space of spectral symmetries, if exists, would be a six dimensional affine space over the space of classical symmetries.
\end{rem}
 
The two conditions a), b) above,  and the structure equation for symmetry vector fields in Sec.~\ref{sec:JacobiSymmetry} show  that $\mcs$ is determined by the set of low order coefficients $\{ V_{\xi}, V_{\xib}, V_0, V_1, V_{\bar{1}}, V_2, V_{\bar{2}}\}$ as follows. 

Set  
\begin{align}\label{eq:spectralsymmetry}
p&=V_2 h_2^{-1}+ V_{\xi} h_2^{-1}h_3=\overline{V_2} \hb_2^{-1}+\overline{V_{\xi}} \hb_2^{-1}\hb_3=\bar{p}, \n \\
V_j &=- V_{\xi} h_{j+1}+\left( (j-2)+\frac{j}{2}p \right)h_j - V_{\xib} T_j, \quad \tn{for}\;\;  j\geq3,  \\
\ed  V_{\xi}-\im  V_{\xi}\rho&\equiv \left(\delta V_0-1-\frac{p}{2}\right)\xi+\hb_2 V_0\xib,\mod \iinf.\n
\end{align}
One may check that these equations are compatible mod $\iinfh$ and formally defines a symmetry vector field satisfying a), b). 
Note that the reality of $p$ implies 
\be\label{eq:Sr^2}
\mcs(h_2\hb_2)=0.
\ee

\two
We claim that the equation \eqref{eq:spectralsymmetry} cannot be solved in terms of local variables on $\xinfh$. Let us content ourselves with a heuristic argument for this. From the last equation for $\ed V_{\xi}$ an inspection shows that $V_{\xi}$ may depend at most on $h_2$, and it cannot involve $h_j$ for $j\geq 3$ (otherwise $V_0$ or/and essentially $V_2$ depends on higher $h_j$'s for $j\geq 3$ and this contradicts the middle equation for $V_3$). Hence $V_3$ depends at most on $h_4$. From the structure equation for symmetry it follows that roughly
$$ V_2 \sim h_3,\quad V_1 \sim h_2.$$
Since $V_1=2\delx V_0$, this implies that the generating Jacobi field $V_0$ should involve the non-local variable\footnotemark\footnotetext{This is only a heuristic explanation.  Lem.~\ref{lem:phi0onZ} shows that this argument is not exactly true.} 
$$V_0\sim \int h_2^{\frac{1}{2}}\xi.$$

\two
The spectral symmetry is instead defined on the integrable extension $Z$ (and hence on the double cover $\Zh$ by taking a lift) as a generalized intermediate symmetry in the following way.  
\begin{defn}\label{defn:shadow}
Let $\Pi:Z \to\xinf$ be the integrable extension in Definition \ref{defn:integrableextension}. Let $\mce$ continue to denote the horizontal lift of the Jacobi operator to $Z$. 

\begin{enumerate}[\qquad a) ]
\item
a scalar function $u:Z\to\C$ is a \tb{non-local Jacobi field} when
$$\mce(u)=0,$$
and $u$ is not a pull-back of a Jacobi field on $\xinf$.

\item
a $\Pi$-horizontal vector field $V$ on $Z$ is a \tb{shadow} when
\be\label{eq:shadow}
\mcl_V \iinf \subset \mcj,
\ee
and $V$ is not a lift of a symmetry on $\xinf$.

\item
a symmetry of $(Z,\mcj)$ is a \tb{non-local symmetry} of $(\xinf,\iinf)$ if it is not a lift of a symmetry on $\xinf$.

\item
a conservation law of $(Z,\mcj)$ is a \tb{non-local conservation law} of $(\xinf,\iinf)$ if it is not a lift of a conservation law on $\xinf$.
\end{enumerate}
\end{defn}
Note that a shadow is an intermediate object and is not necessarily a symmetry of $(Z,\mcj)$. Note also that a non-local symmetry is not necessarily $\Pi$-horizontal.
 
\one
By definition and from Prop.~\ref{prop:symmetry}, a non-local Jacobi field uniquely generates a shadow without $\delx, \delxb$ components.
 
We are now ready to define a spectral symmetry.
\begin{lem}\label{lem:spectralsymmetry}
On the integrable extension $Z\to\xinf$, the real valued function $v_0$ is a non-local Jacobi field. 
Set 
\be\begin{array}{rll}
V_0&=\im v_0,&\n\\
V_1&=2\im  (v_{1,0}+ h_2 v_{0,1}), &V_{\bar{1}}=-\ol{V_1}, \n\\
V_2&= 2\im(  h_2 v_{\rho} +\im h_2+h_3 v_{0,1}), &V_{\bar{2}}=-\ol{V_2},  \n\\
V_{\xi}&= -2\im v_{0,1}, &V_{\xib}=-\ol{V_{\xi}}. \n
\end{array}\ee
Then the \tb{\emph{spectral symmetry}}  $\mcs$ defined by Eqs.~\eqref{spectralsymmetry}, \eqref{eq:spectralsymmetry} with these formulae is a shadow.
\end{lem}
\begin{proof}
It follows from the construction of the integrable extension $(Z,\mcj)$.
\end{proof}

From the fact that $Z\to\xinf$ is an affine bundle, a global consideration leads to the following monodromy invariant of a CMC surface, which is somewhat dual to Kusner's momentum class.
\begin{prop}\label{prop:spectralmonodromy}
Let $\Sigma\hook X$ be a connected integral surface of the EDS for CMC surfaces. Let $\g$ be the Lie algebra of Killing vector fields. 
The spectral symmetry $\mcs$ defined in Lemma \ref{lem:spectralsymmetry} induces \tb{\emph{spectral monodromy}} homomorphism
\be\label{eq:spectralmonodromy}
\mu^{\mcs}:\pi_1(\Sigma) \lra \g.
\ee
\end{prop}
\begin{proof}
When restricted to an integral surface, the system of equations \eqref{eq:spectraldeform} is compatible. Let $\widetilde{\Sigma}\to\Sigma$ be the universal covering. Existence and uniqueness theorem of ODE implies that there exists a lift $\widetilde{\Sigma}\hook Z$ of $\Sigma$ as an integral surface of $\mcj$, which is equivalent to that there exists a six dimensional affine space of spectral symmetries defined globally on $\widetilde{\Sigma}$ modeled on the vector space of classical Killing fields. By the linearity of the equations \eqref{eq:spectraldeform}, the action by deck transformation induces the monodromy homomorphism $\mu^{\mcs}$. Since $\g$ is a (Abelian) vector space, the homomorphism $\mu^{\mcs}$ is independent of choice of a base point for $\pi_1(\Sigma).$
\end{proof}
 
\begin{rem}
Since $\g$ is a vector space, spectral monodromy factors through\one \\  
\centerline{\xymatrix{\pi_1(\Sigma) \ar[d]\ar[dr]^{\mu^{\mcs}} &  \\ 
H_1(\Sigma,\R)\ar[r]^{\;\;\;\mu^{\mcs}} & \g  }}\\ 
and we have
$$\mu^{\mcs}:H_1(\Sigma,\R)\to\g.$$

\noi
Recall Kusner's momentum class $\mu_K\in \g^*\otimes H^1(\Sigma,\R),$ or equivalently
$$\mu_K:\g\to  H^1(\Sigma,\R).$$
As a consequence we have a linear map
$$\mu:=\mu_K\circ\mu^{\mcs}:H_1(\Sigma,\R) \to H^1(\Sigma,\R).$$
\end{rem}
\begin{quest}
Is $\mu$ skew-symmetric?
\end{quest}

\one
The following lemma characterizes the local functions on $\xinfh$ that are stable under the spectral symmetry. 
\begin{lem}\label{lem:Sstable}
Let $f\in C^{\infty}(\xinfh_*)$ be a scalar function. Suppose
$$\mcs(f)\in C^{\infty}(\xinfh_*)\subset C^{\infty}(\Zh).$$
Then $f$ is a function in the variables $\{h_2\hb_2, z_j, \zb_j\}$.
\end{lem}
\begin{proof}
Collect $\mcs(f)$ in the non-local variables $\{V_{\xi}, V_{\xib}, V_0, V_{1}, V_{\bar{1}},V_{2}, V_{\bar{2}}\}$.
In order for $\mcs(f)\in  C^{\infty}(\xinfh_*)$ we first must have 
$$f^0, f^1, f^{\bar{1}}=0.$$
The remaining terms are
\begin{align}
\mcs(f)&\equiv V_{\xi}\left(f_{\xi}-h_{j+1}f_j+\frac{j}{2}h_2^{-1}h_3 h_j f_j- T_{\bar{j}}f_{\bar{j}}\right)
+V_2\left(f_2+\frac{j}{2}h_2^{-1}h_j f_j\right) \n\\
&\quad+\tn{conjugate terms}\mod  C^{\infty}(\xinfh_*).\n
\end{align}
It is clear from this that $f$ is a function of $\{h_2\hb_2, z_j, \zb_j\}$.
\end{proof}
As a derivation the kernel of $\mcs$ is nontrivial on the integrable extension $Z$.
\begin{lem}
The non-local Jacobi field $v_0$ has weight zero with respect to $\mcs$,
$$\mcs(v_0)=0.$$
\end{lem}
\begin{proof}
Direct computation.
\end{proof}

\two
The commutation relation \eqref{eq:commuterelation} will be analyzed in Sec.~\ref{sec:commutation}.



\subsection{Spectral conservation law}\label{sec:spectralcvlaw}
Consider the quotient complex $(\Omega^*(Z)/\mcj, \underline{\ed})$, where $\underline{\ed}=\ed\mod\mcj$. A conservation law  for $(Z,\mcj)$ is by definition an element in the characteristic cohomology $H^1(\Omega^*(Z)/\mcj, \underline{\ed}).$ In the present case when $(Z,\mcj)$ is an integrable extension of $(\xinf,\iinf)$, a natural dual of shadow can be defined as an intermediate differentiated conservation law.
\begin{defn}\label{defn:horizontalcvlaw}
Let $\Pi:Z \to\xinf$ be the integrable extension from Definition \ref{defn:integrableextension}.
Let $\Phi\in\Pi^*\Omega^2(\iinf)$ be a $\Pi$-horizontal 2-form on $Z$. It is a \tb{horizontal conservation law} when
\be
\ed\Phi \equiv 0\mod \mcj^{\pi}. \n
\ee
\end{defn}
Recall that a conservation law of $(Z,\mcj)$ is called a non-local conservation law  of $(\xinf,\iinf)$.
A horizontal conservation law is an intermediate object and it is not necessarily a (differentiated) conservation law of $(Z,\mcj)$, nor $(\xinf,\iinf)$.

\one
By construction of the integrable extension, we have
\begin{lem}
Set
\be
\Phi_{v_0}= v_0\Psi+\im\theta_0\w\big(v_{1,0}\xi-v_{\bar{1},0}\xib+v_{0,1}\eta_1-v_{0,\bar{1}}\etab_1\big).
\ee
Then $\Phi_{v_0}$ is a horizontal conservation law.
\end{lem}
\begin{proof}
From Eq.~\eqref{eq:spectraldeform}, the inhomogeneous terms $\im h_2, -\im\hb_2$ occur only at $\delx v_{1,0}, \delxb v_{\bar{1},0}$ respectively, and they do not contribute to $\ed\Phi_{v_0}$ mod $\mcj^{\pi}$. The rest follows by comparing with the structure equation for classical conservation laws Eq.~\eqref{eq:classicalJacobi}. 
\end{proof}

We wish to show that in fact there exists a non-local conservation law $\varphi_{v_0}\in\Omega^1(Z)$ which corresponds to $\Phi_{v_0}$. Let  
\be 
\varphi_{0}=\im   \left( v_{0,\bar{1}} \xi- v_{0,1} \xib \right).\n
\ee
Differentiating this, one gets the identity
\be
\ed\varphi_0=\Phi_{v_0}+\im \left(\pi_{0,\bar{1}}\w\xi-\pi_{0,1}\w\xib\right)
+\im\delta\left(v_0\xi\w\xib-2v_{0,\bar{1}}\theta_0\w\xi+2v_{0,1}\theta_0\w\xib\right).
\ee

On the other hand, let 
$$\chi_0=v_{1,0}\xi-v_{\bar{1},0}\xib.$$
Then
$$\ed\chi_0\equiv \gamma^2 v_0\xi\w\xib\mod\mcj.$$
From this it follows that 
\be\label{eq:spectralcvlaw}
\varphi_{v_0}=\varphi_0-\im\frac{\delta}{\gamma^2}\chi_0
\ee
represents a non-local conservation law.
\begin{defn}
Let $Z\to X$ be the integrable extension for spectral symmetry. The \tb{spectral conservation law}  is the characteristic cohomology class 
$$[\varphi_{v_0}]   \in    H^1(\Omega^*(Z)/\mcj, \underline{\ed})$$
defined by Eq.~\eqref{eq:spectralcvlaw}.
\end{defn}
\begin{quest} 
Is the cohomology class $[\varphi_{v_0}]\in H^1(\Omega^*(Z)/\mcj, \underline{\ed})$   nontrivial? 
\end{quest}
 
For a CMC surface $\Sigma$, let $\langle \mcc^0 \rangle \subset H^1(\Sigma,\R)$ be the subspace spanned by the restriction of classical conservation laws. The non-local conservation law $[\varphi_{v_0}]$ has the following geometric interpretation.
\begin{thm}[Secondary conservation law] \label{thm:spectralcvlaw}
Let $\Sigma\hook M$ be a CMC surface. The spectral conservation law
$[\varphi_{v_0}] \in    H^1(\Omega^*(Z)/\mcj, \underline{\ed})$,   Eq.~\eqref{eq:spectralcvlaw},
restricts to represent an element in the quotient space
$$ [ \varphi_{v_0} ]\in  H^1(\Sigma,\R)/\langle \mcc^0 \rangle .$$
\end{thm}
\begin{proof}
By the existence and uniqueness theorem of ODE, the universal cover $\widetilde{\Sigma}$ of $\Sigma$ admits a lift to $Z$ as an integral surface of $\mcj$. For a choice of base point in $\widetilde{\Sigma}$ let $[\gamma]\in\pi_1(\Sigma)$ be a deck transformation. Then by construction we have
$$[\gamma]^*\varphi_{v_0}=\varphi_{v_0}+\varphi^{\mu_{\mcs}([\gamma])},$$
where $\varphi^{\mu_{\mcs}([\gamma])}$ is the classical conservation law corresponding to the Killing vector field $\mu_{\mcs}([\gamma])\in\g$ of the spectral monodromy $\mu_{\mcs}.$
Note also that $\varphi_{v_0}$ is smooth.
\end{proof}
The spectral conservation law can thus be considered as a \tb{secondary characteristic cohomology class} of the EDS for CMC surfaces.
\begin{cor}
Let $\Sigma\hook X$ be a CMC surface. Suppose the image of the classical conservation laws $\langle \mcc^0 \rangle \subset H^1(\Sigma, \R)$ is trivial. Then the spectral conservation law $[ \varphi_{v_0} ]$ is a well defined element in $H^1(\Sigma,\R).$
\end{cor}
 


\subsection{Commutation relations}\label{sec:commutation}
In this section various properties of the spectral symmetry are established. In particular, we show that the spectral symmetry commutes with the horizontal lift of Jacobi operator as operators acting on the scalar functions $C^{\infty}(\xinfh)\subset C^{\infty}(\Zh)$.\footnotemark\footnotetext{We identify $C^{\infty}(\xinfh)$ with its image $\Pi^* C^{\infty}(\xinfh) \subset C^{\infty}(\Zh)$.} One may consider this as an analogue of the familiar K\"ahler identities.

\two
We start with a lemma which shows the robustness of the characteristic vector field $h_2^{-\frac{1}{2}}\delx$, which is dual to $\sqrt{\ff}=h_2^{\frac{1}{2}}\xi.$ 
\begin{lem}\label{lem:shadowdelx}
Let $V$ be a shadow (or symmetry) on $\Zh$. Then the $\Pi$-horizontal component  $[V, h_2^{-\frac{1}{2}}\delx]^{\tn{h}}$ is a shadow. It follows that
$$ [V, h_2^{-\frac{1}{2}}\delx]^{\tn{h}} \in\langle \;\delx, \delxb\; \rangle.$$
\end{lem}
\begin{proof}
By definition of $\Pi$-horizontal direction as $(\hat{\mcj}^{\pi})^{\perp}$,
$$\mcl_{[V, h_2^{-\frac{1}{2}}\delx]}(\iinfh)=\mcl_{[V, h_2^{-\frac{1}{2}}\delx]^{\tn{h}}}(\iinfh).
$$
The claim that $[V, h_2^{-\frac{1}{2}}\delx]^{\tn{h}}$ is a shadow follows from the observations
$$\mcl_{\delx} \iinfh\subset\iinfh, \quad \mcl_{\delx}\hat{\mcj}\subset\hat{\mcj},\quad
\mcl_{V}(\iinfh)\subset\hat{\mcj}.$$

Let $f_0$ be the generating non-local Jacobi field for the shadow $[V, h_2^{-\frac{1}{2}}\delx]^{\tn{h}}$. For the second claim it suffices to show that $f_0=0$. Let
$$V=g_0 E_0+g_1 E_1 +g_{\bar{1}} E_{\bar{1}}+\, ... \, .
$$
By the defining property of a shadow one has $(g_0)_{\xi}=\frac{1}{2}g_1$. One computes then
\begin{align}
f_0&=\theta_0([V, h_2^{-\frac{1}{2}}\delx])\n\\
&=-\ed\theta_0(V, h_2^{-\frac{1}{2}}\delx)-h_2^{-\frac{1}{2}}\delx(\theta_0(V))\n\\
&=\frac{1}{2}\theta_1\w\xi(V, h_2^{-\frac{1}{2}}\delx) -h_2^{-\frac{1}{2}}\frac{1}{2}g_1=0.\n
\end{align}
\end{proof}
\begin{rem}
The subspace $\langle \;\delx, \delxb\; \rangle$ of trivial symmetries behaves under Lie bracket like an ideal in the space of shadows (symmetries).
\end{rem}

Recall the Jacobi operator originally defined on $\xinf$,
$$\mce=\delx \delxb +\frac{1}{2}(\gamma^2+h_2\hb_2). 
$$
The horizontal lift of this operator to $\Zh$ will be also denoted by $\mce$.
\begin{prop}\label{prop:Jacobicommute}
Let $\mcs$ be the spectral symmetry \eqref{spectralsymmetry} specified by Eq.~\eqref{eq:spectralsymmetry}. When considered as operators acting on the scalar functions $C^{\infty}(\xinfh)$, we have the commutation relations
\begin{align}\label{eq:Jacobicommute}
\quad&a)\quad [\mcs, h_2^{-\frac{1}{2}}\delx]^{\tn{h}}=+ h_2^{-\frac{1}{2}}\delx,\\
\qquad&\;\;\;\quad[\mcs, \hb_2^{-\frac{1}{2}}\delxb]^{\tn{h}}=-\hb_2^{-\frac{1}{2}}\delxb,\n\\
\quad&b)\quad [\mcs, \mce ]=0.\n
\end{align}
\end{prop}
\begin{proof}
b) Since $\mcs(h_2\hb_2)=0$, $\mcs$ commutes with the scalar multiplication operator by $h_2\hb_2$. Hence b) follows from a).
\begin{align}
\mcs \circ (h_2^{-\frac{1}{2}}\delx) \circ (\hb_2^{-\frac{1}{2}}\delxb)
&=\left( (h_2^{-\frac{1}{2}}\delx) \circ\mcs + h_2^{-\frac{1}{2}}\delx\right)\circ (\hb_2^{-\frac{1}{2}}\delxb)\n\\ 
&=(h_2^{-\frac{1}{2}}\delx) \circ\left( (\hb_2^{-\frac{1}{2}}\delxb)\circ\mcs-\hb_2^{-\frac{1}{2}}\delxb\right)
+(h_2^{-\frac{1}{2}}\delx)\circ (\hb_2^{-\frac{1}{2}}\delxb)\n\\
&=(h_2^{-\frac{1}{2}}\delx) \circ (\hb_2^{-\frac{1}{2}}\delxb)\circ\mcs.\n
\end{align}
Since $[\mcs,  h_2\hb_2]=0$, this implies 
$$[\mcs, \delx\delxb]=0.$$

a) Let
\begin{align}\label{eq:commutevector}
\Big([\mcs, h_2^{-\frac{1}{2}}\delx]- h_2^{-\frac{1}{2}}\delx\Big)^{\tn{h}}
&=U \\
&=u_{\xi}\delx+u_{\xib}\delxb+u_0E_0+\sum_{j\geq 1}u_jE_j+u_{\bar{j}}E_{\bar{j}}\n
\end{align}
be the $\Pi$-horizontal part of the vector field associated with the commutator.
By construction and the weighted homogeneous property of $T_j=\delxb h_j$, this derivation (vector field) annihilates $h_2\hb_2, z_j, \zb_j.$ Applying to $z_j$  we get
\be\label{eq:Ucompatibility}
u_j+(-\frac{j}{2}h_2^{-1}h_j)u_2+(-\frac{j}{2}h_2^{-1}h_3h_j+h_{j+1})u_{\xi}+ T_j u_{\xib}=0.
\ee

On the other hand from Lemma \ref{lem:shadowdelx} we have $ u_j, u_{\bar{j}}=0$ for all $ j\geq 1.$ Since Eq.~\eqref{eq:Ucompatibility} holds for all $j\geq 3$, this forces $u_{\xi}, u_{\xib}=0$.
\end{proof}
  
\subsection{Homogeneous decomposition}\label{sec:homogeneousdecomp}
Note the following homogeneous objects with weights under the spectral symmetry.
\[\begin{array}{rcrl}
&\vline&\tn{spectral weight}&\\
\hline
h_2^{-\frac{1}{2}}\delx, \omb &\vline& +1& \\
\hb_2^{-\frac{1}{2}}\delxb, \omega &\vline& -1& \\
v_0, h_2\hb_2&\vline& 0&\\
z_j,\,\zeta_j&\vline&+(j-2)&\tn{for}\;j\geq 3 \\
\zb_j,\,\zetab_j&\vline&-(j-2)&\tn{for}\;j\geq 3
\end{array}\]

\one\noi
Since Lie derivative commutes with the exterior derivative, these weights are invariant under the exterior differentiation. The commutation relation \eqref{eq:Jacobicommute} implies that the Jacobi operator $\mce$ also preserves the weight.

\two
We record a useful lemma. It shows that a homogeneous conservation law corresponds to a homogeneous Jacobi field.
\begin{lem}\label{lem:hrepJacobi}
Suppose $\varphi=f\omega+g\omb\in\Omega^1(\xinfh_*)$ be a conservation law which is homogeneous of weight $m$ under the spectral symmetry such that
$$\mcl_{\mcs}\varphi\equiv m\varphi\mod\mcjh.$$
Let $\Phi_u\in\Cv{(\infty)}$ be the corresponding differentiated conservation law generated by a Jacobi field $u$.
Then $u$ is homogeneous of weight $m+2$ under the spectral symmetry.
\end{lem}
\begin{proof}
We shall make use of the classification of Jacobi fields to be proved in Thm.~\ref{thm:higherNoether}:
the space of Jacobi fields is generated by the weighted homogeneous higher-order Jacobi fields of distinct odd spectral weights and the classical Jacobi fields. 

Let $u=\sum_{i=0}^k u_i$, where $u_i, i\geq 1,$ is higher-order and $u_0$ is classical.
By definition there exist a 1-form $\Theta\in\iinf$ such that
\be\label{eq:homcv}
\ed\varphi+\ed\Theta=\sum_{j=0}^{k} \Phi_{u_j}.
\ee
By Cor.~\ref{cor:zurelation} and \cite{Fox2011}, each $\Phi_{u_i}, i\geq 1$, is weighted homogeneous of odd spectral weight, say $w_i$.

Suppose the highest order (nonzero) Jacobi field $u_k=z_{2k+3}+\, ... \, $ so that up to constant scale the highest weight terms of $\Phi_{u_k}$ are
$$\Phi_{u_k} =\, ... \,+ \zeta_2\w\zeta_{2k+1}- \zeta_3\w\zeta_{2k}+\, ... \, .$$
Applying the spectral symmetry to Eq.~\eqref{eq:homcv}, we get 
\begin{align*}
m \ed\varphi+\ed\Theta'&\equiv\, ... \,+(2k-1) \zeta_2\w\zeta_{2k+1}- (2k-1)\zeta_3\w\zeta_{2k}+\, ... \, \\
&\;+\sum_{j=1}^{k-1} w_i \Phi_{u_j} +\mcl_{\mcs}\Phi_{u_0}\mod\mcjh^{\pi}
\end{align*}
for a 1-form $\Theta'\in\mcjh$.
Subtracting $m\cdot$Eq.~\eqref{eq:homcv} from this,
\begin{align*}
\ed\Theta''&\equiv\, ... \,+((2k-1)-m) \zeta_2\w\zeta_{2k+1}- ((2k-1)-m)\zeta_3\w\zeta_{2k}+\, ... \, \\
&\;+\sum_{j=1}^{k-1} (w_i-m) \Phi_{u_j} +(\mcl_{\mcs}-m)\Phi_{u_0}\mod\mcjh^{\pi}
\end{align*}
for a 1-form $\Theta''\in\mcjh$.

It is easily checked from the structure equation that this forces $m=2k-1$, and, since $w_i<m$,  $u_i$'s are trivial for $k-1\geq i \geq 1$. The equation is now reduced to
$$\ed\Theta''=(\mcl_{\mcs}-m)\Phi_{u_0}\mod\mcjh^{\pi}$$
for a classical conservation law $\Phi_{u_0}.$
One may check by direct computation that this implies $\Phi_{u_0}=0$ and $u_0$ is also trivial. 
\end{proof}
   
\subsection{Extension to affine Kac-Moody algebra}\label{sec:KacMoody}
The original EDS for CMC surfaces is defined on the unit tangent bundle $X\to M$ of the 3-dimensional space form $M$. Recall $\g$ is the Lie algebra of the group of isometries of $M$. As a consequence, the structure 1-form $\psi$ of a CMC surface is $\g$-valued. The main idea of the analysis in the previous sections is that by inserting the spectral parameter $\lambda$ the 1-form $\psi$ is extended to an associated  loop algebra valued 1-form $\psi_{\lambda}$, which remains compatible and satisfies the Maurer-Cartan equation. The naturally attached formal Killing field led to an enhanced prolongation and the associated formal Killing field coefficients. From this we shall extract several useful conclusions on the structure of the Jacobi fields and conservation laws later on. 

We introduce in this section another extension of $\psi_{\lambda}$ to a compatible affine Kac-Moody algebra valued 1-form.  Let us comment on a practical implication of this extension. So far one of the most useful tools for generating solutions of integrable PDE's is dressing transformation, \cite{Terng2000}. Based on the generally transcendental process of loop group factorization, the notion of dressing captures effectively the possible hidden symmetry structure on the moduli space of solutions. In this perspective, the extension to affine algebra itself represents an extended symmetry of the EDS for CMC surfaces and at the same time it may provide a setting for generating wider class of special solutions.

In addition,  we suspect that this extension will also play an important role in connecting the theory of CMC surfaces with the related  conformal field theory in physics, \cite{Babelon2007}. 

\two
Assume the structure constant $\gamma^2>0$. Recall $\g^{\C}=\so(4,\C)$, and the twisted based loop algebra $\mcl^{\sigma,\tau}(\g^{\C})$, \eqref{defn:twistedloop}.
For a CMC surface $\Sigma$, the formal Killing field $\tb{X}_{\lambda}$ whose components are given by $a^{2j+1},b^{2j},c^{2j}$'s (see Sec.~\ref{sec:recursion}) induces a covariant constant map
\be\label{eq:loopalgebramap}
\tb{X}_{\lambda}:\Sigmah\to\mcl^{\sigma,\tau}(\g^{\C}).
\ee
Note from the construction that the $\lambda^0$-th component $\tn{X}_{0}$ of $\tb{X}_{\lambda}$ is set to zero and 
$$\tb{X}_{\lambda}=...\;+\lambda^{-1}\tn{X}_{-1}+\lambda^{0}\tn{X}_{0}+\lambda^{1}\tn{X}_{1}+\;...$$
where the first few terms are
\begin{align}
\tn{X}_0&=0,\n\\
\tn{X}_1&= 
\bp
\cdot&-c^2&-\im c^2&\cdot\\
c^2&\cdot&\cdot&b^2\\
\im c^2&\cdot&\cdot&-\im b^2\\
\cdot&-b^2&\im b^2&\cdot
\ep,\n\\
\tn{X}_{-1}&=\ol{\tn{X}}_1.\n
\end{align}

The affine Kac-Moody algebra $\widehat{\mcl^{\sigma,\tau}(\g^{\C})}$  associated with $\mcl^{\sigma,\tau}(\g^{\C})$ is obtained from the twisted loop algebra by a central extension and a semi-direct product by the derivation $\lambda\frac{\partial}{\partial\lambda}$. As a vector space, we set
$$\widehat{\mcl^{\sigma,\tau}(\g^{\C})}\simeq \C d\oplus\mcl^{\sigma,\tau}(\g^{\C})\oplus\C c,$$  
where the element $d=\lambda\frac{\partial}{\partial\lambda}$ is the weight derivation and the element $c$ is a center. For $x,y\in\g^{\C}$ the Lie bracket is given by
\begin{align}
[d, x\lambda^n]&=nx\lambda^n, \n\\
[x\lambda^n,y\lambda^m]&=[x,y]\lambda^{n+m}+\langle x,y\rangle n\delta_{n+m,0}c,\n
\end{align}
while all other brackets are trivial. Here $\delta_{n+m,0}$ is Dirac delta function, and $\langle x,y\rangle=\tn{tr}(xy^\top)=-\tn{tr}(xy)$ is the Killing form of $\g^{\C}$.

\two
We wish to show that the formal Killing field ${\tb{X}}_{\lambda}$ extends to a $\widehat{\mcl^{\sigma,\tau}(\g^{\C})}$-valued affine Killing field $\widehat{\tb{X}}_{\lambda}$. Consider the ansatz
\begin{align}\label{eq:affinealgebramap}
\widehat{\tb{X}}_{\lambda}&:=(t, \tb{X}_{\lambda}+t\tn{X}_0, 4\im\gamma u_0):\widetilde{\Sigmah}\to\widehat{\mcl^{\sigma,\tau}(\g^{\C})},  \\
\ed \widehat{\tb{X}}_{\lambda}&+[(0,\psi_{\lambda},-4\im\sigma),\widehat{\tb{X}}_{\lambda}]=0,\n
\end{align}
(note that we add the $\lambda_0$-th component $t\tn{X}_0$ to $ \tb{X}_{\lambda}$).
Here the center component 1-form $\sigma$ should give
\be \label{eq:positivefield}
\ed\sigma=(\gamma^2+h_2\hb_2)\frac{\im}{2}\xi\w\xib 
\ee
so that the $\widehat{\mcl^{\sigma,\tau}(\g^{\C})}$-valued 1-form $(0,\psi_{\lambda},\sigma)$ satisfies the Maurer-Cartan structure equation
$$ \ed(0,\psi_{\lambda},\sigma)+\frac{1}{2}[(0,\psi_{\lambda},\sigma),(0,\psi_{\lambda},\sigma)]=0.$$

\one
It is straightforward to check from the construction that such $\widehat{\tb{X}}_{\lambda}$ is covariant constant when the components satisfy the following relations.
\begin{align}
t&\;\;\;\tn{is a constant},\n\\
\ed u_0&=h_2^{\frac{1}{2}}\xi+\hb_2^{\frac{1}{2}}\xib,\label{eq:centralfield}\\
\tn{X}_0&=
\bp
\cdot&\cdot&\cdot&\cdot\\
\cdot&\cdot&+1&\cdot\\
\cdot&-1&\cdot&\cdot\\
\cdot&\cdot&\cdot&\cdot
\ep.\n 
\end{align}

The relevant question is whether this extension can be canonically defined without explicit reference to the underlying CMC surface. Lem.~\ref{lem:phi0onZ} shows that the equation \eqref{eq:centralfield} for $u_0$ cannot be solved on $\Zh$. From this we propose, for now, the following definitions restricted to an individual CMC surface.
\begin{defn}
Let $\Sigma\hook \xinf$ be a CMC surface.
Let $\Sigmah\hook\xinfh$ be its double cover defined by the square root $\omega=\sqrt{\ff}.$
A \tb{central field} is a scalar function $u_0\in C^{\infty}(\Sigmah)$ which satisfies 
\be\label{eq:centralfielddefi}
\ed u_0= \omega+\omb.
\ee
A \tb{central connection} is a  (singular) 1-form $\sigma\in\Omega^1(\Sigmah)$ which satisfies
\be\label{eq:centralconnectiondefi}
\ed\sigma= (\gamma^2+h_2\hb_2)\frac{\im}{2}\xi\w\xib.
\ee
A \tb{central potential} is a (singular) scalar function $q$ on $\Sigmah$ which satisfies
\be\label{eq:centralpotnetialdefi}
q_{\xi\xib}+\frac{1}{2}(\gamma^2+h_2\hb_2)=0
\ee
so that it gives rise to a central connection
$$\sigma_q= \frac{\im}{2}(q_{\xi}\xi-q_{\xib}\xib).$$
\end{defn}
  
We proceed to the question of existence for central field and central connection.

\subsubsection{Central field $u_0$}
Recall that the 1-form 
$$2\tn{Re}(\omega)=h_2^{\frac{1}{2}}\xi+\hb_2^{\frac{1}{2}}\xib$$ 
represents a conservation law on $\Xh{1}$.  
Lem.~\ref{lem:phi0onZ} shows that one needs another integrable (Abelian) extension over $\Zh$ to accommodate the anti-derivative for $2\tn{Re}(\omega)$. 

In any case, given a CMC surface the central field $u_0$ is well defined up to constant and up to the period of the conservation law $2\tn{Re}(\omega)$.
\subsubsection{Central potential $q$}
Let us argue for simplicity on a given compact CMC surface. Observe that
$$\Big(\log(h_2\hb_2)\Big)_{\xi\xib}=-2(\gamma^2-h_2\hb_2).$$
It follows from the defining equation of the central potential 
$$q_{\xi\xib}=-\frac{1}{2}(\gamma^2+h_2\hb_2) 
$$ that roughly one expects  
$$ q\sim c \log(h_2\hb_2)\mod \tn{(smooth functions)}
$$
for a constant $c$ locally near each umbilic.

Assume that the given CMC surface $\Sigma$ is compact and $\tn{genus}(\Sigma)\geq 2$. By Gau\ss-Bonnet theorem, there exists a unique function $q^0\in C^{\infty}(\Sigma)$ which satisfies
\begin{align}
\int_{\Sigma} q^0\,\frac{\im}{2}\xi\w\xib&=0,\n\\
q^0_{\xi\xib}&=\left(\gamma^2-\frac{2\pi\chi(\Sigma)}{\tn{Area($\Sigma$)}}\right)-h_2\hb_2.\n
\end{align}
Here $\chi(\Sigma)$ is the Euler characteristic of $\Sigma$ and $\tn{Area($\Sigma$)}=\int_{\Sigma}\frac{\im}{2}\xi\w\xi.$
The (singular) central potential is then given by
$$ q= \frac{\tn{Area($\Sigma$)}}{2\pi\chi(\Sigma)} \left(\frac{1+\gamma^2}{2} q^0
+ \frac{ (1+\gamma^2)-\frac{2\pi\chi(\Sigma)}{\tn{Area($\Sigma$)}}}{4} \log(h_2\hb_2)\right). 
$$
Note that the associated central connection $\sigma= \frac{\im}{2}(q_{\xi}\xi-q_{\xib}\xib)$ has only simple pole type singularities at the umbilics.


\section{Recursion for Jacobi fields on $\Xh{\infty}_*$}\label{sec:recursion}
In this section we resume the analysis of formal Killing field equation in Sec.~\ref{sec:primitivemap},  from Remark \ref{rem:babyrecursion}. One of our main objectives is to obtain a recursion relation for higher-order Jacobi fields on $\Xh{\infty}_*$ that can be used to prove their very existence.  We will derive the recursion from the symmetry of the formal Killing field structure equation, guided by the goal of having the structure equations for functions that are defined on $\Xh{\infty}_*$. 
As a result we will be able to simply read off the infinite sequence of higher-order Jacobi fields and conservation laws from the structure equation for the enhanced prolongation.
A similar local construction was given for the Toda field equation associated to $\SU(3)/\SO(2)$ in \cite{Fox2012}.

\subsection{Motivation}\label{sec:motivation}
Recall the index notation \eqref{eq:weightindex} according to the powers of the spectral parameter $\lambda$. Consider the second, and third equation of Eqs.~\eqref{eq:formalKilling_n},
\begin{align}\label{eq:recursionpair}
&\ed  b^n  -\im  b^n \rho+\frac{\im \gamma }{2} e^n \xi +\frac{\im}{2}  \hb_{2}a^n \xib  \equiv 0,\\
&\ed c^n +\im c^n \rho+\frac{\im}{2}  h_2 e^n  \xi+\frac{ \im \gamma }{2} a^n \xib \equiv 0\mod\;\iinfh,\n
\end{align}
where $a^n,\,b^n,\,c^n,\,e^n$ are now considered as functions on $\Fh{\infty}$.
Following the formal integrability of the Killing fields equation \eqref{eq:KillingEquation}, suppose for a given Jacobi field $a^n$ there exist $b^n, \,c^n$ which satisfy these equations (they are unique up to addition by constant multiple of $h_2^{-\frac{1}{2}}, \,h_2^{\frac{1}{2}}$ respectively).

Let us write
\begin{align}\label{eq:formalinversedb}
\delxb b^n&=- \frac{\im\hb_2}{2}  a^n,  \\
\delxb c^n&=- \frac{\im\gamma}{2} a^n.\n
\end{align}
Then from Eq.~\eqref{eq:recursionpair} we have the recursive relation for Jacobi fields
\begin{align}\label{eq:formalrecursion}
e^n&=-\frac{2}{\im\gamma}\delx b^n,  \\
      &= -\frac{2}{\im h_2}\delx c^n.\n
\end{align}
Note that the ambiguity of solutions $b^n,\,c^n$ mentioned above corresponds to addition by constant multiples of the Jacobi field $z_3$ to $\,e^n$.
\begin{lem}\label{lem:bncn}
Let $a^n$ be a Jacobi field on $\xinfh$. Consider the first order differential equations for functions $b^n,\,c^n$ on $\Fh{\infty}$.
\begin{align}\label{eq:bncn}
&\ed  b^n  -\im  b^n \rho+\frac{\im \gamma }{2} e^n \xi +\frac{\im}{2}  \hb_{2}a^n \xib  \equiv 0,\\
&\ed c^n +\im c^n \rho+\frac{\im}{2}  h_2 e^n  \xi+\frac{ \im \gamma }{2} a^n \xib \equiv 0 \mod\;\iinfh,\n 
\end{align}
where
$$\im\gamma e^n=h_2^{-1}(-\im\delx^2a^n+h_3 b^n)
= h_2^{-2} (-\im  h_2 \delx^2a^n +\im h_3 \delx a^n - \gamma h_3 c^n).$$
These equations are formally compatible separately, i.e., $\ed^2\equiv 0\mod\;\iinfh$ is a formal identity for each equation.

Suppose, away from the locus $h_2=0$, there exists a solution to these equations. Then the scalar function $e^n$  is a Jacobi field.
\end{lem}
\begin{proof}
Let us examine the equation for $\,b^n$ only. Differentiating the given equation $e^n=-\frac{2}{\im\gamma}\delx b^n$, we have
\begin{align*}
\delxb e^n&=-\frac{2}{\im\gamma}\delxb\delx b^n\\
              &=-\frac{2}{\im\gamma}(\delx \delxb b^n -\frac{R}{2}b^n)\\
              &=-\frac{2}{\im\gamma}\left(\delx (-\frac{\im \hb_2}{2} a^n) -\frac{R}{2}b^n \right).
\end{align*}
Since $e^n$ is supposed to be a Jacobi field, the Jacobi equation $\mce(e^n)=0$ gives after simplification
\[ \im\gamma e^n=h_2^{-1}(-\im\delx^2a^n+h_3 b^n).
\]

For the compatibility of the closed differential equation for $b^n$, it suffices to check
$\delxb( \im\gamma e^n)-\delx( \im\hb_2  a^n)=R b^n.$ This follows similarly from repeated application of the identity \eqref{3anticommut}.
\end{proof}
We should remark here that although the formal integrability of Eq.~\eqref{eq:bncn} is a necessary condition for the solvability of these equations, it is nevertheless not a sufficient condition. In fact the space of un-differentiated conservation laws is the obstruction space. The recursion process relies essentially on the fact that the obstruction to solving Eq.~\eqref{eq:bncn} vanishes when $a^n$ is a Jacobi field.

\two
In order to keep track of jet orders, let us rearrange the index notation for the coefficients of formal Killing field to
\[ \{  a^0, b^0, c^0, e^0, f^0, g^0, \, ...\} \to
\{a^{1}, -b^{2}, -c^{2}, a^{3}, -b^{4}, c^{4},\,... \}.
\]
The structure equation for the enhanced prolongation with the new index notation is given below 
for positive index $n\geq 1.$
\begin{align}
a^{1}&= 0,\, b^2=-\im\gamma h_2^{-\frac{1}{2}},\,c^2=\im h_2^{\frac{1}{2}},\n\\
a^3   &= h_2^{-\frac{3}{2}}h_3,\n\\
\ed  a^{2n-1}&\equiv  (\im \gamma c^{2n}+\im h_2 b^{2n})\xi +(\im \gamma  b^{2n-2}+\im \hb_{2} c^{2n-2})\xib,\label{eq:newrecursion}\\
\ed  b^{2n} - \im  b^{2n} \rho&\equiv  \frac{\im \gamma }{2} a^{2n+1} \xi +\frac{\im}{2}  \hb_{2}a^{2n-1} \xib,\n\\
\ed  c^{2n} +\im c^{2n}  \rho&\equiv \frac{\im}{2}  h_2 a^{2n+1}  \xi +\frac{ \im \gamma }{2} a^{2n-1} \xib,\n\\
\ed  a^{2n+1}&\equiv  (\im \gamma c^{2n+2}+\im h_2 b^{2n+2})\xi + (\im \gamma  b^{2n}+\im \hb_{2} c^{2n})\xib,\n\\
\ed  b^{2n+2} - \im  b^{2n+2} \rho&\equiv \frac{\im \gamma }{2} a^{2n+3} \xi + \frac{\im}{2}  \hb_{2}a^{2n+1} \xib ,\n\\
\ed  c^{2n+2} +\im c^{2n+2}  \rho&\equiv \frac{\im}{2}  h_2 a^{2n+3}\xi  + \frac{ \im \gamma }{2} a^{2n+1} \xib, \mod\;\iinf.\n
\end{align}
The case for $n\geq 1$ suffices for our purpose, and we shall not restate the reality condition \eqref{eq:newreality}. The first few terms  are recorded in the table.

Consider the following schematic diagram.

\be\label{recursiondiagram}
\begin{split}
\xymatrix{  & b^{2n} \ar[ld]_{\frac{2}{\im\hb_2}\delxb} \ar[rd]^{\frac{2}{\im\gamma}\delx} &
                 & b^{2n+2} \ar[ld]_{\frac{2}{\im\hb_2}\delxb} \ar[rd]^{\frac{2}{\im\gamma}\delx}  & \\
           ...\quad   a^{2n-1}  & & a^{2n+1} & & a^{2n+3}\quad...\\
                    & c^{2n}\ar[lu]^{\frac{2}{\im\gamma}\delxb}\ar[ru]_{\frac{2}{\im h_2}\delx}  &
                    & c^{2n+2}\ar[lu]^{\frac{2}{\im\gamma}\delxb}\ar[ru]_{\frac{2}{\im h_2}\delx}   &   }
\end{split}
\ee
In view of  Lem.~\ref{lem:bncn} the middle $a^{2n+1}$-sequence is a sequence of Jacobi fields,
and Eq.~\eqref{eq:bncn} is the recursion relation to be solved to generate this sequence.
\begin{figure}
\be\label{eq:recursioncoeff}\begin{array}{ccl}\n
 &\vline \\
\hline
a^1 &\vline&  0\\
\hline
b^2 &\vline&  -\im\gamma h_2^{-\frac{1}{2}}  \\
\hline
c^2 &\vline&   \im h_2^{\frac{1}{2}} \\
\hline
a^3 &\vline& z_3 \\
\hline
b^4 &\vline& -\frac{\im}{2} h_2^{-\frac{1}{2}}(z_4-\frac{5}{4}z_3^2)\\
\hline
c^4 &\vline&  -\frac{\im}{2\gamma} h_2^{\frac{1}{2}}(z_4-\frac{7}{4}z_3^2)\\
\hline
a^5 &\vline& -\frac{1}{\gamma}(z_5-5 z_4z_3 +\frac{35}{8}z_3^3 )\\
\hline
b^6 &\vline& {\frac {\im}{2\gamma}} h_2^{-\frac{1}{2}}\left(z_{{6}}
-7z_{{5}}z_{{3}}
-\frac{21}{4}z_{{4}}^{2}
+\frac{231}{8}z_{{4}}z_{{3}}^{2}
-\frac{1155}{64}z_{{3}}^{4}  \right)\\
\hline
c^6 &\vline& {\frac {\im}{2\gamma^2}} h_2^{\frac{1}{2}}\left(z_{{6}}
-8 z_{{5}}z_{{3}}
-\frac{19}{4} z_{{4}}^{2}
+\frac{259}{8} z_{{4}}z_{{3}}^{2}
-\frac{1365}{64} z_{{3}}^{4}  \right)
\end{array}
\ee
\caption{Formal Killing field coefficients}
\end{figure}
\subsection{Recursion on $\Xh{\infty}_*$}\label{sec:recursionxinf}
In this section we give the details of the recursion relation described in Lemma \ref{lem:bncn}. Define the polynomial ring
\[\mcr:=\C[z_3,z_4,\, ... \;].
\]
The recursion will produce a sequence of Jacobi fields which are weighted homogeneous polynomials in $\mcr$ with the spectral weights given in Sec.~\ref{sec:homogeneousdecomp}.  

\two
Let us start by re-writing Eqs.~\eqref{eq:bncn} in the new index notation.
Let $a^{2n+1}$ be a Jacobi field. We wish to solve the first order differential equations for functions $\,b^{2n+2},\,c^{2n+2}$ on $\Fh{\infty}$;
\begin{align}\label{eq:b2n+2}
&\ed b^{2n+2} -\im  b^{2n+2} \rho  \equiv \frac{\im \gamma }{2} a^{2n+3} \xi +\frac{\im}{2}  \hb_{2}a^{2n+1} \xib  \equiv 0,\\
&\ed c^{2n+2} +\im c^{2n+2} \rho  \equiv \frac{\im}{2}  h_2 a^{2n+3}  \xi+\frac{ \im \gamma }{2} a^{2n+1} \xib \equiv 0, \mod\;\iinf,\n 
\end{align}
where
\begin{align}
a^{2n+3}&=\frac{1}{\gamma h_2}\Big(- \delx^2a^{2n+1}+\im h_3 b^{2n+2}\Big),\n\\
&= \frac{1}{\gamma h_2}\Big(- \delx^2a^{2n+1}+  \frac{h_3}{ h_2} (\delx a^{2n+1}-\im\gamma c^{2n+2})\Big).\n
\end{align}
These equations suggest to consider the scaled function $\hat{b}^{2n+2}=h_2^{\frac{1}{2}}b^{2n+2}$, which is well defined on $\xinfh_*$.  
Set
\[ 2\im \Upsilon(a^{2n+1}):= h_2^{-\frac{1}{2}}(\delx^2 a^{2n+1})\xi - h_2^{\frac{1}{2}}\hb_2 a^{2n+1}\xib.
\]
Then Eq.~\eqref{eq:b2n+2} on $\Fh{\infty}$ is equivalent to  the following equation on $\xinfh_*$.
\be\label{eq:bh2n+2}
\ed \hat{b}^{2n+2}\equiv \Upsilon(a^{2n+1})\mod \Ih{\infty}.
\ee

Note from Lemma \ref{lem:bncn} that $\ed\Upsilon(a^{2n+1})\equiv 0 \mod \Ih{\infty}$, and $\Upsilon(a^{2n+1})$ defines an un-differentiated conservation law. Our claim is that this is trivial in cohomology.
\begin{lem}\label{lem:recursionscheme}
Let $a^{2n+1}\in\mcr$ be a nontrivial Jacobi field which is weighted homogeneous of odd spectral weight $2n-1$. From the analysis of Sec.~\ref{sec:Jacobifields}, up to scaling we may assume (linear in the highest weight term  $z_{2n+1}$)
\be\label{eq:a2n+1scale}
 a^{2n+1}=(-\frac{1}{\gamma})^{n-1} z_{2n+1}+\tn{(lower order terms)}.
\ee
Then the associated conservation law $[\Upsilon(a^{2n+1})]$ is trivial and there exists a function  $\hat{b}^{2n+2}\in\mcr$ which solves Eq.~\eqref{eq:bh2n+2}.
\end{lem}
\begin{proof}
Given $[\Upsilon(a^{2n+1})]$, let $\Phi_u$ be the corresponding differentiated conservation law with the generating Jacobi field $u$. From the weighted homogeneity of $\Upsilon(a^{2n+1})$ and by Lem.~\ref{lem:hrepJacobi}, $u\in\mcr$ is also weighted homogeneous and it is a possibly trivial Jacobi field of even spectral weight $2n+2$.  From the classification of Jacobi fields, Thm.~\ref{thm:higherNoether}, such $u$ must be trivial.

The existence of an anti-derivative  $\hat{b}^{2n}\in\mcr$ will be proved in the next section, Prop.~\ref{prop:inductiveformula} and  Thm.~\ref{thm:abcnormalform}.
\end{proof}
The differential algebraic recursion formulae to be defined in the next section, which does not involve solving a differential equation, will in fact show that the following refinement of this lemma is true.
\begin{lem}
Let $a^{2n+1}\in\mcr$ be a Jacobi field as in Eq.~\eqref{eq:a2n+1scale} which is weighted homogeneous of odd spectral weight $2n-1$. Let $\hat{b}^{2n+2}$ be a solution to the recursive equation Eq.~\eqref{eq:bh2n+2}.
\begin{enumerate}[\qquad a)]
\item we may take $\hat{b}^{2n+2}\in\mcr$ such that it is weighted homogeneous of even spectral weight $2n$, and with the normal form
\[\hat{b}^{2n+2}=\frac{\im(-1)^n}{2\gamma^{n-1}} z_{2n+2}+\tn{(lower order terms)}.
\]
\item the new Jacobi field $a^{2n+3}\in\mcr$ is weighted homogeneous of odd spectral weight $2n+1$,
and with the normal form
$$a^{2n+3}=(-\frac{1}{\gamma})^n z_{2n+3}+\tn{(lower order terms)}.$$
\end{enumerate}
\end{lem}
See the analysis of next section for the proof. 
A similar statement is true for the coefficient $c^{2n+2}$, and shall be omitted.

\two
Before we proceed to the differential algebraic inductive formulae for the formal Killing coefficients, let us give a description of the recursion relation in terms of the adapted $\omega$-frame and $z_j, \zb_j$ coordinates.  

For a scalar function $A:\Xh{\infty}_* \to \C$, define the covariant derivative 
\[
\ed A \equiv A_{\omega}\omega + A_{\omb}\omb \; \mod \Ih{\infty}
\]
where, for reference, we note that $A_{\omega}=h_1^{-1}A_{\xi}$.  
From the structure equation \eqref{eq:zetastruct},
$$\mce(A)=r A_{\omega, \omb}+\frac{1}{2} (\gamma^2 + r^2) A.$$

Note in passing,
\begin{lem}\label{lem:mcez}
$$\mce(z_j)=\hat{T}_{j+1}-\frac{j}{2}(\hat{T}_3z_j+\hat{T}_jz_3)+\frac{1}{2}(\gamma^2+r^2)z_j.
$$
It follows that
$$\mce(\mcr)\subset \mcr\oplus (h_2\hb_2)\mcr.
$$
\end{lem}\noi
We remark that the various characteristic differential algebraic properties of the CMC system can be attributed to this and Lem.~\ref{lem:Tjh}.

\two
The Jacobi equation $\mce(A)=0$ is now written as
\begin{equation}
A_{\omega, \omb}+\frac{1}{2} (\gamma^2 + r^2) r^{-1}A=0.
\end{equation}
The recursion for Jacobi field is expressed in this setting as follows.
\begin{defn}[Jacobi field Recursion]\label{defn:Jacobirecursion}
Define the recursion operator
\[
\mcp:\ker(\mce) \cap \mcr \to \ker(\mce) \cap \mcr
\]
as follows. Given a Jacobi field $A:\Xh{\infty}_* \to \C$ in the kernel of $\mce$, define
\begin{equation}\label{eq:upsilon}
\Upsilon(A)= \left(A_{\omega, \omega}+\frac{1}{2} z_3 A_{\omega}\right) \omega + \gamma^2 r^{-1}A \omb
\end{equation}
and a new function $C\in\mcr$  by the equation
$$\ed C \equiv \Upsilon(A)  \mod \Ih{\infty},$$  
and the condition that $C(0)=0$ (as a polynomial in $z_j$).  Then define  
\begin{equation}\label{eq:newA}
\mcp(A):=A'=A_{\omega, \omega} - \frac{1}{2}z_3\left(  A_{\omega} + C\right).
\end{equation}
\end{defn}
The above lemma's show that
\begin{prop}
The operator $\mcp$ is well-defined.
\end{prop}
It follows that starting from the initial Jacobi field $A^1=z_3 \in \mcr$, repeatedly applying $\mcp$ creates an infinite sequence of independent, higher-order Jacobi fields which are weighted homogeneous polynomials in $\mcr$.
\begin{rem}
In \cite{Fox2012} it was shown that, for the Toda field system associated to the $6$-symmetric space $\SU(3)/\SO(2)$, one can derive a pair of recursions for Jacobi fields, $\mcp$ and $\mcn$, that are almost inverses of each.  For that system it seems essential to use the pair of them in order to obtain vanishing results needed for a complete characterization of the conservation laws.  For the Toda field equation associated to $\SO(4)/\SO(2)$, the single recursion $\mcp$ suffices.  For this reason we will not introduce a second recursion, $\mcn$, in the present article as currently it seems unnecessary.  However the reader may want to be aware of its existence.
\end{rem}

Here are the first few higher-order Jacobi fields on $\Xh{\infty}_*$:
\begin{align*}
A^3 & =z_3,\\
A^5&  =z_5 - 5 z_3 z_4   +\frac{35}{8}z_3^3,\\
A^7 & =z_{{7}}-\frac{21}{2}\,z_{{3}} z_{{6}} -{\frac {35}{2}}\,z_{{4}} z_{{5}} 
+{\frac {483}{8}}\,z_{{3}}^{2} z_{{5}} +{\frac {651}{8}}\,z_{{3}}z_{{4}}^{2} 
-231\,z_{{3}}^{3} z_{{4}} +{\frac {15015}{128}}\,z_{{3}}^{5},\\
A^9 &=z_{{9}}-18\,z_{{3}}z_{{8}}-42\,z_{{4}}z_{{7}}+{\frac {1419}{8}}\, z_{{3}}^{2}z_{{7}}-63\,z_{{5}}z_{{6}}
+{\frac {2871}{4}}\,z_{{3}}z_{{4}} z_{{6}}  -{\frac {19305}{16}}\,z_{{3}}^{3}z_{{6}} \n\\
&\quad+{\frac {3597}{8}}\,z_{{3}}z_{{5}}^{2}+{\frac {4851}{8}}\,z_{{4}}^{2} z_{{5}} 
-{\frac {98241}{16}}\,z_{{3}}^{2}z_{{4}}z_{{5}}  +{\frac {770055}{128}}\,z_{{3}}^{4}z_{{5}}
-{\frac {22165}{8}}\,z_{{3}}z_{{4}}^{3} \n\\  
&\quad+{\frac {1044615}{64}}\,z_{{3}}^{3}z_{{4}}^{2}-{\frac {2807805}{128}}\,z_{{3}}^{5}z_{{4}}
+{\frac {8083075}{1024}}\,z_{{3}}^{7}.\n
\end{align*}

\subsection{Inductive formula}\label{sec:inductiveformula}
The recursion relation for the formal Killing coefficients $\{a^{2n+1}, b^{2n+2}, c^{2n+2}\}_{n=1}^{\infty}$ described in the previous section, which involves integrating a closed 1-form, can be purely differential algebraically solved in terms of an inductive formula originally given in \cite[p419, Proposition 3.1]{Pinkall1989}. For related ideas compare \cite{Burstall1993}  and  \cite{Terng2005}.

\two
Set
\be\label{eq:aij}
\hat{\tn{a}}_{ij}=a^{2i+1}a^{2j+3}-2b^{2i+2}c^{2j+2}-2b^{2j+2}c^{2i+2},\quad 1\leq i\leq j.
\ee
Then a short computation shows that
\[\delxb \hat{\tn{a}}_{ij}=(\delxb a^{2i+1})a^{2j+3}-(\delxb a^{2i+3})a^{2j+1}.
\]
Consider the partial sum
\be\label{eq:sumn}
\hat{\tn{s}}_n=\sum_{\substack{i+j=n\\ i\leq j}}\hat{\tn{a}}_{ij},\quad n\geq 1.
\ee
From the above formula for $\delxb \hat{\tn{a}}_{ij}$ one finds that
\begin{align}
\delxb  \hat{\tn{s}}_n&=(\delxb a^3)a^{2n+1}-(\delxb a^5)a^{2n-1}+(\delxb a^5)a^{2n-1}-\,...\, \n\\
&=\left\{
\begin{array}{l}
(\delxb a^3)a^{2n+1}-(\delxb a^{n+3})a^{n+1}\quad\tn{when $n$ is even},\\
(\delxb a^3)a^{2n+1}-(\delxb a^{n+2})a^{n+2}\quad\tn{when $n$ is odd}.
\end{array}\right.\n
\end{align}

Set accordingly
\be\label{eq:mhatn}
\hat{\tn{m}}_n=
\left\{
\begin{array}{ll}
\hat{\tn{s}}_n+\frac{1}{2}(-\hat{\tn{a}}_{\frac{n}{2},\frac{n}{2}}+a^{n+1}a^{n+3})&\quad\tn{when $n$ is even},\\
\hat{\tn{s}}_n+\frac{1}{2}(a^{n+2})^2&\quad\tn{when $n$ is odd}.
\end{array}\right.
\ee
Then we have
\begin{align}
\delxb \hat{\tn{m}}_n&=(\delxb a^3)a^{2n+1}=(\im\gamma b^2+\im\hb_2 c^2)a^{2n+1}\n\\
&=h_2^{-\frac{1}{2}}(\gamma^2-h_2\hb_2)a^{2n+1}.\n
\end{align}
On the other hand,
\[
-2\im \delxb(\gamma c^{2n+2}-h_2 b^{2n+2})=(\gamma^2-h_2\hb_2)a^{2n+1}.
\]
Since these are weighted homogeneous polynomials by definition, it follows from Lemma \ref{lem:lemma5.4} that
\[\frac{\im}{2}h_2^{\frac{1}{2}}\hat{\tn{m}}_n=\gamma c^{2n+2}-h_2 b^{2n+2}.
\]

From the structure equation \eqref{eq:newrecursion} we also have
\[
-\im \delx a^{2n+1} =\gamma c^{2n+2}+h_2 b^{2n+2}.
\]
As a result we obtain an inductive formula for $c^{2n+2}, b^{2n+2}$:
\begin{align}\label{eq:bc2n+2}
b^{2n+2}&=\frac{\im}{2h_2}\left(- \delx a^{2n+1} -\frac{1}{2}h_2^{\frac{1}{2}}\hat{\tn{m}}_n\right),\\
c^{2n+2}&=\frac{\im}{2\gamma}\left(- \delx a^{2n+1} +\frac{1}{2}h_2^{\frac{1}{2}}\hat{\tn{m}}_n\right).\n
\end{align}
\begin{prop}\label{prop:inductiveformula}
Let $\{a^{2n+1}, b^{2n+2}, c^{2n+2}\}_{n=0}^{\infty}$ be the sequence of coefficients for the formal Killing field equation \eqref{eq:newrecursion}. The coefficients $b^{2n+2}, c^{2n+2}$ admit the inductive formulae \eqref{eq:bc2n+2}. The Jacobi field $a^{2n+3}$ therefore admits the inductive formula
\begin{align}\label{eq:a2n+3}
a^{2n+3}&=-2\im h_2^{-1}\delx c^{2n+2}\\
&=\frac{1}{\gamma h_2}\left(- \delx^2 a^{2n+1}
+\frac{1}{2}\delx(h_2^{\frac{1}{2}}\hat{\tn{m}}_n)\right).\n
\end{align}
\end{prop}
It is clear by construction that each element of the sequence 
$$\{a^{2n+1}, h_2^{\frac{1}{2}}b^{2n+2}, h_2^{-\frac{1}{2}}c^{2n+2}\}_{n=0}^{\infty}$$ 
is a weighted homogeneous polynomial in $\mcr$. Let us summarize and record some of their properties which are relevant for our analysis. It shows that the recursion scheme indeed generates an infinite sequence of distinct higher-order Jacobi fields starting from the initial Jacobi field $a^3=z_3$.
\begin{thm}\label{thm:abcnormalform}
The recursion scheme proposed in Lemma \ref{lem:recursionscheme} works for all  $n\geq 1$.
The formal Killing coefficients $\, a^{2n+1},\,b^{2n+2},\,c^{2n+2}, \, n \geq 1,$ have the following properties: 
\begin{enumerate}[\qquad a)]
\item
they are elements in the ring $\mcr=\C[z_3,\,z_4,\,... \, ]$ of the following form, up to scaling by $\, h_2^{\frac{1}{2}},\,h_2^{-\frac{1}{2}}$ (no $\zb_j$-terms).
\begin{align}
a^{2n+1}&=(-\frac{1}{ \gamma})^{n-1} \left[ z_{2n+1} + (\tn{lower order terms}) \,+t^{2n+1}_a z_3^{2n-1}\right],
\label{eq:abcnormal} \\
b^{2n+2}&= h_2^{-\frac{1}{2}} \frac{\im(-1)^n}{2\gamma^{n-1}} \left[  z_{2n+2} + \,(\tn{lower order terms}) \,+t^{2n+2}_b z_3^{2n}\right],\n\\
c^{2n+2}&= h_2^{\frac{1}{2}} \frac{\im(-1)^n}{2\gamma^{n}} \left[ z_{2n+2}+ \,(\tn{lower order terms}) \,+t^{2n+2}_c z_3^{2n}\right].\n
\end{align}
\item
they are weighted homogeneous with the following spectral weights.
\begin{center}
\begin{tabular}{ccl}
&\vline& \tn{spectral weight}\\
\hline
$a^{2n+1}$ &\vline& $2n-1$\\
$h_2^{\frac{1}{2}}b^{2n+2}$ &\vline& $2n$\\
$h_2^{-\frac{1}{2}}c^{2n+2}$ &\vline& $2n$\\
\end{tabular}
\end{center}
\end{enumerate}
\end{thm}
\begin{proof}
The coefficients of the highest order terms follows from the inductive formula. 
We omit the rest of details.
\end{proof}
Note from the explicit formula \eqref{eq:Tj} for $T_j=\delxb h_j$ that  the coefficients of the polynomials inside the bracket $[\, \cdot\, ]$ in the expressions above for $a^{2n+1}, b^{2n+2}, c^{2n+2}$, including $t^{2n+1}_a, t^{2n+2}_b, t^{2n+2}_c$, are in fact rational numbers. Hence up to scale they are elements of the subring
$$ \mathbb{Q}[z_3,\,z_4,\,... \, ]\subset\mcr.$$
\begin{rem}\label{rem:BCAbca}
Recall the coefficients $\{ \,B^{2n}, C^{2n}, A^{2n+1}\,\}_{n=1}^{\infty}$ for the enhanced prolongation from Sec.~\ref{sec:kfpro}.
They correspond to
\begin{align*}
B^{2n}&=\im h_2^{\frac{1}{2}} b^{2n}, \\
C^{2n}&=\im h_2^{-\frac{1}{2}}c^{2n}, \\
A^{2n+1}&=\frac{1}{2}a^{2n+1}.
\end{align*}
\end{rem}
  
\subsection{Resolution of trivial Jacobi field}\label{sec:resolution}
Recall the structure equation \eqref{eq:classicalJacobi} for the classical Jacobi fields from Sec.~\ref{sec:classicallaws}. When the structure constant $\gamma^2\ne 0$, it takes the following form when written mod $\iinf$.
\be\label{eq:classicaloperator}
\ed \begin{pmatrix} A\\A^1\\A^{\xi}\\B-\delta A \end{pmatrix} +
\im \begin{pmatrix} \cdot \\-A^1\\A^{\xi}\\ \cdot\end{pmatrix}\rho
\equiv\begin{pmatrix}
A^{\xi}+h_2A^1& * \\
\frac{1}{2}(B-\delta A)&-\frac{1}{2}\hb_2A\\
\frac{1}{2}h_2(B-\delta A)&-\frac{\gamma^2}{2}A\\
* & -\gamma^2 A^1-\hb_2 A^{\xi} \end{pmatrix}
\begin{pmatrix} \xi \\ \xib  \end{pmatrix},\mod \iinf.
\ee
This contains the following elliptic first order linear differential equation for the three variables
$(A, \frac{1}{\im} A^1, \frac{1}{\im\gamma} A^{\xi})$.
\be\label{operatorL}
L\bp A\\ \frac{1}{\im} A^1\\ \frac{1}{\im\gamma} A^{\xi} \ep:=
\bp
-\delx &\im h_2&\im\gamma\\
\frac{\im}{2}\hb_2&-\delxb&\cdot\\
\frac{\im\gamma}{2}&\cdot&-\delxb
\ep
\bp A\\ \frac{1}{\im}A^1\\ \frac{1}{\im\gamma} A^{\xi}  \ep =0.
\ee
It is clear that up to scaling by constants and by powers of $f=h_2^{-\frac{1}{2}}$,  the differential operator $L$  is equivalent to the recursion operator $\mcp$ of Definition \ref{defn:Jacobirecursion}.
\begin{lem}\label{lem:resolution}
Suppose $A$ is a Jacobi field. Let $(A,\frac{1}{\im} A^1,\frac{1}{\im\gamma} A^{\xi} )^t$ be an element in  $\ker (L)$. Then  $A^1_{\xi} =h_2^{-1} A^{\xi}_{\xi}$ is a Jacobi field.
\end{lem}
\begin{proof}
From the previous arguments, both $A^1_{\xi}$ and $h_2^{-1} A^{\xi}_{\xi}$ are Jacobi fields. Note that $\delxb(A^1_{\xi} - h_2^{-1} A^{\xi}_{\xi})=0.$ The only holomorphic Jacobi field is $0$, and we must have  $A^1_{\xi}=h_2^{-1} A^{\xi}_{\xi}$.
\end{proof}
This lemma shows that the recursive property of the solutions to the Jacobi operator $\mce$ is already embedded in the structure equation for the classical Jacobi fields.

\two
We give an alternative analytic interpretation of the recursion scheme in Sec.~\ref{sec:recursionxinf} as a 
resolution of Jacobi field.

Let $a^{2n-1}$ be a Jacobi field. Solve Eq.~\eqref{operatorL} for  $(a^{2n-1}, b^{2n},c^{2n}).$
Lemma \ref{lem:resolution} suggests to write the following structure equation as in \eqref{eq:newrecursion}.
\begin{align}\label{eq:resolution-1}
\ed  a^{2n-1}&\equiv  (\im \gamma c^{2n}+\im h_2 b^{2n})\xi +(*)\xib,\\
\ed  b^{2n} - \im  b^{2n} \rho&\equiv  \frac{\im \gamma }{2} a^{2n+1} \xi +\frac{\im}{2}  \hb_{2}a^{2n-1} \xib,\n\\
\ed  c^{2n} +\im c^{2n}  \rho&\equiv \frac{\im}{2}  h_2 a^{2n+1}  \xi +\frac{ \im \gamma }{2} a^{2n-1} \xib,\n\\
\ed  a^{2n+1}&\equiv a^{2n+1}_{\xi}\xi + (\im \gamma  b^{2n}+\im \hb_{2} c^{2n})\xib,\mod\;\iinfh.\n
\end{align}
Here $a^{2n+1}_{\xib}$ is obtained from the commutative identity $\delxb (b_{\xi})=\delx(b_{\xib})-\frac{R}{2}b$.

Then $a^{2n+1}$ is the new Jacobi field. Solving Eq.~\eqref{operatorL}  again for $(a^{2n+1}, b^{2n+2}, c^{2n+2})^t$ and iterating this process, the enhanced prolongation by formal Killing fields is realized by a sequence of resolutions of the initial trivial Jacobi field $a^1=0.$
 \begin{rem}
The resolution of a classical Jacobi field does not  connect to the higher-order Jacobi fields. In the notation of \eqref{eq:classicaloperator}, it only produces constant linear combinations of the classical Jacobi fields $A, \, B$. 
\end{rem}
 
The resolution of Jacobi field by the operator $L$ prompts the following definition.
\begin{defn}\label{defn:resolutionlength}
Let $\Sigma\hook\xinf$ be an integral surface. 
Let $\Sigmah\hook\xinfh$ be its double cover.
Let $A$ be a (possibly singular) Jacobi field on $\Sigmah.$
The \tb{resolution length} of $A$ is the dimension of the vector space of Jacobi fields generated by the resolution sequence of $A$ by the operator $L$.
\end{defn}
Note the sequence of Jacobi fields $a^{2j+1}$'s are singular at the umbilics. Although the differential operator $L$ is smooth and elliptic, when the analysis is carried out on a (compact) CMC surface it is not obvious how to apply the existing elliptic theory to the resolution of higher-order Jacobi fields $a^{2j+1}$'s.

\two
In hindsight, the resolution by the operator $L$, which follows from the structure equation for classical Jacobi field (or classical conservation laws) and a computation of the first few terms, may have led directly to the enhanced prolongation without reference to the formal Killing field. On the other hand, the consideration of formal Killing field also allows one to make use of the powerful loop group methods, particularly for finite-type CMC surfaces. This will be examined in Sec.~\ref{sec:finitetype}.
  
\subsection{Recursion for conservation laws}\label{sec:recursionforcvlaws}
Set
\be\label{eq:varphin}
\varphi_n=c^{2n+2}\xi+b^{2n}\xib\in\Omega^1(\Xh{2n+1}).
\ee
The structure equation \eqref{eq:newrecursion} shows that
\[ \ed\varphi_n\equiv 0 \mod\iinfh,
\]
and $\varphi_n$ represents a conservation law. Note that  $\varphi_n$ is weighted homogeneous with
$$\tn{weight}(\varphi_n)=2n-1
$$
in the sense that
$$\mcl_{\mcs}\varphi_n\equiv(2n-1)\varphi_n\mod\mcjh.
$$
\begin{prop}\label{prop:nontrivialvarphin}
For  each $n\geq 0$,
\begin{enumerate}[\qquad a)]
\item  the conservation law $[\varphi_n]\in \Cv{(\infty)}$ is nontrivial,
\item  there exists a nonzero constant $\tn{t}_n$ such that
\be
[ \ed\varphi_n ] = \tn{t}_n  [\Phi_{a^{2n+3}}] \in \Cv{2n+2}.
\ee
\end{enumerate}
\end{prop}
\begin{proof}
a)
For the proof presented below we apply the results from Sec.~\ref{sec:extension}, \ref{sec:spectralsym}, especially Lemma \ref{lem:noZintegral}.

The case $n=0$ has been established in Exam.~\ref{ex:HopfJacobi}. Assume $n\geq 1$.
 
\two
Suppose there exists a scalar function $f$ on $\xinfh_*$ such that
\be\label{eq:exactvarphin}
\ed f \equiv \varphi_n \mod \iinfh.
\ee
We first claim that  up to scaling by constant $f$ is a weighted homogeneous polynomial in the variables $z_j$'s of weight $2n-1$ ($f\in\mco(2n+1)$).\footnotemark\footnotetext{This follows from Lem.~\ref{lem:delbpoly0} in the next section. We record here an alternative proof using the spectral symmetry.}

Given that
$$\mcl_{\mcs}\varphi_n\equiv (2n-1)\varphi_n \mod\mcjh,
$$
the Lie derivative of Eq.~\eqref{eq:exactvarphin} therefore gives (Lie derivative commutes with exterior derivative)
\be\label{eq:Sexactvarphin}
\ed \big(\mcs(f)\big) \equiv (2n-1) \varphi_n \mod \mcjh.
\ee
It follows that
$$ \ed \big(\mcs(f) -(2n-1)f \big) \equiv 0 \mod \mcjh.$$

\one
$\bullet$ Case $\epsilon\ne 0$:
By Lemma \ref{lem:noZintegral}
$$\mcs(f) = (2n-1)f +\tn{constant}.$$
By Lemma \ref{lem:Sstable} such $f$ is necessarily a function in the variables $h_2\hb_2, z_j, \zb_j$'s. The defining equation  \eqref{eq:exactvarphin} then shows that  $f$ is in fact a function in the variables $z_j$'s only.
From this it follows that up to adding a constant to $f$,
$$\mcs(f) = (2n-1)f$$
and our claim is verified.

\one
$\bullet$ Case $\epsilon=0$:
We still claim that up to constant $f$ is a weighted homogeneous polynomial in the variables $z_j$'s of weight $2n-1$.

From the defining equation  \eqref{eq:exactvarphin}, let
$$\ed f\equiv \varphi_n+\sum_{j=3}^{2n+1}   f_j \zeta_j   \mod \theta_0,\theta_1,\thetab_1,\theta_2.$$
Then
\begin{align}
\ed \big(\mcs(f)-(2n-1)f\big)&\equiv \sum_{j=3}^{2n+1}  [\mcs(f_j)-(2n-1-(j-2))f_j ] \zeta_j\n\\
& \mod \theta_0,\theta_1,\thetab_1,\theta_2,\thetab_2, \mcjh^{\pi},\n\\
&\equiv 0 \mod \mcjh^{\tn{FI}} \quad\tn{by Lemma \ref{lem:noZintegral}}.\n
\end{align}

Since $f$ does not depend on the variable $h_2\hb_2$, it follows that for each $j\geq 3$ the derivative $f_j$ is a constant coefficient weighted homogeneous polynomial in $z_j$'s. Hence one may write
$$ f = g+v$$
where $g$ is the weighted homogeneous part and $v\in C^{\infty}(\Xh{1}).$ One computes
$$h_2^{-\frac{1}{2}}(f_{\xi}-g_{\xi})=h_2^{-\frac{1}{2}}v_{\xi}=h_2^{-\frac{1}{2}}(v^{\xi}+v^{2}h_3).
$$
The left hand side of this equation is homogeneous in the variable $z_j$'s of weight $2n\geq 2$ while the right hand side is at most linear in $z_3$. They are therefore linearly independent and we have $v_{\xi}=0$, and our claim is verified.

The rest of proof for the nontriviality of $[\varphi_n]$ follows from Cor.~\ref{cor:zurelation} and the arguments for Lemma 8.15, 8.16 in \cite{Fox2011}.

\two
We give a sketch of an alternative and more direct proof by constructing 
a class of simple integral cycles of $(\xinfh,\iinfh)$ with nonzero periods for the conservation laws.
From the argument above we are allowed to work globally in $\xinfh_*.$

Consider for example the case $\epsilon=1, \delta=0$, and the ambient space $M=\s{3}.$ 
Let $\x:\s{1}\to X$ be an integral curve such that; its image under the projection $X\to M$ is a great circle, and the  normal vector  ($e_3$) rotates with a constant speed (say $v$, an integer, with respect to the arc-length) in the 2-plane orthogonal to the 2-plane containing the great circle. A computation shows that 
on the infinite prolongation of this integral curve the structure functions are 
\begin{align}
h_{2}&=-\im v,\n\\
h_{2k}&= c'_{2k}( v^{2k-1}+\,\tn{polynomial in $v$ of lower degree} ), \n\\
h_{2k+1}&=0, \quad \tn{for}\; \, k\geq 1, \n
\end{align}
for nonzero constants $c'_{2k}$.
From this one finds that $a^{2k+1}=0$ for all $k\geq 1$,
and that the conservation laws have the following form
$$\varphi_n= \textnormal{$c''_{n}$}\left(v^{\frac{2n+1}{2}}+\tn{powers in $v^{\frac{1}{2}}$ of lower degree}\right)\ed\phi 
$$
for nonzero constants $c''_{n}$.
The claim follows for $v$ is an arbitrary integer.

b)
Lem.~\ref{lem:hrepJacobi}.
\end{proof}

\section{Noether's theorem}\label{sec:Noethertheorem}   
The analysis for symmetries and conservation laws is summarized in the higher-order extension of Noether's theorem.
\begin{thm}[Noether's theorem]\label{thm:higherNoether}
The Noether's theorem for classical symmetries and conservation laws Cor.~\ref{cor:Noether} admits the following higher-order extension.
\begin{enumerate}[\qquad a)]
\item
The space of conservation laws is the direct sum of the classical conservation laws and the higher-order conservation laws:
\begin{align}
\Cv{(\infty)}&=\Cv{0}\oplus\cup_{k=1}^{\infty}\Cv{2k},\n\\
\mcc^{(\infty)}&=\mcc^0\oplus\cup_{k=1}^{\infty}\mcc^{2k-1}.\n
\end{align}
Here $\tn{dim}\,\Cv{0}=\tn{dim}\, \mcc^0=6$, and for $n\geq 1,$ 
\begin{align}
\Cv{n}&=\left\{
\begin{array}{rll}
&0  &\tn{when $n$ is odd,} \n \\
&\langle \Phi_{a^{2k+1}}, \Phi_{\ol{a}^{2k+1}} \rangle &\tn{when $n=2k$.}  \n 
\end{array}\right. \n \\
\mcc^{n}&=\left\{
\begin{array}{rll}
&\langle [\varphi_{k-1}], [\ol{\varphi_{k-1}}] \rangle &\tn{when $n=2k-1$,}  \n\\
&0&\tn{when $n$ is even}.\n
\end{array}\right. \n
\end{align}
\item
Recall the exact sequence Eq.~\eqref{eq:exactsequence} from Sec.~\ref{sec:highercvlaws1},
\be 
0\to E^{0,1}_1\hook E^{1,1}_1\to E^{2,1}_1.\n
\ee
The injective map $E^{0,1}_1\hook E^{1,1}_1$ is also surjective and we have the isomorphisms
\[ \mathfrak{J}^{(\infty)} \simeq \mathfrak{S}_v \simeq \mcc^{(\infty)} \simeq \Cv{(\infty)}.\]
\end{enumerate}
\end{thm}
\begin{proof}
a)  
It follows from Thm.~\ref{thm:dimbound} and Prop.~\ref{prop:nontrivialvarphin}.

b)
A direct computation shows that the only Jacobi fields on $\Xh{1}$ are classical.
From the results so far then, particularly Lem.~\ref{lem:Jacobinormal}, it suffices to show that there does not exist a Jacobi field $f\in\mco(2k), k\geq 2,$ of the form
$$ f=z_{2k}+\mco(2k-1).$$
By taking  complex conjugate this would also imply that there does not exist a Jacobi field $f\in\mco(-2k)$ of the form
$$ f=\zb_{2k}+\mco(-(2k-1)).$$

First we introduce a set of notations.
Let
\begin{align}
P_d&=\{ \tn{weighted homogeneous polynomials of degree $d\geq 0$ in  $z_j$} \}, \n\\
\mcp_d&=\oplus_{i=0}^{d} P_i, \n\\
\mcp_d(\ell)&=\mcp_d \cap \mco(\ell).\n\\
Q_d&=P_d\oplus (h_2\hb_2) P_d, \n\\
\mcq_d&= \oplus_{i=0}^{d} Q_{i},\n\\
\mcq_d(\ell)&=\mcq_d \cap \mco(\ell).\n
\end{align}
We start the proof with a lemma.  
\begin{lem}\label{lem:delbpoly0}
Let $v\in\mco(k), k\geq 3$.  
Suppose
$$h_2^{\frac{1}{2}}\delxb v\in\mcq_{d}(k). 
$$
Then $$v\in\mcp_{d+1}(k).$$
\end{lem}
\begin{proof}
Consider the case $k=3.$ For functions in $\mco(3)$ we have the commutation relation
$[E_3,\delxb]=0.$ Hence by applying $E_3$ repeatedly to $\delxb v$ we get
$$h_2^{\frac{1}{2}}\delxb E^{m}_3(v)=0$$
for some $m\leq d+1$. 
By Lem~.\ref{lem:lemma5.4}, $v$ is a polynomial in $z_3$, and we write
$$v= v_m z_3^m+v_{m-1}z_3^{m-1}+\, ... \, +v_1 z_3+v_0,  
$$
for $v_j\in\mco(2)$ with $v_m$ being a constant.
Substituting this to the given equation for $h_2^{\frac{1}{2}}\delxb v$, we find the recursive equation
$$h_2^{\frac{1}{2}}\delxb v_j+(j+1)v_{j+1}R=c'_j\gamma^2+c_j^{"}h_2\hb_2, \quad j=m, m-1, \, ... \, $$
for constants $c'_j,c_j^{"}$ ($v_{m+1}=0$).
Since $v_m$ is a constant and $h_2^{\frac{1}{2}}\delxb z_3=R=\gamma^2-h_2\hb_2$,
an inductive argument by Lem.~\ref{lem:lemma5.4'}  in decreasing $j$ shows that
all the coefficients $v_j$ must be constant, and $v\in\mcp_{d+1}(3)$.

Suppose the claim is true up to $\mco(k-1)$. 
Let $v\in\mco(k)$.  For functions in $\mco(k)$ we have the commutation relation
$[E_k,\delxb]=0.$ Hence similarly as above  $v$ is a polynomial in $z_k$, and we write
$$v= v_m z_k^m+v_{m-1}z_k^{m-1}+\, ... \, +v_1 z_k+v_0,  
$$
for $v_j\in\mco(k-1)$, $m(k-2)\leq d+1$, and $v_m$ being a constant.
Substituting this to the given equation for $h_2^{\frac{1}{2}}\delxb v$, we find the recursive equation
$$h_2^{\frac{1}{2}}\delxb v_j+(j+1)v_{j+1}\hat{T}_k\in \mcq_{d-j(k-2)}(k-1), 
\quad j=m, m-1, \, ... \,.$$
By Lem.~\ref{lem:Tjh}, $\hat{T}_k\in  \mcq_{k-3}(k-1)$.
It follows by the similar inductive argument (using the induction hypothesis) with decreasing $j$ 
that  $v_j\in \mcp_{d-j(k-2)+1}(k-1)$ for each $j$, 
and consequently $v\in\mcp_{d+1}(k)$.
\end{proof}
\begin{cor}\label{lem:delbpoly}
Let $u\in\mco(k), k\geq 3$.  Let $u^{k}=E_{k} (u)=\frac{\del u}{\del h_{k }}$.
Suppose
$$h_2^{\frac{1}{2}}\delxb (h_2^{ \frac{k}{2}} u^{k }) \in\mcq_{d}(k). 
$$
Then $h_2^{ \frac{k}{2}}u^k\in\mcp_{d+1}(k)$, and hence
$$u\in\mcp_{d+(k-1)}(k)\mod \mco(k-1).$$
\end{cor}
\begin{proof}
Substitute $v=h_2^{ \frac{k}{2}} u^{k }$ from Lem.~\ref{lem:delbpoly}.
\end{proof}  

We now give the proof of Noether's theorem.
Let
$$ f=z_{2k}+u_{(1)},\quad u_{(1)}\in\mco(2k-1)$$
be a Jacobi field. Applying the Jacobi operator one finds,
$$-\mce(z_{2k})\equiv h_2^{\frac{1}{2}}\delxb(h_2^{ \frac{(2k-1)}{2}}u_{(1)}^{2k-1})z_{2k}\mod\mco(2k-1).$$
By Lem.~\ref{lem:mcez}, $\mce(z_{2k})\in\mcq_{2k-2}(2k)$.
By Cor.~\ref{lem:delbpoly} we have 
$$ u_{(1)} = p_{(1)}+\mco(2k-2)$$
for $p_{(1)}\in \mcp_{2k-2}(2k-1).$

Suppose by induction we arrive at the formula
$$ f=z_{2k}+p_{(1)}+p_{(2)}+\, ... \, +p_{(j)}+u_{(j+1)},\quad u_{(j+1)}\in\mco(2k-(j+1)),$$
where each $p_{(i)}\in \mcp_{2k-2}(2k-i)$ such that
$$ q_{(i)}:=\mce\left(z_{2k}+p_{(1)}+p_{(2)}+\, ... \, +p_{(i)}\right)\in\mco(2k-i).
$$
By Lem.~\ref{lem:mcez} we then have $q_{(j)}\in\mcq_{2k-2}(2k-j)$.
Applying the Jacobi operator to the refined normal form of $f$ we get
$$ -q_{(j)}\equiv h_2^{\frac{1}{2}}\delxb(h_2^{ \frac{2k-(j+1)}{2}}u_{(j+1)}^{2k-(j+1)})z_{2k-j}\mod\mco(2k-(j+1)).$$
By Cor.~\ref{lem:delbpoly} we may write
\begin{align*}
u_{(j+1)}&=p_{(j+1)}+u_{(j+2)}, \\
p_{(j+1)}&\in\mcp_{2k-2}(2k-(j+1)),\\
u_{(j+2)}&\in\mco(2k-(j+2)).
\end{align*}

Continuing this process we arrive at the normal form
\be\begin{array}{rll}
f&=p+u_{(2k-2)}, &u_{(2k-2)}\in\mco(2),\n  \\
p&=z_{2k}+p_{(1)}+p_{(2)}+\, ... \, +p_{(2k-3)},&
\end{array}
\ee
where $p_{(2k-3)}\in \mcp_{2k-2}(3)$ such that
$$ q_{(2k-3)}:=\mce(p)\in\mcp_{2k-2}(3).
$$
 
By the complex conjugate of all of this argument, we may assume the Jacobi field $f$ decomposes into
$$f =P+g$$
where  $P$ is an un-mixed polynomial in $z_j, \zb_j$, and $g$ is a function on $\Xh{1}$.
Since $\mce(g)\in\mco(3,-3)$ is at most linear in $z_3,\zb_3$ and
the operator 
$\mce:\mcp_{2k-2}(2k)\to\mcq_{2k-2}(2k)$ preserves the spectral weight,
and since $z_3, \zb_3$ are Jacobi fields, 
we must have
$$\mce(P)=c'+c"h_2\hb_2$$
for some constants $c',c".$ Up to adding a constant to $P$, we finally have
\begin{align}
\mce(P)&=c, \n\\
\mce(g)&=-c,\n
\end{align}
for a constant $c.$\footnotemark\footnotetext{When the curvature of the ambient space $\epsilon\ne 0$, a direct computation shows that for a function $g$ on $X$ this is only possible when $c=0$ and $g$ is a classical Jacobi field. When $\epsilon=0$, there exists such a function $g$ on $X$.
This is an example of inhomogeneous Jacobi field, Sec.~\ref{sec:scalingHopf}.}
Since the Jacobi operator $\mce$ preserves the spectral weight and $\mce(1)=\frac{1}{2}(\gamma^2+h_2\hb_2)$,
the equation $\mce(P)=c$ implies that $c=0$ and that $P$ has no constant term.
Hence $P$ is a pure polynomial Jacobi field in the variables $z_j, \zb_j$ with the leading term $z_{2k}$.
By Cor.~\ref{cor:zurelation} such a Jacobi field corresponds to a higher-order Jacobi field of the elliptic sinh-Gordon equation with the even order leading term, a contradiction to the classification result of  \cite{Fox2011}.
\end{proof}
Due to the commutation relation $[\mcs,\mce]=0$, the analysis in fact gives a direct proof that \emph{the space of Jacobi fields is the direct sum of the space of higher-order, weighted homogeneous, un-mixed polynomial Jacobi fields in $z_j, \zb_j$, and the classical Jacobi fields.} By Cor.~\ref{cor:zurelation}, the classification result of \cite{Fox2011} shows that there exist two Jacobi fields for each odd order, namely the inductively defined sequence $a^{2j+1}, \ab^{2j+1}$.

\np
\part{Finite type surfaces}\label{part:FT}
\section{Structure equations arising from the recursion}\label{sec:recursionfinitetype}
We introduce yet another adapted structure equation on $\xinfh_*$ that arises from the recursion relation \eqref{eq:newrecursion} and the analysis in Sec.\ref{sec:recursionxinf}. This structure equation is suited for our analysis of the finite-type solutions, for which the induced structure equation closes up at a finite prolongation. Let us mention two relevant properties:

$\bullet$\,
the sequence of higher-order Jacobi fields $\,\{ \tn{A}^{n}\}_{n=1}^{\infty}$ is embedded in the structure equation with the right constant scale so that
$\tn{A}^{n}=z_{2n+1}+\tn{(lower order terms)}$,

$\bullet$\,
the structure equation only contains the real structure constant $\gamma^2$, and the equation for complex conjugate which is necessary  for our analysis is uniform independent of the sign of $\gamma^2$.

\two

The recursion in Sec.\ref{sec:recursionxinf} can be rearranged to define the following set of structure equations. Note here the index is such that $\tn{A}^n\sim a^{2n+1}$ up to constant scale.
\begin{lem}
There exists a sequence of functions $\{ \tn{A}^n,\,\tn{B}^n,\,\tn{C}^n\}_{n=1}^{\infty}$ on $\xinfh_*$ that satisfy the following structure equation. Here we have normalized so that $r=h_2$ is real valued, $\tn{B}^0=0,\,\tn{C}^0=-2$, and the complex $1$-form $\omega=r^{\frac{1}{2}}\xi$.
\begin{align}
\ed \omega& \equiv 0,  \nonumber\\
\ed r & \equiv r  \Re(\tn{A}^1 \omega), \nonumber \\
\ed \tn{A}^n&\equiv \tn{B}^n \omega -r^{-1} \left(\gamma^2 \tn{B}^{n-1}+\frac{1}{2}(\gamma^2-r^2) \tn{C}^{n-1} \right) \omb, \label{eq:InvariantStrEq} \\
\ed \tn{B}^n&\equiv\left( \tn{A}^{n+1}+ \frac{1}{2} \tn{A}^1 \left( \tn{B}^n+ \tn{C}^n \right) \right)\omega -\frac{1}{2} (\gamma ^2+r ^2) r ^{-1} \tn{A}^n \omb,  \nonumber\\
\ed \tn{C}^n &\equiv- \left( \tn{A}^{n+1} +\frac{1}{2} \tn{A}^1 \tn{C}^n \right)\omega +\gamma^2 r ^{-1} \tn{A}^n \omb,\quad \mod\iinfh, \quad n\geq 1. \n
\end{align}
Moreover each $\tn{A}^n$ is a higher-order Jacobi field which satisfies
\begin{equation}
r^{-1} \mce(\tn{A}^n)=\tn{A}^n_{\omega, \omb}+\frac{1}{2} (\gamma^2 + r^2) r^{-1}\tn{A}^n=0.\n
\end{equation}
\end{lem}

\begin{proof}
We derive these structure equations from the recursion in Sec.~\ref{sec:recursionxinf}.  The proof is a straightforward computation. All the congruences below are mod $\iinfh$.

Let $\ed \tn{A}^n \equiv \tn{A}^n_{\omega}\omega + \tn{A}^n_{\omb} \omb$, and set $\tn{B}^n=\tn{A}^n_{\omega}$.
The mixed partial condition from $\ed^2 \tn{A}^n \equiv 0$ provides the formula for $\tn{B}^n_{\omb}$.  Using the fact that $\tn{B}^n_{\omega}=\tn{A}^n_{\omega,\omega}$ and the recursion relation
$\tn{A}^{n+1}=\tn{A}^n_{\omega,\omega} - \frac{1}{2}\tn{A}^1\left(\tn{A}^n_{\omega} + \tn{C}^n\right)$, we get the formula for $\tn{B}^n_{\omega}$.

The structure equation for $\ed \tn{C}^n$ comes directly from Eq.~\eqref{eq:upsilon} if we use $\tn{A}^1$ in place of $z_3$.  All we have left is to derive the expression for $\tn{A}^n_{\omb}$.  First, using $\tn{A}^1=z_3$ and the structure equation for $z_3$ on $\Xh{\infty}_*$ we find that $\ed \tn{A}^1 \equiv \tn{B}^1\omega+ r^{-1}(\gamma^2 - r^2) \omb$.  With this in hand, the identity from $\ed^2 \tn{C}^{n-1} \equiv 0$ provides the expression for $\tn{A}^n_{\omb}$.  That expression then requires that  $\tn{B}^0=0$ and $\tn{C}^0=-2$ if it is to hold for $n=1$.
\end{proof}
 
In the next section we will prove that there exist many finite-type solutions away from the umbilic locus $r=0$ by utilizing this structure equation and the Lie-Cartan theorem for closed differential systems, \cite[Appendix]{Bryant2001}.
 
\section{Linear finite-type surfaces}\label{sec:finitetype}
In this section we prove that locally there exist many CMC surfaces called `linear finite-type surfaces'  which satisfy constant coefficient linear relations among the canonical Jacobi fields. We examine a few classes of low order examples. 
The geometry of these surfaces including the polynomial Killing fields, spectral curves, etc, will be further explored in the next sections.

\two
The basic finite-type condition used by Pinkall \& Sterling \cite{Pinkall1989} is
\begin{defn}\label{defn:finitetype}
Let $\{\tn{A}^j\}_{j=1}^{\infty}$ be  the sequence of normalized Jacobi fields defined by Eqs.~\eqref{eq:InvariantStrEq}.
An integral surface $\Sigma\hook\xinf$ of the EDS for CMC surfaces is of \tb{linear finite-type} when the associated double cover $\Sigmah\hook\xinfh$ has the following property; there exists an integer $n\geq 0$ such that when pulled back to $\Sigmah$  the sequence of Jacobi fields satisfy the linear equation
\begin{equation}\label{eq:finitetype}
\tn{A}^{n+1}=\sum_{j=0}^{n}\left( U_j \tn{A}^j + V_j \ol{\tn{A}}^j \right), \quad (\tn{set}\;\tn{A}^0=0) 
\end{equation}
for constant coefficients $U_j, V_j$.
The \tb{level} of a linear finite-type surface is the least integer $n$ such that Eq.~\eqref{eq:finitetype} holds.
\end{defn}

Vinogradov refers to setting $\tn{A}^{n+1}=0$ as the basic method for finding solutions invariant under the symmetry generated by the Jacobi field $\tn{A}^{n+1}$.  We will show below that for the system under study, the condition of Eq.~\eqref{eq:finitetype} implies that $\tn{A}^{n+k}=0$ for some $k\geq 0$.

\begin{exam}[Level $0$]
The linear finite-type surfaces of level 0 are defined by the equation
\[ \nA^1=z_3=h_1^{-3}h_3=0.
\]
From the well known characterization of CMC surfaces  with transitive group of symmetry, 
such a surface is a generalized cylinder.
\end{exam}
\begin{exam}[Level $1$]
Consider the next simplest case of linear finite-type surfaces of level $1$ defined by
\be\label{eq:FTlevel1}
\nA^2=U_1 \nA^1 +V_1 \nAb^1.
\ee
From the structure equation \eqref{eq:InvariantStrEq},
\[\ed\nA^1=\nB^1\omega+\frac{(\gamma^2-r^2)}{r}\omb.
\]
Differentiating Eq.~\eqref{eq:FTlevel1}  and collecting the $\omb$-terms, one gets
\begin{equation}\label{eq:finitetype1}
\nC^{1}=\frac{-2}{\gamma^2-r^2}\left( \gamma^2\nB^1 +rV_1\nBb^1 \right) - 2 U_1.
\end{equation}
Requiring that $\ed^2 \nB^1=0$ with this relation, the compatibility equation factors and the analysis is divided into two cases;  $|V_1|^2=\gamma^2$, or $|V_1|^2\ne\gamma^2.$

\one
$\bullet$ Case  $|V_1|^2=\gamma^2$: This requires $\gamma^2>0.$
The equations $\ed^2 \nA^1=0,\ed^2 \nB^1=0$ are identities. This leaves the following closed structure equation (mod $\iinfh$);
\begin{align}
\ed \omega&\equiv 0, \label{eq:n=1InvariantStrEq}  \\
\ed r &\equiv r  \Re(\nA^1 \omega), \nonumber \\
\ed \nA^1&\equiv \nB^1\omega +  r^{-1} (\gamma ^2-r ^2)\omb, \n \\
\ed \nB^1&\equiv
\left( U_1 \nA^1 +V_1 \nAb^1 +\frac{1}{2}\nA^1(\nB^1+\nC^1) \right)\omega
- \frac{r^{-1}}{2}(\gamma ^2+r ^2) \nA^1 \omb.\n
\end{align}
A direct computation shows that $\ed^2\equiv0\mod\iinfh$ is an identity for this putative structure equation.
By Lie-Cartan  theorem, for any choice of initial conditions of $r, \nA^1, \nB^1, \omega$ at a point satisfying $r \neq 0, \gamma^2 -r^2\ne 0$, and the constants $\gamma,U_1,V_1$ with $|V_1|=\gamma^2$, locally there exists an integral surface with a complex $1$-form $\omega$ and the functions $r, \nA^1, \nB^1$ satisfying this structure equation.

$\bullet$ Case  $|V_1|^2\ne\gamma^2$:
From the compatibility equation for $\ed^2 \nB^1=0$ we get
\[ \nB^1=\frac{\nA^1}{\nAb^1}\frac{\gamma^2-r^2}{r}.\]
This leaves the following closed structure equation (mod $\iinfh$);
\begin{align}
\ed \omega&\equiv 0,  \nonumber\\
\ed r &\equiv r  \Re(\nA^1 \omega), \nonumber \\
\ed\nA^1&=\frac{(\gamma^2-r^2)}{r}
 \left({\frac {\nA^{{1}}}{\nAb^1}}\omega+\omb\right).\n
\end{align}
A direct computation shows that $\ed^2\equiv0\mod\iinfh$ is again an identity for this putative structure equation. By Lie-Cartan  theorem, for any choice of initial conditions of $r, \nA^1,\omega$ at a point satisfying $r \neq 0, \gamma^2 -r^2\ne 0$, and the constants $\gamma,U_1,V_1$ with $|V_1|\ne\gamma^2$, locally there exists an integral surface with a complex $1$-form $\omega$ and the functions $r, \nA^1$ satisfying this structure equation.
\end{exam}

An analysis shows that the case $\gamma^2<0$ is generally not compatible with the linear finite-type condition. We therefore make the assumption on the structure constant $\gamma^2$:

\begin{center}\framebox{$\gamma^2>0$ for linear finite-type surfaces.}\end{center}

\one
\begin{prop}\label{prop:finitetypeexistence}
Consider the structure equations \eqref{eq:InvariantStrEq}$\mod\iinfh$.  
Impose a linear relation
 \begin{equation}\label{eq:finitetypen}
\tn{A}^{n+1}=\sum_{j=1}^{n}\left( U_j \tn{A}^j + V_j \ol{\tn{A}}^j \right),  
\end{equation}
for constants $U_j,V_j$. Then the $\omb$-derivative of this equation gives
\begin{align}\label{eq:ccondition}
\tn{C}^{n}=
&\frac{2}{\gamma^2-r^2}\sum_{j=1}^{n-1}
\left[ U_{j+1}\left(\gamma^2\nB^j+\frac{\gamma^2-r^2}{2}\nC^j\right)-r V_j\nBb^j\right] \\
&+\frac{-2}{\gamma^2-r^2}\left(\gamma^2\nB^n+r V_n\nBb^n\right)-2U_1.\n
\end{align}
With this relation, the following set of equations among the constant coefficients
\begin{equation}\label{eq:constants} 
|V_n|^2=\gamma^2 \quad \& \quad V_j=-V_n \Ub_{j+1} \;\; {\rm for} \;\; 1 \leq j \leq n-1   
\end{equation}
is sufficient to imply the identity $\ed^2 \tn{B}^n=0$. As a consequence,
the structure equations \eqref{eq:InvariantStrEq}$\mod\iinfh$ for $j \leq n$ along with Eqs.~\eqref{eq:finitetypen},~\eqref{eq:ccondition},~\eqref{eq:constants} imposed define a system of compatible closed structure equations that formally satisfies $\ed^2=0$.  By Lie-Cartan theorem there exist local solutions for initial conditions satisfying $r \neq 0$ and $r^2 \neq \gamma^2$.
\end{prop}
\begin{proof}
The proof is a straightforward calculation.
\end{proof}
\begin{rem}
We will see that if Eq.~\eqref{eq:finitetypen} is satisfied then $\tn{A}^{2n+1}=0$.  Thus  linear finite-type CMC surfaces admit polynomial Killing fields. See Sec.~\ref{sec:polynomialKF} for a proof. We refer the reader to \cite{Burstall1993} for a related work on the finite-type harmonic tori in symmetric spaces.
\end{rem}

With the linear finite-type equation \eqref{eq:finitetypen} imposed it becomes necessary to include all of its differential consequences into consideration. The result that this leads to a compatible structure equation under the assumption \eqref{eq:constants} on the constant coefficients as soon as it closes up after recovering the variables $\nC^n$ by Eq.~\eqref{eq:ccondition} is by no means trivial. 
It is due to the fact that Eq.~\eqref{eq:finitetypen} is stable under the Laplacian;
$$ \delx\delxb\left(Eq.~\eqref{eq:finitetypen}\right)\equiv 0\mod Eq.~\eqref{eq:finitetypen}$$
(here we use the Jacobi equation for $\nA^{j}$).
 
Based on the robustness of this phenomenon and the well known results on the spectral geometry of finite-type harmonic maps it would not be surprising that the linear finite-type equation inherits the symmetries generated by the canonical Jacobi fields.
\begin{conj}\label{conj:LFTsymmetry}
Let $\mu:\hat{Y}^{(n)} \subset \Xh{2n+2}$ be a submanifold defined by a $n$-th level linear finite-type equation.  Then $(\hat{Y}^{(n)},\mu^*\Ih{2n+2})$ is a Frobenius system.  Moreover, to each Jacobi field $\tn{A}^k, k\leq n$, there exists an associated symmetry vector field $V_{\tn{A}^k}$ (locally) defined on $\hat{Y}^{(n)}$ so that the corresponding flow $\phi_k(t):\hat{Y}^{n} \to \hat{Y}^{n}$ is a symmetry of $\mu^*\Ih{2n+2}.$
\end{conj}

\begin{exam}[Level $1$, $\vert V_1 \vert^2=\gamma^2$ case continued]
Let $\mu: \hat{Y}^{(1)}\subset \Xh{4}$ be the submanifold defined by Eqs.~\eqref{eq:FTlevel1}, \eqref{eq:finitetype1}.
Let $\mu^*\Ih{4}$ be the induced ideal on $\hat{Y}^{(1)}$.
Let $V_{\nA^1}\in H^0(T\Xh{4})$ be the symmetry vector field generated by the Jacobi field $\nA^1.$
Conj.~\ref{conj:LFTsymmetry} claims that the restriction of Jacobi field $(\nA^1)_{\vert_{\hat{Y}^{(1)}}}$ generates a symmetry of the Frobenius system $(\hat{Y}^{(1)}, \mu^*\Ih{4})$.  It is likely that $V_{\nA^1}$ is in general not tangent to $\hat{Y}^{(1)}$.

Note the conservation laws
\begin{align*}
\phi_0:&=\omega\\
&=(\varphi_0)_{\vert_{Y^{(1)}}}, \\
\phi_1:&=(\nB^1+\frac{1}{4}(\nA^1)^2)\omega-2r\omb \\
&=2\im\gamma(\varphi_1)_{\vert_{Y^{(1)}}}.
\end{align*}
From the defining equation \eqref{eq:FTlevel1} we suspect that there exists the corresponding constant coefficient linear relation among the conservation laws so that
$$ [\phi_1]\in \langle [\phi_0], [\ol{\phi}_0] \rangle\subset \mcc^{(\infty)}_{\vert_{Y^{(1)}}}.$$
\end{exam}

\section{Polynomial Killing field}\label{sec:polynomialKF}
Beginning this section we return to adopt the formal Killing field coefficients $a^{2j+1}, b^{2j+2}, c^{2j+2}$'s for our analysis of  linear finite-type CMC surfaces.

Let $\Sigma\hook\xinf$ be a linear finite-type CMC surface of level $m$; on the double cover
$\Sigmah\hook\xinfh$ the sequence of canonical Jacobi fields satisfy the linear relation
\be\label{eq:Plevelm}
a^{2m+1}=\sum_{j=1}^{m-1} U_j a^{2j+1}+V_j\ab^{2j+1} 
\ee
for constant coefficients $U_j, V_j$. We show that such $\Sigmah$ admits a $\sla(2,\C)$-valued Killing field, which becomes a polynomial of degree $2m-1$ in the spectral parameter $\lambda$ under the action of spectral symmetry.  The existence of a polynomial Killing field is a well known feature in the integrable systems theory, which comes with a standard package of integration methods. 
For example it leads to the spectral curve, an associated complex algebraic curve, 
which effectively linearizes the CMC surface on its Jacobian.

\two
To begin we analyze the differential consequences of the defining equation \eqref{eq:Plevelm}. We first show that Eq.~\eqref{eq:Plevelm} and its derivatives  determine all the Jacobi fields of order $\geq 2m+1$ in terms of the Jacobi fields of order $\leq2m-1$. 

Recall the structure equation \eqref{eq:newrecursion}. Differentiating Eq.~\eqref{eq:Plevelm} by $\delx$ one gets
\begin{align}\label{eq:Pdel}
\gamma c^{2m+2}+h_2 b^{2m+2}&=\sum_{j=1}^{m-1} U_j(\gamma c^{2j+2} +h_2 b^{2j+2}) \\
                                                      &\,-\sum_{j=1}^{m-1}V_j (\gamma \bb^{2j} +h_2 \cb^{2j}).\n
\end{align}
Differentiating this equation by $\delx$ again, one gets
\be\label{eq:Pdeldel}
a^{2m+3}-\sum_{j=1}^{m-1}( U_ja^{2j+3}+V_j\ab^{2j-1} )
=\frac{h_3}{\im\gamma h_2} \left(-b^{2m+2}+ \sum_{j=1}^{m-1}(U_jb^{2j+2}-V_j\cb^{2j}) \right).
\ee

We claim that the term on the right hand side is given by
\be\label{eq:b2m+2}
P^{2m+2}:=h_2^{\frac{1}{2}}\left(-b^{2m+2}+ \sum_{j=1}^{m-1}(U_jb^{2j+2}-V_j\cb^{2j}) \right)=\im\gamma U_0
\ee
for a constant $U_0$. For this, differentiate $P^{2m+2}$ by $\delx$ and  one finds
\[ \delx(P^{2m+2})\equiv 0\mod \eqref{eq:Pdeldel}.
\]
Differentiating $P^{2m+2}$ by $\delxb$ one finds
\[  \delxb(P^{2m+2})\equiv0\mod \eqref{eq:Plevelm},
\]
and the claim follows.

By definition $a^3=h_2^{-\frac{3}{2}}h_3=-2\delx(h_2^{-\frac{1}{2}})$. Substituting Eq.~\eqref{eq:b2m+2} to Eq.~\eqref{eq:Pdeldel}, we have
\begin{lem}\label{lem:shift}
Suppose the sequence of canonical Jacobi fields of a CMC surface, not necessarily of level $m$, 
satisfy the linear relation \eqref{eq:Plevelm}.
Then they also satisfy
\be\label{eq:a2m+3}
a^{2m+3}=\sum_{j=1}^{m-1}( U_ja^{2j+3}+V_j\ab^{2j-1} ) +U_0 a^3
\ee
for a constant $U_0$.
\end{lem}
By repeated application, the lemma implies that all the higher-order Jacobi fields $a^{2n+1}, \, n\geq m,$ lie in the constant coefficient linear span of  $\{ a^{2j+1},\,\ab^{2j+1}\}_{j=1}^{m-1}$.

Note by Eq.~\eqref{eq:b2m+2} the equation \eqref{eq:Pdel} splits into two parts and we have from $b^2=-\im\gamma h_2^{-\frac{1}{2}}, c^2=\im h_2^{\frac{1}{2}}$ that
\begin{align}\label{eq:bc2m+2}
  b^{2m+2}&=\sum_{j=1}^{m-1}(U_jb^{2j+2}-V_j\cb^{2j}) -\im\gamma U_0 h_2^{-\frac{1}{2}},\\
                  &=\sum_{j=1}^{m-1}(U_jb^{2j+2}-V_j\cb^{2j})+U_0 b^2,\n\\
  c^{2m+2} &=\sum_{j=1}^{m-1} ( U_j  c^{2j+2} - V_j  \bb^{2j}) + \im U_0 h_2^{\frac{1}{2}},\n\\
                  &=\sum_{j=1}^{m-1} ( U_j  c^{2j+2} - V_j  \bb^{2j})+U_0 c^2.\n
\end{align}
\begin{cor}\label{cor:abcgeneratingrelation}
For a linear finite-type CMC surface  of level $m$,
all the higher-order Killing coefficients $\{ a^{2n+1}, b^{2n+2}, c^{2n+2}\}_{n\geq m}$ lie in the constant coefficient linear span of the generating  set $\,\{ a^{2j+1}, \ab^{2j+1}, b^{2j+2}, \bb^{2j+2}, c^{2j+2},\cb^{2j+2}\}_{j=0}^{m-1}$.
\end{cor}

\two
Lemma \ref{lem:shift} has another application.
\begin{lem}\label{lem:ftypenormalform}
The defining equation \eqref{eq:Plevelm} for  linear finite-type surfaces  can be normalized so that
\[ a^{2m+1}=\sum_{j=1}^{m-1} U'_j a^{2j+1}+V'_j\ab^{2j+1},
\]
where the constant coefficients $U'_j, V'_j$ satisfy the relation
\[ |V'_{m-1}|^2=1,\quad\&\quad V'_j=-V'_{m-1}\ol{U}'_{j+1}\quad\tn{for}\; 1\leq j \leq m-1.
\]
\end{lem}
\begin{proof}
We refer the reader to  \cite[p425, Proposition 4.3]{Pinkall1989}. Note that we do not use
the operation (\textrm{vi})  in  \cite[p424, bottom]{Pinkall1989} (which is essentially the spectral symmetry). This accounts for the introduction of the unit complex number $V'_{m-1}$.
\end{proof}
\begin{defn}
The equation \eqref{eq:Plevelm} for linear finite-type surfaces is \tb{adapted}
when the constant coefficients $U_j, V_j$ satisfy the algebraic relation
\be\label{eq:UVadapted}
|V_{m-1}|^2=1,\quad\&\quad V_j=-V_{m-1}\ol{U}_{j+1}\quad\tn{for}\; 1\leq j \leq m-1.
\ee
\end{defn}
Note from Prop.~\ref{prop:finitetypeexistence} that there exist many local  linear finite-type CMC surfaces satisfying the adapted equation. We shall work with the adapted linear finite-type equation from now on.
\begin{rem}
It is clear from the weighted homogeneity, or spectral symmetry, of the higher-order Jacobi fields that the $\s{1}$-family of associate surfaces of a linear finite-type surface are linear finite-type. In the original paper \cite[p425, Proposition 4.3]{Pinkall1989} it is shown that one may further normalize by spectral symmetry so that $V_{m-1}=1$, i.e., there exists an associate surface of a linear finite-type surface which satisfies the adapted linear finite-type  equation with $V_{m-1}=1$.
\end{rem}

There is another identity obtained from differentiating Eq.~\eqref{eq:Plevelm}, this time assuming the  linear finite-type equation is adapted. This identity will be used for the construction of polynomial Killing field.

Differentiate Eq.~\eqref{eq:Plevelm} by $\delxb$ and one gets
\be\label{eq:Pdelb}
\gamma b^{2m}+\hb_2 c^{2m}=\sum_{j=1}^{m-1}\left( U_j(\gamma b^{2j}+\hb_2 c^{2j})
                                                                               - V_j(\gamma \cb^{2j+2}+\hb_2 \bb^{2j+2})\right).
\ee
Differentiating this equation again by $\delxb$ and taking a complex conjugate one gets
\be\label{eq:Pdelbdelb}
-\cb^{2m}+\sum_{j=1}^{m-1} ( \ol{U}_j \cb^{2j}- \ol{V}_j b^{2j+2} )
=\frac{\im\gamma h_2}{h_3}\left(-\ab^{2m-1}+\sum_{j=1}^{m-1}(\ol{U}_j\ab^{2j-1}+\ol{V}_j a^{2j+3})\right).
\ee
Since $a^1=0$, comparing the right hand side with Eq.~\eqref{eq:Plevelm} together with the adapted condition on the coefficients $U_j,V_j$, we have
\[-\ab^{2m-1}+\sum_{j=1}^{m-1}(\ol{U}_j\ab^{2j-1}+\ol{V}_j a^{2j+3})
=\ol{V}_{m-1}U_1a^3.
\]
Hence Eq.~\eqref{eq:Pdelbdelb} gives
\[-\cb^{2m}+\sum_{j=1}^{m-1} ( \ol{U}_j \cb^{2j}- \ol{V}_j b^{2j+2} )
=-\ol{V}_{m-1}U_1b^2.
\]
As a consequence, after complex conjugation of this equation,  Eq.~\eqref{eq:Pdelb} splits into two parts and we have
(recall from \eqref{eq:newrecursion} that $\hb_2\bb^2=\im\gamma\hb^{\frac{1}{2}}=-\gamma\cb^2$)
\begin{align}\label{eq:cb2m}
c^{2m}&=\sum_{j=1}^{m-1} ( U_j c^{2j}- V_j \bb^{2j+2} )+V_{m-1}\ol{U}_1\bb^2,\\
 b^{2m}&=\sum_{j=1}^{m-1}\left( U_j b^{2j} - V_j  \cb^{2j+2})\right)+V_{m-1}\ol{U}_1\cb^2.\n
\end{align}
Note that they are complex conjugate of each other up to scaling by $\ol{V}_{m-1}$.

\two
With this preparation we proceed to the construction of a polynomial Killing field. The construction originally due to Pinkall \& Sterling in \cite{Pinkall1989}, which we follow closely here, is essentially local, compared to for example Hitchin's for harmonic tori in the 3-sphere which made use of the global holonomy of the associated flat $\sla(2,\C)$-connection,  \cite{Hitchin1990}. A consequence of the local construction is that the geometry of a  linear finite-type surface is a posteriori global; there exists as we will find the corresponding globally well defined rank 2 distribution on essentially a space of certain polynomials in an auxiliary parameter $\lambda$ for which a linear finite-type surface can be viewed as a part of integral leaf.

A main idea of construction is the Lie algebra decomposition
\be\label{eq:so4cdecomposition}
\so(4,\C)=\sla(2,\C)\oplus\sla(2,\C).
\ee
The reducibility of $\so(4,\C)$ already appeared in the recursion equation \eqref{eq:newrecursion} which repeats a 3-step,  rather than a 6-step, relation. We shall make use of the decomposition and construct an $\sla(2,\C)$-valued polynomial Killing field first.  We observe then that this can  be extended, or doubled, to an $\so(4,\C)$-valued polynomial Killing field.

\two
By Lemma ~\ref{lem:ftypenormalform}, assume that $\Sigma$ is linear finite-type of level $m$ and the sequence of Jacobi fields satisfy the adapted constant coefficient linear equation \eqref{eq:Plevelm}. We shall suppress the spectral parameter $\lambda$ from the analysis for the moment and treat the sequence of Killing coefficients $a^{2j+1}, b^{2j+2}, c^{2j+2}$ simply as the sections of the bundles $\hat{K}^0, \hat{K}^{-1}, \hat{K}^1$ over the double cover $\Sigmah$ respectively (here $\hat{K}\to\Sigmah$ is the canonical bundle).

Based on Cor.~ \ref{cor:abcgeneratingrelation}, set for an ansatz
\begin{align}\label{eq:Killingansatz}
\hat{\tb{a}}&=\sum_{j=1}^{m-1} s_j a^{2j+1}+t_j\ab^{2j+1}, \\
\hat{\tb{b}}&=\sum_{j=1}^{m-1} (p_j b^{2j+2}-q_j\cb^{2j})+p_0b^2-q_m\cb^{2m},  \n \\
\hat{\tb{c}}&=\sum_{j=1}^{m-1} (p_j c^{2j+2}-q_j\bb^{2j})+p_0c^2-q_m\bb^{2m},\n
\end{align}
for constant coefficients $s_j, t_j, p_j, q_j$. In order to make these equations transversal to Eq.~\eqref{eq:cb2m}, let us set $q_m=0$. Define
\begin{align}
\psi_+&=
\bp\frac{\im}{2}\rho&-\frac{1}{2}(\gamma \xib+h_2\xi)\\
\frac{1}{2}(\gamma \xi+\hb_2\xib) &-\frac{\im}{2}\rho
\ep,\label{eq:tbXdefi}\\
\bf{X}&=
\bp
-\im \hat{\tb{a}}& 2 \hat{\tb{c}}\\
2\hat{\tb{b}} &\im\hat{\tb{a}}
\ep.\n
\end{align}
One finds that $\psi_+$ satisfies the structure equation $\ed\psi_+ + \psi_+\w\psi_+=0$.
We wish to determine the set of constants  $\{ s_j, t_j, p_j, q_j\}$ so that
\be\label{eq:tbXstructure}
\ed\tb{X}+[\psi_+, \tb{X}]=0.
\ee
This is written component-wise
\be\label{eq:dbhabc}
\bp \ed\hat{\tb{a}} \\ \ed\hat{\tb{b}}-\im\rho\hat{\tb{b}}  \\  \ed\hat{\tb{c}}+\im\rho\hat{\tb{c}} \ep=
\bp
 \im\gamma \hat{\tb{c}}+\im h_2\hat{\tb{b}} & \im\gamma\hat{\tb{b}}+\im \hb_2 \hat{\tb{c}}\\
 \frac{\im}{2} \gamma\hat{\tb{a}}& \frac{\im}{2}\hb_2\hat{\tb{a}} \\
 \frac{\im}{2} h_2\hat{\tb{a}} &\frac{\im}{2} \gamma\hat{\tb{a}}
\ep \bp \xi \\ \xib \ep.
\ee
Note that this structure equation is obtained from Eq.~\eqref{eq:newrecursion} by setting $a^{2n+1}=\hat{\tb{a}}, b^{2n+2}=\hat{\tb{b}}, c^{2n+2}=\hat{\tb{c}}$ for all $n$.

We proceed with the computation of Eq.~\eqref{eq:dbhabc}. First, imposing that the set of compatibility equations arising from the equations for $\ed\hat{\tb{b}}, \ed\hat{\tb{c}}$ are proportional to Eq.~\eqref{eq:Plevelm} and its complex conjugate, we get the following set of linear equations (we omit the details of straightforward computation).
\be
\bp  q_m \\ q_j-t_j \\p_j-s_j
\ep =\beta_1
\bp -1\\ \ol{U}_j \\ \ol{V}_j
\ep, \quad
\bp p_{m-1} \\p_{j-1}-s_j \\q_{j+1}-t_j
\ep=\beta_2
\bp -1\\ U_j \\ V_j
\ep.
\ee
Here $\beta_1, \beta_2$ are arbitrary constants. Since $q_m=0$, we must have $\beta_1=0$. This is then a set of $4m-2$ linear equations for the $4m-2$ unknowns $\{ q_j, p_j, t_j, s_j\}$. One easily finds that this system of equations is non-degenerate and  $\{ q_j, p_j, t_j, s_j\}$ are uniquely determined up to scale by $\{ \beta_1=0, \beta_2,  U_j, V_j\}$. Without loss of generality, let us scale $\beta_2=1$. The solution is given by
\begin{align}\label{eq:qtps}
q_m&=0, \\
q_j=t_j&=-(V_{m-1}+V_{m-2}+\,...\, \;+V_{j+1}+V_j ),  \n\\
&\hspace{2cm} \rule{40pt}{0.2pt}  \n \\
p_j=s_j&= -1+U_{m-1}+U_{m-2}-\, ... \, +U_{j+1}, \n \\
p_0&= -1+U_{m-1}+U_{m-2}+ \hspace{1cm}  ... \hspace{1cm}+U_{1}. \n
\end{align}

Substitute these equations to $\ed \hat{\tb{a}}$, and imposing that the resulting compatibility equations are proportional to Eq.~\eqref{eq:cb2m} and its complex conjugate, one gets an identity.\footnotemark
\footnotetext{Here one needs to take into account the linear relation
\[ \gamma c^2+h_2 b^2=0.
\]}
As a result, \emph{there exists up to scale a unique non-trivial $\sla(2,\C)$-Killing field defined by the ansatz \eqref{eq:Killingansatz}.}

\two
It is at this point that the spectral symmetry of CMC surfaces plays its role. The symmetry gives an extension of the Killing field $\tb{X}$ to a polynomial $\sla(2,\C)$-Killing field $\tb{X}(\lambda)$ in the spectral parameter $\lambda$. By introducing the auxiliary parameter $\lambda$, we are able to capture  the underlying algebraic structure of the linear finite-type surfaces. 

For a unit complex number $\lambda$, set the $\mathbb{S}^1$-family of $\su(2)$-valued 1-form 
\be\label{eq:psi+lambda}
\psi_+(\lambda) =
\bp\frac{\im}{2}\rho&-\frac{1}{2}(\gamma \xib+\lambda^{-1}h_2\xi)\\
\frac{1}{2}(\gamma \xi+\lambda \hb_2\xib) &-\frac{\im}{2}\rho
\ep, \ee
which is obtained from $\psi_+$ by spectral symmetry. It is straightforward to check that the Maurer-Cartan compatibility equation 
$$\ed\psi_+(\lambda)+\psi_+(\lambda) \w\psi_+(\lambda) =0$$
continues to hold when $\lambda$ is extended to a nonzero complex number $\lambda\in\C^*$. 
The algebro-geometric construction that follows, which may seem ad-hoc at first sight, can be attributed to this extension property of the spectral symmetry to the $\C^*$-symmetry.   

We now state the main result of this section.
\begin{thm}\label{thm:polyKilling}
Let $\Sigma\hook\xinf$ be an integral surface of the EDS for CMC surfaces. 
Suppose $\Sigma$ is linear finite-type of level $m$ so that on the double cover $\Sigmah\to\Sigma$
the canonical Jacobi fields satisfy an adapted linear equation \eqref{eq:Plevelm}.
Then $\Sigmah$ admits a $\sla(2,\C)$-valued polynomial Killing field $\lambda^{\frac{1}{2}}\tb{X}(\lambda)$ of degree $2m-1$.
\end{thm}
\begin{proof}
First, apply the spectral symmetry
\[ h_j \to \lambda^{-1}h_j, \quad \hb_j \to \lambda \hb_j, \quad\tn{for}\;j\geq 2,
\]
for a unit complex number $\lambda$. Then $a^{2n+1}, b^{2n+2}, c^{2n+2}$ transform to
\begin{align}\label{eq:abclambda}
a^{2n+1}\longrightarrow  a^{2n+1}(\lambda)&=\lambda^{n-\frac{1}{2}}a^{2n+1}, \quad
\ab^{2n+1}\longrightarrow \ol{a^{2n+1}(\ol{\lambda}^{-1})} =\lambda^{-n+\frac{1}{2}}\ab^{2n+1},  \\
b^{2n+2}\longrightarrow b^{2n+2}(\lambda)&=\lambda^{n+\frac{1}{2}}b^{2n+2},\quad
\,\bb^{2n+2}\longrightarrow \ol{b^{2n+2}(\ol{\lambda}^{-1})} =\lambda^{-n-\frac{1}{2}}\bb^{2n+2},\n \\
c^{2n+2}\longrightarrow c^{2n+2}(\lambda)&=\lambda^{n-\frac{1}{2}}c^{2n+2},\quad
\,\cb^{2n+2}\longrightarrow \ol{c^{2n+2}(\ol{\lambda}^{-1})} =\lambda^{-n+\frac{1}{2}}\cb^{2n+2}.\n
\end{align}
The coefficients $U_j, V_j$ for the  linear finite-type equation \eqref{eq:Plevelm} transform accordingly
\begin{align}
U_j\longrightarrow U_j(\lambda)&=U_j \lambda^{m-j},\\
V_j\longrightarrow V_j(\lambda)&=V_j \lambda^{m+j-1}.\n
\end{align}
Substituting these into  $\hat{\tb{a}}, \hat{\tb{b}},\hat{\tb{c}}$ in Eq.~\eqref{eq:Killingansatz}, we get from Eq.~\eqref{eq:qtps}, 
\begin{align}
\hat{\tb{a}}(\lambda)&=\lambda^{-\frac{1}{2}}\sum_{j=1}^{2m-2}\hat{\tb{a}}_j\lambda^j,\n\\
&=\lambda^{-\frac{1}{2}}( -a^3\lambda+\,...\, -V_{m-1}\ab^3\lambda^{2m-2} ),\n\\ 
\hat{\tb{b}}(\lambda)&=\lambda^{ \frac{1}{2}}\sum_{j=0}^{2m-2}\hat{\tb{b}}_j\lambda^j,\n\\
&=\lambda^{\frac{1}{2}}(  -b^2+\,...\, +V_{m-1}\cb^2\lambda^{2m-2}         ),\n\\
\hat{\tb{c}}(\lambda)&=\lambda^{-\frac{1}{2}}\sum_{j=0}^{2m-2}\hat{\tb{c}}_j\lambda^j,\n\\
&=\lambda^{-\frac{1}{2}}( -c^2+\,...\, +V_{m-1}\bb^2\lambda^{2m-2}       ).\n
\end{align} 

Set
\[(\tb{a}(\lambda), \tb{b}(\lambda), \tb{c}(\lambda))
=( \lambda^{\frac{1}{2}}\hat{\tb{a}}(\lambda), \lambda^{-\frac{1}{2}}h_2^{\frac{1}{2}}\hat{\tb{b}}(\lambda),
\lambda^{\frac{1}{2}}h_2^{-\frac{1}{2}}\hat{\tb{c}}(\lambda)).
\] Then
\begin{align}
\tb{a}(\lambda)&=\sum_{i=1}^{2m-2}\tb{a}_i\lambda^i,\\
\tb{b}(\lambda)&=\sum_{i=0}^{2m-2}\tb{b}_i\lambda^i,\n\\
\tb{c}(\lambda)&=\sum_{i=0}^{2m-2}\tb{c}_i\lambda^i,\n
\end{align}
and each $\tb{a}_i, \tb{b}_i, \tb{c}_i$ is a polynomial in $z_j$'s only 
with coefficients in the vector space $\C\langle r, r^{-1}\rangle.$  
Note that
\be\label{eq:abcbottomtop}\begin{array}{rlrl}
\tb{a}_1&=-a^3, & \tb{a}_{2m-2}&=-V_{m-1}\ab^3, \\
\tb{b}_0&=\im\gamma,& \tb{b}_{2m-2}&=-\im V_{m-1} r, \\
\tb{c}_0&=-\im, & \tb{c}_{2m-2}&=\im\gamma V_{m-1}r^{-1}.
\end{array}\ee

A direct computation shows that the closed structure equation \eqref{eq:dbhabc}  translates to
\begin{align}\label{eq:dbabc}
\ed
\bp \tb{a}(\lambda) \\ \tb{b}(\lambda) \\ \tb{c}(\lambda) \ep&=
\bp
 \im\gamma   \tb{c}(\lambda) +\im   \tb{b}(\lambda)
& \im\gamma  r^{-1}\lambda \tb{b}(\lambda) +\im r \lambda \tb{c}(\lambda) \\
 \frac{\im}{2} \gamma   \lambda^{-1} \tb{a}(\lambda)  +\frac{1}{2}z_3 \tb{b}(\lambda)
& \frac{\im}{2} r \tb{a}(\lambda)  \\
 \frac{\im}{2}  \lambda^{-1} \tb{a}(\lambda) -\frac{1}{2}z_3 \tb{c}(\lambda)
&\frac{\im}{2} \gamma r^{-1} \tb{a}(\lambda)
\ep \bp \omega \\ \omb \ep.
\end{align}
Note that $\tb{a}(\lambda)$ is divisible by $\lambda$ and  the expression $\lambda^{-1} \tb{a}(\lambda)$ is a polynomial in $\lambda$. Since the structure equation \eqref{eq:dbabc} holds for arbitrary unit complex number $\lambda$, it in fact holds for arbitrary nonzero complex number $\lambda\in\C^*$.

The lowest, and highest degree terms in \eqref{eq:abcbottomtop} indicate that at generic points on the surface $\tb{a}(\lambda), \tb{b}(\lambda),\tb{c}(\lambda)$ are polynomials of given degree. It follows from Eq.~\eqref{eq:tbXstructure} that $\lambda^{\frac{1}{2}}\tb{X}(\lambda)$ is the desired $\sla(2,\C)$-valued polynomial Killing field of degree $2m-1$.
\end{proof}

Recall the recursive structure equation for the formal Killing fields \eqref{eq:newrecursion}. For a given  linear finite-type surface of level $m$, it follows from Eq.~\eqref{eq:tbXstructure} that with the substitution
\begin{align*}
\hat{\tb{a}}_j &\longrightarrow a^{2j+1}  \\
\hat{\tb{b}}_j &\longrightarrow  b^{2j}  \\
\hat{\tb{c}}_j &\longrightarrow  c^{2j},  \quad j\geq1,
\end{align*}
the formal Killing field structure equation \eqref{eq:newrecursion} holds. By the uniqueness of canonical Jacobi fields in Prop.~ \ref{prop:closure}, this shows that
\[ a^{4m-1} =0.
\]
As a result,  
\begin{cor}
A linear finite-type CMC surface of level $m$ admits an $\so(4,\C)$-valued polynomial Killing field 
of degree $4m-2$.
\end{cor}

The closed structure equation \eqref{eq:dbabc} shows that the structure of local moduli space of  linear finite-type CMC surfaces of given level  is  fairly rigid, compare the discussion below with \cite{Fernandez2012}.
Let $\tb{P}_n$ be the vector space of polynomials in $\lambda$ of degree $n$. Then $(\lambda^{-1}\tb{a}(\lambda) ,\tb{b}(\lambda) ,\tb{c}(\lambda) )$ takes values in $\tb{P}(m):=\tb{P}_{2m-3}\times\tb{P}_{2m-2}\times\tb{P}_{2m-2}$.
Since the remaining variable $r$ satisfies
\be\label{eq:dalpha} \ed r=r\Re(z_3\omega),
\quad z_3=a^3=-\tb{a}_{1},
\ee
it follows that the system of equations \eqref{eq:dbabc}, \eqref{eq:dalpha} define a rank two Frobenius distribution on the space $\tb{P}(m)\times\R$. Here the latter $\R$ factor is for the variable $r$. In this way a linear finite-type CMC surface corresponds to an integral surface of a Frobenius system on a space of polynomials.

\section{Spectral curve}\label{sec:spectralcurve}
The Lie algebra $\sla(2,\C)$-valued polynomial Killing field can be considered as a particular form of first integral which for the class of linear finite-type surfaces represents in an invariant way the underlying algebraic structure of the prolonged structure equation. For compact linear finite-type surfaces, the geometric information extracted from the polynomial Killing field relates to a set of algebraic geometry objects, e.g., hyperelliptic curve, meromorphic differential, Jacobi variety. This set of algebraic data satisfying certain transcendental period conditions would be sufficient to invert the procedure and reconstruct the original  linear finite-type surface. 

In this section we introduce the spectral curve for compact linear finite-type surfaces and record its basic properties. The main result is that the monodromies of the associated $\C^*$-family of connections commute with each other. This agrees with the recent result of Gerding on the factorization of 
a compact high genus linear finite-type CMC surface through a branched covering of a torus, \cite{Gerding2011}.

\two
Recall
\be\label{eq:tbX2}
\tb{X}(\lambda)=
\bp
-\im \hat{\tb{a}}(\lambda) & 2\hat{\tb{c}}(\lambda)\\
2\hat{\tb{b}}(\lambda) &\im\hat{\tb{a}}(\lambda)
\ep. \ee
By construction, the polynomial Killing field coefficients $\hat{\tb{a}}, \hat{\tb{b}},\hat{\tb{c}}$ have a real involutive symmetry. 
\begin{lem}\label{lem:abcsymmetry}
\be\label{eq:abcsymmetry}
\begin{array}{rl}
V_j(\lambda)&=-V_{m-1}\lambda^{2m-2}\ol{U_{j+1}(\ol{\lambda}^{-1})}, \\
\tn{$\hat{\tb{a}}$}(\lambda)&=V_{m-1}\lambda^{2m-2}\ol{\tn{$\hat{\tb{a}}$}(\ol{\lambda}^{-1})},\\
\tn{$\hat{\tb{b}}$}(\lambda)&=-V_{m-1}\lambda^{2m-2}\ol{\tn{$\hat{\tb{c}}$}(\ol{\lambda}^{-1})},\\
\tn{$\hat{\tb{c}}$}(\lambda)&=-V_{m-1}\lambda^{2m-2}\ol{\tn{$\hat{\tb{b}}$}(\ol{\lambda}^{-1})}.
\end{array}
\ee
\end{lem}
\begin{proof}
The identities for $V_j(\lambda), \tn{$\hat{\tb{a}}$}(\lambda)$ are straightforward. For the identities for $\tn{$\hat{\tb{b}}$}(\lambda), \tn{$\hat{\tb{c}}$}(\lambda)$, one needs to take into account Eq.~\eqref{eq:cb2m}.
\end{proof}

Consider $\lambda^{\frac{1}{2}}\tb{X}(\lambda)$ as an $\sla(2,\C)$-valued function on the product space $\Sigmah\times\C^*$. The structure equation \eqref{eq:tbXstructure} shows that $\lambda\det(\tb{X}(\lambda))$ is constant along the fibers of the projection
$\Sigmah\times\C^*\to\C^*$, and it is a well defined polynomial on the $\lambda$-plane $\C^*$.

Set 
$$P(\lambda):=\det(\tb{X}(\lambda)).$$
Then
\begin{align}\label{eq:Plambda}
P(\lambda)&=\hat{\tb{a}}^2(\lambda)-4\hat{\tb{b}}(\lambda)\hat{\tb{c}}(\lambda) \\
&=\lambda^{-1}\tb{a}^2(\lambda)-4\tb{b}(\lambda) \tb{c}(\lambda) \n \\
&=\sum_{j=0}^{4m-4} P_j\lambda^j.\n
\end{align}
Since $\tb{a}(\lambda)$ is divisible by $\lambda$ and has degree $2m-2$, it does not contribute to the extremal terms $P_0, P_{4m-4}$. From Eq.~\eqref{eq:abcbottomtop} we have
\begin{align}\label{eq:P04m-4}
P_0&=-4\tb{b}_0\tb{c}_0=-4\gamma, \\
P_{4m-4}&=-4\tb{b}_{2m-2}\tb{c}_{2m-2}=-4\gamma V_{m-1}^2.\n
\end{align}
This shows that $P(\lambda)$ is a polynomial of degree $4m-4$ with nonzero $0$-th degree term (hence not divisible by $\lambda$).

\begin{cor}\label{cor:Psymmetry}

\be\label{eq:Prealinvolution}
P(\lambda)=V_{m-1}^2\lambda^{4m-4}\ol{P(\frac{1}{\ol{\lambda}})}.
\ee
The roots of $P(\lambda)$ are symmetric with respect to the unit circle in the $\lambda$-plane.  In particular, no odd degree root of $P(\lambda)$ lie on the unit circle.
\end{cor}

\one
Let $\PP^1$ be the projectivization of the $\lambda$-plane $\C^*$. Let $V=\PP^1\times \C^2$ be the trivial rank 2 vector bundle with the standard $\sla(2,\C)$ representation.  By definition (for a fixed generic point of $\Sigmah$), the polynomial Killing field can be considered as a section
$$\lambda^{\frac{1}{2}}\tb{X}(\lambda) \in H^0(\PP^1,\mco(2m-1)\otimes\tn{End}(V)).
$$
Let $\pi_{\mco(2m-1)}:\mco(2m-1)\to\PP^1$ be the projection map. Let $\mu$ be the tautological section of the pulled back bundle $\pi_{\mco(2m-1)}^*\mco(2m-1)\to \mco(2m-1)$.
Set
\be\label{eq:spectralPoly}
P(\mu, \lambda):=\det(\mu\tn{I}_2 -\lambda^{\frac{1}{2}}\tb{X}(\lambda))=\mu^2+\lambda P(\lambda).
\ee
Then 
\be\label{eq:Pmulambda}
P(\mu, \lambda)\in H^0\left(\mco(2m-1),\pi_{\mco(2m-1)}^*\mco(2(2m-1))\right).
\ee
\begin{defn}\label{defn:spectralcurve}
Let $\Sigma\hook\xinf$ be a  linear finite-type CMC surface of level $m$.
Let $\Sigmah\hook\xinfh$ be its double cover.
Let  $\lambda^{\frac{1}{2}}\tb{X}(\lambda)$, \eqref{eq:tbX2}, be the $\sla(2,\C)$-valued polynomial Killing field of degree $2m-1$ defined on $\Sigmah$.  Let  Eq.~\eqref{eq:spectralPoly} be the associated algebraic equation. 
The \tb{spectral curve} $\mathcal{C}$ of $\Sigma$ is the algebraic curve
\[\mcc=P(\mu,\lambda)^{-1}(0)\subset\mco(2m-1)
\]
in the complex surface $ \mco(2m-1)$.
\end{defn}
From the adjunction formula, the arithmetic genus $p_a$ of $\mcc$ is given by
\[ p_a=1+\frac{-4+2(2m-1)}{2}=2m-2.
\]
Let us denote the hyperelliptic involution of $\mcc$ by
\be\label{eq:involution}
\mathfrak{i}: (\mu, \lambda) \to (-\mu, \lambda).
\ee
From the symmetry  \eqref{eq:Prealinvolution}, the spectral curve $\mcc$ admits a real involution $\varrho:\mcc\to\mcc$ defined by
\be\label{eq:realinvolution}
\varrho: (\mu,\lambda)\to (V_{m-1}\ol{\lambda}^{(-2m+1)}\ol{\mu}, (\ol{\lambda})^{-1}).
\ee 

Specifically, let $\{ \lambda_i, \ol{\lambda}_i^{-1} \}_{i=1}^{\ell}$ be the set of odd degree roots of $P(\lambda)$ with multiplicity $2n_i+1$ respectively ($\vert \lambda_i \vert\ne 1$).
Let $\{ \lambda^e_i \}_{i=1}^{q}$ be the set of even degree roots of $P(\lambda)$ with multiplicity $2n^e_i$ respectively.  Counting the degrees we have
\[ 2m-2=p_a=\sum_{i=1}^{\ell} (2n_i+1) +\sum_{i=1}^{q} (n^e_i).
\]
From the symmetry \eqref{eq:Prealinvolution}, these roots with multiplicity are invariant under the involution $\lambda \to \ol{\lambda}_i^{-1}$.

\begin{defn}\label{defn:hyperellipticcurve}
Let $\mcc$ be the spectral curve of a  linear finite-type CMC surface $\Sigma\hook\xinf$ as described above.
Let $\{ \lambda_i, \ol{\lambda}_i^{-1} \}_{i=1}^{\ell}$ be the set of odd degree roots (with multiplicity $2n_i+1$) of the associated polynomial $P(\lambda)$. The \tb{hyperelliptic curve} $\hat{\mcc}$ of $\Sigma$ is the Riemann surface of the algebraic curve $\mcc$ branched over the $2\ell+2$ points $0, \infty$ and   $\lambda_i, \ol{\lambda}_i^{-1}, \, i=1, 2, \, ... \, {\ell},$ of $\PP^1$.
\end{defn}
From the Riemann-Hurwitz formula, the geometric genus $p_g$ of $\mcc$ is given by
\[ p_g=\tn{genus}(\hat{\mcc})=\ell.
\]

By a generic case we mean the case when there exist no even degree roots and the multiplicity of each odd degree roots is $2n_i+1=1$. We consequently have for the generic case $\hat{\mcc}=\mcc$, and $p_a=p_g=2m-2$.

\subsection{Eigenline bundle}\label{sec:eigenline}
One may consider the polynomial Killing field $\lambda^{\frac{1}{2}}\tb{X}(\lambda)$ simply as matrices, or operators acting on $V=\C^2$, parametrized by $\Sigmah\times\C^*$. The algebraic construction of the spectral curve can then be considered as a completion of the space of eigenvalues of these $\Sigmah\times\C^*$-family of operators in the present geometric situation, see \cite{Mulase1994}  for a lucid overview of the related ideas in terms of KP equation.  
When interpreted from the more or less trivial duality

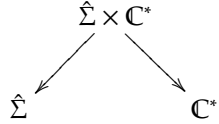
\two
\centerline{$\xymatrix@!R@R=7mm@C=3mm{
  & \ar[dl] \;\,\Sigmah\times\C^* \ar[dr] & \\
 \Sigmah &&\C^*}
$}
\two\one

\noindent
the eigenspaces form a line bundle over the spectral curve $\mcc$, the completion of certain branched double cover of $\C^*\subset\PP^1$, parametrized by $\Sigmah$.  

\two
For a point $z\in\Sigma\setminus\mcu$, consider the eigenspace 
\be\label{eq:eigenline}
\mce_z=\ker(\mu\tn{I}_2-\lambda^{\frac{1}{2}}\tb{X}(\lambda)_z)\subset V.
\ee
For $(\mu,\lambda)\in\mcc$ with $\lambda\in\C^*\subset\PP^1$ not a root of $P(\lambda)$, $\mce_z$ is a one dimensional subspace well defined at $(\mu,\lambda)$.  Since $\lambda^{\frac{1}{2}}\tb{X}(\lambda)$ is a polynomial in $\lambda$, by analytic continuation and the real involutive symmetry \eqref{eq:abcsymmetry} it follows that $\mce_z$ defines a line bundle over $\mcc$.\footnotemark
\footnotetext{Note that $\tb{X}(\lambda)$ is defined on the double cover $\Sigmah$. But under the sign change $h_2^{\frac{1}{2}}\to-h_2^{\frac{1}{2}}$ we have $\tb{X}(\lambda)\to-\tb{X}(\lambda)$ and $\mce_z$ is well defined for $z\in\Sigma$.}

In terms of a local coordinate $z$ of $\Sigma$ centered at a umbilic,
we observe from Lemma~\ref{lem:zj} that a singularity of $\lambda^{\frac{1}{2}}\tb{X}(\lambda)$ is locally of the form
\[ \frac{\tn{smooth matrix}}{z^k}+\frac{\tn{smooth matrix}}{\ol{z}^{\ell}}
\]
for positive integers $k,\ell$. Since the smooth part is again a polynomial in $\lambda$, an argument similar as above shows that the family of line bundles $\mce_z$ admit a unique analytic extension across $\mcu\subset\Sigma$, and $\mce_z$ is well defined all over $\Sigma$.
\begin{defn}
Let $\Sigma\hook\xinf$ be a linear  finite-type surface. Let $\lambda^{\frac{1}{2}}\tb{X}(\lambda)$ be the polynomial Killing field.
The \tb{eigenline bundle} $\mce_z\to\mcc$ over the spectral curve $\mcc$ of $\Sigma$ parametrized by a point $z\in\Sigma$ is the line bundle defined by the eigenspace \eqref{eq:eigenline}.
\end{defn}

\subsection{Spectral data}\label{sec:spectraldata}
Let  $\Sigma\hook\xinf$ be a compact  linear finite-type CMC surface of level $m$. We wish to argue that such $\Sigma$ with genus $\geq 2$ necessarily factors through a branched covering of a torus, and hence that  this analysis does not produce any compact high genus examples of immersed CMC surfaces.

\two
Let $\x^*\mcf\to\Sigma$ be the induced $\SO(2)$-bundle of oriented orthonormal coframes. Recall $V=\C^2$. Under the representation
\[ \chi: e^{\im t}\in\SO(2) \mapsto
\bp  e^{\frac{\im t}{2}}&\cdot \\
\cdot& e^{-\frac{\im t}{2}}
\ep \in\SL(2,\C),
\]
let $\mcv=\mcf_1\times_{\chi} V$ be the associated rank 2 vector bundle on $\Sigma$. The $\C^*$-family of $\sla(2,\C)$-valued 1-form $\psi_+(\lambda)$ on $\mcf$ then defines a $\C^*$-family of flat connection on $\mcv$. Let us denote this connection for a fixed $\lambda$ by $D_{\lambda}$.

Let $g=\tn{genus}(\Sigma)$. Let $\{a_k, b_k\}_{k=1}^g$ be a standard generator of the fundamental group $\pi_1(\Sigma)$ such that they have the intersection pairing $$(a_i, a_j)=(b_i, b_j)=0, (a_i,b_j)=\delta_{ij}.$$
Let $z_0\in\Sigma\setminus\mcu$ be a base point for $\pi_1(\Sigma)$. Let $\{A_k(\lambda)_{z_0}, B_k(\lambda)_{z_0}\}_{k=1}^g$ be the monodromy of the  flat connection $D_{\lambda}$ along $\{a_k, b_k\}_{k=1}^g$ based at $z_0$. 
Let
\begin{align*}
\alpha_k^2-\tn{tr}(A_k(\lambda))\alpha_k+1=0,\\
\beta_k^2-\tn{tr}(B_k(\lambda))\beta_k+1=0,  
\end{align*}
be the characteristic equations of the monodromies $\{A_k(\lambda)_{z_0}, B_k(\lambda)_{z_0}\}_{k=1}^g$ considered as 2-by-2 matrices.

Let $\{\alpha^{\pm 1}_k(\lambda), \beta^{\pm 1}_k(\lambda)\}_{k=1}^g$ be the corresponding eigenvalues. A different choice of base point results in the monodromies which are conjugate to the given ones, and $\{\alpha^{\pm 1}_k(\lambda), \beta^{\pm 1}_k(\lambda)\}_{k=1}^g$ are 2-valued functions on the $\lambda$-plane $\C^*$. The branching occurs at the points where the equations
\begin{align}\label{eq:ABdiscrminant}
\Delta_{A_k(\lambda)}&:= \tn{tr}(A_k(\lambda))^2-4=0,\\
\Delta_{B_k(\lambda)}&:= \tn{tr}(B_k(\lambda))^2-4=0,  \n
\end{align}
have odd degree zeros.
Set 
\be\label{eq:PhiPsi}
\Phi_k=\pm\ed\log\alpha_k,\quad  \Psi_k=\pm\ed\log\beta_k.
\ee
They are holomorphic 2-valued differentials on $\C^*$.

\two
In the high genus case the fundamental group $\pi_1(\Sigma)$ is not commutative, and we cannot directly conclude that the monodromy share a common eigenspace as was done in the tori case, \cite{Hitchin1990}. 
We make use of the polynomial Killing field $\lambda^{\frac{1}{2}}\tb{X}(\lambda)$ instead. It turns out that the analysis of the monodromy is relatively easier with this additional data. For example the finiteness of the set of branch points for the eigenvalues follows from an algebraic argument, compare this with  \cite[p642, Proposition (2.3)]{Hitchin1990}.

The following observation is crucial for the analysis in this section. For $\lambda\in\C^*$, let $\mce^{\pm}_z(\lambda)\subset V$ be the 2-valued eigenspace for $\lambda^{\frac{1}{2}}\tb{X}(\lambda)$ defined earlier.
\begin{lem}
The eigenspace $\mce^{\pm}_z(\lambda)$ is the common eigenspace for the entire monodromy $\{A_k(\lambda)_z, B_k(\lambda)_z\}_{k=1}^g$ based at $z$ for all $z\in\Sigma$. 
\end{lem}
\begin{proof}
Assume first $z\in\Sigma\setminus \mcu$, and $\lambda\in\C^*\setminus P^{-1}(0)$.
The polynomial Killing field $\lambda^{\frac{1}{2}}\tb{X}(\lambda)$ is defined on $\Sigma$ up to sign, and from the structure equation $\ed\tb{X}(\lambda)+[\psi_+(\lambda),\tb{X}(\lambda)]=0,$ it follows by considering the parallel transport along $a_k$ that
\be\label{eq:XAcommute}
\lambda^{\frac{1}{2}}\tb{X}(\lambda)\vert_z=A_k^{-1}(\lambda)_z \big(\lambda^{\frac{1}{2}}\tb{X}(\lambda)\vert_z\big) A_k(\lambda)_z.
\ee
The similar identity holds for $B_{\ell}(\lambda)_z$. The eigenspaces $\mce^{\pm}_z(\lambda)$ correspond to distinct eigenvalues for $\lambda\in\C^*\setminus P^{-1}(0)$, and the claim follows for this generic case. Since each term of Eq.~\eqref{eq:XAcommute} is analytic in $\lambda\in\C^*$, the claim is in fact true for all $\lambda\in\C^*$.
All this argument is true for $z$ away from the umbilics, and by analytic continuation it follows that $\mce^{\pm}_z(\lambda)$ is the simultaneous eigenline for all of the monodromy. 
\end{proof}

The eigenspaces of $\lambda^{\frac{1}{2}}\tb{X}(\lambda)$ are generically distinct, and by continuity we have:
\begin{cor}\label{cor:monocommute}
For a compact, linear finite-type CMC surface $\Sigma$, the set of monodromies $\{A_k(\lambda)_z, B_k(\lambda)_z\}_{k=1}^g$ of the associated $\C^*$-family of flat connections $D_{\lambda}$ commute with each other.
\end{cor}
We proceed to show that the hyperelliptic curve  $\hat{\mcc}$ 
is the spectral curve for each of the monodromies.
\begin{lem}\label{lem:eigenvaluebranching}
The odd degree zeros of the characteristic equations for the monodromies $\Delta_{A_k(\lambda)}=0, \Delta_{B_k(\lambda)}=0$, Eq.~\eqref{eq:ABdiscrminant}, occur exactly at the odd degree zeros of $P(\lambda)$ in $\C^*$.
\end{lem}
\begin{proof}
Let $\lambda^e\in\C^*$ be an even degree zero of $P(\lambda)$ (including the case $P(\lambda^e)\ne 0$). With an abuse of notation, consider the following equation in a small neighborhood $U$ of $\lambda^e$ for a generic choice of $z\in\Sigma$.
\[  A_k(\lambda)_z\mce^{\pm}_z(\lambda)=\alpha^{\pm 1}_k(\lambda)\mce^{\pm}_z(\lambda)
\]
(and similarly for $B_k(\lambda)$). Here "$\mce^{\pm}_z(\lambda)$" actually means a nonzero section of the bundle.
Since $\lambda^e$ is an even degree zero of $P(\lambda)$, the terms $A_k(\lambda)_z,  \mce^{\pm}_z(\lambda)$ in the equation above are analytic functions on $U$. Hence $\alpha_k(\lambda)$ cannot have a real branch at $\lambda^e$.

Suppose $\lambda^o\in\C^*$ is an odd degree zero of $P(\lambda)$, while on the other hand it is an even degree zero of $\Delta_{A_k(\lambda)}=0$. Arguing similarly as above, this time since $A_k(\lambda)_z,  \alpha^{\pm 1}_k(\lambda)$ are analytic functions in a neighborhood of $\lambda^o$, by elementary computation their eigenspaces are also analytically determined. But $\mce^{\pm}_z(\lambda)$ is the eigenspace of polynomial Killing field whose eigenvalue branches at $\lambda^o$, and $\mce^{\pm}_z(\lambda)$ must have a real branch at $\lambda^o$, a contradiction.
\end{proof}
\begin{cor}\label{cor:commonspectral}
The hyperelliptic (spectral) curves $\hat{\mcc}_{\alpha_k}, \hat{\mcc}_{\beta_k}$ for the eigenvalues of the monodromies  $\{A_k(\lambda)_z, B_k(\lambda)_z\}_{k=1}^g$ exist, and they are all isomorphic to the hyperelliptic curve $\hat{\mcc}$ defined earlier.
\end{cor}
We remark that this is a stringent condition on the connection form $\psi_+(\lambda)$. 
In fact such a  linear finite-type surface necessarily factors through a torus.
\begin{thm}[Gerding \cite{Gerding2011}]\label{thm:Gerding} 
A compact, linear finite-type CMC surface $\Sigma$ of genus $\geq 2$ necessarily factors through a branched covering of a torus,
$$ \Sigma \to T^2 \to X.$$
\end{thm}
We shall not pursue to give the complete proof for this theorem, and refer the reader to the original paper \cite{Gerding2011} for the remaining details.

\two
We conclude this section with a description of the eigenvalues of the monodromy.
Recall $\lambda:\mcc\to\PP^1$.
\begin{cor}
The eigenvalues $\{ \alpha_k(\lambda),\alpha^{-1}_k(\lambda), \beta_k(\lambda), \beta^{-1}_k(\lambda)\}$ 
for each of the monodromy $\{A_k(\lambda)_z, B_k(\lambda)_z\}, k=1,2,\,...\,g,$ are well defined holomorphic functions on $\mcc\setminus \lambda^{-1}\{0,\infty\}$.
In particular, the differentials \eqref{eq:PhiPsi} are holomorphic differentials on $\mcc\setminus \lambda^{-1}\{0,\infty\}$.
\end{cor}

Denote the holomorphic functions $\alpha^{\pm 1}_k(\lambda), \beta^{\pm 1}_k(\lambda)$ now defined on $\mcc\setminus \lambda^{-1}\{0,\infty\}$ by the same notation, and similarly for $\Phi_k, \Psi_k$.

\two
We examine the asymptotics of the eigenvalues $\alpha^{\pm 1}_k(\lambda), \beta^{\pm 1}_k(\lambda)$ at the two branch points $\lambda^{-1}\{0,\infty\}$. Note first the symmetry
\[ \ol{\psi_+(\ol{\lambda}^{-1})}^t =-\psi_+(\lambda).
\]
It follows that
\[  \ol{A_k(\ol{\lambda}^{-1})}^t =A_k(\lambda)^{-1}, \ol{B_k(\ol{\lambda}^{-1})}^t =B_k(\lambda)^{-1},
\]
and hence
\[  \ol{\alpha_k^{\pm 1}(\ol{\lambda}^{-1})}  =\alpha^{\mp 1}_k(\lambda),  \ol{\beta_k^{\pm 1}(\ol{\lambda}^{-1})}  =\beta^{\mp 1}_k(\lambda).
\]
The asymptotics at $\lambda^{-1}(\infty)$ are thus determined by the asymptotics at  $\lambda^{-1}(0)$, and we will examine the latter case.

Recall $\omega=\sqrt{\ff}$, the square root of the Hopf differential which is a 2-valued holomorphic 1-form on $\Sigma$. Assume that the generators of $\pi_1(\Sigma)$ lie in $\Sigma\setminus\mcu$ so that a branch of $\omega$ is well defined on each $a_k, b_k$.
Denote the periods of the chosen branch $\omega$ by
\be 
\int_{a_k}\omega=m_{a_k},\quad
\int_{b_k}\omega=m_{b_k}.
\ee
\begin{prop}
Let $\eta=\lambda^{\frac{1}{2}}$ be the local coordinate of $\mcc$ in a neighborhood of $\lambda^{-1}(0)$. The eigenvalues have the following asymptotics at $\eta=0$.
\begin{align}
\pm\log\alpha_k(\lambda)&\equiv-\frac{\im \sqrt{\gamma}}{2}m_{a_k}\eta^{-1}+\im\pi n_k,\n\\
  \pm\log\beta_k(\lambda)&\equiv-\frac{\im \sqrt{\gamma}}{2}m_{b_k}\eta^{-1}+\im\pi n'_k, \mod \mco(\eta), \n
\end{align}
where $n_k, n'_k$ are integers.
Hence
\begin{align}
\pm\Phi_k&\equiv \frac{\im \sqrt{\gamma}}{2}m_{a_k}\frac{\ed\eta}{\eta^2},\n\\
\pm\Psi_k&\equiv \frac{\im \sqrt{\gamma}}{2}m_{b_k}\frac{\ed\eta}{\eta^2}, \mod \tn{analytic terms}. \n
\end{align}
The differentials $\Phi_k, \Psi_k$ are meromorphic differentials on $\mcc$ of the second kind with double poles at $\lambda^{-1}\{0,\infty\}$ and otherwise regular.
\end{prop}
\begin{proof}
We argue in a small neighborhood of the point $\eta=0$. For the given polynomial Killing field $\lambda^{\frac{1}{2}}\tb{X}(\lambda)$,
the eigenvalues $\pm\mu$ have the local expansion
\[ \pm\mu=2\sqrt{b^2c^2}\eta+\, ... \, = 2\sqrt{\gamma}\eta+\tn{higher-order terms in $\eta$}.
\]
Choose $+\mu$, and let
\[ v=v_0+v_1\eta+\,...\,, \quad v_i=\bp v_i^1\\v_i^2 \ep
\]
be the expansion of a non-vanishing local section of the eigenspace $\mce(\eta)$ (without lower script $z$) for $\mu$, now considered as a line bundle over $\Sigma$ for a fixed $\eta=\lambda^{\frac{1}{2}}$ near $\eta=0$. A direct computation shows that we may take
\[ v_0=\bp 1\\ 0 \ep, \quad v_1=\bp v_1^1 \\ v_1^2 \ep, \; v_1^2=-\frac{\sqrt{\gamma}}{c^2}=\im\sqrt{\gamma}h_2^{-\frac{1}{2}}.
\]
The sub-bundle $\mce(\eta)\subset\mcv$ is invariant under the flat connection $D_{\lambda}$, and there exists a 1-form $\theta$ such that
\be\label{eq:thetadefi} D_{\eta^2}v=\ed v+\psi_+(\lambda) v=\theta v.
\ee
Let $\theta=\theta_{-1}\eta^{-1}+\theta_0+\, ... \, $ be the expansion of $\theta$ in $\eta$ (there are no higher degree negative terms like $\theta_{-2}\eta^{-2}$, etc). Collecting the $\eta^{-1}$-terms from \eqref{eq:thetadefi} one gets
\[\theta_{-1}=-\frac{v_1^2}{2}h_2\xi=-\frac{\im \sqrt{\gamma}}{2}h_2^{\frac{1}{2}}\xi.
\]
Integrating this over the cycles $a_k, b_k$ gives the desired formulae. Note that away from $\eta=0$ the monodromy is a function of $\lambda=\eta^2$, and this explains the value of constant terms in $\im\pi\mathbb{Z}$.
\end{proof}
The present analysis shows that 
Gerding's result on the factorization of a linear finite-type high genus CMC surface is quite plausible.
It also gives an indication of how to construct the factored CMC torus. 

 

\section{Nonlinear finite-type surfaces}\label{sec:nft}
In this section we introduce a class of nonlinear finite type CMC surfaces.
The umbilics are allowed, and by Rem.~\ref{rem:umbilicandFT} in the below,
these surfaces are generally not linear finite type.
The Smyth surfaces (Mr. and Mrs. Bubbles) with rotationally symmetric metric
and an umbilic point of arbitrary degree at the center are the well known examples in this class.

\subsection{Flat structure 3-web}
Recall that a planar $m$-web is a set of $m$ pairwise transversal foliations on a 2-dimensional surface.
\begin{defn}\label{defn:kthHopf}
Given a CMC surface $\Sigma$,  the \tb{$m$-th order Hopf differential} is the section 
$$\Phi_m:=h_m\xi^m\in H^0(\Sigma, K^m).$$
In case $\Phi_m$ does not vanish identically, the \tb{$m$-th order structure web} $\mcw_m$ is the $m$-web of possibly singular foliations defined by the line fields
$$\tn{Im}(\Phi_m)=0.$$
\end{defn}
Note the structure web $\mcw_m$ has the property that the corresponding $m$ line fields divide the tangent plane at each point in $m$ equal angles.

Compared to the general $m\geq 4$ webs, the local geometry of a planar $3$-web up to diffeomorphism is relatively simple;  the primary local invariant is the web curvature 2-form. The class of CMC surfaces we shall be interested in are those for which the structure 3-web $\mcw_3$ is flat and the web curvature vanishes.

\two
Let us give a brief analytic derivation of the web curvature of a 3-web. We refer the reader to 
\cite{Chern1982} \cite{Pereira2009}   for the details on planar 3-webs. 
Up to scaling, suppose that a planar 3-web $\mcw$ is locally defined by 1-forms 
$\eta^a, a=1,2,3,$ on a surface satisfying the normalization
\be\label{eq:etasum}
\eta^1+\eta^2+\eta^3=0.
\ee
In this case, there exists a unique web connection 1-form $\psi$ such that
$$\ed\eta^a=\psi\w\eta^a, \quad a=1,2,3.$$
The web curvature 2-form $\Omega_{\mcw}$ is defined by
$$\Omega_{\mcw}=\ed\psi.$$
It is easily checked that $\Omega_{\mcw}$ is independent of choice of $\eta^a$'s satisfying Eq.~\eqref{eq:etasum}.

\two
The CMC surfaces with flat structure 3-web are analytically characterized by the following nonlinear equation on the structure functions.
\begin{lem}\label{lem:flat3web}
Given a CMC surface, suppose the third order Hopf differential $\Phi_3$ is nonzero 
and the structure 3-web $\mcw_3$ is well defined. 
Then $\mcw_3$ is flat whenever the structure functions satisfy
\be\label{eq:flat3web}
\tn{Im}\left(\frac{z_4}{z_3^2}\right)=0.
\ee
\end{lem}
\begin{proof}
Assuming $h_3\ne 0$, choose a frame for which $h_3=\hb_3$ is real. Differentiating this equation, one gets
$$\rho=-\frac{\im}{3(h_3+\hb_3)}\left((h_4-\hb_2 R)\xi-(\hb_4-h_2 R)\xib\right).$$
By inspection the web connection 1-form $\psi=*\rho$, where '$*$' is the $\C$-linear Hodge star operator acting on 1-forms by $*\xi=-\im \xi.$ The web curvature 2-form is then given by
$$\Omega_{\mcw_3}=\ed(*\rho)=\frac{R}{3h_3\hb_3}(\hb_2\hb_4-h_2h_4)\xi\w\xib.
$$
We are assuming that $\Phi_3=h_3\xi^3\ne 0$, and hence the curvature $R$ does not vanish identically either.
Since the function $\tn{Im}\left(\frac{z_4}{z_3^2}\right)$ in Eq.~\eqref{eq:flat3web} is well defined on the CMC surface, this is equivalent to $\hb_2\hb_4-h_2h_4=0$ for this particular choice of frame.
\end{proof}
\begin{rem}
Note 
\be\label{eq:dh23}
\ed (h_2\hb_2)\w\ed(h_3\hb_3)=
-h_3\hb_3( \frac{z_4}{z_3^2}-\frac{\zb_4}{\zb_3^2})\xi\w\xib.
\ee
\end{rem}
Differentiating Eq.~\eqref{eq:flat3web}, one may solve for $h_5,\hb_5$ in terms of the lower order variables
$\{\,h_2,\hb_2,h_3,\hb_3, h_4,\hb_4\,\}$. 
Thus the structure equation  closes up at order $5$ and 
Eq.~\eqref{eq:flat3web} leads to a nonlinear finite type equation. 
We claim that the resulting structure equation is compatible.
\begin{prop}\label{prop:flat3webcompat}
Consider the system of equations generated by Eq.~\eqref{eq:flat3web} and its $\delxb, \delx$-derivatives.
The resulting putative structure equation, which is closed at order 5, is compatible and formally satisfies $\ed^2=0$.
By Lie-Cartan theorem there exist local solutions for initial conditions satisfying $h_2, h_3\ne 0.$
\end{prop}
\begin{proof}
Eq.~\eqref{eq:flat3web} is equivalent to
\be\label{eq:flat3web'}
h_2\hb_3^2h_4=\hb_2h_3^2\hb_4.
\ee
Solving for $h_4, \hb_4$ we get
$$
h_4 = \frac{\hb_2h_3^2\hb_4}{h_2\hb_3^2},\quad
\hb_4 =\frac{h_2\hb_3^2h_4}{\hb_2h_3^2}.
$$
Differentiating this we get
$$
h_5 = \delx\left(\frac{\hb_2h_3^2\hb_4}{h_2\hb_3^2}\right), \quad
\hb_5 =\delxb\left(\frac{h_2\hb_3^2h_4}{\hb_2h_3^2}\right).
$$
Substituting this to the structure equation for $\ed h_4, \ed\hb_4$, the claim is verified by direct computation.
\end{proof}

\subsection{Un-coupled structure equation}
The defining equation ~\eqref{eq:flat3web'} represents a relation among the phases of the complex coefficients $\{\,h_2, h_3, h_4\,\}.$ This suggests to choose a local frame such that
\begin{align}\label{eq:uncoupledparam}
h_2&= A e^{\im\phi},      \\
h_3&= B e^{\im\phi},      \n    \\
h_4&= C e^{\im\phi},    \n
\end{align}
where $A,B,C,\phi$ are real valued. 
Let $\xi=\omega^1+\im\omega^2$ be the real and imaginary parts.
Then the structure equation un-couples in this parametrization as follows.
\begin{align}\label{eq:uncoupledstrt}
\ed A&=B\omega^1,   \\
\ed B&=(C-A^3+\gamma^2 A)\omega^1, \n \\
\ed C&=\left(\frac{2C^2}{B}+\frac{(-B^2+2A^4-2\gamma^2A^2)C}{BA}+(5\gamma^2 B-7BA^2)\right)\omega^1, \n \\
\ed\phi&=\frac{(-2CA+3B^2-2A^4+2\gamma^2A^2)}{BA}\omega^2, \n \\
\rho&=\frac{(CA-B^2+A^4-\gamma^2A^2)}{BA}\omega^2. \n
\end{align}

From the structure equation $\ed\omega^1=-\rho\w\omega^2=0$,  
and one may write locally $\omega^1=\ed s$ for a parameter $s$.
Assuming $\frac{(-2CA+3B^2-2A^4+2\gamma^2A^2)}{BA}\ne 0$, one may take $(s,\phi)$ as a local coordinate and write $\omega^2 =q(s) \ed\phi$, where $q(s)=\frac{BA}{(-2CA+3B^2-2A^4+2\gamma^2A^2)}.$
It follows that the metric of the  CMC surface is written in this coordinate in warped product form
$$\xi\xib=(\omega^1)^2+(\omega^2)^2=(\ed s)^2+q(s)^2(\ed\phi)^2.
$$

This in particular implies that the metric admits a Killing field. To see this,
let $\{\,e_1, e_2\,\}$ be the dual frame of $\{ \, \omega^1, \omega^2\,\}$.
Let $f$ be a function which satisfies
$$\ed f + f \frac{(CA-B^2+A^4-\gamma^2 A^2)}{BA}\omega^1=0.$$
Then it is easily checked that the vector field $f e_2$ is a Killing field of $\xi\xib.$
Note that the leaves of the $e_1$-foliation are geodesics.

\two
In order to see how this surface may look like, consider for a concrete example the case of a minimal surface  in the 3-sphere, $\delta=0, \epsilon=1$.
Let $\tn{v}=(\tn{v}_0,\tn{v}_1, \tn{v}_2, \tn{v}_3)$ be the adapted $\SO(4)$-frame along the minimal surface
such that $\tn{v}_0$ describes the surface, $\tn{v}_3$ is normal to the surface, and $\{ \tn{v}_1, \tn{v}_2\}$ are tangent to the surface that correspond to the frame $\{ e_1,e_2 \}$. A computation shows that they satisfy the following structure equation.
$$
\ed\tn{v}\equiv\tn{v}
\bp
\cdot &-1&\cdot&\cdot\\
1&\cdot&-A\cos(\phi)&\cdot\\
\cdot &A\cos(\phi)&\cdot&-A\sin(\phi)\\
\cdot &\cdot&A\sin(\phi)&\cdot
\ep\omega^1 \mod \omega^2.
$$
As the parameter $\phi$ varies from $0$ to $\pi$, the corresponding leaves of the geodesic $e_1$-foliation oscillate between a planar curve lying in a totally geodesic $\s{2}\subset\s{3}$ 
and a great circle. The leaves in-between these two extremes are generalized helices.

\subsection{Examples}
\subsubsection{Cohomogeneity-$1$ surfaces} 
Consider the degenerate case $\ed\phi=0$ and the phase is constant.
It is known that this is equivalent to that the CMC surface is invariant under a Killing field of the ambient space form. The cohomogeneity-$1$ CMC surfaces are well understood, \cite{Brendle2013}.

\subsubsection{Torus} 
We claim that if a compact torus belongs to this class, then it is necessarily of cohomogeneity-$1$.

Choose a uniformization $z$ of the given torus so that the Hopf differential is
$$h_2\xi^2=A e^{\im\phi}(\omega^1+\im\omega^2)^2=(\ed z)^2.
$$
When restricted to an $e_1$-curve, this becomes
$$ A e^{\im\phi}(\omega^1)^2=(\ed z)^2.$$
Since $\ed\phi\equiv 0\mod\omega^2$, the phase $e^{\im\phi}$ is constant along the $e_1$-curve.
Hence the locus of each $e_1$-curve on the torus is a straight line with respect to the uniformization $z$.

Suppose $\ed\phi\not\equiv0$. Then by continuity there exists a $e_1$-curve which is dense on the torus. But since $e^{\im\phi}$ is constant on the curve, this implies that $e^{\im\phi}$ is constant everywhere, a contradiction.

\two
Since the metric on the surface must admit a Killing field, it is unlikely that
a compact high genus CMC surface belongs to this class either.

\subsubsection{Smyth surfaces}\label{sec:Smyth}
Brian Smyth discovered a one parameter family of complete CMC-1 planes in $\E^3$ with  radial metric
and with an umbilic of arbitrary degree at the origin (center), \cite{Smyth1993}\cite{Timmreck1994}.
They are also known as Mr. and Mrs. Bubbles, \cite{Geometrie}.
We refer the reader to \cite{Reis2003} for the analytic details on Smyth surfaces.

Since the metric is invariant under a one parameter family of symmetry,
the curvature must satisfy the functional relation
$$\ed R \w \ed(\vert \nabla  R \vert )=0.$$
It is easily checked that this is equivalent to \eqref{eq:dh23},
and a Smyth surface has the flat structure 3-web.

\np
\part{Periods}\label{part:period}
\section{Fundamental divisor}\label{sec:divisor}
We now turn our attention to the geometry of the periods of conservation laws for compact high genus CMC surfaces.
In this section we start by recording a relevant parity property of the zero divisor of the square root of Hopf differential.
This is an elementary observation. We present the details here to fix notation and for future reference.

\two
Let $\Sigma$ be a compact CMC surface of genus $g\geq 2$. Let
$$\pi:\Sigmah\to\Sigma$$
be the double cover defined by the square root $\omega=\sqrt{\ff}\in H^0(\Sigmah,\hat{K})$. 
Let $\hat{g}$ denote the genus of $\Sigmah.$ 
 Recall that  $\pi$ is branched over the odd degree zeros of the Hopf differential.
\begin{defn}
Let $\Sigma$ be a compact CMC surface of genus $\geq 2$. 
Let $\pi:\Sigmah\to\Sigma$ be the double cover defined by $\omega=\sqrt{\ff}.$ 
The \tb{fundamental divisor} of $\Sigmah$ is the divisor (of zeros)
$$\hat{\mcu}:=(\omega)_0\in \Sigmah^{(2\hat{g}-2)}.$$
Note by definition the umbilic divisor $ \mcu=(\ff)_0 = \pi_*(\hat{\mcu}).$
\end{defn}
\one
We claim that 
\begin{lem}
Without counting multiplicities, the fundamental divisor $\hat{\mcu}$ consists of even number of points.
\end{lem}
\begin{proof}
Counting the total number of zeros of the quadratic differential $\ff$ on $\Sigma$ with multiplicities, by Riemann-Roch we have
\be\label{eq:ffzeros}
4g-4=\Sigma_{i=1}^{\ell^e} d_i^e + \Sigma_{i=1}^{\ell^o} d_i^o.
\ee
Here 

\begin{center}

\qquad\qquad\;$\ell^e$: number of even degree zeros of $\ff$,

\qquad\qquad$\ell^o$: number of odd degree zeros of $\ff$. 
\end{center}

\one\noi
The corresponding degrees of zero are denoted by $d_i^e,\, d_i^o$ respectively. 
From the parity of the equation \eqref{eq:ffzeros}, we have

\begin{center} $\ell^o$ is even.\end{center}

\noi
By definition $\ell^o$ is the total branching number of the double cover $\Sigmah\to\Sigma$. 
By  Riemann-Hurwitz formula one consequently gets the genus formula
\be\label{eq:genushat}
\hat{g}=2g-1+\frac{\ell^o}{2}.
\ee

Counting the total number of zeros of the 1-form $\omega$ on $\Sigmah$ with multiplicities, by Riemann-Roch we have
\be\label{eq:omegazeros}
2\hat{g}-2=\Big(\Sigma_{i=1}^{\ell^e} \frac{d_i^e }{2}+\Sigma_{i=1}^{\ell^e} \frac{d_i^e }{2}\Big)
+ \Sigma_{i=1}^{\ell^o} (d_i^o+1).
\ee
Here the two identical sums $\Sigma_{i=1}^{\ell^e} \frac{d_i^e }{2}$ in the parenthesis represent the zeros (with multiplicities) on the two distinct copies of $\Sigmah$ over the even degree zero points of $\ff$. This yields that the total number of zero points of $\omega$ is
\begin{align}
\tn{$\sharp$ of zero points (without multiplicities) of $\omega$ on $\Sigmah$}:=r&=2\ell^e+\ell^o\n\\
&=\tn{even}.\n
\end{align}
\end{proof}

\section{Jacobi fields and conservation laws at umbilics}\label{sec:umbilics}
The coefficients of the sequence of higher-order conservation laws $\varphi_n$ are 
up to scale polynomials in the variables $z_j=h_2^{-\frac{j}{2}}h_j.$
As a result they are singular at the fundamental divisor.
We examine their singularities and give an interpretation as smooth sections 
of the twisted line bundle $\Omega^1(\Sigmah)\otimes(\hat{K})^{4n}.$
This  leads to the new concept of twisted characteristic cohomology.
Locally we establish a refined normal form for the conservation laws
adapted at the umbilics. This indicates that the residues are generally non-trivial
and detect certain higher-order quantitative invariants for CMC surfaces at the umbilics.

 
\subsection{Twisted conservation law}\label{sec:twistedcvlaw}
Recall the asymptotics Eq.~\eqref{eq:abcnormal} of the formal Killing coefficients:
\begin{align}
B^{2n}&:=\im h_2^{\frac{1}{2}} b^{2n}
 =   \frac{ (-1)^{n}}{2\gamma^{n-2}} \left[  z_{2n} + \,(\tn{lower order terms}) \,+t^{2n}_b z_3^{2n-2}\right], \label{eq:bcanormal} \\
C^{2n}&:=\im h_2^{-\frac{1}{2}} c^{2n}
 =  \frac{ (-1)^{n }}{2\gamma^{n-1}} \left[ z_{2n}+ \,(\tn{lower order terms}) \,+t^{2n}_c z_3^{2n-2}\right], \n\\
A^{2n+1}&:=\frac{a^{2n+1}}{2}
 =\frac{ (-1)^{n-1 }}{2\gamma^{n-1}} \left[ z_{2n+1} + (\tn{lower order terms}) \,+t^{2n+1}_a z_3^{2n-1}\right].\n
\end{align}
Here $\{\, B^{2n}, C^{2n}, A^{2n+1}\,\}$ are the coefficients of the enhanced prolongation from Sec.~\ref{sec:kfpro}.
Since $z_j=h_2^{-\frac{j}{2}}h_j$, they appear to be singular at the fundamental divisor.

On the other hand $\omega=h_2^{\frac{1}{2}} \xi$ is smooth.
Accordingly, when restricted to a CMC surface, we may consider them as smooth sections of the powers of canonical line bundle 
$\hat{K}\to\Sigmah$.
\begin{lem}\label{lem:BCAtwisted}
Consider the formal Killing coefficients twisted by powers of $\omega$,
$$\{\, B^{2n}\omega^{4n-4},\;C^{2n}\omega^{4n-4}, \; A^{2n+1}\omega^{4n-2} \,\}.$$
Then they are smooth sections of $\hat{K}^{4n-4}, \hat{K}^{4n-4}, \hat{K}^{4n-2}$ respectively. 
\end{lem}
\begin{proof}
For the initial case $n=1$, we have
$ ( B^{2}\omega^{0},\;C^{2}\omega^{0}, \; A^{3}\omega^{2})
= (\gamma, -1, \frac{1}{2}h_3h_2^{-\frac{1}{2}}\xi^2)$.
We claim that $h_3h_2^{-\frac{1}{2}}\xi^2$ is smooth.
 
Consider first an even degree umbilic point of $\ff$ so that a section $\omega=h_2^{\frac{1}{2}}\xi$ of $K$ is locally defined (on $\Sigma$). 
Take a local coordinate $z$ such that $\xi=e^u\ed z, \eta_1=e^{-u} z^{m}\ed z$, and 
$\ff=\eta_1\circ\xi=z^{m}(\ed z)^2, \, m\geq 1.$
From the structure equation one finds
\begin{align*}
h_2&=e^{-2u} z^{m}, \\
h_3&\equiv m e^{-3u} z^{m-1}\mod z^{m}.
\end{align*}
Hence $h_3h_2^{-\frac{1}{2}}\equiv m e^{-2u} z^{\frac{m}{2}-1}$ modulo smooth terms.
This suffices for the case $m\geq 2$ is even. 

For an odd degree umbilic point, choose a local coordinate $w$ on $\Sigmah$ such that $z=w^2$ and
$\xi=2e^{u}w\ed w.$ Then 
$h_3h_2^{-\frac{1}{2}}\xi^2\equiv 4 m   w^m (\ed w)^2$ modulo smooth terms.

For the general case $n\geq2$, consider $B^{2n}\omega^{4n-4}$.
From the weighted homogeneity, one only needs to check the lowest order term
$z_3^{2n-2} \omega^{4n-4}=(h_3h_2^{-\frac{1}{2}}\xi^2)^{2n-2}.$
The other cases follow similarly.
\end{proof}
\begin{cor}\label{cor:varphintwisted}
Let $\Sigma\hook \xinf$ be a CMC surface.
Let $\Sigmah\to\xinfh$ be its double cover.
Let $\varphi_n=c^{2n+2}\xi+b^{2n}\xib, n=0,1, \, ... \,,$ be the sequence of 1-forms representing the higher-order conservation laws.
Then $\varphi_n\otimes \omega^{4n}$ is a smooth section of the twisted line bundle
$$\varphi_n \otimes \omega^{4n} \in H^0(\Omega^1(\Sigmah)\otimes (\hat{K})^{4n}).$$
\end{cor}
This shows that from the global perspective  
it is more natural to think of the higher-order conservation laws 
as smooth sections of the twisted line bundle $\Omega^1(\Sigmah)\otimes (\hat{K})^{4n}$.
It is only if we want to think of them globally as untwisted 1-forms 
that we must allow singularities, similarly as for the homomorphic 1-forms on the projective space $\C\PP^N$.
The singularities are just an implication of topology on global analytic matters. 
This also suggests to introduce a general notion of twisted characteristic cohomology.

\subsection{Residue}\label{sec:residue}
The residues of the 1-forms $\varphi_n$ which are singular at the umbilics
can be used in turn to detect and give invariants for  the umbilics.
In this section we give a refined normal form for the conservation laws
at the umbilics  toward understanding their residues.

\two
We follow the local formulation of CMC surfaces adapted to an umbilic point from Sec.~\ref{sec:localextension},  Eqs.~\eqref{eq:strtum}, \eqref{eq:PDEfum}.
Given a CMC surface $\Sigma$, choose a local coordinate $z$ centered at an umbilic point such that
\be\label{eq:localum}
\xi=e^{u}\ed z, \;\eta_1=e^{-u} z^p \ed z
\ee
for a real valued function $u$ and a positive integer $p$. 
The Hopf differential is written in this coordinate as
\be\label{eq:Hopfum}
\ff=\eta_1\circ\xi=z^p (\ed z)^2.
\ee
Introduce a square root $w$  such that $z=w^2$.
At an even degree umbilic ($p$ even, no branching),
the computation of residues in $w$-coordinate would only double the values 
and hence does not affect the vanishing/non-vanishing argument.

We claim that in terms of $w$ the formal Killing field coefficients have the following normal form.
Here the notation $\mco(w^{\ell})$ means a polynomial in $w$ of degree $\leq \ell$
with coefficients in the ring of functions in  $u$ and its successive $z$-derivatives.
\begin{lem}\label{lem:formalKillingum}
With respect to the local formulation of CMC surfaces in Eqs.~\eqref{eq:localum},\eqref{eq:Hopfum}
adapted to an umbilic point, 
the formal Killing field coefficients have the following normal form.
\begin{align}\label{eq:abcinw}
b^{2n}&=\frac{1}{w^{(2n-1)p+(4n-4)}}\mco(w^{4n-4}), \\
c^{2n}&=\frac{1}{w^{(2n-3)p+(4n-4)}}\mco(w^{4n-4}), \n\\
a^{2n+1}&=\frac{1}{w^{(2n-1)p+(4n-2)}}\mco(w^{4n-2}), \qquad  n\geq 1. \n
\end{align}
Here $z=w^2$.
\end{lem}
\begin{proof}
Recall the notations and the inductive formulae for the Killing coefficients from 
Sec.~\ref{sec:inductiveformula}.
In terms of the normal form \eqref{eq:abcinw} a short computation shows that 
we would have
$$\hat{\tn{a}}_{ij}=\frac{1}{w^{2(i+j)p+4(i+j)}}\mco(w^{4(i+j)}).
$$
Hence assuming the claim \eqref{eq:abcinw} is true up to $n$,  we have
$\hat{\tn{s}}_n=\frac{1}{w^{2np+4n}}\mco(w^{4n})$. 
It follows that
$$\hat{\tn{m}}_{n}=\frac{1}{w^{2np+4n}}\mco(w^{4n}).
$$
Note $h_2=e^{-2u} z^p=e^{-2u} w^{2p}.$
Substituting these equations to Eqs.~\eqref{eq:bc2n+2}, \eqref{eq:a2n+3}, 
the claim is verified by induction. We omit the details.
\end{proof}
\begin{rem}\label{rem:umbilicandFT}
One finds in the normal form above that up to nonzero scale
the numerators $\mco(w^{4n-4}),\mco(w^{4n-2})\mod w$ are constants depending only on $p$.
Examining the explicit normal form for the Jacobi field $a^{2n+1}$ for the first few terms $n=1, 2, 3, \, ... $, 
it is easily checked (and can be proved) that 
as $w\to 0$ this limit of the numerator $\mco(w^{4n-2})$ for $a^{2n+1}$ 
is nonzero for any positive value of $p$
(it may become zero for certain negative integer values of $p$).
This shows that a CMC surface with an umbilic point cannot be of linear finite type.
\end{rem}
Note that $\xi \sim w\ed w, \xib \sim \ol{w}\ed\ol{w}$ up to nonzero scale.
From Lem.~\ref{lem:formalKillingum}, the corresponding refined normal form for conservation laws follow.
\begin{cor}\label{cor:varphininw}
With respect to the local formulation of CMC surfaces in Eqs.~\eqref{eq:localum},\eqref{eq:Hopfum}
adapted to an umbilic point, 
the higher-order conservation laws have the following normal form.
Here $z=w^2$.
\be\label{eq:varphininw}
\varphi_n=\frac{1}{w^{(2n-1)p+(4n-1)}}\mco(w^{4n}) \ed w 
+\frac{1}{w^{(2n-1)p+(4n-4)}}\mco(w^{4n-4})  \ol{w}\ed \ol{w},\quad n\geq 0.
\ee
\end{cor}

Recall $$\omega=h_2^{\frac{1}{2}}\xi=2w^{p+1}\ed w.$$
It is  clear from this that $\varphi_n \otimes\omega^{4n}$ is a smooth section of 
$\Omega^1(\Sigmah)\otimes (\hat{K})^{4n}.$
 
\two
Since the curvature of the induced metric is $R=\gamma^2-h_2\hb_2=\gamma^2-e^{-4u}\vert z \vert^{2p}$, 
the coefficients of  the polynomials $\mco(w^{4n}), \mco(w^{4n-4})$ in the normal form for $\varphi_n$ 
are elements in the ring of functions
generated by the curvature $R$ and its successive $z$-derivatives.
The residues for the conservation laws 
in this way capture  certain higher-order numerical invariants of the curvature at the umbilics.
 
\two
We check the  case $\varphi_1$ for an example.
By a direct computation one finds
$$\varphi_1=\frac{\im}{w^p}\left[ 
\frac{\left(3 {p}^{2}+4\,p -16pu_{{1}} w^{2} 
+16(u_{2}+u_{1}^{2})w^{4} \right)}{4\gamma w^3} \ed w
- 2\gamma e^{2u}\ol{w}\ed\ol{w}  \right]
$$
(here $u_1=u_z, u_2=u_{zz},$ etc. Recall $u$ is defined on $\Sigma$).
Consider the Taylor  series expansion
\begin{align*}
 u&=\mbox{$\sum_{i,j=0}^{\infty} c_{ij}z^i\zb^j$},\\
&=\mbox{$\sum_{i,j=0}^{\infty} c_{ij}w^{2i}\ol{w}^{2j}$}. 
\end{align*}
Then the residue of $\varphi_1$ at $w=0$ is expressed by an algebraic formula 
involving the coefficients $c_{ij}$'s.
Similar interpretation would hold for the higher-order conservation laws.

\two
As this observation may show, the residues of conservation laws are 
subtle invariants. We currently do not have any further information regarding their behavior , e.g., when they vanish.
\begin{exam}[Smyth surfaces from Sec.~\ref{sec:Smyth}]
We show that for the umbilics of Smyth surfaces
the residues of conservation laws are all trivial.

Recall the parametrization \eqref{eq:uncoupledparam} and the un-coupled structure equation \eqref{eq:uncoupledstrt}. A direct computation shows that
$$z_k=\{A,B,C\} e^{-\left(\frac{k-2}{2}\right)\im\phi},
$$
where $\{A,B,C\}$ is a generic notation for a function in the variables $A,B,C$  only (not involving $\phi$).
From this we have
\begin{align}
c^{2n+2}&=\{A,B,C\} e^{-\left(\frac{2n-1}{2}\right)\im\phi},\n\\
b^{2n}&=\{A,B,C\} e^{-\left(\frac{2n-1}{2}\right)\im\phi}.\n
\end{align}
Hence the conservation laws are of the form
\begin{align}
\varphi_0&\equiv e^{ \frac{1}{2} \im\phi} \{A,B,C\} \ed\phi, \n\\
\varphi_n&\equiv e^{-\left(\frac{2n-1}{2}\right)\im\phi}  \{A,B,C\}\ed\phi,\mod\omega^1.\n
\end{align}

Given a Smyth surface with an umbilic at the center, 
the intrinsic $\s{1}$ metric symmetry is in the direction of $(\omega^1)^{\perp}$.
It follows that the residue of a conservation law is obtained by setting $\omega^1=0$ and integrating with respect to $\ed\phi$. Note that $\ed A, \ed B, \ed C\equiv 0\mod \omega^1$ 
and they are constants on such locus of integration (an orbit of  $\s{1}$ isometry).

By definition $\varphi_0=\sqrt{\ff}$ has no residue. 
From the given normal form for $\varphi_0$ above, the interval of integration for $\phi$ then must be an integer multiple of $4\pi$. Equivalently, a Smyth surface with umbilic 
(or its double cover in case of odd degree umbilic) is locally a union of even number of sectors parametrized 
as in \eqref{eq:uncoupledparam} for $\phi\in[0, 2\pi].$
\end{exam}
  
We remark that  the conservation laws  may  have nontrivial residues at the poles of the Hopf differential.
They therefore have possible applications
to detect and distinguish the ends of complete CMC surfaces too.
 
\section{Abel-Jacobi maps}\label{sec:abel}
In this section we initiate the study of periods of conservation laws.
We introduce a natural rationality criterion for compact CMC surfaces 
in terms of an Abel-Jacobi map defined by the periods.

Assume that $\Sigma$ is a compact CMC surface of genus $g\geq 1.$
We continue the analysis from Sec.~\ref{sec:divisor}.

\subsection{Spectral Jacobian}\label{sec:spectralJacobian} 
Recall $r$ is the number of zero points of $\omega$. Set a positive integer
\begin{align} m&=\frac{1}{2}(r+2\hat{g})=\frac{r}{2}+\hat{g} \in \mathbb{Z}_+ \\
&=\ell^e+\frac{\ell^o}{2}+\Big(2g-1+\frac{\ell^o}{2}\Big)\n\\
&=2g-1+(\ell^e+\ell^o)\leq 2g-1+(4g-4)=6g-5.\n
\end{align}
Note that when $g\geq 2$, $m$ is bounded below by
$$2g\leq m.$$ 


Recall the sequence of higher-order conservation laws 
$$\varphi_j, \bar{\varphi}_j, \quad j=0,1,\, ... \, .$$
Consider a set of 
$$n\geq m$$ 
linearly independent conservation laws
\begin{align}
\phi^A&=\sum_{j=0}^{\infty} \Big(U^A_j\varphi_j+V^A_j\bar{\varphi}_j\Big)+\phi^A_{-1},
\quad A=1,2,\, ... \, n,\n\\
&\quad\phi^A_{-1}\in\mcc^{0}.\n
\end{align}
Here $U^A_j,  V^A_j$'s are constant coefficients with finitely many nonzero terms.
Set the column of conservation laws
\be\label{eq:veccv}
\vec{\phi}=(\phi^1,\phi^2,\, ... \, \phi^n)^t.
\ee

\one
Let $\{ \hat{p}_i \}_{i=1}^{r}\subset\Sigmah$ be the set of zero points of $\omega.$ Define the residues
\be
\tn{r}_i =\tn{Res}_{\hat{p}_i}(\vec{\phi}) \in\C^n,\quad i=1,2,\, ... \, r.
\ee
By Stokes theorem we have
$$\sum_i \tn{r}_i=0.$$

Let $\{ \hat{a}_k, \hat{b}_k \}_{k=1}^{\hat{g}}$ be a basis of $H_1(\Sigmah,\Z)$.
Define the periods
\be
\hat{\tn{a}}_k=\int_{a_k} \vec{\phi}, \quad \hat{\tn{b}}_k=\int_{b_k} \vec{\phi}, 
\qquad k=1,2,\, ... \, \hat{g}.\n
\ee
By definition the periods are defined modulo the integral lattice generated by the set of residues.
\begin{defn}
Let $\Sigma$ be a compact CMC surface. Let $\Sigmah\to\Sigma$ be its double cover.
For a choice of column of conservation laws \eqref{eq:veccv},
the associated \tb{spectral lattice}  
is the integral lattice generated by
\be
L_{\vec{\phi}} =\Z\langle \tn{r}_i, \hat{\tn{a}}_k, \hat{\tn{b}}_k \rangle  \subset \C^n. \n
\ee
\end{defn} 
The rank of the lattice $L_{\vec{\phi}}$ is bounded above by
$$(r-1)+2\hat{g},$$ 
which is odd for $r$ is known to be even. Note the condition $n\geq m$ ensures that
$$\tn{dim}_{\R}\C^n=2n\geq 2m =r +2\hat{g} >(r-1)+2\hat{g}\geq \tn{rank}(L_{\vec{\phi}}).
$$
\begin{defn}\label{defn:spectralJacobian}
For a  compact CMC surface $\Sigma$ with a choice of column of conservation laws \eqref{eq:veccv},
the associated \tb{spectral Jacobian}   is the quotient space
\be
T_{\vec{\phi}}=\C^n/L_{\vec{\phi}}.
\ee
\end{defn}
Similarly as in the Riemann surface theory, the spectral Jacobian comes equipped with a canonical Abel-Jacobi map.
\begin{defn}\label{defn:AJmap}
For a  compact CMC surface $\Sigma$ with a choice of column of conservation laws \eqref{eq:veccv},
the associated \tb{Abel-Jacobi map} is defined as follows:
pick an arbitrary base point $\hat{p}_0\in\Sigmah$ away from the fundamental divisor $\hat{\mcu}$.
By construction one has the well defined map
\be
\tn{u}_{\vec{\phi}}:\Sigmah\setminus\hat{\mcu} \to T_{\vec{\phi}} \n
\ee
defined by
\be
\tn{u}_{\vec{\phi}}(\hat{p}) = [ \int_{\hat{p}_0}^{\hat{p}}\vec{\phi}\;] \n\quad  \tn{for $\hat{p}\in\Sigmah.$ }
\ee
Here $[\;\cdot\;]$ denotes the equivalence class of a vector in $\C^n$ as a point in $T_{\vec{\phi}}.$
\end{defn}

In the following we shall examine the properties of the spectral Jacobian and its associated objects
for a simple choice of  column of conservation laws.

 



  
\subsection{CMC surfaces of rational type}\label{sec:rationalCMC}
We consider the column of higher-order truncated conservation laws $\phi=(\varphi_1,\varphi_2, \, ... \, \varphi_n)^t.$ To this end, set 
$$\ol{H}^{(1,0)}_n:=\C\langle\varphi_1,\varphi_2,\, ... \, \varphi_n\rangle\simeq\C^n$$
be the vector space of the first $n$ higher-order conservation laws after the always smooth $\varphi_0.$
We treat the following two cases separately depending on the presence of local residues for the conservation laws.

\subsubsection{Case: No residues for $\varphi_j$}\label{sec:casenoresidues}  
Then we have a natural map
\be\label{eq:pullbackmap}
\ol{H}^{(1,0)}_n\to H^1(\Sigmah,\C)\to H^{(1,0)}(\Sigmah),
\ee
where the first map is pull-back and the second map is just projection. 
Dualizing this one obtains the linear map
\be\label{eq:pullbackdualmap}
\tn{u}_n:\left(H^{(1,0)}(\Sigmah)\right)^*\to \left(\ol{H}^{(1,0)}_n\right)^*.
\ee
\begin{defn}\label{defn:levelnoresidue}
The \tb{level} of a compact CMC surface is the least integer $n$ such that  the linear map $\tn{u}_n$ in \eqref{eq:pullbackdualmap} is injective.
\end{defn}
In the discussion below we assume that the CMC surface has level $\leq n$ and $\tn{u}_n$ is injective.
 
\two
Define the period lattice
\be\label{defn:Gammaperiod}
\Gamma_n:=\Z\langle \hat{a}_k, \hat{b}_k \rangle \subset \left(\ol{H}^{(1,0)}_n\right)^*.
\ee
Here  the homology basis $\{ \hat{a}_k, \hat{b}_k \}_{k=1}^{\hat{g}}$ of $\Sigmah$ are considered as elements in  $\left(\ol{H}^{(1,0)}_n\right)^*$ by the period homomorphism.

Let 
$$[\omega]\in H^{(1,0)}(\Sigmah)$$
denote the cohomology class represented by $\omega=\sqrt{\ff}$. 
Let $[\omega]^{\perp}\subset \left(H^{(1,0)}(\Sigmah)\right)^*$ be the dual hyperplane.
\begin{defn}[No residues]\label{defn:noresiduerational}
Let $\Sigma$ be a compact CMC surface. Let $\Sigmah\to\Sigma$ be the double cover defined by the square root of Hopf differential $\omega=\sqrt{\ff}$.  Let $\Lambda\subset \left(H^{(1,0)}(\Sigmah)\right)^*$ be the period lattice so that the Jacobian of $\Sigmah$ is given by
$$\tn{Jac}(\Sigmah)=\left(H^{(1,0)}(\Sigmah)\right)^*/\Lambda\simeq T^{\hat{g}}.$$

Suppose the conservation laws  $\varphi_j \in \ol{H}^{(1,0)}_n$ have no residues 
at the fundamental divisor $\mcuh.$
The CMC surface $\Sigma$ is \tb{rational type} of \tb{order $n$} when 
\begin{enumerate}[\qquad a)]
\item the class $[\omega]\in H^{(1,0)}(\Sigmah)$ is rational so that 
$[\omega]^{\perp}/\Lambda\subset \tn{Jac}(\Sigmah)$ is a compact torus,
\item under the linear map \eqref{eq:pullbackdualmap}, 
$$\tn{u}_n(\Lambda)\subset\Gamma_n.
$$
\end{enumerate}

It follows that for a rational type CMC surface $\Sigma$ there exists the associated linear \tb{Abel-Jacobi map}  
\be\label{eq:AbelJacobimapnor}
\tn{u}_n:\left(T_{[\omega]^{\perp}},\tn{Jac}(\Sigmah)\right) \to 
\left(\tn{u}_n(T_{[\omega]^{\perp}}), (\ol{H}^{(1,0)}_n)^*/\Gamma_n\right).
\ee
Here $T_{[\omega]^{\perp}}= [\omega]^{\perp}/\Lambda.$
\end{defn}
 
Note the following commutative diagram for $n'\leq n$.

\be
\xymatrix{ \tn{Jac}(\Sigmah) \ar[r]^{\tn{u}_n} \ar[rd]^{\tn{u}_{n'}}  & \quad \left(\ol{H}^{(1,0)}_{n}\right)^*/\Gamma_{n}  \ar[d]^{\pi} \\
 &  \quad\left(\ol{H}^{(1,0)}_{n'}\right)^*/\Gamma_{n'} }
\ee
Here $\pi:\left(\ol{H}^{(1,0)}_n\right)^*/\Gamma_n\to\left(\ol{H}^{(1,0)}_{n'}\right)^*/\Gamma_{n'}$ is the projection induced by restriction map. It follows that if a CMC surface is rational type of order $n$, then it is rational type of order $n'$ for any $n'\leq n$.
 
\subsubsection{Case: Residues for $\varphi_j$}\label{sec:caseresidues}
Similarly as before we have a natural map
\be\label{eq:pullbackmap}
\ol{H}^{(1,0)}_n\to H^1(\Sigmah\setminus\mcuh,\C)\to H^{(1,0)}(\Sigmah\setminus\mcuh),
\ee
where the first map is  pull-back and the second map is just projection.\footnotemark
\footnotetext{The cohomology $H^{(1,0)}(\Sigmah\setminus\mcuh)$ of the punctured Riemann surface $\Sigmah\setminus\mcuh$ is generated by the set of meromorphic 1-forms with at most simple poles on $\mcuh$.}
Dualizing this one obtains the linear map
\be\label{eq:pullbackdualmapr}
\tn{u}_n:\left(H^{(1,0)}(\Sigmah\setminus\mcuh)\right)^*\to \left(\ol{H}^{(1,0)}_n\right)^*.
\ee
\begin{defn}\label{defn:levelnoresidue}
The \tb{level} of a compact CMC surface is the least integer $n$ such that  the linear map $\tn{u}_n$ in \eqref{eq:pullbackdualmapr} is injective.
\end{defn}
In the discussion below we assume that the CMC surface has level $\leq n$ and $\tn{u}_n$ is injective.
 
\two
Define the period lattice
\be\label{defn:Gammaperiod}
\Gamma_n:=\Z\langle \hat{p}_j, \hat{a}_k, \hat{b}_k \rangle\subset \left(\ol{H}^{(1,0)}_n\right)^*.
\ee
Here each zero point $\hat{p}_j$ of $\omega$ is considered as an element in $\left(\ol{H}^{(1,0)}_n\right)^*$ by residue homomorphism. The homology basis $\{ \hat{a}_k, \hat{b}_k \}_{k=1}^{\hat{g}}$ of $\Sigmah$ are considered as elements in $\left(\ol{H}^{(1,0)}_n\right)^*$ by the period homomorphism which is defined modulo the integral lattice generated by residues (and hence the lattice $\Gamma_n$  is well defined).

\begin{defn}[With residues]
Let $\Sigma$ be a compact CMC surface. Let $\Sigmah\to\Sigma$ be the double cover defined by the square root of the Hopf differential $\ff$.
Let $\Lambda\subset \left(H^{(1,0)}(\Sigmah\setminus\mcuh)\right)^*$ be the residue/period lattice so that
the Jacobian is given by
$$\tn{Jac}(\Sigmah\setminus\mcuh)=\left(H^{(1,0)}(\Sigmah\setminus\mcuh)\right)^*/\Lambda.
$$

Suppose the conservation laws  $\varphi_j \in \ol{H}^{(1,0)}_n$ have  
residues at the fundamental divisor  $\mcuh$.
The CMC surface $\Sigma$ is \tb{rational type} of \tb{order $n$} when 
\begin{enumerate}[\qquad a)]
\item the class $[\omega]\in H^{(1,0)}(\Sigmah\setminus\mcuh)$ is rational so that 
$[\omega]^{\perp}/\Lambda\subset \tn{Jac}(\Sigmah\setminus\mcuh)$ is a compact torus,
\item under the linear map \eqref{eq:pullbackdualmap},
$$\tn{u}_n(\Lambda)\subset\Gamma_n.
$$
\end{enumerate}

It follows that for a rational type CMC surface $\Sigma$ there exists the associated linear \tb{Abel-Jacobi map}  
\be\label{eq:AbelJacobimapr}
\tn{u}_n:\left(T_{[\omega]^{\perp}},\tn{Jac}(\Sigmah\setminus\mcuh)\right) \to 
\left(\tn{u}_n(T_{[\omega]^{\perp}}), (\ol{H}^{(1,0)}_n)^*/\Gamma_n\right).
\ee
Here $T_{[\omega]^{\perp}}= [\omega]^{\perp}/\Lambda.$
\end{defn}
Note that similarly as in the zero residue case if a CMC surface is rational type of order $n$, then it is rational type of order $n'\leq n$.
\section{Polarization}\label{sec:Polarization}
In this section we introduce a polarization on the space of truncated conservation laws $\ol{H}^{(1,0)}_n$ for a given compact CMC surface. We continue to assume that the CMC surface  has level $\leq n$
so that the Abel-Jacobi map is injective.

\two
Let $\Sigma$ be a compact CMC surface of genus $\geq 1$. For each integer $\ell\geq 2n$, 
consider the bilinear form $Q_{\ell}$ on $\ol{H}^{(1,0)}_n$ defined as follows: 
for $\phi_1,\phi_2\in \ol{H}^{(1,0)}_n$ let 
\be\label{eq:Qform}
Q_\ell(\phi_1,\phi_2):=\int_{\Sigmah}(h_2\hb_2)^{\ell}\phi_1\w *\ol{\phi}_2.
\ee
The weight factor $(h_2\hb_2)^{\ell}$ is introduced to guarantee convergence of the integral at the umbilics, Cor.~\ref{cor:varphintwisted}.
It is clear that $Q_\ell$ is a positive definite Hermitian form.
\begin{defn}
Let $\Sigma$ be a compact CMC surface of genus $\geq 1.$ Let $\Sigmah\to\Sigma$ be its double cover.
Let $\ol{H}^{(1,0)}_n=\C\langle\varphi_1,\varphi_2, \, ... \, \varphi_n\rangle$ be the vector space of truncated higher-order conservation laws.
The \tb{polarization} of $\ol{H}^{(1,0)}_n$ of degree $\ell\geq 2n$ associated with $\Sigma$ is the bilinear form $Q_\ell$ \eqref{eq:Qform}.
\end{defn}

Let 
$$V_{\Sigma} =\left(H^{(1,0)}(\Sigmah\setminus\mcuh)\right)^*.$$
\begin{defn}\label{defn:polarization}
Let $\Sigma$ be a compact CMC surface of genus $\geq 1$.  Let $\Sigmah\to\Sigma$ be its double cover.
The \tb{polarization} of $V_{\Sigma}$ of order  $n$ and  degree  $\ell\geq 2n$ 
is the bilinear form $\tn{u}_n^*Q_\ell$ on $V_{\Sigma}$
induced by the linear Abel-Jacobi map (either \eqref{eq:pullbackdualmap}, or \eqref{eq:pullbackdualmapr})
$$ \tn{u}_n: V_{\Sigma} \hook  (\ol{H}^{(1,0)}_n)^*. 
$$
\end{defn}
Note by the assumption on the level of $\Sigma$ that the induced bilinear form on $V_{\Sigma}$ is nondegenerate.

The vector space $V_{\Sigma}$ is canonically associated with the pair $(\Sigmah, \omega)$, 
which is a holomorphic object. The polarization gives a globally defined structural invariant 
for   compact CMC surfaces.

\section{Action of symmetries on conservation laws}\label{sec:families}
In this section we examine how the symmetry acts on the space of conservation laws. It turns out the action is almost trivial, except the action by the classical symmetry on the classical conservation laws. This on the one hand  indicates the rigid property of the conservation laws and their periods, and it also provides an indirect justification to consider deformation by non-canonical Jacobi fields for the study of periods of conservation laws; 
the canonical Jacobi fields keep the periods and do not incur any variation of periods for the higher-order conservation laws.

\two
Consider first the classical symmetries and conservation laws. Let $A, B$ be   classical Jacobi fields. Let $V_A, V_B$ be the infinite prolongation of corresponding classical symmetries, and let $\Phi_A, \Phi_B$ be the corresponding classical conservation laws respectively. 
From Eqs.~\eqref{eq:classicalCartanform}, \eqref{eq:Jacobicvlaw}, we have the equation
\[ \Phi_B=V_B\lhk\Upsilon_0\]
for the 0-th order Poincar\'e-Cartan 3-form $\Upsilon_0$.
Taking the Lie derivative one finds
\begin{align}
\mcl_{V_A}\Phi_B&=\mcl_{V_A} (V_B\lhk\Upsilon_0)   \n\\
&=V_B\lhk (\mcl_{V_A}\Upsilon_0)+[V_A,V_B]\lhk\Upsilon_0.\n
\end{align}
Since $\ed\Upsilon_0=0$ we have
$$\mcl_{V_A}\Upsilon_0=\ed (V_A\lhk\Upsilon_0)=\ed\Phi_A=0.$$
We consequently get
\be\label{eq:VAonPhiB}
\mcl_{V_A}\Phi_B=[V_A,V_B]\lhk\Upsilon_0.
\ee

Introduce the Poisson bracket $\{\,\cdot \,\}$ on the space of classical Jacobi fields by the equation
\be\label{eq:classicalPoisson}
 V_{\{A,B\}}:=[V_A,V_B].
\ee
Then the action \eqref{eq:VAonPhiB} is written as
\be\label{eq:PhiAB}
\mcl_{V_A}\Phi_B=\Phi_{\{A,B\}}.
\ee
We wish to examine the analogous relation for the higher-order symmetries and conservation laws.

\two
To begin it will be convenient to define a Poisson bracket on the space of Jacobi fields. Recall  $\mathfrak{J}^{(\infty)}$ is the space of Jacobi fields. For $P,Q\in\mathfrak{J}^{(\infty)}$ consider a 1-form
$$\varphi_{P,Q}:=-P\JAI\ed Q +Q\JAI \ed P.
$$
Since $\mce(P)=\mce(Q)=0$ it follows that
$$\ed \varphi_{P,Q}\equiv 0\mod\iinfh,$$
and $\varphi_{P,Q}$ represents a conservation law.
\begin{defn}\label{defn:Poissonbracket}
Let $\mathfrak{J}^{(\infty)}$ be the space of Jacobi fields. Define the \tb{Poisson bracket} 
$$  \{\; , \, \}:  \mathfrak{J}^{(\infty)}\times \mathfrak{J}^{(\infty)}\to \mathfrak{J}^{(\infty)}$$
as follows. For $P, Q\in\mathfrak{J}^{(\infty)}$, define $\{P,Q\}\in\mathfrak{J}^{(\infty)}$ by
$$\Phi_{\{P,Q\}}:=[\ed\varphi_{P,Q}]\in\Cv{(\infty)}.$$
Here $\Phi_{u}\in\Cv{(\infty)}$ is the differentiated conservation law (closed reduced 2-form) generated by the Jacobi field $u.$
\end{defn}
We will see in the below that when restricted to the subspace of classical Jacobi fields this Poisson bracket agrees with the one defined in \eqref{eq:classicalPoisson}.

The Poisson bracket is nonetheless almost trivial.
\begin{lem}\label{lem:Poissoncommute}
Let $P,Q$ be higher-order Jacobi fields. Then $\{P,Q\}=0$.
\end{lem}
\begin{proof}
Consider the case $P=a^{2m+1},Q=a^{2n+1}, m<n.$ Then the 1-form $\varphi_{P,Q}$ is weighted homogeneous of even spectral weight $2m+2n$. Hence by Lem.~\ref{lem:hrepJacobi} and   Thm.~\ref{thm:higherNoether}, $[\ed\varphi_{P,Q}]$ is trivial. 
 
The case $P=a^{2m+1},Q=\ol{a}^{2n+1}$ follows by a similar argument.
\end{proof}
We will see in the below that $\{P,Q\}=0$  if one of the Jacobi fields, $Q$, is higher-order.

\two
With this preparation we now turn to the action of symmetries on the conservation laws.  
The following lemma shows that the action of symmetry by Lie derivative on the un-differentiated conservation laws is well defined independent of a particular representative.
\begin{lem}
Suppose that $[\p]=[\pt] \in \mcc^{(\infty)}$.  Let $V\in H^0(T\xinfh)$ be a symmetry.  Then $[\mcl_V \p]=[\mcl_V \pt] \in \mcc^{(\infty)}$.
\end{lem}
\begin{proof}
By definition of conservation law, $[\p]=[\pt] $ is equivalent to the existence of a scalar function $f$, 
and a 1-form $\Theta \in \iinfh$ such that $\pt=\p+\ed f + \Theta$.  Then we calculate that
\[
[\mcl_V \pt] =[ \mcl_V \p +  \mcl_V \ed f+  \mcl_V \Theta]  =[\mcl_V \p +  \ed \left(V( f) \right)+  \mcl_V \Theta]=[\mcl_V \p]
\]
for $\mcl_V \iinfh \subset \iinfh$.
\end{proof}

This justifies the following definition.
\begin{defn}
The action of a symmetry vector field $V\in H^0(T\xinfh)$ on a conservation law $[\p] \in \mcc^{(\infty)}$ is defined by the Lie derivative $[\mcl_V \p]$.
\end{defn}

 
We wish to extend the identity Eq.~\eqref{eq:PhiAB} to the entire space of Jacobi fields $\mathfrak{J}^{(\infty)}.$
To this end we make use of the differentiated conservation laws in normal form.  Recall that associated to a Jacobi field $Q$ we have the corresponding differentiated conservation law  $\Phi_Q$
which is a closed reduced 2-form
\begin{equation}
\Phi_Q \equiv \theta_0 \w (-\JAI \ed Q) + Q \Psi \mod F^2\Omega^2.
\end{equation}
\begin{prop}\label{prop:symmetryaction}
Let $P, Q\in\mathfrak{J}^{(\infty)}$. Suppose 
$\ed  \p_Q \equiv \Phi_Q  \mod F^2\Omega^2$ and $\p_Q$ is a \tn{mod} $F^2\Omega^2$-representative for $\Phi_Q.$
Then $$\mcl_{V_P} \p_Q \equiv \p_{\{P,Q\}}  \mod \iinfh.$$
 It follows that 
\be\label{eq:PhiPQ}
[\mcl_{V_P}\Phi_Q]\equiv \Phi_{\{P,Q\}} \in\Cv{(\infty)},
\ee
and the identity  \eqref{eq:PhiAB} extends to all Jacobi fields.
\end{prop}
\begin{proof}
Let $V_P=V_{\xi}\delx+ V_{\xib}\delxb+ P E_0 + V_j E_j+V_{\ol{j}}E_{\ol{j}}$ (when $P$ is vertical we set $V_{\xi}=V_{\xib}=0.$  These terms do not affect the computation below).
\begin{align*}
\mcl_{V_P} \p_Q &= V_P \lhk \ed \p_Q + \ed (V_P \lhk \p_Q)  \\
& \equiv V_P \lhk \Phi_Q + \ed (V_P \lhk \p_Q) \mod (F^1\Omega^1=\iinfh)  \\
& \equiv  -P \JAI \ed Q -\frac{\im}{2} Q (V_1 \xi - V_{\ol{1}}\xib )+ \ed (V_P \lhk \p_Q) \mod \iinfh\\
& \equiv -P\JAI \ed Q+ Q\JAI\ed P+ \ed (V_P \lhk \p_Q) \mod \iinfh\\
& \equiv \p_{\{P,Q\}} + \ed (V_P \lhk \p_Q) \mod \iinfh.
\end{align*}
\end{proof}
This immediately implies the following extension of Lemma \ref{lem:Poissoncommute}.
\begin{cor}\label{cor:Poissoncommute}
Let $P,Q\in \mathfrak{J}^{(\infty)}.$  Suppose $Q$ is a higher-order Jacobi field. 
Then $\{P,Q\}=0$.
\end{cor}
\begin{proof}
It suffices to check the case $P$ is classical. It is geometrically clear that a classical symmetry $V_P$ leaves invariant $\omega, h_2\hb_2, z_j$.
It follows  that for all $n\geq 0$ 
$$\mcl_{V_P}\varphi_n=0\mod \tn{nothing}.$$
\end{proof}
\begin{rem}
A similar argument shows that the Lie bracket
$[ V_P, V_Q]$ of the symmetry vector fields vanish
whenever $Q$ is a higher-order Jacobi field.
\end{rem}
Thus the Poisson bracket for Jacobi fields we introduced is almost trivial. It reduces to the bracket on the classical Jacobi fields.

\begin{cor}\label{cor:exactaction}
The action of symmetry on the conservation laws is nontrivial only when a classical symmetry acts on a classical conservation law.
\end{cor}
 

The analysis in this section has the following geometric implication. 
\begin{thm}\label{thm:invariantperiods}
Let $\x_t:\Sigma \hook \xinf$ be a smoothly varying $1$-parameter family of smooth 
CMC surfaces preserving the umbilic divisor $\mcu.$ 
Let $\xh_t:\Sigmah \hook \xinfh$ be its double cover.
Suppose the variational vector field $\dd{\xh_t}{t}=V_t$ is the symmetry vector field associated to a Jacobi field for each time $t$.  Then for any higher-order conservation law $[\p] \in \mcc^{(\infty)}$
its periods   (residues)  
\begin{equation}
\int_{a}\xh^*_t[\p] \hspace{1cm}{\rm for}\; \;\; [a] \in H_1(\Sigmah \setminus\mcuh, \Z)
\end{equation}
are independent of $t$ (here $\mcuh$ is the fundamental divisor).
\end{thm}
\begin{proof}
We compute using Cor.~\ref{cor:exactaction}:
\begin{equation}
\dd{}{t}_{|_{t=0}}\int_{a}\xh^*_t[\p] = \int_{a} \xh^*(\mcl_{V_0}  \p)=0. \n
\end{equation}
\end{proof}
Thus for any smooth family of CMC surfaces with the fixed umbilic divisor for which the infinitesimal deformations are all given by Jacobi fields, the periods of higher-order conservation laws are invariant of the family.  It follows that in the umbilic-free case
the  sequence of linear Abel-Jacobi maps $\tn{u}_n$ on the Jacobian $\tn{Jac}(\Sigmah)$
from Eqs.~\eqref{eq:AbelJacobimapnor}, \eqref{eq:AbelJacobimapr}  
remain invariant under smooth deformation by Jacobi fields.
\begin{rem}
Let $\mcm$ be the formal moduli space of integral surfaces of $(\xinfh, \iinfh)$.
The space of symmetry vector fields $\mathfrak{S}$ can be considered as the tangent space
$$\mathfrak{S}=H^0(T\mcm).$$
From the analysis above the higher-order symmetry vector fields 
define a sequence of commuting flows on $\mcm.$
The theorem shows that the residues/periods are invariant under these flows.
\end{rem}

\section{Variation of periods}\label{sec:VP}
Recall that the Jacobi equation, and hence its solutions Jacobi fields, on a CMC surface describes the conformal CMC deformation with the prescribed Hopf differential.\footnotemark\footnotetext{This can be verified by a moving frame computation.}

Consider the case of a CMC torus. The Hopf differential is nowhere zero, and the infinite sequence of canonical Jacobi fields $a^{2j+1}$'s are all smooth. A geometric conclusion can be drawn from this that, since the kernel of an elliptic differential operator on a compact manifold is finite dimensional, a CMC torus is necessarily of  linear  finite-type. An alternative way to view this phenomenon would be that;
\emph{a CMC torus generally comes in a family  
and  admits conformal CMC deformations with the prescribed Hopf differential.} 
This qualitative statement could be made rigorous analytically in terms of the well developed spectral theory for harmonic tori, \cite{Hitchin1990}.

Consider the case of a CMC surface of genus $\geq 2.$ The canonical sequence of Jacobi fields and conservation laws are singular at the umbilics. In order to incorporate the singularities and obtain smooth objects, one is  led to consider the notion of Jacobi fields and conservation laws twisted by powers of canonical line bundle, Sec.~\ref{sec:umbilics}. In analogy with the torus case an alternative way to view this phenomenon would be that;
\emph{a CMC surface of genus $\geq 2$ is generally rigid under conformal CMC deformations with the prescribed Hopf differential.} 

It is more likely in this case that a pair of Riemann surface (of genus $\geq 2$) equipped with a holomorphic quadratic differential uniquely determines the compatible conformal metric. From this we argue that for the deformation theory of high genus CMC surfaces one should consider the deformation of underlying Riemann surface equipped with a Hopf differential, and the corresponding variation of Hodge-like structures (VHS) naturally emerges.

\two
In this section we initiate the study of VHS for CMC surfaces by considering the deformation of CMC surfaces that corresponds to scaling the Hopf differential. The objective is to determine the Picard-Fuchs equation for the periods of conservation laws, which is more or less computable  in this case by local analysis.

\subsection{Scaling Hopf differential}\label{sec:scalingHopf}
Let $\Sigma$ be a Riemann surface equipped with a pair $(\xi\circ\xib, \ff)$ of a conformal Riemannian metric $g=\xi\circ\xib$ and a Hopf differential $\ff\in H^0(\Sigma, K^2)$ that satisfies the compatibility equation
$$ R_{g} =\gamma^2 -\vert\ff\vert^2_{g}.$$
Consider a smooth conformal deformation of such pair on $\Sigma$ given by
\begin{align}
g(t)&=\xi(t)\circ \xib(t)=e^{\tilde{u}(t)}\xi\circ e^{\tilde{u}(t)}\xib= e^{2\tilde{u}(t)}\xi\circ\xib, \n\\
\ff(t)&=e^{-2t} \ff =e^{-2t} h_2\xi^2,\n
\end{align}
which scales the Hopf differential by the constant real parameter $e^{-2t}, t \in \R$. Here $\tilde{u}(t)$ is a $t$-dependent real valued scalar function which represents the corresponding conformal change of metric. Note that
$$h_2(t)=e^{-2t-2\tilde{u}(t)} h_2.
$$
The compatibility equation (Gau\ss\, equation) for each fixed $t$ is given by
\be\label{eq:scaleGauss}
\tilde{u}_{\xi\xib}+\frac{1}{4}(\gamma^2 e^{2\tilde{u}}-e^{4t-2\tilde{u}}h_2\hb_2)
=\frac{1}{4}(\gamma^2-h_2\hb_2).
\ee

Set the initial condition
$$ \tilde{u}(0)=0,$$
and consider the expansion
$$ \tilde{u}(t)=u t +\mco(t^2)$$
for a real valued scalar function $u$ on $\Sigma$. For simplicity we shall adopt the notation such as
$$ \dot{\tilde{u}} = \frac{\ed}{\ed t}\tilde{u}(t)\vert_{t=0} =u.$$
The upper dot  $(\;\dot{}\;)$  would mean $t$-derivative at $t=0$. Note that  $(\;\dot{}\;)$ is a complex linear operation.

Differentiating Eq.~\eqref{eq:scaleGauss} one gets the deformed compatibility equation
\be\label{eq:inhomJacobi}
\mce(u)=u_{\xi\xib}+\frac{1}{2}(\gamma^2+h_2\hb_2)u =h_2\hb_2, 
\ee
which is an \tb{inhomogeneous Jacobi equation}. We wish to determine the corresponding first order deformation of the higher-order derivatives $h_j$'s, and ultimately of the sequence of higher-order conservation laws $[\varphi_j],  j=0, 1, \, ... \, .$
\begin{exam}\label{exam:vrho}
Recall the integrable extension $Z\to\xinf$ for the spectral symmetry, Sec.~\ref{sec:extension}.
From Eq.~\eqref{eq:vrho} the real valued scalar function 
$$ \im v_{\rho}
$$
satisfies $\mce(\im v_{\rho})=h_2\hb_2$, and it is an \tb{inhomogeneous Jacobi field}.
\end{exam}

\two
Let us start with the following observations.
\begin{align}
\dot{\xi}&= u \xi, \n   \\
\dot{\rho}&=\im ( u_{\xi}\xi - u_{\xib}\xib), \n  \\
\dot{h}_2&= -(2u+2)h_2,   \n  \\
\dot{\tn{$(h_2 \bar{h}_2)$}}&= -(4u+4)h_2\hb_2,   \n  \\
\dot{(h_2^{\frac{1}{2}}\xi)}&=-h_2^{\frac{1}{2}}\xi. \n
\end{align}
From this we get,
\begin{lem}
\begin{align}\label{eq:hzdot}
\dot{h}_j &=-(j u+2)h_j - \delta_j,   \\
\dot{z}_j &=(j -2)z_j  - h_2^{-\frac{j}{2}}\delta_j,  \n  \\
&=(j -2)z_j  -\epsilon_j.\n
\end{align}
where $\delta_2=0$ and for $j\geq 3$,
\be\begin{array} {rll}
\delta_j&=\sum_{k=2}^{j-1} \tn{C}_j^k u_{j-k} h_k,    &      \n      \\
C^2_{j}&=4, \quad \tn{set} \; \;C^{j}_{j}=2j, \;\;&j\geq 3, \n \\  
\tn{C}_j^k&=\tn{C}_{j-1}^{k-1}+\tn{C}_{j-1}^k, &k\leq j-1, \,\;j\geq 4.\n
\end{array}\ee
Here we denote the derivatives of $u$ by  ($u$ is real)
$$ u_k =\delx^k u, \quad u_{\bar{k}} =\delxb^k u =\ol{u_k}.
$$
\end{lem}
\begin{proof}
By definition
$$\ed h_j(t) +\im j h_j(t)\rho(t)\equiv h_{j+1}(t)\xi(t) \mod \xib.$$
Since the $t$-derivative $(\;\dot{}\;)$ commutes with the exterior derivative $\ed$, one computes
$$ \delx \dot{h}_j \equiv ( \dot{h}_{j+1}+h_{j+1} u + jh_j u_1)\xi\mod \xib.$$
and we have the inductive formula for $j\geq 2$,
$$ \dot{h}_{j+1} =(\dot{h}_j)_{\xi} - u h_{j+1}- j u_1 h_j.$$
This is equivalent to
$$\delta_{j+1}=(\delta_j)_{\xi}+ 2 j u_{1} h_j.$$
The rest follows from this.
\end{proof}
\begin{rem}
Set 
$$ w_k = h_2^{-\frac{k}{2}} u_k, \quad w_{\bar{k}} = \hb_2^{-\frac{k}{2}} u_{\bar{k}},
$$
and assign the new weight with respect to the $t$-derivative
$$ \tn{weight} (z_k)=\tn{weight} (z_{\bar{k}})=\tn{weight} (w_k)=\tn{weight} (w_{\bar{k}})=k-2.
$$
The lemma shows  that the $t$-derivative operation $(\;\dot{}\;)$ preserves 
this weight on the variables $z_j, w_j, z_{\bar{j}}, w_{\bar{j}}.$
\end{rem}
  
The $t$-derivative formulae \eqref{eq:hzdot} can be extended to the formal Killing field coefficients $a^{2j+1}, b^{2j+2}, c^{2j+2}$ by taking $(\;\dot{}\;)$ of the recursion formulae \eqref{eq:bc2n+2}, \eqref{eq:a2n+3} from Sec.~\ref{sec:inductiveformula}.  This is a direct computation and we shall not present the details. Let us instead content ourselves by recording the first few terms of the sequence in the table. Note that  each element of the sequence $\{\da^{2n+1}, \db^{2n+2}, \dc^{2n+2}\}_{n=0}^{\infty}$ is (up to scaling) weighted homogeneous in the variables $z_j, w_j.$
\begin{figure}
\be\label{eq:dotrecursioncoeff}\begin{array}{ccl}
 &\vline \\
\hline
\da^1 &\vline&  0\\
\hline
\db^2 &\vline&  b^2- u\im\gamma h_2^{-\frac{1}{2}}  \\
\hline
\dc^2 &\vline&  -c^2-u \im h_2^{\frac{1}{2}} \\
\hline
\da^3 &\vline& a^3 -4 w_1 \\
\hline
\db^4 &\vline& 3 b^4  +\frac{\im}{8} h_2^{-\frac{1}{2}}\left( 16 w_2+u_0(5 z_3^2-4  z_4)\right)\\
\hline
\dc^4 &\vline& c^4+\frac{\im}{8\gamma} h_2^{\frac{1}{2}}\left(16w_2-16w_1z_3+u_0(-7z_3^2+4z_4)\right)\\
\hline
\da^5 &\vline& 3 a^5 -\frac{1}{2\gamma}\left(-8w_3+12w_2z_3+w_1(-5 z_3^2+4  z_4)\right )
\end{array}
\ee
\caption{Deformation of formal Killing coefficients}
\end{figure}

\subsection{Picard-Fuchs equation}\label{sec:PicardFuchs}
Suppose  $\Sigma$ be a compact orientable surface of genus $g\geq 2$. The moduli space of complex structures on $\Sigma$ has dimension $(3g-3)_{\C}$, and by Riemann-Roch $\tn{dim}_{\C}H^0(\Sigma, K^2)=3g-3$ also. Since the compatibility equation for the conformal metric is elliptic, and the fact that the canonical sequence of Jacobi fields are singular at the umbilics, one suspects that the total moduli space of admissible triples for CMC surfaces on $\Sigma$ is finite dimensional. The deformation in the previous section which scales the Hopf differential by a real parameter can be considered as a particular subset of this total moduli and when combined with the spectral symmetry, which scales the Hopf differential by unit complex numbers, they form a $\C^*$-family of deformations of the compatible triple data for CMC surfaces. 

Suppose more concretely that there exists a submersion
$$ \widetilde{\Sigma} \to \mcb =\C^*$$
each of whose fiber is a copy of $\Sigma$ equipped with a compatible triple parametrized by the scaling factor $\C^*$. Then from the previous analysis the higher-order conservation laws define on each fiber an infinite sequence of elements 
$[\varphi_j] \in H^1(\Sigmah\setminus\mcuh,\C)$. 
 This is a somewhat familiar picture as in the theory of VHS on complex projective varieties. Specifically the analogue of Picard-Fuchs equation would describe in this case the variation under the $\C^*$-family of the linear Abel-Jacobi map 
$$ \tn{u}_n: H^{(1,0)}(\Sigmah\setminus\mcuh)^* \to (\ol{H}^{(1,0)}_n)^*.$$
Equivalently we have a map
$$\mcb\simeq \C^* \to H^{(1,0)}(\Sigmah\setminus\mcuh)\otimes  (\ol{H}^{(1,0)}_n)^*.
$$

On the other hand compared to VHS the space of higher-order conservation laws is infinite dimensional, and they are  locally computable.

\two
In this section we examine the first order truncated Picard-Fuchs equation for the higher-order conservation laws $[\varphi_j], j=0,1,\, ... \,,$ under the scaling deformation of Hopf differential. 
 
Recall $\varphi_j=c^{2j+2}\xi+b^{2j}\xib.$ Applying $(\;\dot{}\;)$,
\begin{align}\label{eq:dotvarphij}
\dot{\varphi}_j&=\dc^{2j+2}\xi+\db^{2j}\xib +u\varphi_j  \\
&=(2j-1)\varphi_j +\psi_j. \n
\end{align}
We thus denote by $\psi_j$ the variation of $\varphi_j(t)$ at $t=0$ modulo the weighted scaling part $(2j-1)\varphi_j.$ By construction $\psi_j$ is homogeneous linear in the derivatives $w_k$'s of the inhomogeneous Jacobi field $u$.

By construction, again since $(\;\dot{}\;)$ operation commutes with the exterior derivative $\ed$, it is clear that $$\ed \psi_j =0$$
and $\psi_j $ is a (possibly trivial) conservation law. In fact this is true for any first order deformation given by an inhomogeneous Jacobi field $u$.

\begin{prop}
Let $\Sigma$ be a CMC surface. Let $u$ be a $\C$-valued inhomogeneous Jacobi field satisfying \eqref{eq:inhomJacobi}. 
Then the sequence of 1-forms $\psi_j$ in Eq.~\eqref{eq:dotvarphij} representing the variation of conservation laws are closed. Thus an inhomogeneous Jacobi field generates an associated sequence of conservation laws $\psi_j, j=0,1,\, ... \, .$
\end{prop}
\begin{proof}
When $u$ is real, this follows from the first order deformed structure equation for the formal Killing coefficients $\da^{2j+1}, \db^{2j+2}, \dc^{2j+2}.$  
Suppose we augment $u$ by an imaginary Jacobi field. This does not affect the first order deformation of the metric $\xi(t)\circ\xib(t).$ Note also that the coefficients of $\psi_j$ consists of $h_2, z_k, w_k$ and do not involve the conjugate variables $\hb_2, \zb_k, \bar{w}_k.$
\end{proof}

\two
The space of inhomogeneous Jacobi fields is by definition an affine space over the space of Jacobi fields. From the arguments at the beginning of Sec.~\ref{sec:VP}, we find that up to Jacobi fields there should exist essentially one smooth inhomogeneous Jacobi field, namely the non-local function 
$$\im v_{\rho}\mod v_0$$ 
induced from the integrable extension $Z\to\xinf$. 
It is defined on a CMC surface $\Sigma$ up to monodromy by classical Jacobi fields. 

We proceed to analyze the first order deformation of $[\varphi_j(t)]$ 
corresponding to  the inhomogeneous Jacobi field\footnotemark \footnotetext{Let us ignore the monodromy issue for the moment.} 
$$u=\im v_{\rho}+\im \delta v_0.$$
Note that  $\im v_0$ is an imaginary (non-local) Jacobi field. We suspect that the  characteristic cohomology classes of $\psi_1, \psi_2$ can be determined similarly for the case $u=\im v_{\rho}$ too. 
The imaginary Jacobi field $\im \delta v_0$ was added for the sake of computation.
The following analysis is carried out on $\Zh$.

By a direct computation, the first few terms are (recall that $w_k=h_2^{-\frac{k}{2}}u_k$)
\begin{align}\label{eq:psij}
\psi_0&=0, \\
\psi_1&=-2\im\left(  \frac{(-w_2+w_1z_3)}{\gamma}\omega+\frac{\gamma u_0}{(h_2\hb_2)^{\frac{1}{2}}}\omb \right), \n \\
\psi_2&=\frac{\im}{4}\Big[ \frac{1}{\gamma^2} \left(-8\,w_{{4}}+28\,w_{{3}}z_{{3}}-35\,w_{{2}} z_{{3}}^{2}
 +35w_{{1}}z_{{3}}^{3} -40w_{{1}} z_{{3}}z_{{4}}   +12w_{{2}}z_{{4}}+8w_{{1}} z_{{5}}  \right) \omega \n\\
&\qquad+\frac{1}{(h_2\hb_2)^{\frac{1}{2}}} \left(8\,w_{{2}}+5u_{{0}} z_{{3}}^{2} -4u_{{0}} z_{{4}}\right) \omb \Big]. \n  
\end{align}
One finds that these 1-forms are closed modulo the inhomogeneous Jacobi equation \eqref{eq:inhomJacobi} and its differential consequences.

\two
We claim that the characteristic cohomology classes represented by $\psi_1, \psi_2$ are computable as follows. Substituting $u=\im v_{\rho}+\im \delta v_0$, note
\begin{align}
\psi_1&\equiv \ed (-4\gamma h_2^{-\frac{1}{2}} v_{0,\bar{1}}), \n \\
\psi_2&\equiv \ed\left( 4 h_2^{\frac{1}{2}}v_{{0,1}} -8 h_2^{-\frac{1}{2}} v_{{1,0}} -2 v_0 z_3 
+ v_{0,\bar{1}} h_2^{-\frac{1}{2}}\left( -2z_{{4}}+\frac{5}{2}z_{{3}}^{2} \right)  \right) \n\\
&\quad +6\im \varphi_0 \mod \mcjh. \n
\end{align}
Hence their characteristic cohomology classes on $\Zh$ are determined as
\begin{align*}
[\psi_1]&=0,\\
[\psi_2]&=6\im [\varphi_0]\in H^1(\Omega^*(\Zh)/\mcjh,\underline{\ed}).
\end{align*}

\two
As a consequence of this initial analysis it is tempting to conjecture that the variational conservation laws $[\psi_j]$'s are not new on $\Zh$ but they are linear combination of the higher-order conservation laws.
\begin{conj}[Griffiths transversality]\label{conj:Gtransversality}
For the first order deformation given by the inhomogeneous Jacobi field $u=\im v_{\rho}+\im\delta v_0$, let $\psi_j, j=0,1,\, ... \,,$ be the sequence of 1-forms obtained from the variation of the higher-order conservation law  $\varphi_j$  by Eq.~\eqref{eq:dotvarphij}. Then as a characteristic cohomology class on $\Zh$ we have
\be\label{eq:Gtransversality}
[\psi_j] \in \langle [\varphi_0], [\varphi_1],\, ... \, [\varphi_{j-2}] \rangle \subset H^1(\Omega^*(\Zh)/\mcjh,\underline{\ed}).
\ee
\end{conj}
For a general deformation of CMC surfaces, which is defined by a nonlinear differential equation, it would be a difficult problem to determine the variation of periods of the conservation laws. 
This conjecture points at the possibility of finding the explicit formulae for the Picard-Fuchs equation 
for higher-order conservation laws under the deformation by scaling Hopf differential. 
 
\two
The proposed equation \eqref{eq:Gtransversality} is reminiscent of the Griffiths transversality theorem for VHS. Here the relevant filtration is
$$ \ldots \, \subset \ol{H}^{(1,0)}_{n-1}\subset  \ol{H}^{(1,0)}_{n}\subset  \ol{H}^{(1,0)}_{n+1}\subset 
\ldots \; \subset\ol{H}^{(1,0)}_{\infty}\subset\mcc^{(\infty)}.
$$
Eqs.~~\eqref{eq:dotvarphij}, \eqref{eq:Gtransversality} show that the variation of $\ol{H}^{(1,0)}_{n}$, up to scaling factor, is bounded by $\ol{H}^{(1,0)}_{n-2}$,
$$\dot{ \left(  \ol{H}^{(1,0)}_{n} \right) }\subset  \ol{H}^{(1,0)}_{n-2} \mod \;\tn{scaling terms},
$$
and the order difference is bounded uniformly by $2\cdot 2=4$ independent of $n$ (recall 
$\varphi_n\in\Omega^1(\Xh{2n+1})$).

\np
\bibliographystyle{amsplain}
\providecommand{\MR}[1]{}\def\cprime{$'$}
\providecommand{\bysame}{\leavevmode\hbox to3em{\hrulefill}\thinspace}
\providecommand{\MR}{\relax\ifhmode\unskip\space\fi MR }
\providecommand{\MRhref}[2]{%
  \href{http://www.ams.org/mathscinet-getitem?mr=#1}{#2}
}

\providecommand{\MR}[1]{}\def\cprime{$'$}
\providecommand{\bysame}{\leavevmode\hbox to3em{\hrulefill}\thinspace}
\providecommand{\MR}{\relax\ifhmode\unskip\space\fi MR }
\providecommand{\MRhref}[2]{%
  \href{http://www.ams.org/mathscinet-getitem?mr=#1}{#2}
}
\providecommand{\href}[2]{#2}

\end{document}